\documentclass[a4paper,10pt]{amsart}

\usepackage[utf8]{inputenc}
\usepackage[T1]{fontenc}
 \usepackage{amsfonts}
\usepackage{dsfont}
\usepackage{amsthm}
\usepackage{amsmath}
\usepackage{amssymb}
\usepackage{pdfpages}
\usepackage{graphicx}
\usepackage{mathrsfs}
\usepackage[a4paper]{geometry}

\allowdisplaybreaks[1]

\newcommand{\N}{\mathbb{N}}
\newcommand{\Z}{\mathbb{Z}}

\newcommand{\R}{\mathbb{R}}
\newcommand{\C}{\mathbb{C}}

\newcommand{\dd}{\mathrm{d}}
\newcommand{\dhy}{\mathrm{d}}
\newcommand{\bcal}{\mathcal{B}}
\newcommand{\G}{\Gamma}
\newcommand{\g}{\gamma}
\newcommand{\LL}{\Lambda}
\newcommand{\tLL}{\tilde{\Lambda}}
\newcommand{\del}{\mathrm{\delta}}

\newcommand{\xx}{\mathrm{X}}
\newcommand{\mm}{\mathrm{M}}
\newcommand{\oo}{{\bf o}}
\newcommand{\ot}{\mathcal{L}}
\newcommand{\otd}{\mathcal{L}_{\delta}}
\newcommand{\otds}{\mathcal{L}_{\delta}^*}
\newcommand{\dit}{\delta+it}

\newcommand{\T}{\mathrm{T}}

\newcommand{\gol}{\mathfrak{l}}
\newcommand{\cg}{e_{\Gamma}}
\newcommand{\Cg}{E_{\Gamma}}
\newcommand{\ros}{\rho^*}
\newcommand{\rois}{\rho_{\infty}^*}

\newcommand{\re}[2][\mathrm{Re}]{#1 \left(#2\right)}
\newcommand{\im}[2][\mathrm{Im}]{#1 \left(#2\right)}
\newcommand{\un}{\mathds{1}}

\newcommand{\mc}{\mathcal{C}}

\theoremstyle{plain}
\newtheorem{prop}{Proposition}[section]
\newtheorem{defi}[prop]{Definition}
\newtheorem{lem}[prop]{Lemma}
\newtheorem{fait}[prop]{Fact}
\newtheorem{coro}[prop]{Corollary}
\newtheorem{pro}[prop]{Property}
\newtheorem{rem}[prop]{Remark}
\theoremstyle{remark}

\theoremstyle{plain}
\newtheorem*{thmA}{Theorem\ A}
\newtheorem*{propA1}{Proposition\ A.1}
\newtheorem*{propA2}{Proposition\ A.2}

\newtheorem*{thmB}{Theorem\ B}
\newtheorem*{propB1}{Proposition\ B.1}
\newtheorem*{propB2}{Proposition\ B.2}

\newtheorem*{thmC}{Theorem\ C}
\newtheorem*{propC1}{Proposition\ C.1}
\newtheorem*{propC2}{Proposition\ C.2}

\newtheorem*{thmR}{Theorem}

\title{Ergodic properties of some negatively curved manifolds with infinite measure}
\begin{document}

\begin{abstract}
Let $\mm=\xx/\G$ be a geometrically finite negatively curved manifold with fundamental group $\G$ acting on $\xx$ by isometries. The purpose 
of this paper is to study the mixing property of the geodesic flow on $\T^1\mm$, the asymptotic equivalent as $R\longrightarrow+\infty$ of the 
number of closed geodesics on $\mm$ of length less than $R$ and of the orbital 
counting function $\sharp\{\g\in\G\ |\ \dhy(\oo,\g.\oo)\leqslant R\}$.

These properties are well known when the Bowen-Margulis measure on $\T^1\mm$ is finite. We consider here divergent Schottky groups whose Bowen-Margulis measure is infinite and ergodic, and we precise these ergodic properties using a suitable symbolic coding.
\end{abstract}

\maketitle
\setcounter{tocdepth}{1}
\tableofcontents

{\bf Acknowledgements.} This paper comes from my PhD at the University of Nantes between 2013 and 2016 under 
the direction of Marc Peigné and Samuel Tapie. During that time, my research was supported by Centre Henri Lebesgue, in
programm ``Investissements d'avenir'' - ANR-11- LABX-0020-01.
I want to thank Sébastien Gouëzel for his help and all his explanations during the writing of this paper.

\section{Introduction}

\subsection{Background and previous results}
Let $\xx$ be a connected, simply connected and complete riemannian manifold with pinched negative sectional curvature. Denote by $\dhy$ the distance on $\xx$ induced by the riemannian structure of $\xx$ and by $\G$ a discrete 
group of isometries of $\left(\xx,\dhy\right)$, acting properly discontinuously without fixed point and let 
$\mm=\xx/\G$. Fix $\oo\in\xx$. The study of quantities like 
the orbital function
$$\mathrm{N}_\G(\oo,R):=\sharp\{\g\in\G\ |\ \dhy(\oo,\g.\oo)\leqslant R\}$$
\noindent
is strongly related to the one of the dynamic of the geodesic flow $\left(g_t\right)_{t\in\R}$ on the unit tangent bundle
$\T^1\mm$ of the quotient manifold. Let us first define precisely this flow: each couple
$\left({\bf p},{\bf v}\right)\in\T^1\mm$ determines a unique geodesic $\left(\boldsymbol{\g}(t)\right)_{t\in\R}$ satisfying
$\left(\boldsymbol{\g}(0),\boldsymbol{\g}'(0)\right)=\left({\bf p},{\bf v}\right)$ and for any $t\in\R$, the action of $g_t$ is given by
$g_t.\left({\bf p},{\bf v}\right)=\left(\boldsymbol{\g}(t),\boldsymbol{\g}'(t)\right)$. It is known (see \cite{OP}) that the topological entropy of the geodesic flow is given by the rate of exponential growth $\delta_\G$ of the orbital function, that is
$$\delta_\G:=\limsup\limits_{R\longrightarrow+\infty}\dfrac{\ln\left(\mathrm{N}_\G(\oo,R)\right)}{R}.$$
\noindent
This last quantity is also the critical exponent of the Poincaré series $\mathcal{P}_\G$ of the group $\G$ defined as follows: for any $s>0$
$$\mathcal{P}_{\G}(s):=\sum\limits_{\g\in\G}e^{-s\dhy(\oo,\g.\oo)}.$$ 
\noindent
S. J. Patterson (in \cite{Pat2}) and D. Sullivan (in \cite{Sul}) used these series to construct a family of measures 
$\left(\sigma_{{\bf x}}\right)_{{\bf x}\in\xx}$, the so-called
Patterson-Sullivan measures. More precisely, each measure $\sigma_{{\bf x}}$ is fully-supported by the limit set $\LL_\G\subset\partial\xx$, which is defined as the set of all accumulation points of one(all) $\G$-orbit(s) in the visual boundary $\partial\xx$ of $\xx$. This set also is the smallest non-empty $\G$-invariant closed subset of $\xx\cup\partial\xx$. It is the closure in the boundary of the set of fixed points of $\G^*:=\G\setminus\{\mathrm{Id}\}$. A group $\G$ is said to be elementary if its limit set is a finite set. S.J. Patterson and D. Sullivan described a process to associate to this family a measure $m_\G$ defined on $\T^1\mm$, which is invariant under the action of the geodesic flow. When the group $\G$ is \emph{divergent}, \emph{i.e.} $\mathcal{P}_\G(\delta_\G)=+\infty$ (otherwise $\G$ is said to be \emph{convergent}), the family 
$\left(\sigma_{\bf x}\right)_{{\bf x}\in\xx}$ is unique up to a normalization, hence $m_\G$ is also unique; this is the case when $m_\G$ has finite mass. We will focus in this paper on the case of divergent groups, which allows 
us to speak about ``the'' Bowen-Margulis measure even when it has infinite mass. Nevertheless, in this introduction, the assumption ``$\mm=\xx/\G$ has infinite Bowen-Margulis measure`` should be in general understood as the fact that 
\emph{any} invariant measure obtained from a Patterson-Sullivan density $\left(\sigma_{\bf x}\right)_{{\bf x}\in\xx}$ has infinite mass. We first study here a property of mixing of the geodesic flow $(g_t)_{t\in\R}$ with respect to this measure. We say that 
the geodesic flow $(g_t)_{t\in\R}$ is mixing with respect to a measure $m$ with finite total mass $||m||$ on $\T^1\mm$, if for any 
$m$-measurable sets $\mathfrak{A},\mathfrak{B}\subset\T^1\mm$, one gets
\begin{equation}\label{defimelangemesurefinie}m\left(\mathfrak{A}\cap g_{-t}.\mathfrak{B}\right)\longrightarrow\dfrac{m(\mathfrak{A})m(\mathfrak{B})}{||m||}\ \text{when}\ t\longrightarrow\pm\infty.\end{equation}
\noindent
When the measure $m$ has infinite mass, this definition may be extended saying that the flow $(g_t)_{t\in\R}$ is mixing if
$$m\left(\mathfrak{A}\cap g_{-t}.\mathfrak{B}\right)\longrightarrow0\ \text{when}\ t\longrightarrow\pm\infty,\ \text{where}\ 
\mathfrak{A}\ \text{and}\ \mathfrak{B}\ \text{have finite measure}.$$
\noindent
When the measure $m_\G$ is finite, Property \eqref{defimelangemesurefinie} was first proved in \cite{Hed} for finite volume surfaces in constant curvature, by F. Dal'bo and M. Peigné for Schottky groups with parabolic isometries acting on Hadamard manifolds with pinched negative curvature (see \cite{DP2}) and by M. Babillot in \cite{Bab2}. The following result of 
T. Roblin \cite{Rob} gathers all the information known in such a general content.
\begin{thmR}[Roblin]
If the Bowen-Margulis measure $m_\G$ has finite mass $||m_\G||$ (resp. infinite mass), the flow $(g_t)_{t\in\R}$ satisfies
$$m_\G\left(\mathfrak{A}\cap g_{-t}.\mathfrak{B}\right)\underset{t\longrightarrow\pm\infty}{\longrightarrow}\dfrac{m_\G(\mathfrak{A})m_\G(\mathfrak{B})}{||m_\G||}\ \left(\text{resp.}\ m_\G\left(\mathfrak{A}\cap g_{-t}.\mathfrak{B}\right)\underset{t\longrightarrow\pm\infty}{\longrightarrow}0\right).$$
\end{thmR}
\begin{rem}
The definition of mixing in infinite measure seems to be weak (see the third chapter of \cite{Rob} about this fact). Nevertheless, our Theorem A below will furnish an asymptotic of the form
$$m_\G(\mathfrak{A}\cap g_{-t}.\mathfrak{B})\underset{t\longrightarrow\pm\infty}{\sim} \varepsilon(|t|)m_\G(\mathfrak{A})m_\G(\mathfrak{B}),\ \text{for an explicit function}\ \varepsilon,$$
\noindent
which can be understood as a mixing property, up to a renormalization.
\end{rem}
\noindent
On the one hand, this property is interesting from the point of view of the ergodic theory. On the other hand, in the case 
of geometrically finite manifolds with finite measure, the property of mixing of the geodesic flow may be used to find an asymptotic of 
the orbital function 
$\mathrm{N}_\G(\oo,R)$. This idea was initially developped in G.A. Margulis' thesis \cite{Mar} for compact manifold with negative curvature: the mixing of the geodesic flow implies a property of equidistribution of spheres on $\mm$, which leads to the orbital counting. In the constant curvature case, other proofs of the orbital counting have been developped using 
spectral theory (see for instance \cite{LP}) or symbolic coding (see \cite{La}). In \cite{Rob}, the author generalizes 
the ideas of Margulis and deduces the asymptotic for the orbital counting function from the mixing of the geodesic flow. He shows the following
%
%
\begin{thmR}[Roblin]
Let $\mm=\xx/\G$ be a complete manifold with pinched negative sectional curvatures. If the Bowen-Margulis measure $m_\G$ has finite mass $||m_\G||$ (resp. infinite mass), the asymptotic behaviour of the orbital function $\mathrm{N}_\G(\oo,R)$ is given by 
$$\mathrm{N}_\G(\oo,R)\sim\dfrac{||\sigma_{\oo}||^2}{||m_\G||}e^{\delta_\G R}\ \left(\text{resp.}\ \mathrm{N}_\G(\oo,R)=o\left(e^{\delta_\G R}\right)\right)\ \text{when}\ R\longrightarrow+\infty,$$
\noindent
where $||\sigma_{\oo}||$ is the mass of the Patterson-Sullivan measure $\sigma_\oo$.
\end{thmR}
We eventually focus on finding an asymptotic equivalent for the number of closed geodesics on $\mm$ of length less than 
$R$, as $R$ goes to infinity. Such asymptotic was first found out by Selberg for compact hyperbolic surfaces (see \cite{Sel}), then by Margulis (\cite{Mar}) for 
compact manifold of negative variable curvature and extended in \cite{PP} to periodic orbits of axiom-A flows. This was generalized by T. Roblin in \cite{Rob} as follows.
\begin{thmR}[\cite{Rob}]
Let $\mm=\xx/\G$ be a geometrically finite complete manifold with sectional curvatures less than $-1$, whose Bowen-Margulis measure has finite mass. For all $R>0$, let $\mathrm{N}_{\mathscr{G}}(R)$ be the number of closed geodesics on $\mm$ of length less than $R$. Then as $R$ goes to infinity,
$$\mathrm{N}_{\mathscr{G}}(R)\sim\dfrac{e^{\delta_\G R}}{\delta_\G R}.$$
\end{thmR}
\noindent
When the Bowen-Margulis measure is infinite, Roblin's method does not yield to such asymptotic. We will detail why in the next paragraph.
\subsection{Assumptions and results}
In this article, we will focus on some manifolds $\mm=\xx/\G$, where $\G$ is a divergent group whose Bowen-Margulis measure $m_\G$ on $\T^1\mm$ has infinite mass and whose Poincaré series are controlled at infinity. For such manifolds, we establish a speed of convergence to $0$ of the quantities $m_\G\left(\mathfrak{A}\cap g_{-t}.\mathfrak{B}\right)$ for any $\mathfrak{A},\mathfrak{B}\subset\T^1\mm$ with finite measure, an asymptotic equivalent for the orbital counting function 
$\mathrm{N}_\G(\oo,R)$ and an asymptotic lower bound for the number of closed geodesics $\mathrm{N}_{\mathscr{G}}(R)$.
The groups $\G$ which we consider are exotic Schottky groups, whose construction was explained in the articles \cite{DOP} and \cite{Pe1} and will be recalled in the second 
section. The main idea of 
these papers is the following: let $\mathrm{M}$ be a geometrically finite hyperbolic manifold with dimension $N\geqslant2$, with a cusp and whose fundamental 
group is a non elementary Schottky group. Theorems A and B in \cite{DOP} ensure that the group $\G$ is of divergent type and that $m_\G$ is finite. The 
proofs of these results give a way to modify the metric in the cusp in order to obtain a manifold $\mathrm{M}'$ isometric to a quotient $\xx/\G$ 
where 
\begin{itemize}
 \item[-] the manifold $\xx$ is a Hadamard manifold with pinched negative curvature;
 \item[-] the group $\G$ acts by isometries on $\xx$ and is of convergent type; the measure $m_\G$ is thus infinite.
\end{itemize}
The article \cite{Pe1} extends the previous construction and allows us to modify the metric in the cusp of $\mathrm{M}$ in such a way that the group $\G$ is of 
divergent type with respect to the metric on $\xx$ and the measure $m_\G$ is still infinite. This article furnishes examples 
of manifolds $\mm=\xx/\G$ on which our work applies. 

Let $\xx$ be a Hadamard manifold with pinched negative curvature between $-b^2$ and $-a^2$, where $0<a<1\leqslant b$ and 
let $\G$ be a Schottky group, \emph{i.e.} $\G$ is generated by elementary groups $\G_1,...,\G_{p+q}$ in Schottky position (see Paragraph 2.1.2), where
$p,q\geqslant1$ and $p+q\geqslant3$. Assume that for some $\beta\in]0,1]$, the groups $\G_1,...,\G_{p+q}$ satisfy the following family of assumptions $(H_\beta)$:
\begin{enumerate}
 \item[$(D)$]{\it The group $\G=\G_1\star...\star\G_{p+q}$ is of divergent type.}
 \item[$(P_1)$]{\it For any $j\in[\![1,p]\!]$, the group $\G_j$ is parabolic, of convergent type and its critical exponent is equal to $\delta_\G$}.
 \item[$(P_2)$]{\it There exists a slowly varying function 
 \footnote{A function $L:\R^{+}\longrightarrow\R^{+}$ is slowly varying at infinity if for any $x>0$, it satisfies
$\lim\limits_{t\longrightarrow+\infty}\dfrac{L(xt)}{L(t)}=1$.} $L$
such that for any $j\in[\![1,p]\!]$, the tail of the Poincaré series at $\delta_\G$ of the group $\G_j$ satisfies
$$\sum\limits_{\alpha\in\G_j\ |\ \dhy(\oo,\alpha.\oo)>T}e^{-\del_\G\dhy(\oo,\alpha.\oo)}\underset{T\longrightarrow+\infty}{\sim}C_j\frac{L(T)}{T^{\beta}}$$
\noindent
for some constant $C_j>0$.}
 \item[$(N)$]{\it For any $j\in[\![p+1,p+q]\!]$, the group  $\Gamma_j$ satisfies the following property}
$$\sum\limits_{\alpha\in\G_j\ |\ \dhy(\oo,\alpha.\oo)>T}e^{-\del_\G\dhy(\oo,\alpha.\oo)}\underset{T\longrightarrow+\infty}{\sim}o\left(\frac{L(T)}{T^{\beta}}\right).$$
We make {\bf additionally} the following assumption: 
 \item[$(S)$]\it{For any $\Delta>0$, there exists $C=C_\Delta>0$ such that for any $j\in[\![1,p+q]\!]$ and any
 $T>0$ large enough}
$$\sum\limits_{\alpha\in\G_j\ |\ T-\Delta\leqslant\dhy(\oo,\alpha.\oo)<T+\Delta}e^{-\del_\G\dhy(\oo,\alpha.\oo)}\leqslant C\dfrac{L(T)}{T^{1+\beta}}.$$
\end{enumerate}
We will say that the parabolic groups $\G_1,...,\G_p$ are ''influent``, since their properties will determine all the dynamical properties of $\G$ on $\partial\xx$. The reader should notice that each of these subgroups $\G_1,...,\G_p$ is convergent and has the same critical exponent $\delta_\G$ as the whole group $\G$. The existence of at least a factor having these properties is needed to get a Schottky group with infinite Bowen-Margulis measure, see \cite{DOP} and Section 2 below.
On the contrary, the groups $\G_{p+1},...,\G_{p+q}$ are said to be ``non-influent'' and their own critical exponent may be in particular strictly less than $\delta_\G$.

By a theorem of Hopf, Tsuji and Sullivan, the divergence of $\G$ ensures that the geodesic flow $(g_t)_{t\in\R}$ is 
totally conservative with respect to the measure $m_\G$. Under these assumptions, we first show the following theorem, which precises 
the rate of mixing of the geodesic flow.
\begin{thmA}
Let $\G$ be a Schottky group satisfying Hypotheses $(H_\beta)$ for some $\beta\in]0,1]$ and $\mathfrak{A},\mathfrak{B}\subset\mathrm{T}^1\xx/\G$ be two $m_\G$-measurable sets with finite measure. 
\begin{itemize}
 \item[$\bullet$]If $\beta\in]0,1[$, there exists a constant $C=C_{\beta,\G}>0$ such that
 $$m_\G(\mathfrak{A}\cap g_{-t}.\mathfrak{B})\sim C\dfrac{m_\G(\mathfrak{A})m_\G(\mathfrak{B})}{|t|^{1-\beta}L(|t|)}\ \text{when}\ t\longrightarrow\pm\infty;$$
 \item[$\bullet$]if $\beta=1$, there exists a constant $C=C_{\G}>0$ such that
 $$m_\G(\mathfrak{A}\cap g_{-t}.\mathfrak{B})\sim C\dfrac{m_\G(\mathfrak{A})m_\G(\mathfrak{B})}{\tilde{L}(|t|)}\ \text{when}\ t\longrightarrow\pm\infty,$$
\noindent
where $\tilde{L}(t)=\displaystyle{\int_1^t}\frac{L(x)}{x}\dd x$ for any $t\geqslant1$.
\end{itemize}
\end{thmA}
The proof of this result relies on the study of a coding of the limit set of $\G$ and on a symbolic representation of the geodesic flow given in the fourth section.

In Section 7, we establish an asymptotic lower bound for the number of closed geodesics with length less than $R$. We prove
\begin{thmB}
Let $\mm=\xx/\G$ be a manifold with pinched negative curvature whose fundamental group $\G$ satisfies Hypotheses $(H_\beta)$ 
for $\beta\in]0,1[$. Then
$$\liminf_{R\longrightarrow+\infty}\dfrac{\delta_\G R}{\beta e^{\delta_\G R}}N_{\mathscr{G}}(R)\geqslant1.$$
\end{thmB}
We could not improve our proof to get a full asymptotic equivalent for $N_{\mathscr{G}}(R)$.
This lower bound remains nevertheless surprising. Let us recall the following result proved in \cite{Rob}. 
For all $R\geqslant0$, let $\mathscr{G}_\G(R)$ be the set of closed orbits for the geodesic flow on $\T^1\mm$ with period less than $R$. For any closed orbit $\wp$, let $\mathcal{D}_\wp$ be the normalized Lebesgue measure along $\wp$. Then as $R\longrightarrow+\infty$,
\begin{equation}\label{roblingeodfermees}
\delta_\G Re^{-\delta_\G R}\sum\limits_{\wp\in\mathscr{G}_\G(R)}\mathcal{D}_{\wp}\longrightarrow\dfrac{m_\G}{||m_\G||}
\end{equation}
in the dual of the set of continuous functions with compact support in $\T^1\mm$. In particular, when $\G$ is convex-cocompact ($m_\G$ is thus finite 
in this case), the set $\T^1\mm$ is compact and \eqref{roblingeodfermees} applied with
$\un_{\T^1\xx/\G}$ implies the counting result $\mathrm{N}_{\mathscr{G}}(R)\sim\frac{e^{\delta_\G R}}{\delta_\G R}$. When $\mm$ is geometrically finite with finite Bowen-Margulis measure, Roblin shows that \eqref{roblingeodfermees} still 
implies $\mathrm{N}_{\mathscr{G}}(R)\sim\frac{e^{\delta_\G R}}{\delta_\G R}$. When the Bowen-Margulis measure has infinite mass, Roblin shows that, still in the dual of compactly supported continuous functions of $\T^1\mm$,
$$\delta_\G Re^{-\delta_\G R}\sum\limits_{\wp\in\mathscr{G}_\G(R)}\mathcal{D}_{\wp}\longrightarrow0.$$
\noindent
Therefore, we could have expected that the $\mathrm{N}_{\mathscr{G}}(R)$ would have been negligible with respect to $\frac{e^{\delta R}}{R}$ as $R$ goes to $+\infty$; Theorem B above thus contradicts this intuition. 

We eventually establish the following asymptotic equivalent for the orbital counting function.
\begin{thmC}
Let $\Gamma$ be a Schottky group satisfying the family of assumptions $(H_\beta)$ for some $\beta\in]0,1]$.
\begin{itemize}
 \item[$\bullet$]If $\beta\in]0,1[$, there exists $C'=C'_{\beta,\G}>0$ such that
 $$\mathrm{N}_{\G}(\oo,R)\sim C'\frac{e^{\delta_\G R}}{R^{1-\beta}L(R)}.$$
 \item[$\bullet$]If $\beta=1$, there exists $C'=C'_{\G}>0$ such that
 $$\mathrm{N}_{\G}(\oo,R)\sim C'\frac{e^{\delta_\G R}}{\tilde{L}(R)},$$
 \noindent
where $\tilde{L}(t)=\displaystyle{\int_1^t}\frac{L(x)}{x}\dd x$ for any $t\geqslant1$.
\end{itemize}
\end{thmC}
To prove Theorem C, we need to extend the coding of the points of the limit set $\LL_\G$ to the $\G$-orbit of some point $x_0\notin\LL_\G$ in the boundary at infinity.

The constants appearing in Theorems A and C will be precised in the proofs.
\begin{rem}
In his seminal work, T. Roblin always assumes the non-arithmeticity of the length spectrum, i.e. the set of lengths of 
closed geodesics in $\xx/\G$ is not contained in a discrete subgroup of $\R$. This assumption is satisfied in our setting, because the quotient manifold has cusps (see \cite{DP}).
\end{rem}
\begin{rem}
In the case of a Schottky group $\G=\G_1\star\G_2$ with only two factors, satisfying hypotheses $(H_\beta)$ for some 
$\beta\in]0,1]$ and with at least one influent factor, the results presented above are still valid but their proofs are slightly more technical. Indeed, the transfer operator will then have two dominant eigenvalues $+1$ and $-1$ (see \cite{BP2} and \cite{DP}). The proof of our result hence would have to be adapted in this case, similarly to the arguments in 
\cite{DP}, which we will not do here.
\end{rem}
Let us now explain why we present separately the additional assumption $(S)$ in the family $(H_\beta)$. In \cite{DPPS}, the 
authors prove a result similar to Theorem C for $\beta\in]\frac{1}{2},1]$, without this assumption. But they can not obtain it 
for $\beta\leqslant\frac{1}{2}$, their proof (at section 6) being based on the renewal theorem of \cite{Er}, which does not ensure 
any more a true limit when $\beta\leqslant\frac{1}{2}$. Our arguments rely on the article \cite{Gou} and we avoid this distinction between $\beta\leqslant\frac{1}{2}$
and 
$\beta>\frac{1}{2}$ thanks to the additional 
assumption $(S)$. We do not know whether this hypothesis is a consequence of the first four in our geometric setting.
\begin{rem}\label{remark4}
We may notice that the assumption $(S)$ is equivalent to each of both following statements:
\begin{itemize}
 \item[$(S')$]\it{For any $\Delta>0$, there exists a constant $C=C_\Delta>0$ such that for any  $j\in[\![1,p+q]\!]$ and
 any $T>0$ large enough, the following inequality is satisfied}
 $$\sharp\{\alpha\in\G_j\ |\ T-\Delta\leqslant\dhy(\oo,\alpha.\oo)<T+\Delta\}\leqslant C\dfrac{L(T)e^{\delta_\G T}}{T^{1+\beta}}.$$
 \item[$(S'')$]\it{There exists a constant $C>0$ such that for any $j\in[\![1,p+q]\!]$ and any $T>0$ large enough}
 $$\sharp\{\alpha\in\G_j\ |\ \dhy(\oo,\alpha.\oo)\leqslant T\}\leqslant C\dfrac{L(T)e^{\delta_\G T}}{T^{1+\beta}}.$$
\end{itemize}
The equivalence between the statements $(S)$ and $(S')$ is clear and the fact that $(S'')$  implies  $(S')$ follows from definitions. We may find a proof of the equivalence between the statements $(S')$ and $(S'')$ in Proposition 2.5 of \cite{DPPS2} (in the finite volume case). In section 3, we detail another proof of the converse property involving Karamata and Potter's lemmas.
\end{rem}
\begin{rem}
In the family of hypotheses $(H_\beta)$ for $\beta\in]0,1]$, the ``influent'' parabolic groups $\G_1,...,\G_p$ are supposed to be elementary. Their rank may be larger than $2$. Nevertheless, in the proofs of Proposition \ref{localexp} and \ref{localexpb1} and of Facts \ref{faitintermediaire} and \ref{fact7}, we work with parabolic groups of rank $1$ in order to simplify the notations. The arguments 
are always true in higher rank.
\end{rem}

\subsection{Outline of the paper}
The second section is devoted to the construction of Schottky groups satisfying hypotheses $(H_\beta)$ as explained in \cite{DOP} and \cite{Pe1}.

The third section is devoted to the presentation of some properties of stable laws with paramater $\beta\in]0,1]$ for 
random variables, together with results on regularly varying functions. These will be crucial tools in the sequel 
due to our assumptions $(P_2)$ and $(N)$.

In the fourth section, we define a coding of the limit set of $\G$ and of the geodesic flow. We then introduce a family of transfer operators associated to 
this coding. We finally end this part with a study of the spectrum and of the regularity of the spectral radius of these operators.

The proof of the mixing rate given in Theorem A in the case $\beta\in]0,1[$ is exposed in Section 5: it is based on the one of Theorem 1.4 in \cite{Gou}. The case $\beta=1$ is presented in Section 6 and the approach is inspired by \cite{MT}. 

We establish in Section 7 the asymptotic lower bound for closed geodesics given in Theorem B.

In section 8, we extend the previous coding of the limit set to include the orbit of a base point $x_0\in\partial\xx\setminus\LL_\G$ and study the family of extended transfer operators associated to this new coding. These extended operators will be central in the proof of Theorem C. 

The final section is dedicated to the proof of Theorem C, which follows the same steps as Theorem A.\\

\underline{Notations:}

In this paper, we will use the following notations. For two functions $f,g\ :\R^+\longrightarrow\R^+$, we will write 
$f\underset{C}{\preceq}g$ (or $f\preceq g$) if $f(R)\leqslant Cg(R)$ for a constant $C>0$ and $R$ large enough and 
$f\underset{C}{\asymp} g$ (or $f\asymp g$) if $f\underset{C}{\preceq}g$ and $g\underset{C}{\preceq}f$. Similarly, for any real numbers $a$ and $b$ and $C>0$, the notation $a\overset{C}{\sim}b$ means $|a-b|\leqslant C$. 

If $A,B$ are two subsets of $E$, we denote by
$A\overset{\Delta}{\times} B:=\{(a,b)\in A\times B\ |\ a\neq b\}$.

The value of the constant $C$ which appears in the proofs may change from line to line.

\section{About ``exotic'' Schottky groups}


In this section, we first recall some definitions and properties about manifolds in negative curvature. Then we give a sketch of the construction of \emph{exotic} Schottky groups following \cite{DOP} and \cite{Pe1}.
\subsection{Negatively curved manifolds and Schottky groups}
\subsubsection{Notations}
Let $\xx$ be a Hadamard manifold of pinched negative curvature  $-b^2\leqslant\kappa\leqslant-a^2$ with $0<a<1<b$, endowed with the 
distance $\dhy$ 
induced by the metric. We denote by $\partial\xx$ its boundary at infinity (see \cite{Bal}). For a point 
$x\in\partial\xx$ and two points ${\bf x},{\bf y}\in\xx$, the \emph{Busemann cocycle} $\mathcal{B}_{x}({\bf x},{\bf y})$ is defined as the
limit of $\dhy({\bf x},{\bf z})-\dhy({\bf z},{\bf y})$ when ${\bf z}$ goes to $x$. This quantity represents the algebraic distance between the horospheres 
centered at $x$ and passing through ${\bf x}$ and ${\bf y}$ respectively. This function satisfies the following property : for any $x\in\partial\xx$ and ${\bf x},{\bf y},{\bf z}\in\xx$
\begin{equation}\label{rajoutbusemancocycle}
\mathcal{B}_{x}({\bf x},{\bf y})=\mathcal{B}_{x}({\bf x},{\bf z})+\mathcal{B}_{x}({\bf z},{\bf y}).
\end{equation}
The \emph{Gromov product} of two points $x$ and $y$ of 
$\partial\xx$ seen from the point $\oo\in\xx$ is given by the following formula
$$(x|y)_{\oo}=\frac{1}{2}\bigg(\bcal_{x}(\oo,{\bf z})+\bcal_{y}(\oo,{\bf z})\bigg)$$
\noindent
where $z$ is any point on the geodesic $(xy)$ with endpoints $x$ and $y$; this product does not depend on the point ${\bf z}$. The curvature being bounded from above by $-a^2$, we may find in \cite{Bou} a proof of the fact that
the quantity $\dhy_{\oo}(x,y):=\exp(-a(x|y)_{\oo})$ defines a distance on $\partial\xx$, which satisfies the following ``visibility'' property: there 
exists a constant $C>0$ depending only on the bounds of the curvature of $\xx$ such that for any $x,y\in\partial\xx$
\begin{equation}\label{visibility}
\dfrac{1}{C}\dhy(\oo,(xy))\leqslant\dhy_{\oo}(x,y)\leqslant C\dhy(\oo,(xy)).
\end{equation}
\noindent
As a consequence of this property, we mention the following important lemma.
\begin{lem}[triangular ``quasi-equality'']\label{quasiegalitetriangulaire}
Let $E,F\subset\xx$ such that $\overline{E}\cap\overline{F}\cap\partial\xx=\emptyset$. There exists a constant $C=C(E,F)>0$ such that 
$\dhy(\oo,{\bf x})+\dhy(\oo,{\bf y})-C\leqslant\dhy({\bf x},{\bf y})$ for any ${\bf x}\in E$ and ${\bf y}\in F$.
\end{lem}
This result can be proved for instance using the arguments given in section 2.3 in \cite{Sch}. This lemma thus furnishes a complement to the classical triangular inequality; many results of this paper are based on it.
Denote by $\mathrm{Isom}(\xx)$ the group of orientation-preserving isometries of $\xx$. The action of 
$\g\in\mathrm{Isom}(\xx)$ can be extended to $\partial\xx$ by homeomorphism. It follows from the previous definitions that for 
any $\g\in\mathrm{Isom}(\xx)$ and any $x,y\in\partial\xx$
\begin{equation}\label{mvr}
\dhy_{\oo}(\g.x,\g.y)=\sqrt{e^{-a\bcal_{x}(\g^{-1}.\oo,\oo)}e^{-a\bcal_{y}(\g^{-1}.\oo,\oo)}}\dhy_{\oo}(x,y).
\end{equation}
\noindent
We thus talk about ``conformal action'' of $\mathrm{Isom}(\xx)$ on the boundary at infinity; the conformal factor $\left|\g'(x)\right|_{\oo}$ 
of an isometry $\g$ at the point $x\in\partial\xx$ is given by the formula $\left|\g'(x)\right|_{\oo}=e^{-a\mathcal{B}_{x}(\g^{-1}.\oo,\oo)}$. From \eqref{rajoutbusemancocycle}, we deduce that
the function $b(\g,x):=-\frac{\log\left|\g'(x)\right|_{\oo}}{a}=\mathcal{B}_{x}(\g^{-1}.\oo,\oo)$ satisfies the following cocycle relation: for any $\g_1,\g_2\in\G$ and any $x\in\partial\xx$
$$b(\g_1\g_2,x)=b(\g_1,\g_2.x)+b(\g_2,x).$$
\subsubsection{Schottky product groups}
Let $k\geqslant1$ and
$\mathfrak{a}_1,...,\mathfrak{a}_k$ be isometries satisfying the following property: there exist non-empty pairwise disjoint closed subsets $D_1,...,D_k$ of $\partial\xx$ such that for all $j\in[\![1,k]\!]$ and all $n\in\Z^*$, one gets 
$\mathfrak{a}_j^n.\left(\partial\xx\setminus D_j\right)\subset D_j$. The isometries $\mathfrak{a}_1,...,\mathfrak{a}_k$ are 
said to be in Schottky position. Klein's ping-pong lemma implies that the group $\G$ generated by the isometries $\left(\mathfrak{a}_i\right)_{1\leqslant i\leqslant k}$ is free and acts properly discontinuously and without fixed point on $\xx$. The group 
$\G$ is called a Schottky group. The limit set $\LL_\G$ of such a group $\G$ is a perfect nowhere dense set.

Let us give an example in the model of the Poincaré half-plan. 
Let $p\ :\ z\longmapsto z+1$ be a parabolic isometry; the non-trivial powers of $p$ send
$[0,1]$ into $]\infty,0]\cup[1,\infty[$, hence we set $D_p=]\infty,0]\cup[1,\infty[\cup\{\infty\}$. On the other hand, let us conjugate
the hyperbolic isometry $h\ :\ z\longmapsto64z$ by $\g(z)=\frac{\frac{3}{4}z+\frac{1}{2}}{z+2}$. The so-obtained isometry $h'$ is hyperbolic with 
fixed points $\frac{1}{4}$ and $\frac{3}{4}$; its negative powers send $\R\cup\{\infty\}\setminus\left[\frac{7}{52},\frac{25}{76}\right]$ into 
$\left[\frac{7}{52},\frac{25}{76}\right]$, and its positives send $\R\cup\{\infty\}\setminus\left[\frac{37}{52},\frac{37}{44}\right]$
into $\left[\frac{37}{52},\frac{37}{44}\right]$. We finally set $D_h=\left[\frac{7}{52},\frac{25}{76}\right]\cup\left[\frac{37}{52},\frac{37}{44}\right]$ 
(see figure \ref{schottkyposition} below).
\begin{figure}[htbp]
\begin{minipage}[c][\height]{.45\linewidth}
\begin{center}
\includegraphics[width=8cm]{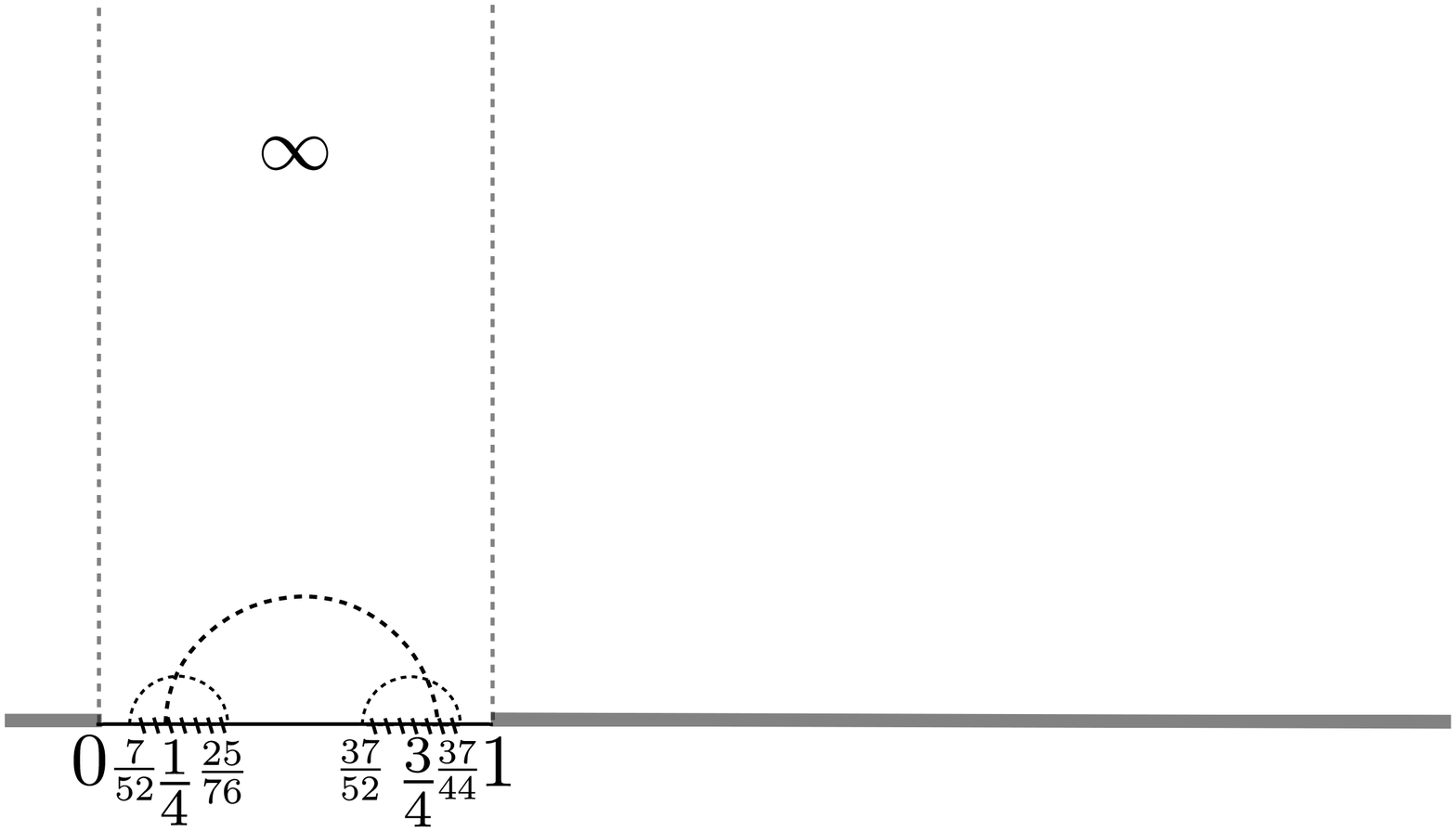}
\end{center}
\end{minipage}
\hfill
\begin{minipage}[c][\height]{.45\linewidth}
\begin{center}
\includegraphics[width=5cm]{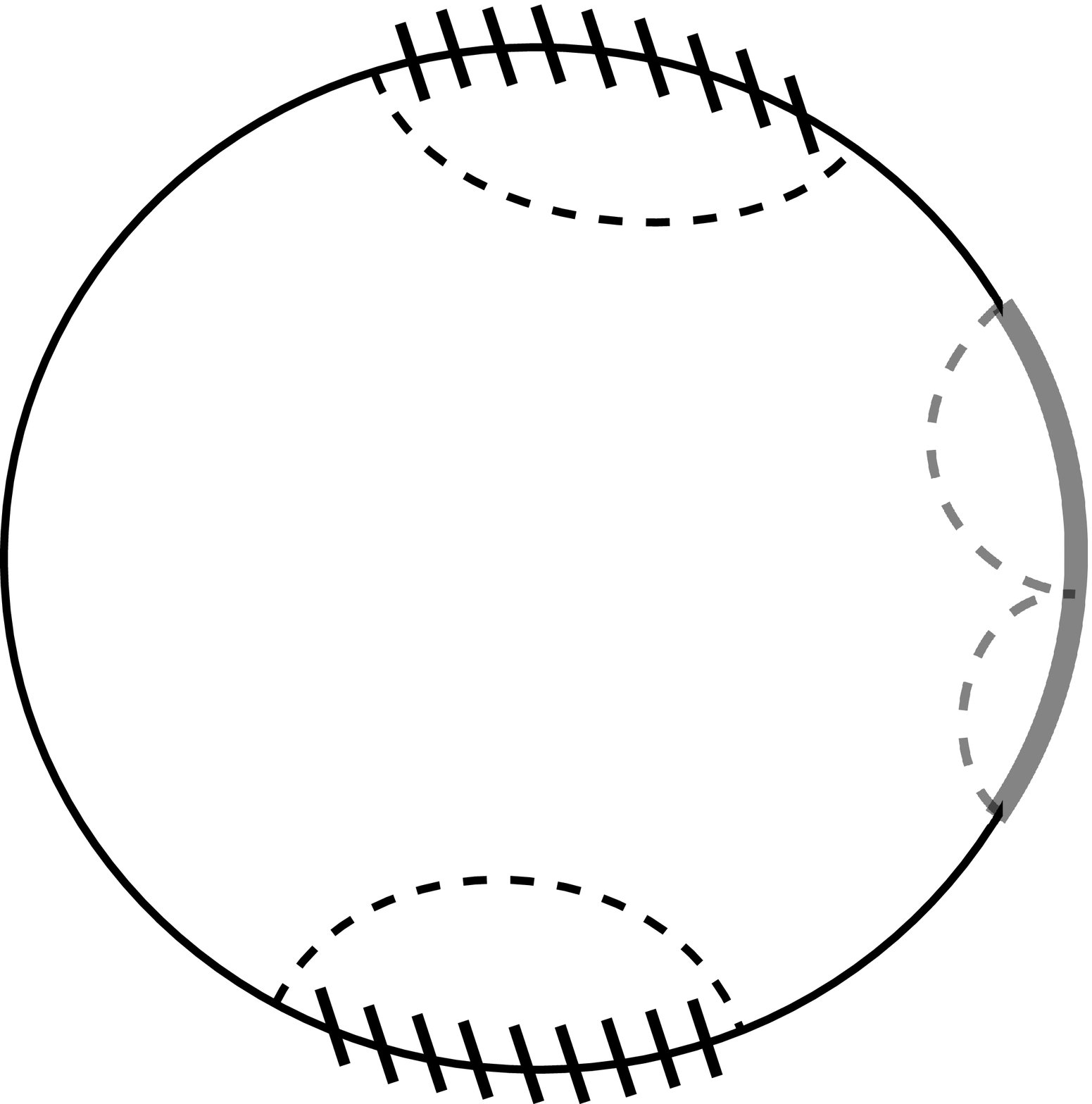}
\end{center}
\end{minipage}
\begin{center}
\includegraphics[height=1cm]{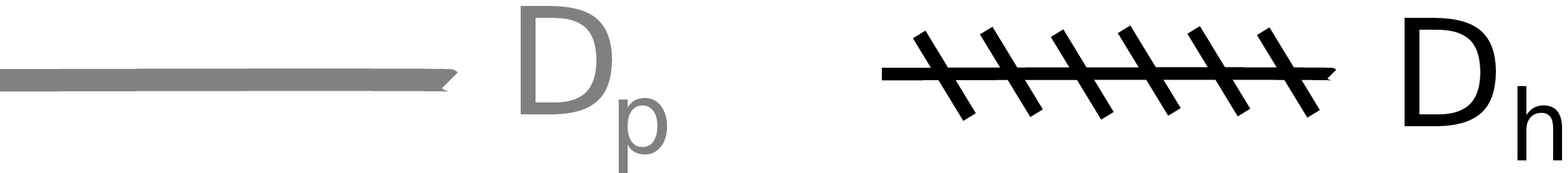}
\end{center}
\caption{\label{schottkyposition} Schottky position in the half-plan or in the disc model of $\mathbb{H}^2$}
\end{figure}
We extend this definition for product of groups. Fix $k\geqslant2$. We say that $k$ groups $\G_1,...,\G_k$ are 
in Schottky position if there exist non-empty pairwise disjoint closed subsets $D_1,...,D_k$ of $\partial\xx$ such that for all $j\in[\![1,k]\!]$, one gets 
$\G_j^*.\left(\partial\xx\setminus D_j\right)\subset D_j$. We may note that $D_j$ contains the limit set $\LL_{\G_j}$ of $\G_j$. Thus the group $\G$ generated by 
$\G_1,...,\G_k$ is the free product $\G_1\star...\star\G_k$; it is called Schottky product of the groups $\G_1,...,\G_k$. Each group $\G_i$, $1\leqslant i\leqslant k$, is called a Schottky factor of 
$\G$.

In the sequel, we will need to consider
subsets $\left({\bf D}_j\right)_{1\leqslant j\leqslant k}$ of $\overline{\xx}$ with the same dynamical properties as the sets 
$\left(D_j\right)_{1\leqslant j\leqslant k}$ under the action of $\G$. For any $j\in[\![1,k]\!]$, we introduce sets ${\bf D}_j\subset\overline{\xx}$, which are geodesically convex and connected 
(resp. admits two geodesically convex connected components) 
when $\G_j$ is generated by a parabolic (resp. hyperbolic) isometry $\mathfrak{a}_j$ and whose 
intersection with $\partial\xx$ contains $D_j$; in addition, we assume that $\G_j^*.\left(\overline{\xx}\setminus{\bf D}_j\right)\subset{\bf D}_j$ and that 
these sets are pairwise disjoint. 
Figure \ref{schottkypositionplein} illustrates the situation for the above isometries $p$ and $h'$.
\begin{figure}[htbp]
\begin{minipage}[c][\height]{.45\linewidth}
\begin{center}
\includegraphics[width=8cm]{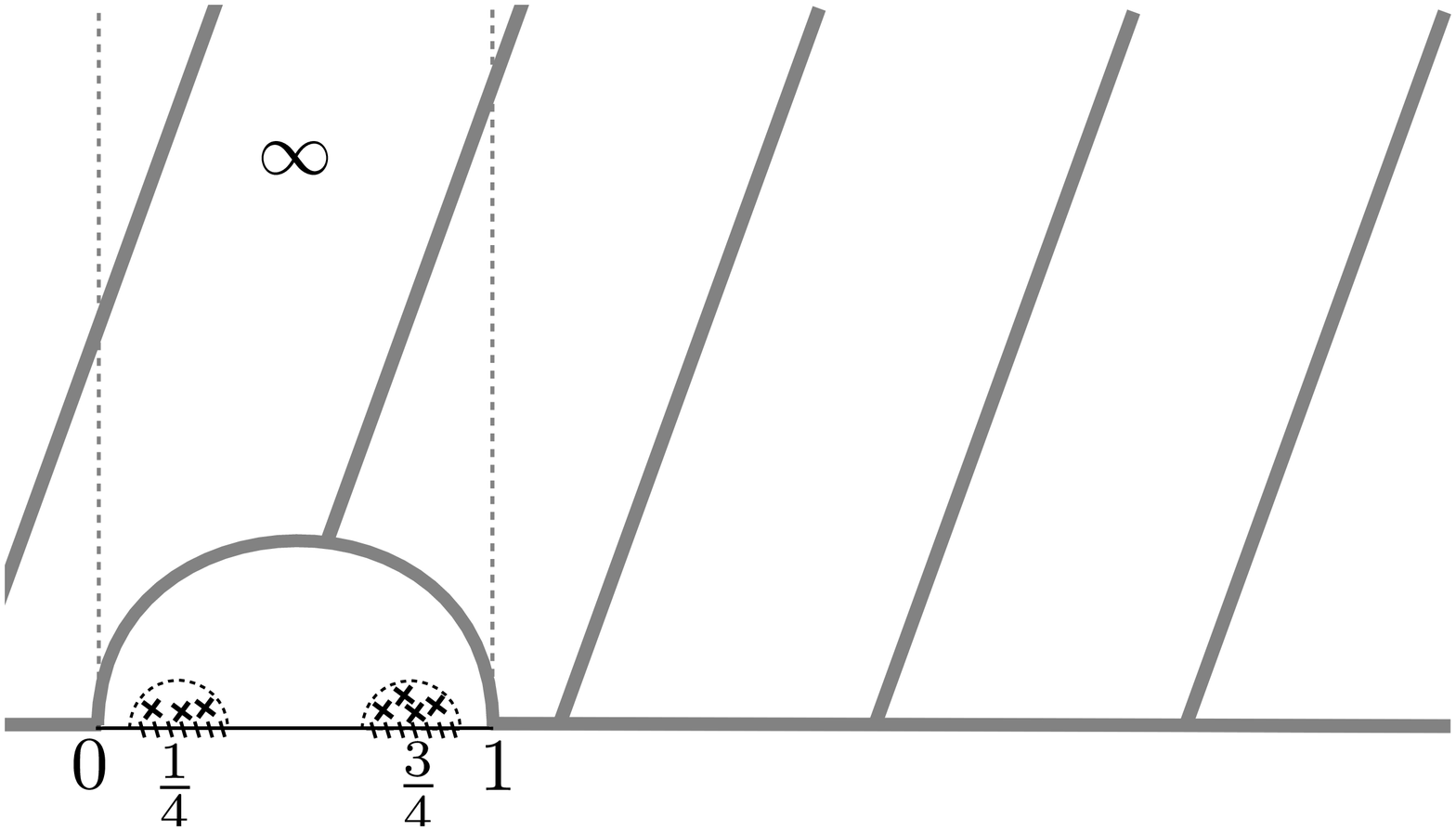}
\end{center}
\end{minipage}
\hfill
\begin{minipage}[c][\height]{.45\linewidth}
\begin{center}
\includegraphics[width=5cm]{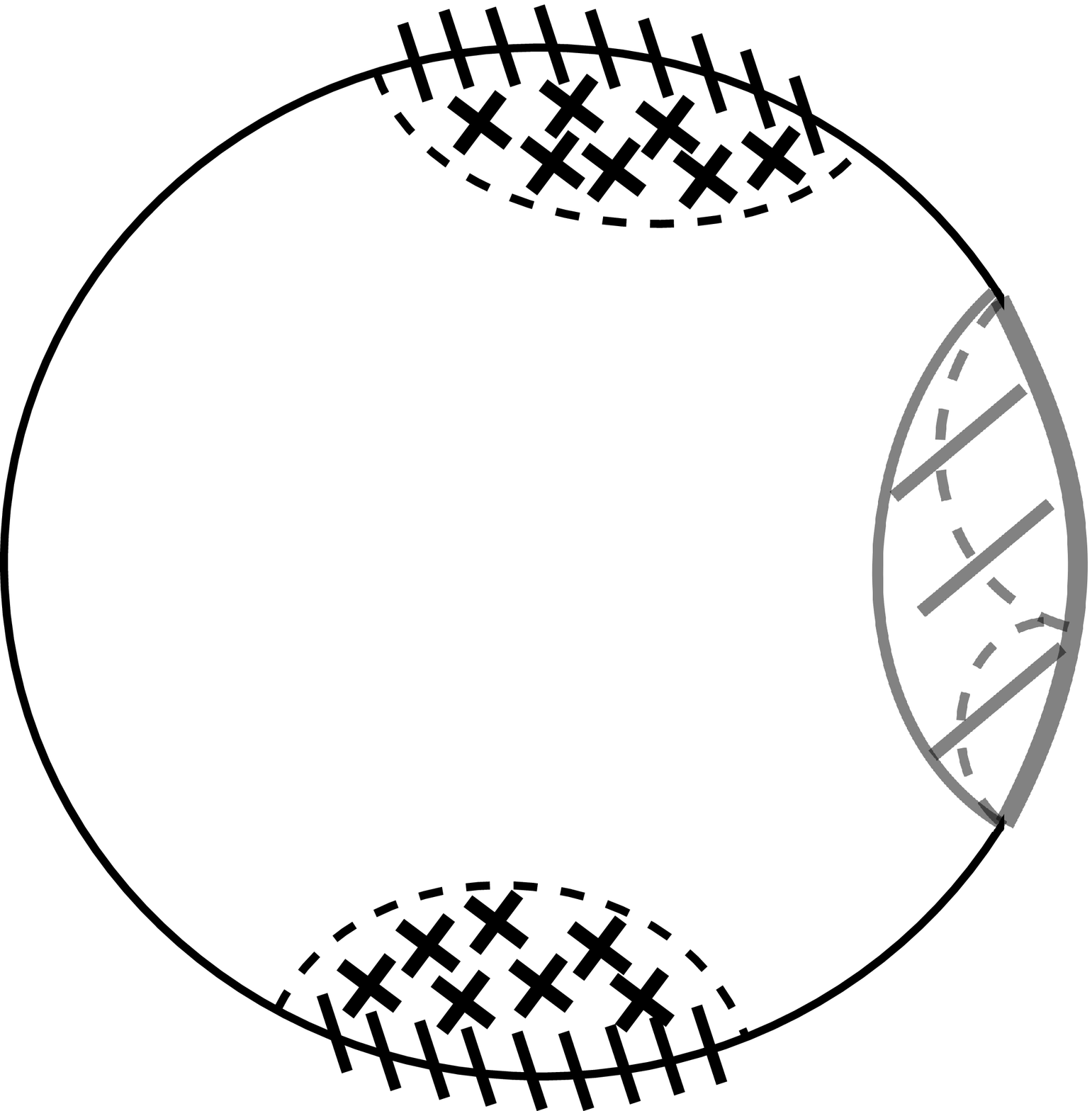}
\end{center}
\end{minipage}
\begin{center}
\includegraphics[height=1cm]{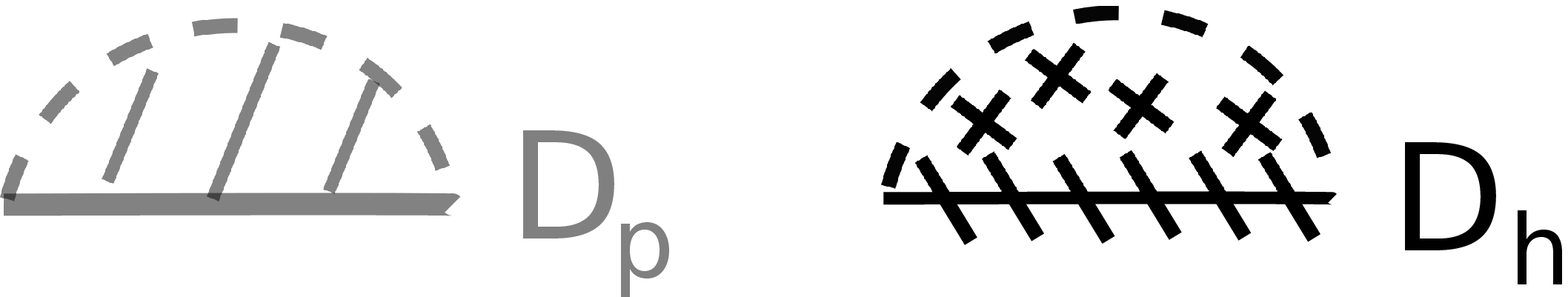}
\end{center}
\caption{\label{schottkypositionplein}Sets ${\bf D}_j$}
\end{figure}

\subsubsection{Geodesic flow and Bowen-Margulis measure}
Using Hopf coordinates, we identify the unit tangent bundle $\T^1\xx$ to the set $\partial\xx\overset{\Delta}{\times}\partial\xx\times\R$: the point
$({\bf x},{\bf v})\in\T^1\xx$ determines a unique triplet $(x^-,x^+,r)$ in $\partial\xx\overset{\Delta}{\times}\partial\xx\times\R$, where $x^-$ and $x^+$ are the 
endpoints of the oriented geodesic passing through ${\bf x}$ at time $0$ with tangent vector ${\bf v}$ and $r=\bcal_{x^+}(\oo,{\bf x})$. The group $\G$ acts on  $\partial\xx\overset{\Delta}{\times}\partial\xx\times\R$ as follows: for any $\g\in\G$
$$\g.(x^-,x^+,r)=\left(\g.x^-,\g.x^+,r+\bcal_{x^+}(\g^{-1}.\oo,\oo)\right)$$
\noindent
and the action of the geodesic flow $(\tilde{g}_t)_{t\in\R}$ is given by
$$\tilde{g}_t(x^-,x^+,r)=(x^-,x^+,r+t)$$
\noindent
for any $t\in\R$. These two actions commute and, quotienting $\T^1\xx$ by $\G$, define the action of the geodesic flow $(g_t)_{t\in\R}$ on 
$\partial\xx\overset{\Delta}{\times}\partial\xx\times\R/\G\simeq\T^1\mm$. By \cite{Eb}, the \emph{non-wandering set} of $(g_t)_{t\in\R}$ on $\T^1\mm$ is  
$\Omega_{\G}:=\LL_{\G}\overset{\Delta}{\times}\LL_{\G}\times\R/\G$.

By Patterson's construction (see \cite{Pat2} and \cite{Sul}), there exists a family $\sigma=(\sigma_{{\bf x}})_{{\bf x}\in\xx}$ of finite measures on $\partial\xx$ 
supported on $\LL_\G$ and satisfying, for any ${\bf x},{\bf x'}\in\xx$, any $x\in\LL_{\G}$ and $\g\in\G$:
 \begin{equation}\label{deltaconformité}\dfrac{\dd\sigma_{{\bf x'}}}{\dd\sigma_{{\bf x}}}(x)=e^{-\delta_\G\bcal_{x}({\bf x'},{\bf x})}\ 
\mathrm{and}\ \g^*\sigma_{{\bf x}}=\sigma_{\g^{-1}.{\bf x}}
\end{equation}
\noindent
where $\g^*\sigma(A)=\sigma(\g A)$ for any Borel subset of $\partial\xx$ and $\delta_\G$ is the Poincaré exponent of $\G$. As soon as $\G$ is divergent and 
geometrically finite, the measures $\sigma_{\mathbf{x}}$ do not have atomic part (see \cite{DOP}). As observed by Sullivan \cite{Sul}, the Patterson measure of $\G$ may be used to construct an invariant measure for the geodesic flow 
with support $\Omega_\G$. It follows from \eqref{deltaconformité} and \eqref{mvr} that the measure $\mu$ defined on $\LL_{\G}\overset{\Delta}{\times}\LL_{\G}$ by
$$\dd\mu(y,x)=\dfrac{\dd\sigma_{\oo}(y)\dd\sigma_{\oo}(x)}{\dhy_{\oo}(y,x)^{\frac{2\delta_\G}{a}}}$$
\noindent
is $\G$-invariant so that $\tilde{m}_\G=\mu\otimes\dd t$ on $\LL_{\G}\overset{\Delta}{\times}\LL_{\G}\times\R$ is both $(\tilde{g}_t)$ and $\G$-invariant. It thus induces 
on $\Omega_{\G}$ an invariant measure $m_\G$ for $(g_t)$. When this measure has finite total mass, this is the unique measure which maximises the measure-theoretic entropy of the 
geodesic flow restricted to its non wandering set: it is called the \emph{Bowen-Margulis measure} (see \cite{OP}). 
When this measure is infinite, there is 
no finite invariant measure which maximises the entropy: however, we still call it the Bowen-Margulis measure. 
\subsection{Construction of exotic Schottky groups}
Let us recall the genesis of the setting in which we will work, {\it i.e.}
the construction of some exotic Schottky groups introduced in \cite{DOP} and \cite{Pe1}.
\subsubsection{Divergents groups and finite Bowen-Margulis measure}
The article \cite{DOP} gives the first known example of geometrically finite manifolds $\mm=\xx/\G$ with pinched negative curvature and whose Bowen-Margulis measure has infinite mass.

The authors construct such examples by providing convergent Schottky groups $\G$ (which hence are geometrically finite). Therefore, the Bowen-Margulis measure on $\T^1\mm$ has infinite mass.

To get such examples, the authors first show that the group $\G$ needs to contain a parabolic subgroup $P<\G$ whose 
critical exponent is $\delta_P=\delta_\G$.
\begin{thmA}[\cite{DOP}]
Let $\G$ be a geometrically finite group with parabolic transformations. If $\delta_\G>\delta_P$ for any parabolic subgroup $P$, then $\G$ is of divergent type.
\end{thmA}
The assumption $\delta_\G>\delta_P$ for any parabolic subgroup $P$ is called the ``critical gap property'' of the group $\G$. It follows from Proposition 2 of \cite{DOP}: if $\G=\G_1\star...\star\G_k$ is a Schottky group, whose each factor $\G_i$, $1\leqslant i\leqslant k$, are divergent, then it has the critical gap property.

When the group $\G$ is divergent, but still has parabolic elements, necessary and sufficient condition for the finiteness of the Bowen-Margulis measure is given by the following criterion.
%
%
\begin{thmB}[\cite{DOP}]
Let $\G$ be a divergent geometrically finite group containing parabolic isometries. The measure $m_\G$ is finite if and only if 
for any parabolic subgroups $P$ of $\G$, the series $\sum_{p\in P}\dhy(\oo,p.\oo)e^{-\delta_\G\dhy(\oo,p.\oo)}$ converges.
\end{thmB}
We can deduce at least two things from both previous theorems. On the one hand, a geometrically finite group containing parabolic isometries 
satisfying the critical gap property is divergent and admits a finite measure $m_\G$. On the other hand, we understand that a first step to 
obtain a group with infinite measure $m_\G$ involves the construction of parabolic groups of convergent type. This is the purpose of the next paragraph, which is based on \cite{DOP}.
\subsubsection{Construction of convergent parabolic groups}
Let us first consider the situation in constant curvature $-1$. Fix $N\geqslant2$. We may identify $\mathbb{H}^N$ with the product $\R_x^{n-1}\times\R_y^{+*}$ 
endowed with the metric $\frac{\dd x^2+\dd y^2}{y^2}$. Let $P$ an elementary parabolic group acting on 
$\mathbb{H}^N$. Up to a conjugacy, we may suppose that the elements of $P$ fix the point at infinity $\infty$. Denote by $\mathscr{H}$ the 
horoball centered at $\infty$ and passing through $\mathbf{i}:=(0,0,...,0,1)$. The group
$P$ acts by euclidean isometries on the horosphere $\partial\mathscr{H}$. By a Bieberbach's theorem (see \cite{Bie} and \cite{Bie2}), there exists a finite index abelian subgroup 
$Q$ of $P$
which acts by translations on a subspace $\R^k\subset\R^{N-1}$, $1\leqslant k\leqslant N-1$. There thus exist $k$ linearly independant 
vectors $v_1,...,v_k$ and a finite set $F\subset P$ such that any element $p\in P$ decomposes into 
$\tau_{v_1}^{n_1}...\tau_{v_k}^{n_k}f$ for $\overline{n}:=(n_1,...,n_k)\in\Z^k$ and $f\in F$, where $\tau_{v_i}^{n_i}$ is the $n_i$-th power of the 
translation of vector $v_i$. In this case, the Poincaré series $\mathcal{P}_{P}$ of $P$ is given by: for $s>0$ 
$$\mathcal{P}_{P}(s)=\sum\limits_{p\in P}e^{-s\dhy(\mathbf{i},p.\mathbf{i})}
=\sum\limits_{f\in F}\sum\limits_{\overline{n}\in\Z^k}e^{-s\dhy(\mathbf{i},\tau_{v_1}^{n_1}...\tau_{v_k}^{n_k}f.\mathbf{i})}.$$
\noindent
The quantity 
$$\dhy(\mathbf{i},\tau_{v_1}^{n_1}...\tau_{v_k}^{n_k}f.\mathbf{i})-2\ln\left(||n_1v_1+...+n_kv_k||\right)$$
\noindent
is bounded when $n_1^2+...+n_k^2\longrightarrow+\infty$, where $||\cdot||$ is the euclidean norm in $\R^{n-1}$. The previous series thus 
behaves like the following
$$\sum\limits_{\underset{\overline{n}\neq\overline{0}}{\overline{n}\in\Z^k}}\dfrac{1}{||n_1v_1+...+n_kv_k||^{2s}},$$
\noindent
which diverges at its critical exponent $\frac{k}{2}$. 

In the sequel, following \cite{DOP}, we will modify the metric in the horoball $\mathscr{H}$ in such a way that the parabolic group 
$P$ will still have critical exponent $\delta_P=\frac{k}{2}$, but its Poincaré series will converge at $\frac{k}{2}$. In this purpose, we consider another model 
of the hyperbolic space, which will be more suitable to understand the action of $P$ on the horospheres. The classical 
upper half space model of the hyperbolic space, 
$\mathbb{H}^N\cong\left(\R^{N-1}\times\R_+^*,\frac{\dd x^2\oplus\dd y^2}{y^2}\right)$ is isometric to 
$\left(\R^{N-1}\times\R,e^{-2t}\dd x^2\oplus\dd t^2\right)$ via the diffeomorphism
$$\Psi\ :\left\{\begin{array}{ll}
                      \R^{N-1}\times]0,+\infty[ & \longrightarrow\R^N\\
                      (x,z) & \longmapsto (x,\ln(z))=(x,t)
                      \end{array}\right..$$
\noindent 
Let us denote by 
$\mathscr{H}_t=\{(x,s)\ |\ x\in\R^{n-1},\ s\geqslant t\}$ the horoball of level $t$ centered at infinity in this model; one gets 
$\Psi(\mathscr{H})=\mathscr{H}_0$. Fix $x,y\in\R^{N-1}$ and let us denote $x_t=(x,t)$ and $y_t=(y,t)$ for $t>0$; these two points both belong 
to the horosphere $\partial\mathscr{H}_t$, and the distance between them, with respect to the metric on $\partial\mathscr{H}_t$ induced by 
the hyperbolic metric on $\R^N$, is equal to $e^{-t}||x-y||$, where $||\cdot||$ is the Euclidean norm on $\R^{n-1}$. Therefore, on 
the horosphere of level $t=\ln\left(||x-y||\right)$, the distance induced on the horosphere between $x_t$ and $y_t$ is 
$1$. Since the curve $[x_0x_t]\cup[x_ty_t]\cup[y_ty_0]$ is a quasi-geodesic, we can deduce from \cite{HeiIm}
that the quantity $\dhy(x_0,y_0)-2\ln\left(||x-y||\right)$ is bounded. Let us now consider on $\R^{N-1}\times\R$ the metric 
$g_T=T^2(t)\dd x^2+\dd t^2$, where $T:\ \R\longrightarrow\R^+$ is chosen such that $g_T$ has pinched negative curvatures. Let us 
write $\dhy_T$ for the distance induced by $g_T$ on $\R^n$. The same argument as previously given for the hyperbolic space 
shows that if $x_0=(x,0)\in\R^{N-1}\times\R$, $y_0=(y,0)\in\R^{N-1}\times\R$, then 
$\dhy_{T}(x_0,y_0)-2u\left(||x-y||\right)$ is bounded uniformly in $x,y\in\R^{N-1}$, where $u:\ \R^{+*}\longrightarrow\R$ 
is defined by the implicit equation $T(u(s))=\frac{1}{s}$ for all $s>0$. When $u(s)=\ln(s)$ and $T(t)=e^{-t}$, we obtain the previous model of the hyperbolic space. One of the steps in \cite{DOP} section 3 and \cite{Pe1} section 2 is to explain how the functions $u$ and $T$ have to be chosen 
so that the sectional curvature $K(t)$ remains negative and pinched on $\R^n$ endowed with $g_T=T^2(t)\dd x^2+\dd t^2$. More precisely, Lemma 2.2 in \cite{Pe1} 
states the following.
\begin{lem}\label{nouvellemetrique}
Fix a constant $\kappa\in]0,1[$. For any $\beta\geqslant0$, there exist $s_\beta\geqslant1$ and a $\mathcal{C}^2$ non-decreasing function 
$u_\beta\ :\ \R^{+*}\longrightarrow\R$ satisfying:
\begin{itemize}
 \item[$\bullet$]$u_\beta(s)=\ln(s)$ if $s\in]0,1]$;
 \item[$\bullet$]$u_\beta(s)=\ln(s)+(1+\beta)\ln\left(\ln(s)\right)$ if $s\geqslant s_\beta$;
 \item[$\bullet$]if $T_\beta(u_\beta(s))=\frac{1}{s}$ for any $s>0$ and $g_\beta=T_\beta^2(t)\dd x^2+\dd t^2$, then $K_{g_\beta}\leqslant -\kappa^2$;
 \item[$\bullet$]$\lim\limits_{s\longrightarrow+\infty}K_{g_\beta}(x,s)=-1$ for any $x\in\R^{n-1}$ and the derivatives of $K_{g_\beta}$ tend to $0$ as $s$ goes to $+\infty$, uniformly in $x$. 
\end{itemize}
\end{lem}
\noindent
We may notice that this metric coïncides with the hyperbolic one on the set $\R^{N-1}\times\R^-$; we can enlarge this 
area shifting the metric $g_\beta$ along the $t$-axis (see Paragraph 2.2 in \cite{Pe1} and Paragraph 2.2.4 of this paper).

On $\left(\R^{n},g_{\beta}\right)$, the group $P$ defined above still acts by isometries and its Poincaré series behaves 
like 
$$\sum\limits_{\underset{\overline{n}\neq\overline{0}}{\overline{n}\in\Z^k}}
\dfrac{e^{-sO(\overline{n})}}{||n_1v_1+...+n_kv_k||^{2s}\ln\left(||n_1v_1+...+n_kv_k||\right)^{2s(1+\beta)}}.$$
\noindent
This series still admits $\frac{k}{2}$ as critical exponent and is convergent if and only if $\beta>0$. In the next 
paragraph, we will see how to adapt the above construction of metric $g_\beta$, $\beta>0$, to highlight the existence of convergent 
parabolic group satisfying the assumptions $(P_2)$ and $(S)$

\subsubsection{On convergent parabolic group satisfying assumptions $(P_2)$ and $(S)$.} 
Here we fix $N=2$,but the following construction may be adapted in higher dimension. Let $\oo$ be the point $(0,0)$
in $\R^2$ and $p$ the translation of vector $(1,0)$. As mentionned previously, for these metrics $g_{\beta}$, 
$\beta>0$, there exists $C>0$ such that for $|n|$ sufficiently large, one gets
$$2\ln|n|+2(1+\beta)\ln\ln|n|-C\leqslant\dhy(\oo,p^n.\oo)
\leqslant2\ln |n|+2(1+\beta)\ln\ln|n|+C.$$
\noindent
As we saw in paragraph 2.2.2, this is enough to ensure the convergence of 
the parabolic group $P=\langle p\rangle$. Nevertheless, this 
estimate is not precise enough to ensure that $P$ satisfies Hypotheses $(P_2)$ and $(S)$. Therefore, in the sequel, we 
present new metrics $g_{\beta}$, $\beta>0$, close to those presented in Lemma \ref{nouvellemetrique}, for which we can
precise the behaviour of the bounded term as $n\longrightarrow\pm\infty$.

Let us fix $\beta>0$. For all real $t$ greater than some ${\mathfrak a}>0$ to  be chosen later, let us set
$$T(t)=T_{\beta,L}(t) = e^{-t}\dfrac{t^{1+\beta}}{L(t)}$$
\noindent
where $L$ is a slowly varying function on $[0,+\infty[$ with values in $\R^{*+}$. Without loss of generality, we assume 
that $L$ is $C^{\infty}$ on $\mathbb R^+$ and  its derivates $L^{(k)}, k \geq 1$, satisfy 
$\ln\left(L^{(k)}(t)\right)\longrightarrow 0$ as $t\to +\infty$ (\cite{BGT}, Theorem 1.3.3); furthermore, for any $\theta>0$, there 
exist $t_\theta \geq 0$ and $C_\theta \geq 1$ such that  for any $t\geq t_\theta$
\begin{equation} \label{majorationslowlyvarying}
\dfrac{1}{C_\theta t^\theta}\leqslant L(t)\leqslant C_\theta t^\theta.
\end{equation}
\noindent
Notice that 
$\displaystyle 
-\frac{T ''(t)}{T (t)}=-\left(1-\frac{2(1+\beta)}{t}+L'(t)\right)^2+\left(\frac{(1+\beta)}{t^2}+L''(t)\right)<0 $ 
for $t\geqslant{\mathfrak a}$. As for Lemma \ref{nouvellemetrique} (whose proof is given in \cite{Pe1}), we extend $T_{\beta,L}$ on 
$\R$ as follows. 
\begin{lem}\label{talpha}
There exists  ${\mathfrak a}={\mathfrak a}(\beta, L)>0$ such that the map $T=T_{\beta, L} : \R \to \R^{*+}$ 
defined by:
\begin{itemize}
\item $T (t) = e^{-t}$ if $t\leqslant 0$;
\item $\displaystyle{T(t)= e^{-t}\frac{t^{1+\beta}}{L(t)}}$ if $t\geq{\mathfrak a}$;
\end{itemize}
\noindent
admits a decreasing and 2-times continuously differentiable  extension on $\mathbb R$ satisfying the following inequalities
$$-b^2\leqslant K(t)= -\dfrac{T''(t)}{T(t)}\leqslant-a^2<0.$$
\end{lem} 
\noindent
Notice that if this property holds for some $\mathfrak a>0$, it holds for any $\mathfrak a'\geq \mathfrak a$. 
{\bf For technical reasons} (see Lemma \ref{groupeconvergentHypotheses}),  we will assume  without loss of generality  
that $\mathfrak a > 4(1+\beta)$.
A direct computation yields the following estimate for the function $u=u_{\beta,L}$.
\begin{lem}\label{lem:UAsymp}
Let $u=u_{\beta,L}\ :\ \R^{+*} \to \R$ be such $T(u(s))=\frac{1}{s}$ for all $s>0$. Then 
$$u(s)= \ln s +(1+\beta)\ln\ln s-\ln L(\ln s)+\epsilon(s)$$
\noindent
with $\epsilon(s) \to 0$ as $s \to +\infty$.
\end{lem}
\noindent
The group $P=\langle p \rangle $ is a parabolic subgroup of the group of isometries of 
$\R^2$ endowed with the metric $g_\beta=g_{\beta,L}=T_{\beta,L}^2(t)\dd x^2+\dd t^2$, fixing the point $\infty $. 
It follows from the arguments presented in the previous paragraph and from Lemma \ref{lem:UAsymp} above that, up 
to a bounded term, $\dhy(\oo,p^n.\oo)$ equals $2\ln|n|+2(1+\beta)\ln\ln|n|-2\ln L(\ln|n|)$ for $|n|$ large enough. The 
group $P$ still has critical exponent $\frac{1}{2}$ and is convergent when $\beta>0$. The following proposition gives a 
precise estimate for $\dhy(\oo,p^n.\oo)$; it is the key-point to prove that $P$ satisfies Assumptions $(P_2)$ and $(S)$.
\begin{prop} \label{prop:PAsymp} 
The parabolic  group $P=\langle p \rangle$ on $(\R^2, g_{\beta})$ satisfies the following property: for all 
$n\in\Z\setminus\{0\}$ such that $|n|$ is large enough,
$$\dhy(\oo, p^n.\oo)=2\ln|n|+2(1+\beta)\ln\ln|n|-2\ln L(\ln|n|)+\epsilon(n)$$
\noindent
with $\displaystyle \lim_{n \to\pm\infty} \epsilon(n)=0$. In particular it is convergent with respect to $g_\beta$.
\end{prop} 
\noindent 
Let $\mathcal{H}=\{(x, t) \mid  t\geq 0\}$ be the upper half plane and $\mathcal{H}/P$ the 
quotient cylinder endowed with the metric $g_{\beta}$. We can not estimate directly the 
distances $\dhy(\oo,p^n.\oo)$, since the metric $g_{\beta}$ is not explicit for $t\in [0, \mathfrak a]$. Let us introduce 
the point ${\bf a}=(0,\mathfrak{a})\in\mathcal{H}$. The union of the three geodesic segments 
$[\oo, {\bf a}], [{\bf a}, p^n.{\bf a})]$ and $[p^n.{\bf a}, p^n.{\bf o}]$ is a quasi-geodesic. Since 
$d(\oo, {\bf a})= d(p^n.\oo, p^n.{\bf a})$ is fixed  and $d({\bf a}, p^n.{\bf a})\to +\infty$, it yields to the following lemma.
\begin{lem}
Under the previous notations, 
$$\lim_{n\to\pm\infty}\left(\dhy(\oo, p^n.\oo)-\dhy({\bf a}, p^n.{\bf a})\right) = 2\mathfrak{a}.$$
\end{lem}
\noindent
Therefore, Proposition \ref{prop:PAsymp} follows from the following lemma.
\begin{lem}\label{groupeconvergentHypotheses}
Assume that $\mathfrak{a}\geqslant4(1+\beta)$. Then
$$\dhy({\bf a}, p^n.{\bf a})=2\ln|n| + 2(1+\beta)\ln\ln|n|-2\ln L(\ln|n|)- 2\mathfrak{a}+ \epsilon(n)$$
\noindent
with $\displaystyle \lim_{n \to\pm\infty} \epsilon(n)=0$. 
\end{lem} 
\begin{proof}
%
Throughout this proof, we work on the  upper half-plane $\R \times [\mathfrak{a}, +\infty[$ whose points are denoted 
$(x, \mathfrak a+t)$ with $x\in\R$ and $t\geq 0$; we set  
$$\mathcal{T}(t)=T_{\beta}(t+\mathfrak a)=e^{-\mathfrak a-t}\frac{(t+\mathfrak a)^{1+\beta}}{L(t+\mathfrak a)}.$$
In these coordinates, the quotient cylinder $\R\times[\mathfrak a, +\infty[/P$ is a surface of revolution 
endowed with the metric  ${\mathcal T}(t)^2\dd x^2+\dd t^2$. 
For any $n \in \mathbb Z$, denote $h_n$ the maximal height at which the geodesic segment 
$\sigma_n=[{\bf a}, p^n.{\bf a}]$ penetrates inside the upper half-plane $\R \times [\mathfrak a, +\infty[$. Note that 
due to the negative upperbound on the curvature, $\displaystyle \lim_{n\to \pm \infty} h_n = +\infty$.
The relation between $n, h_n$ and $d_n:=\dhy({\bf a}, p^n\cdot{\bf a})$ may be deduced from 
the Clairaut's relation (\cite{DC}, section 4.4 Example 5):
$$ 
\dfrac{n}{2}={\mathcal T}(th_n)\int_{0}^{h_n}{\dfrac{\dd t}{\mathcal{T}(t)\sqrt{\mathcal{T}^2(t)-\mathcal{T}^2(h_n)}}}
\qquad {\rm and} \qquad   d_n=2\int_{0}^{  h_n}{\dfrac{\mathcal{T}(t)\dd t}{\sqrt{\mathcal{T}^2(t)-\mathcal{T}^2(h_n)}}}.
$$
These identities may be rewritten as
$$ 
(a)\quad \dfrac{n}{2}=\dfrac{1}{{\mathcal T}(h_n)}\int_{0}^{h_n}\dfrac{f_n^2(s)\dd s}{\sqrt{1-f_n^2(s)}}
\quad {\rm and} \quad  (b) \quad d_n=2 h_n+ 2 \int_{0}^{ h_n}
\Bigl(\dfrac{1}{\sqrt{1-f_n^2(s)}}-1\Bigr)\dd s
$$
\noindent
where $\displaystyle f_n(s):= \dfrac{{\mathcal T}(h_n)}{{\mathcal T}(h_n-s)}\un_{[0, h_n]}(s).$

First,  for any $s \geqslant 0$, the quantity 
$\displaystyle{\frac{f_n^2(s)}{\sqrt{1-f_n^2(s)}}}$ converges towards $\displaystyle{\frac{e^{-2s}}{\sqrt{1-e^{-2s}}}}$
as $n \to\pm\infty$. 
\begin{lem}\label{majorationf}
There exists $N_0>0$ such that for all $n$, $|n|\geq N_0$ and all $s\geq 0$,
$$0\leqslant f_n(s)\leqslant f(s):= e^{-s/2}.$$ 
\end{lem}
\begin{proof}
To enlighten the proof, let us denote $\alpha=1+\beta$. 

 Assume first $\frac{h_n}{2}\leqslant s\leqslant h_n$; 
taking $\theta=\frac{\alpha}{2}$ in \eqref{majorationslowlyvarying}
yields 
\begin{eqnarray*}
0\leqslant f_n(s)&=& 
\left(\dfrac{\mathfrak a + h_n}{\mathfrak a + h_n - s}\right)^\alpha\dfrac{L(\mathfrak a + h_n - s)}{L(\mathfrak a + h_n)}
e^{-s}\\
&\leqslant& C_{\frac{\alpha}{2}}^2\dfrac{(\mathfrak a + h_n)^{\frac{3\alpha}{2}}}{(\mathfrak a + h_n - s)^{\frac{\alpha}{2}}}e^{-s}\\
&\leqslant& \dfrac{C_{\frac{\alpha}{2}}^2}{\mathfrak a^{\frac{\alpha}{2}}}(\mathfrak a + h_n)^{\frac{3\alpha}{2}}e^{-s}\\
&\leqslant& \dfrac{C_{\frac{\alpha}{2}}^2}{\mathfrak a^{\frac{\alpha}{2}}}(\mathfrak a + h_n)^{\frac{3\alpha}{2}}
e^{-\frac{h_n}{4}}e^{-\frac{s}{2}}\leqslant e^{-\frac{s}{2}}
\end{eqnarray*}
where the last inequality holds if $|n|$ is large enough, only depending on $\mathfrak a$ and $\alpha$. 

 Assume now $0\leqslant s \leqslant\frac{h_n}{2}$; it holds  
$\displaystyle {\frac{1}{2}}\leqslant\frac{\mathfrak a+h_n-s}{\mathfrak a +h_n}\leqslant 1$ and 
$0\leqslant\frac{s}{\mathfrak a +h_n}\leqslant\min\left(\frac{1}{2},\frac{s}{\mathfrak a}\right)$. The facts that $L$ is 
slowly varying  and that $0\leqslant\frac{1}{1-v}\leqslant e^{2v}$ for $0\leqslant v \leqslant\frac{1}{2}$ yield for 
all $n\geq N_0$,
\begin{eqnarray*}
0\leqslant f_n(s)&=&\dfrac{L(\mathfrak a + h_n - s)}{L( \mathfrak a + h_n)} 
\left(\dfrac{1}{1-\frac{s}{\mathfrak a + h_n}}\right)^\alpha e^{-s}\\
&\leqslant & \sup_{\stackrel{h\geqslant h_{N_0}}{\frac{1}{2}\leqslant\lambda\leqslant 2}}\dfrac{L(\lambda h)}{L(h)}
e^{-(1-\frac{2\alpha}{\mathfrak a})s}\leqslant e^{-s/2} 
\end{eqnarray*}
where the last inequality holds as soon as $\mathfrak  a> 4\alpha$ and $N_0$ large enough, depending on $L$ and $\mathfrak  a- 4\alpha$. 
\end{proof}
\noindent
We have therefore
$$
0\leqslant\dfrac{f_n^2(s)}{\sqrt{1-f_n^2(s)}}\leqslant  F(s):=\dfrac{f^2(s)}{\sqrt{1-f^2(s)}}
$$
\noindent
where  the function $F$ is integrable on $\R^+$. By the dominated convergence theorem, it yields 
$$\dfrac{n}{2}=\dfrac{1+\epsilon(n)}{{\mathcal T}(h_n)} \int_0^{+\infty}\dfrac{e^{-2s}}{\sqrt{1-e^{-2s}}}\dd s
=\dfrac{1+\epsilon(n)}{{\mathcal T}(h_n)}.$$
\noindent
Consequently $h_n= \ln|n|+(1+\beta)\ln\ln |n|-\ln L(\ln|n|)-\log 2 -\mathfrak a +\epsilon(n)$ for $|n|$ large enough.
Similarly 
$$\lim\limits_{n\to\pm\infty}\int_0^{h_n}
\Bigl(\frac{1}{\sqrt{1-f_n^2(s)}}-1\Bigr)\dd s=\int_0^{+\infty}\Bigl(\frac{1}{\sqrt{1-e^{-2s}}}-1\Bigr)\dd s=\ln 2,$$ 
which yields 
$$d_n= 2\ln|n|+2(1+\beta)\ln\ln |n|-2\ln L(\ln|n|) -2\mathfrak a +\epsilon(n)\ \text{for}\ |n|\ \text{large enough.}$$
\end{proof}
\noindent
The Poincar\'e exponent of $P$  equals $\frac{1}{2}$ and its orbital function satisfies the following property:
$$
\sharp\{p\in P\mid 0\leqslant \dhy(\oo, p.\oo)<T \} \sim   e^{\frac{T}{2}}\dfrac{L(T)}{(T/2)^{\beta+1}}\quad{\rm as}\quad T\to+\infty.
$$
\noindent
Hence, for any $\Delta >0$, 
$$
\sharp\{p\in P\mid T\leqslant \dhy(\oo,p.\oo)<T+\Delta\} \sim \frac{1}{2}\int_{T}^{T+\Delta} e^{t/2}
\dfrac{L(t)}{(t/2)^{\beta+1}}\dd t\quad{\rm as}\quad T\to+\infty
$$
\noindent
and
\begin{equation}\label{avantdemontrerP2}
\lim_{T\to +\infty}\dfrac{T^{\beta+1}}{L(T)}
\sum_{\stackrel{p \in P}{T\leqslant \dhy(\oo,p.\oo)<T+\Delta}} 
e^{-\frac{1}{2}\dhy(\oo,p.\oo)}=2^{\beta}\Delta.
\end{equation}
\noindent
On the one hand, for any $T>0$, decomposing 
$$\sum\limits_{\underset{\dhy(\oo,p.\oo)>T}{p\in P}}e^{-\frac{1}{2}\dhy(\oo,p.\oo)}$$
\noindent
as a sum over annuli whose size is arbitrarily small and using \eqref{avantdemontrerP2}, we obtain
$$\sum\limits_{\underset{\dhy(\oo,p.\oo)>T}{p\in P}}e^{-\frac{1}{2}\dhy(\oo,p.\oo)}\sim
\dfrac{2^\beta}{\beta}\dfrac{L(T)}{T^\beta}$$
which corresponds to Assumption $(P_2)$.
On the other hand, let us note that, as soon as for some $C>0$,
$$
 \lim_{T\to +\infty}\dfrac{T^{\beta+1}}{L(T)}
 \sum _{\stackrel{p \in P}{T\leqslant\dhy(\oo,p.\oo)<T+\Delta}} 
 e^{-\frac{1}{2} d(\oo,p.\oo)}  = C \Delta,
$$
\noindent
there exists $T_0 = T_0(\Delta)>0$ such that for all $T\geq T_0$, we have
$$
\dfrac{T^{\alpha}}{L(T)}\sum _{\stackrel{p \in P}{T\leqslant\dhy(\oo,p.\oo)<T+\Delta}} 
e^{-\frac{1}{2}\dhy(\oo,p.\oo)}\leqslant 2 C,
$$
\noindent
which is precisely Hypothesis $(S)$.
%

\subsubsection{About groups with convergent parabolic subgroup}

Let us now describe explicit constructions of exotic groups, i.e. non-elementary groups $\Gamma$ containing a 
parabolic subgroup $P$ such that $\delta_{P}=\delta_\Gamma$ in the context of the metrics $g_\beta$ presented above. Let 
$\beta>0$ be fixed and $N\geqslant2$. For all $\mathfrak{h}>0$ and $t\in\R$, we write
$$
T_{\beta,\mathfrak{h}}=\left\{\begin{array}{ccc}
					e^{-t} & \mbox{if} & t\leqslant \mathfrak{h}\\
					e^{-\mathfrak{h}}T_{\beta,L}(t-\mathfrak{h}) & \mbox{if} & t\geq \mathfrak{h}
					\end{array}\right.,
$$
where $T_{\beta,L}$ was defined in the previous paragraph. Following \cite{Pe1}, we introduce the metric on 
$\R^{N-1}\times\R$ given at all $(x,t)\in\R^n$ by 
$\dd g_{\beta,\mathfrak{h}}=T_{\beta,\mathfrak{h}}^2(t)\dd x^2+\dd t^2$. It is a complete smooth metric, 
with pinched negative curvature by Lemmas \ref{nouvellemetrique} and \ref{talpha}, isometric to the hyperbolic plane on 
$\R^{N-1}\times(-\infty, \mathfrak{h})$. Note that $g_{\beta,0}=g_\beta$ and $g_{\beta,+\infty}$ is the hyperbolic metric on $\mathbb H^N$ 
(see figure \ref{dessincuspethoroboule}). 
\begin{figure}[htbp]
\begin{minipage}[c][\height]{.45\linewidth}
\begin{center}
\includegraphics[width=8cm]{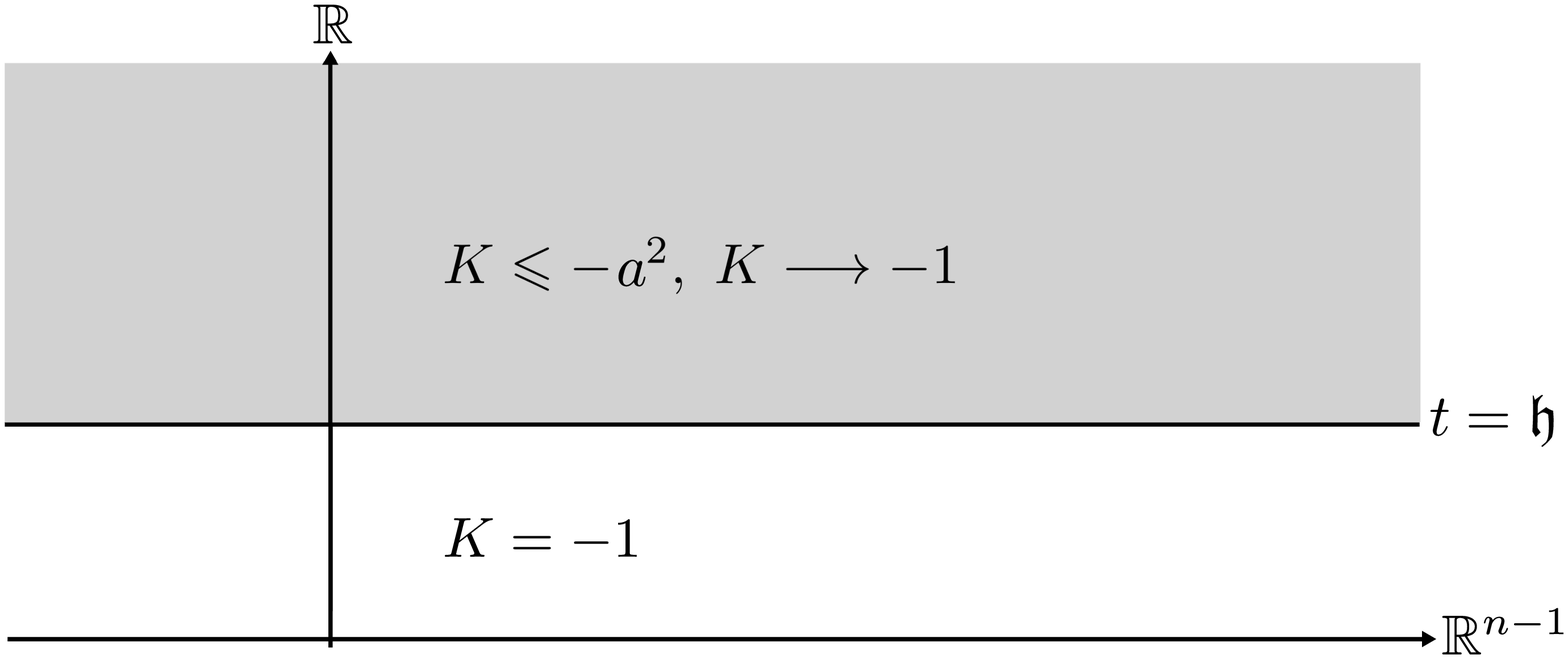}
\end{center}
\end{minipage}
\hfill
\begin{minipage}[c][\height]{.45\linewidth}
\begin{center}
\includegraphics[height=4cm]{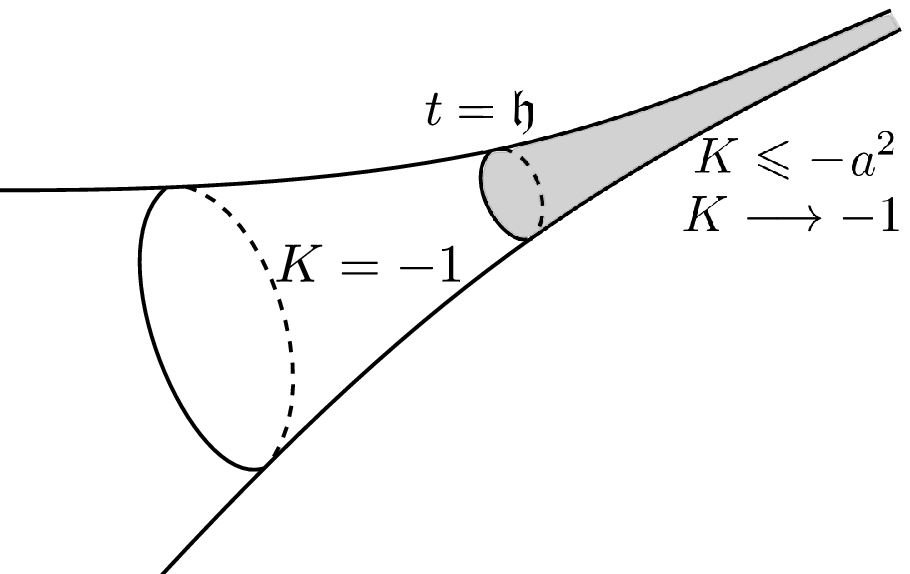}
\end{center}
\end{minipage}
\caption{\label{dessincuspethoroboule}Curvature in the horoball and in the cusp in rank $1$}
\end{figure}
We now construct non-elementary groups containing parabolic subgroups of the above type.
Let us consider a geometrically finite group $\G$, acting freely, properly and discontinuously on $\mathbb{H}^N$ and 
containing parabolic isometries. The quotient manifold $\mathbb{H}^N/\G$ thus has finitely many cusps 
$\mathcal{C}_1$,...,$\mathcal{C}_l$, each of them being 
isometric to the quotient of an horoball $\mathscr{H}_i$ by a parabolic group $P_i$ with rank $k_i\in[\![1,N-1]\!]$. By 
the previous discussion, each group $P_i$ acts by isometries on $\R^n$ endowed with the metric 
$T_{\beta,\mathfrak{h}}^2(t)\dd x^2+\dd t^2$. Let us pick out the 
cusp $\mathcal{C}_1$ and paste the quotient 
$(\R^{N-1}\times[\mathfrak{h},+\infty[)/P_1$ with $\left(\mathbb{H}^N/\G\right)\setminus\mathcal{C}_1$. The previous construction 
ensures that $\G$ still acts by isometries on the universal cover $\xx$ of this quotient manifold endowed with the 
metric $\tilde{g}_{\beta,\mathfrak{h}}$, which is the hyperbolic metric on $\xx$ excepted on the copies of 
$\mathscr{H}_1$ where it coincides with $g_{\beta,\mathfrak{h}}$. 

The fact that $\beta>0$ implies that the group $P_1$ is convergent. Let us denote $\dhy_\mathfrak{h}$ the distance induced by 
the metric $\tilde{g}_{\beta,\mathfrak{h}}$. Following \cite{Pe1} section 4, we 
will use the previous construction of manifolds with a cusp associated to a convergent parabolic group to present 
some discrete group $G$ 
acting on the above space $\xx$, and being convergent or divergent, according to the value of the parameter $\mathfrak{h}$. 
The author of \cite{Pe1} proves the following.
\begin{prop}\label{propMarcmetricetdivergence}
There exist Schottky subgroups $G$ of $\G$ and a positive real $\mathfrak{h}_0$ such that:
\begin{itemize}
 \item[$\bullet$]$G$ admits $\frac{1}{2}$ as critical exponent and is of convergent type on $\left(\xx,\tilde{g}_{\beta,0}\right)$;
 \item[$\bullet$]$G$ has a critical exponent $>\frac{1}{2}$ and is of divergent type on $\left(\xx,\tilde{g}_{\beta,\mathfrak{h}}\right)$ for 
 $\mathfrak{h}\geqslant \mathfrak{h}_0$.
\end{itemize}
\end{prop}
Fix $\oo\in\xx$. This result relies on the existence of a parabolic isometry $p\in P_1$ and a hyperbolic one $h\in\G$ with no 
common fixed point. We may 
shrink the horoball $\mathscr{H}_1$ such that the closed geodesic of $\xx/\G$ obtained as the projection of the axis of $h$ does not 
enter the cusp 
$\mathscr{H}_1/P_1$ (see figure \ref{groupepethfigure}). 
\begin{figure}[htbp]
\begin{center}
\includegraphics[height=5cm]{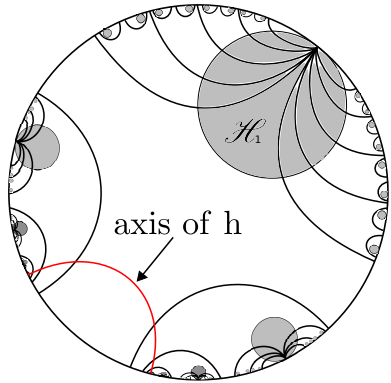}
\end{center}
\caption{\label{groupepethfigure}Isometries $p$ and $h$ and horoball $\mathscr{H}_1$}
\end{figure}

Then this geodesic stays in the area of constant curvature $-1$ and so does its lift. The isometries $p$ and $h$ 
being in Schottky position, there exist 
non-empty disjoint 
closed subsets $\mathbf{D}_p$ and $\mathbf{D}_h$ of $\overline{\xx}$ such that 
$p^n.\left(\overline{\xx}\setminus\mathbf{D}_p\right)\subset\mathbf{D}_p$ and 
$h^n.\left(\overline{\xx}\setminus\mathbf{D}_h\right)\subset\mathbf{D}_h$ for 
any $n\neq0$. Denote by $D_\g=\mathbf{D}_\g\cap\partial\xx$ for $\g\in\{p,h\}$. In this case, for any $s>0$, the Poincaré series of the group $\langle p,h\rangle$ behaves, up to a bounded term, like 
\begin{equation}\label{étoilesomme}
\sum\limits_{l\geqslant1}\sum\limits_{m_i,n_i\in\Z^*}e^{-s\dhy_\mathfrak{h}(\oo,h^{m_1}p^{n_1}...h^{m_l}p^{n_l}.\oo)}.
\end{equation}
\noindent
Using Lemma \ref{quasiegalitetriangulaire} for sets $\mathbf{D}_p$ and $\mathbf{D}_h$, we show that \eqref{étoilesomme} behaves like
\begin{equation}\label{étoilétoilesomme}
\sum\limits_{l\geqslant1}\left(\sum\limits_{m\in\Z^*}e^{-s\dhy_\mathfrak{h}(\oo,h^m.\oo)}\sum\limits_{n\in\Z^*}e^{-s\dhy_\mathfrak{h}(o,p^n.\oo)}\right)^l.
\end{equation}
\noindent
With any metric $\tilde{g}_{\beta,\mathfrak{h}}$, $\mathfrak{h}>0$, we have $\delta_{\langle p\rangle}=\frac{1}{2}\leqslant\delta_{\langle p,h\rangle}$. For 
$\mathfrak{h}=0$, we get $\sum_{n\neq0}e^{-\frac{1}{2}\dhy_0(\oo,p^n.\oo)}<+\infty$. Furthermore 
$\sum_{m\neq0}e^{-\frac{1}{2}\dhy_0(\oo,h^m.\oo)}\asymp e^{-\frac{1}{2}l(h)}$ 
where 
$l(h)$ is the length of $h$. Therefore, if we replace $h$ by a sufficiently large power of $h$, we may assume that 
$$\sum\limits_{m\in\Z^*}e^{-\frac{1}{2}\dhy_0(\oo,h^m.\oo)}\sum\limits_{n\in\Z^*}e^{-\frac{1}{2}\dhy_0(\oo,p^n.\oo)}<1.$$
\noindent
Hence $\mathcal{P}_{\langle p,h\rangle}\left(\frac{1}{2}\right)<+\infty$ so that $\delta_{\langle p,h\rangle}\leqslant\frac{1}{2}$. Finally the subgroup 
$G=\langle p,h\rangle$ of $\G$ is of convergent type with critical exponent $\frac{1}{2}=\delta_{\langle p\rangle}$.

When $\mathfrak{h}$ goes to $+\infty$, the area of $\xx$ with constant curvature grows up and the Poincaré series of the group $\langle p\rangle$ 
tends to $+\infty$ at $\frac{1}{2}$. Since $\sum_{m\neq0}e^{-s\dhy_\mathfrak{h}(\oo,h^m.\oo)}<+\infty$ for any $s>0$, there exist 
$\varepsilon_0>0$ and $\mathfrak{h}_0>0$ such that for any $\varepsilon\in]0,\varepsilon_0[$ and $\mathfrak{h}\geqslant\mathfrak{h}_0$: 
$$\sum\limits_{m\in\Z^*}e^{-\left(\frac{1}{2}+\varepsilon\right)\dhy_\mathfrak{h}(\oo,h^m.\oo)}\sum\limits_{n\in\Z^*}
e^{-\left(\frac{1}{2}+\varepsilon\right)\dhy_\mathfrak{h}(\oo,p^n.\oo)}>1.$$
\noindent
Henceforth \eqref{étoilétoilesomme} implies that $G=\langle p,h\rangle$ satisfies $\delta_G>\frac{1}{2}$ and is 
divergent with finite Bowen-Margulis measure by Theorems A and B of \cite{DOP}.
\subsubsection{About divergent groups with infinite measure $m_\G$}
We now briefly explain the following proposition.
\begin{prop}\label{existenceduaetoile}
There exists a unique $\mathfrak{h}_*\in]0,\mathfrak{h}_0[$ such that the group $G=\langle p,h\rangle$ is divergent with critical exponent 
$\delta_G=\frac{1}{2}=\delta_{\langle p\rangle}$ with respect to the metric $\tilde{g}_{\beta,\mathfrak{h}_*}$. Moreover, $G$ has 
infinite measure when $0<\beta\leqslant1$\footnote{See Remark \ref{remarquebeta=1} for the case $\beta=1$.}.
\end{prop}
\noindent
This result seems to be an intermediate state between the two alternatives given in Proposition 
\ref{propMarcmetricetdivergence}. It relies on the comparison between the Poincaré series $\mathcal{P}_G$ 
(and more precisely \eqref{étoilesomme}) and the potential of a transfer operator $\ot_\mathfrak{h}$ associated to the 
action of $G$ on $\left(\xx,\tilde{g}_{\beta,\mathfrak{h}}\right)$ (see section 4 in \cite{Pe1}). 

Recall that $G=\langle p,h\rangle$, where $\langle p\rangle$ is convergent with critical exponent $\frac{1}{2}$. Fix 
$\mathfrak{h}\in[0,\mathfrak{h}_0]$. We formally introduce the following operator: for any 
$\varphi\in\mathbb{L}^{\infty}(\LL_G,\R)$, any $s>0$ and $x\in\LL_{G}$ 
$$\ot_{\mathfrak{h},s}\varphi(x)=\sum\limits_{\g\in\{p,h\}}\sum\limits_{n\in\Z^*}\un_{D_\g^c}(x)e^{-s\mathcal{B}_x^{(\mathfrak{h})}(\g^{-n}.\oo,\oo)}\varphi(\g^n.x)$$
\noindent
where $\mathcal{B}_x^{(\mathfrak{h})}(\g^{-n}.\oo,\oo)$ is the Busemann cocycle corresponding to the metric 
$\tilde{g}_{\beta,\mathfrak{h}}$. Using the fact that $p$ and $h$ are in Schottky position, we get for any 
$l\geqslant1$ and $s>0$,
$$\sum\limits_{m_i,n_i\in\Z^*}e^{-s\dhy_\mathfrak{h}(h^{m_1}p^{n_1}...h^{m_l}p^{n_l}.\oo,\oo)}
\asymp\left|\ot_{\mathfrak{h},s}^{2l}\un_{\LL_G}\right|_{\infty}$$
\noindent
where $\dhy_\mathfrak{h}$ is the distance induced by the metric $\tilde{g}_{\beta,\mathfrak{h}}$. By \eqref{étoilesomme}, 
this implies that $\mathcal{P}_G(s)$ behaves like 
$\sum_{l\geqslant1}\left|\ot_{\mathfrak{h},s}^{2l}\un_{\LL_G}\right|_{\infty}$. Since 
$\ot_{\mathfrak{h},s}$ is a positive operator, its spectral radius $\rho_{\infty}(\ot_{\mathfrak{h},s})$ on 
$\mathbb{L}^{\infty}(\LL_G,\R)$ is given by 
$$\limsup\limits_{l\longrightarrow+\infty}\left(\left|\ot_{\mathfrak{h},s}^{2l}\un_{\LL_G}\right|_{\infty}\right)^{\frac{1}{2l}}.$$
\noindent
Proposition \ref{propMarcmetricetdivergence} thus implies that
\begin{itemize}
 \item[$\bullet$] the series $\sum_{l\geqslant1}\left|\ot_{0,\frac{1}{2}}^{2l}\un_{\LL_G}\right|_{\infty}$ converges and $\rho_{\infty}(\ot_{0,\frac{1}{2}})\leqslant1$;
 \item[$\bullet$] the series $\sum_{l\geqslant1}\left|\ot_{\mathfrak{h},\frac{1}{2}}^{2l}\un_{\LL_G}\right|_{\infty}$ 
 diverges and $\rho_{\infty}(\ot_{\mathfrak{h},\frac{1}{2}})\geqslant1$ when $\mathfrak{h}\geqslant\mathfrak{h}_0$.
\end{itemize}
\noindent
In \cite{Pe1} is proved that there is a unique $\mathfrak{h}_*\in]0,\mathfrak{h}_0[$ such that 
$\rho_{\infty}(\ot_{\mathfrak{h}_*,\frac{1}{2}})=1$. The group $G$ will have 
the required properties acting on $\left(\xx,\tilde{g}_{\beta,\mathfrak{h}_*}\right)$. The existence of such 
$\mathfrak{h}_*$ is based on the regularity of the map $\mathfrak{h}\longmapsto\rho_{\infty}(\ot_{\mathfrak{h},\frac{1}{2}})$.
This is hard to obtain. Generally the map $\ot\longmapsto\rho_{\infty}(\ot)$ is only semi-lower continuous (\cite{Kato}). To overcome this lack of regularity, 
the author of \cite{Pe1} makes the family of operators 
$\left(\ot_{\mathfrak{h},\frac{1}{2}}\right)_{\mathfrak{h}\in[0,\mathfrak{h}_0]}$ act on the following subspace 
$\mathbb{L}_{\omega}\left(\LL_G\right)$ of 
$\mathbb{L}^{\infty}(\LL_G,\R)$ 
$$\mathbb{L}_{\omega}\left(\LL_G\right):=\left\{\varphi\in\mathcal{C}\left(\LL_G\right)\ |\ |\varphi|_\omega:=|\varphi|_\infty+\left[\varphi\right]_\omega<+\infty\right\}$$
\noindent
where 
$\left[\varphi\right]_\omega:=\underset{\g\in\{p,h\}}{\sup}
\underset{\underset{x\neq y}{x,y\in D_\g}}{\sup}\frac{|\varphi(x)-\varphi(y)|}{\dhy_\oo^{(\mathfrak{h})}(x,y)^\omega}$ 
and $\omega\in]0,1[$ is suitably chosen in order to satisfy Fact 3.7 of \cite{Pe1}. Here the distance 
$\dhy_\oo^{(\mathfrak{h})}(x,y)$ is the Gromov distance on the boundary of $\xx$ 
seen from the point $\oo\in\xx$ and associated with the metric $\tilde{g}_{\beta,\mathfrak{h}}$ (see Section 2.1). Denote by $\rho(a)$ the spectral radius of 
$\ot_{\mathfrak{h},\frac{1}{2}}$ on this space. M. Peigné then shows that this spectral radius is a simple isolated 
eigenvalue in the spectrum of 
$\ot_{\mathfrak{h},\frac{1}{2}}$ and equals to $\rho_{\infty}\left(\ot_{a,\frac{1}{2}}\right)$. This property of 
spectral gap for each 
$\ot_{\mathfrak{h},\frac{1}{2}}$, $\mathfrak{h}\in[0,\mathfrak{h}_0]$, is the key to obtain the required regularity of 
the map  $\mathfrak{h}\longmapsto\rho_{\infty}(\ot_{\mathfrak{h},\frac{1}{2}})$.

Finally, the proof of the main result (paragraph 4.5) explains with detailed  arguments why there exists 
$\mathfrak{h}_*\in[0,\mathfrak{h}_0]$ satisfying 
$\rho_{\infty}\left(\ot_{\mathfrak{h}_*,\frac{1}{2}}\right)=1$ and why $G=\langle p,h\rangle$ is divergent with critical 
exponent $\frac{1}{2}$ for the metric 
$\tilde{g}_{\beta,\mathfrak{h}_*}$. Since $G$ is convergent for $\tilde{g}_{\beta,0}$ and has critical exponent 
$>\frac{1}{2}$ for $\tilde{g}_{\beta,\mathfrak{h}_0}$ (because of Proposition \ref{propMarcmetricetdivergence}), then 
$\mathfrak{h}_*\in]0,\mathfrak{h}_0[$. Then we can apply the 
criterion of finiteness of $m_G$ given by Theorem B of \cite{DOP} above: we obtain that the group $G$ has 
infinite measure if $\beta\leqslant1$ (see the remark below). In the last section of \cite{Pe1}, the author also proves that the 
parameter $\mathfrak{h}_*$ is unique in $]0,\mathfrak{h}_0[$, which achieves the proof of Proposition \ref{existenceduaetoile}.
\begin{rem}\label{remarquebeta=1}
By Theorem B in \cite{DOP}, the measure $m_G$ is infinite if and only if 
$$\sum\limits_{p\in P}\dhy(\oo,p.\oo)e^{-\frac{1}{2}\dhy(\oo,p.\oo)}\asymp\sum\limits_{k\geqslant1}
\sum\limits_{\underset{\dhy(\oo,p.\oo)>k}{p\in P}}e^{-\frac{1}{2}\dhy(\oo,p.\oo)}
\asymp\sum\limits_{k\geqslant1}\frac{L(k)}{k^{\beta}}=+\infty.$$
\noindent
If we denote by 
$\tilde{L}_\beta(t)=\displaystyle{\int_1^t}\frac{L(x)}{x^{\beta}}\dd x$, then $m_G$ is infinite if and only if 
$\lim\limits_{t\longrightarrow+\infty}\tilde{L}_\beta(t)=+\infty$. The measure is thus infinite for 
$\beta<1$.
When $\beta=1$, it may not be true: for example, if we choose as slowly varying function $L=\ln^{-2}$, then using 
Proposition \ref{prop:PAsymp}, one gets
$\dhy(\oo,p^n.\oo)=2\ln(|n|)+4\ln\ln(|n|)+4\ln\ln\ln(|n|)+\varepsilon(n)$ with 
$\lim\limits_{n\longrightarrow+\infty}\varepsilon(n)=0$ and in this case
$\sum_{p\in P}\dhy(\oo,p.\oo)e^{-\frac{1}{2}\dhy(\oo,p.\oo)}\asymp\sum_{k\geqslant1}\dfrac{1}{k\ln(k)^2}$.
\noindent
In our setting of infinite measure $m_\G$, the function $\tilde{L}_\beta$ always will satisfy 
$\lim\limits_{x\longrightarrow+\infty}\tilde{L}_\beta(x)=+\infty$.
\end{rem}
\subsubsection{Comments}
Combining Paragraph 2.2.3 and Proposition \ref{existenceduaetoile}, we may notice that $G=\langle p,h\rangle$ is an 
example of group satisfying the family of assumptions $(H_\beta)$ given in the introduction (except about 
the number of Schottky factors: we may fix this, for instance, adding a hyperbolic generator to $G$). Indeed:
\begin{itemize}
 \item[$\bullet$] $G$ satisfies Assumption $(D)$ by Proposition \ref{existenceduaetoile};
 \item[$\bullet$]the subgroup $P=\langle p\rangle$ is convergent by the choice of the metric $g_{\beta,\mathfrak{h}^*}$; 
moreover Proposition \ref{existenceduaetoile} and Paragraph 2.2.3 ensure that $P$ satisfies Assumptions 
$(P_1)$, $(P_2)$ and $(S)$;
 \item[$\bullet$]similarly, Assumptions $(N)$ and $(S)$ are satisfied by the tail of the Poincaré series of the hyperbolic
group $\langle h\rangle$ at the exponent $\delta_G=\frac{1}{2}$: indeed, the critical exponent of this group is $0$, 
therefore the sums
$$\sum\limits_{n\ |\ \dhy(\oo,h^n.\oo)>T}e^{-\frac{1}{2}\dhy(\oo,h^n.\oo)}\ \text{and} 
\sum\limits_{n\ |\ T\leqslant\dhy(\oo,h^n.\oo)\leqslant T+\Delta}e^{-\frac{1}{2}\dhy(\oo,h^n.\oo)}$$
\noindent
behaves like $e^{-\frac{T}{4}l(h)}$ as $T\longrightarrow+\infty$ and $(N)$ and $(S)$ follow in this case.
\end{itemize}
%
%
\section{Regularly varying functions and stable laws}

This chapter is devoted to the statments of some properties of regularly varying functions and their applications to the study of stable laws. We first recall some facts about regularly varying functions. We then use them to the study of the local behaviour of the characteristic function of a probability law whose tail is controled by regularly varying functions.
\subsection{Slowly varying functions}
\subsubsection{Definitions and classical results}
 
\begin{defi}\
\begin{enumerate}
 \item[i)]A measurable function $L:\R^{+}\longrightarrow\R^{+}$ varies slowly at infinity if for any $x>0$
\begin{equation}\label{defislowly}\lim_{t\longrightarrow+\infty}\dfrac{L(xt)}{L(t)}=1.\end{equation}
 \item[ii)] A measurable function $U:\R^{+}\longrightarrow\R^{+}$ varies regularly with exponent $\beta\in\R$ if for any $t\in\R^+$, it satisfies
$U(t)=t^{\beta}L(t)$ with $L$ slowly varying.
\end{enumerate}
\end{defi}
\noindent
Notice that, by Theorem 1.2.1 in \cite{BGT}, the convergence 
\eqref{defislowly} is uniform in $x\in[a,b]$, for any compact interval $[a,b]\subset]0,+\infty[$.
\subsubsection{Karamata and Potter's lemmas}
The following lemma precises the property of integration of regularly varying functions.
\begin{lem}[Karamata]\label{regvar}
Let $\beta\in\R$ and $L\ :\ \R^+\longrightarrow\R^+$ be a slowly varying function.
\begin{itemize}
 \item[-] If $\beta>1$, then 
$$\int_x^{+\infty}\dfrac{L(y)}{y^\beta}\dd y\sim\dfrac{L(x)}{(\beta-1)x^{\beta-1}}\ when\ 
x\longrightarrow+\infty.$$
 \item[-] If $\beta\leqslant 1$, then the function $\tilde{L}_\beta\ :\ x\longmapsto\displaystyle{\int\limits_1^{x}}\frac{L(y)}{y^\beta}\dd y$ is regularly varying with exponent 
$1-\beta$; moreover, when $x\longrightarrow+\infty$
$$\tilde{L}_\beta(x)\sim\dfrac{x^{1-\beta}L(x)}{1-\beta}\ if\ \beta<1\ and\ 
\dfrac{L(x)}{\tilde{L}_1(x)}\longrightarrow0\ otherwise.$$
\end{itemize}
\end{lem}
\begin{rem}
This lemma is also true in the discrete case. If $\beta>1$, then 
$\sum_{n=N}^{+\infty}\dfrac{L(n)}{n^\alpha}\sim\dfrac{L(N)}{(\beta-1)N^{\alpha-1}}$ when 
$N\longrightarrow+\infty$. If $\beta\leqslant 1$, then, setting $\tilde{L}_\beta(N)= \sum_{n=1}^{N}\frac{L(n)}{n^\beta}$, one gets
$\tilde{L}_\beta(N)\sim\dfrac{N^{1-\beta}L(N)}{1-\beta}$ if $\beta<1$ and  
$\dfrac{L(N)}{\tilde{L}_1(N)}\longrightarrow0$ otherwise, when $N\longrightarrow+\infty$. In the sequel, to simplify the text, we will denote $\tilde{L}_1$ by $\tilde{L}$, when there is no possible confusion.
\end{rem}

The following lemma gives a control of the oscillations of a slowly varying function.
\begin{lem}[Potter's Bound]\label{PB}
If $L$ is a slowly varying function then for any fixed $B>1$ and $\rho>0$, there exists $T=T(B,\rho)$ such that for any $x,y\geqslant T$
$$\dfrac{L(x)}{L(y)}\leqslant B\max\left(\left(\dfrac{x}{y}\right)^{\rho},\left(\dfrac{x}{y}\right)^{-\rho}\right).$$
\end{lem}
For more details about these lemmas, we refer to \cite{BGT}.
\subsection{Applications}
\subsubsection{Local estimates for characteristic functions}
In this paragraph, we study the local behaviour of the characteristic function of probability laws, which are in the ``attraction domain'' of a stable law.
\begin{defi}
Let $\mu$ be a probability measure on $\R^+$. The probability measure $\mu$ is said to be stable if for any $a,b>0$ and any independent random variables $X,Y$ and $Z$ with law $\mu$, there exist $c>0$ and $\alpha\in\R$ such that the laws of the random variables $aX+bY$ and $cZ+\alpha$ are the same.
\end{defi}
This notion of stable law appears in the study of limit distributions of normalized sums 
\begin{equation}\label{sumandstablelaws}S_n=\dfrac{X_1+...+X_n}{a_n}-B_n\end{equation}
\noindent
of independent, identically distributed
random variables $X_1,X_2,...,X_n,...$, where $a_n>0$ and $B_n$ are suitably chosen real constants. We will focus on 
particular stable laws: a probability law is said to be fully asymmetric and stable with parameter $\beta\in(0,1)$ if its characteristic function is given by 
$g_{\beta}(t)=e^{-\G(1-\beta)e^{i\mathrm{sign}(t)\beta\pi/2}}|t|^{\beta}$, where $\G$ is the gamma function  (\cite{G-K} p.162). The density 
of such a distribution is a continuous function 
$\Psi_{\beta}$ supported on $[0,+\infty)$. If a probability law $\mu$ is such that a normalized sum of independent and identically distributed 
random variables $X_n$ with law $\mu$ converges in distribution to a stable law, we say that $\mu$ belongs to the domain of attraction of this stable law. Let us now give some information about the normalizing sequence $(a_n)_n$ appearing in sums of type \eqref{sumandstablelaws} in the case of stable law with parameter $\beta\in(0,1)$. Such
a sequence must satisfy $\frac{a_n^{\beta}}{L(a_n)}=n$ where $L$ is a slowly varying function (\cite{G-K}, p.180). In other terms, setting $A(x)=\frac{x^{\beta}}{L(x)}$, the 
sequence $(a_n)_n$ satisfies $A(a_n)=n$. By Proposition 1.5.12 in \cite{BGT}, there exists an increasing and regularly varying 
function $A^*$ with exponent $\frac{1}{\beta}$ such that $a_n=A^*(n)$. The function $A^*$ also satisfy $A^*(A(x))\sim A(A^*(x))\sim x$ when $x\longrightarrow+\infty$.

The following proposition explains the link between our setting and the class of stable laws with parameter $\beta$. 
\begin{prop}\label{bassinattractionloistable}
Let $\beta\in]0,1[$, $L$ be a slowly varying function and $\mu$ be a probability measure on $\R^+$. If the distribution function $F$ of the law $\mu$ satisfies $1-F(T)\sim\frac{L(T)}{T^\beta}$ when $T\longrightarrow+\infty$, then $\mu$ is in the domain of attraction of a fully asymmetric stable law with parameter $\beta$.
\end{prop}
We deduce from this proposition that a law $\mu$ whose density $f$ satisfies the asymptotic $f(x)\sim\frac{L(x)}{x^{1+\beta}}$ is in the domain of attraction of a fully asymmetric stable law with parameter $\beta$. Moreover, this asymptotic control of the distribution function $F$ is the one imposed on the tail of the Poincaré series of the Schottky factors $(\G_j)_j$ in assumptions $(P_2)$ 
and $(N)$. As we will see later, it will be of interest to precise the local expansion at $0$ of the 
characteristic function $\widehat{\mu}$ of $\mu$, which yields in our setting to a local expansion of the spectral radius of a family of transfer operators. First, we state the following proposition.
\begin{prop}\label{etapecontroltransfert}
Let $\beta\in]0,1]$ and $\mu$ be a probability measure on $\R^+$ whose distribution function $F(t):=\mu\left([0,T]\right)$ satisfies 
$$1-F(T)\asymp\dfrac{L(T)}{T^\beta}\ \left(\text{respectively}\ 
1-F(T)=o\left(\dfrac{L(T)}{T^\beta}\right)\right)\ \text{as}\ T\longrightarrow+\infty$$
\noindent
where $L$ is a slowly varying function. For any $\kappa>0$ and $t\in\R$
\begin{itemize}
 \item[1.] if $\beta<1$, then, as $t,\kappa\longrightarrow0$,
          \begin{align*}
          &\text{a)}\ \int_{0}^{+\infty}\left|e^{ity}-1\right|\mu(\dd y)\preceq|t|^\beta L\left(\dfrac{1}{|t|}\right)\ 
           \left(\text{resp.}\ =o\left(|t|^\beta L\left(\dfrac{1}{|t|}\right)\right)\right);\\
          &\text{b)}\ \int_{0}^{+\infty}\left|e^{-\kappa y}-1\right|\mu(\dd y)\preceq\kappa^\beta L\left(\dfrac{1}{\kappa}\right)\ 
           \left(\text{resp.}\ =o\left(\kappa^\beta L\left(\dfrac{1}{\kappa}\right)\right)\right).
          \end{align*}
 \item[2.] if $\beta=1$, then, as $t,\kappa\longrightarrow0$,
          \begin{align*}
          &\text{a)}\ \int_{0}^{+\infty}\left|e^{ity}-1\right|\mu(\dd y)\preceq|t|\tilde{L}\left(\dfrac{1}{|t|}\right)\ 
           \left(\text{resp.}\ =o\left(|t| \tilde{L}\left(\dfrac{1}{|t|}\right)\right)\right);\\
          &\text{b)}\ \int_{0}^{+\infty}\left|e^{-\kappa y}-1\right|\mu(\dd y)\preceq\kappa \tilde{L}\left(\dfrac{1}{\kappa}\right)\ 
           \left(\text{resp.}\ =o\left(\kappa \tilde{L}\left(\dfrac{1}{\kappa}\right)\right)\right).
          \end{align*}
\end{itemize}
\end{prop}
Its proof follows the one of Lemma 2 in \cite{Er}.
In the following proposition, we focus on the case $\beta=1$.
\begin{prop}\label{asymptoticISetIC}
Let $\mu$ be a probability measure on $\R^+$ with distribution function $F$; for any $t\in\R$ and $\kappa>0$, denote by 
$$I_S:=\int_0^{+\infty}e^{-\kappa y}\sin(ty)(1-F(y))\dd y$$
\noindent
and
$$I_C:=\int_0^{+\infty}e^{-\kappa y}\cos(ty)(1-F(y))\dd y.$$
\noindent
If $F$ satisfies $1-F(T)\sim\frac{L(T)}{T}$ when $T\longrightarrow+\infty$, one gets as $t,\kappa\longrightarrow0$
\begin{align*}
&i)\ \left|I_S\right|\preceq\dfrac{|t|}{\kappa}L\left(\dfrac{1}{\kappa}\right);\\
&ii)\ \left|I_S\right|\preceq L\left(\dfrac{1}{|t|}\right);\\
&iii)\ I_C=\tilde{L}\left(\dfrac{1}{\kappa}\right)(1+o(1))+O\left(\dfrac{|t|}{\kappa}L\left(\dfrac{1}{\kappa}\right)\right);\\
&iv)\ I_C=\tilde{L}\left(\dfrac{1}{|t|}\right)(1+o(1))+O\left(\dfrac{\kappa}{|t|}L\left(\dfrac{1}{|t|}\right)\right).
\end{align*}
\end{prop}
The proof of this proposition is the same as the one of Proposition 6.2 in \cite{MT}.
\subsubsection{Equivalence in Remark \ref{remark4}}
We now prove that Assertion $(S')$ implies $(S'')$ in Remark \ref{remark4}. Fix $j\in[\![1,p+q]\!]$. 
We thus assume that for any $\Delta>0$, there exists $C>0$ such that for any $T>0$, the quantity 
$\sharp\{\alpha\in\G_j\ |\ T-\Delta\leqslant\dhy(\oo,\alpha.\oo)<T+\Delta\}$ is bounded from above by 
$Ce^{\delta_\G T}\frac{L(T)}{T^{1+\beta}}$. We want to bound $\sharp\{\alpha\in\G_j\ |\ \dhy(\oo,\alpha.\oo)\leqslant T\}$ for any $T>0$. A ball centered in $\oo\in\xx$ of radius $T$ is the disjoint union of annulus of width $2\Delta$, so it is sufficient to prove that for any $N$ 
large enough, one gets
$$\Sigma:=\sum\limits_{n=1}^N\dfrac{e^{\delta_\G n}L(n)}{n^{1+\beta}}\preceq\dfrac{e^{\delta_\G N}L(N)}{N^{1+\beta}}.$$
\noindent
We split $\Sigma$ into $\Sigma_1+\Sigma_2$ where
$$\Sigma_1=\sum\limits_{n=1}^{\left[\frac{N}{2}\right]}\dfrac{e^{\delta_\G n}L(n)}{n^{1+\beta}}\ \text{and}\  
\Sigma_2=\sum\limits_{n=\left[\frac{N}{2}\right]+1}^{N}\dfrac{e^{\delta_\G n}L(n)}{n^{1+\beta}},$$
\noindent
where $[\cdot]$ denotes the integer part of $x$. By Karamata's lemma
\begin{equation}\label{5versusstar1}
\Sigma_1\preceq e^{\delta_\G\frac{N}{2}}\sum\limits_{n=1}^{\left[\frac{N}{2}\right]}L(n)
\preceq e^{\delta_\G\frac{N}{2}}\left[\frac{N}{2}\right]L\left(\left[\frac{N}{2}\right]\right)\preceq
NL(N)e^{\delta_\G\frac{N}{2}},
\end{equation}
\noindent
the last inequality following from Potter's lemma with $B=1$, $\rho=1$, $x=\left[\frac{N}{2}\right]$ and $y=N$. Similarly
\begin{equation}\label{5versusstar2}
\Sigma_2=\dfrac{L(N)}{N^{1+\beta}}\sum\limits_{n=\left[\frac{N}{2}\right]+1}^{N}e^{\delta_\G n}\dfrac{L(n)}{L(N)}\dfrac{N^{1+\beta}}{n^{1+\beta}}
\preceq\dfrac{L(N)}{N^{1+\beta}}\sum\limits_{n=\left[\frac{N}{2}\right]+1}^{N}e^{\delta_\G n}
\end{equation}
\noindent
where the last inequality may be deduced from Potter's lemma with $B,\rho=1$, $x=n$ and $y=N$.
The result follows combining \eqref{5versusstar1} and \eqref{5versusstar2} for $N$ large enough.

\section{Coding and transfer operators}

To study the geodesic flow $(g_t)_{t\in\R}$ on a negatively curved manifold, it is classical to conjugate it to a suspension over a shift on a symbolic space (see \cite{Bow} and \cite{PP}). The first part of this section is dedicated 
to the definition of the coding which will be used in 
Section 5.

In a second part, we introduce a family of transfer operators associated to this coding. This will be crucial in the sequel; we will need in particular a precise control on the regularity of this family of operators and their dominant eigenvalue: this will be done in the last part of this section. 
\subsection{Coding of the limit set and of the geodesic flow}
We recall now the setting of the study. We fix $p,q\geqslant1$ such that $p+q\geqslant3$ and consider $p+q$ discrete elementary subgroups $\G_1,...,\G_{p+q}$ of $\mathrm{Isom}(\xx)$ in Schottky position with $\G_1,...,\G_p$ parabolic and let $\G$ be the Schottky product of $\G_1,...,\G_{p+q}$.
We also consider the families of sets $(D_j)_{1\leqslant j\leqslant p+q}$ and $({\bf D}_j)_{1\leqslant j\leqslant p+q}$ introduced in Section 2. We will write $D:=\bigcup_{1\leqslant j\leqslant p+q}D_j$; notice that $D$ is a proper subset of $\partial\xx$. Since $\G=\G_1*...*\G_{p+q}$, any element $\g$ in $\G$ can be uniquely written as the product 
$\alpha_1...\alpha_k$ for some $\alpha_1,...,\alpha_k\in\bigcup_{j}\G_j\setminus\{\mathrm{Id}\}$ with the property that no two consecutive elements $\alpha_j$ and $\alpha_{j+1}$ belong 
to the same subgroup $\G_j,1\leqslant j\leqslant p+q$. The set $\mathcal{A}:=\bigcup_{j}\G_j\setminus\{\mathrm{Id}\}$ is called the \emph{alphabet} of $\G$, and $\alpha_1,...,\alpha_k$ 
will be called the \emph{letters} of $\g$; the word $\alpha_1...\alpha_k$ is said $\mathcal{A}$-admissible. 
The \emph{symbolic length} $|\g|$ of $\g$ is equal to the number $k$ of letters appearing in its decomposition with respect to $\mathcal{A}$.
For any $n\in\N\setminus\{0\}$, set $\G(n):=\{\g\in\G\ |\ |\g|=n\}$.  
Notice that both $\mathcal{A}$ and $\G(n)$ are infinite and countable. The initial and last letters of $\g$ play a special role, so the 
index of the group they belong to will be denoted by $i(\g)$ and $l(\g)$ respectively. 

Let us now give some geometrical properties of the action of such a Schottky group $\G$. First of all, combining Lemma \ref{quasiegalitetriangulaire} and 
the relative position of sets $\left({\bf D}_j\right)_{1\leqslant j\leqslant p+q}$, we deduce the corollaries below: the first one is a reformulation of Lemma \ref{quasiegalitetriangulaire} for triangles with two vertices in different sets ${\bf D}_j$; the second furnishes a well known improvement of the inequality $\mathcal{B}_x(\g^{-1}.\oo,\oo)\leqslant\dhy(\oo,\g^{-1}.\oo)$.
\begin{pro}\label{distsomtrigeod}
There exists a constant $C>0$, which depends only on the bounds of the curvature of $\xx$ and on $\G$, such that $\dhy(\oo,\g_1.\oo)+\dhy(\oo,\g_2.\oo)-C\leqslant\dhy(\g_1.\oo,\g_2.\oo)$ for any $\g_1,\g_2\in\G$ with $i(\g_1)\neq i(\g_2)$.
\end{pro}
\begin{rem}\label{sommefinie}
From this property and Assumptions $(P_1)$ and $(N)$, we deduce that the sum $\sum_{\g\in\G(n)}e^{-\delta_\G\dhy(\oo,\g.\oo)}$ is 
finite for any $n\geqslant1$: if $n=1$, this is a direct consequence of Assumptions $(P_1)$ and $(N)$; for $n\geqslant2$, 
Corollary \ref{distsomtrigeod} implies that 
there exists a constant $C>0$ such that for any $\gamma=\alpha_1...\alpha_n$, one gets
$$\dhy(\oo,\g.\oo)\geqslant \dhy(\oo,\alpha_1.\oo)+\dhy(\oo,\alpha_2.\oo)+...+\dhy(\oo,\alpha_{n}.\oo)-nC,$$
\noindent
so that
$$\sum_{\g\in\G(n)}e^{-\delta_\G\dhy(\oo,\g.\oo)}\leqslant e^{nC}\left(\sum\limits_{\alpha\in\mathcal{A}}e^{-\delta_\G\dhy(\oo,\alpha.\oo)}\right)^n<+\infty.$$
\end{rem}
\begin{pro}\label{busedist}
There exists a constant $C>0$ such that for any $\g\in\G$ and any $x\in\bigcup_{j\neq l(\g)}D_j$
$$\dhy(\oo,\g.\oo)-C\leqslant\bcal_{x}(\g^{-1}.\oo,\oo)\leqslant\dhy(\oo,\g.\oo).$$
\end{pro}
This estimate plays an important role in the proof of the following proposition, which allows us to bound from above the conformal coefficient of an 
isometry $\g\in\G$.
\begin{prop}\label{contract}
There exists a constant $C>0$ and $r\in]0,1[$ such that for any $\g\in\G(n)$, $n\geqslant1$, and for any $x\in\bigcup_{j\neq l(\g)}D_j$ 
$$|\g'(x)|_{\oo}\leqslant Cr^n.$$
\end{prop}
\begin{proof}
Recall that $|\g'(x)|_{\oo}=e^{-a\mathcal{B}_x(\g^{-1}.\oo,\oo)}$. By Property \ref{busedist}, it is sufficient 
to find a constant $A>0$ such that $\dhy(\oo,\g.\oo)\geqslant An$ for all $\g\in\G(n)$. Fix $n\geqslant1$. The 
$\G$-orbits accumulating at infinity, there exists an integer 
$l_0\geqslant1$ such that $\dhy(\oo,g.\oo)\geqslant 1+2C$ for any $g$ with symbolic length at least $l_0$ and where $C>0$ is given in 
Property \ref{distsomtrigeod}. We split the transformation $\g\in\G(n)$ into a product of transformations with length $l_0$. There are 
two cases:
\begin{itemize}
 \item[1)]if $n>2l_0$, we decompose $\g$ as $\g=\g_1...\g_k.\overline{\g}$ where $|\g_i|=l_0$ for any 
$1\leqslant i\leqslant k:=\left[\frac{n}{l_0}\right]$ and $|\overline{\g}|<l_0$. Therefore, Property \ref{distsomtrigeod} implies 
$$\dhy(\oo,\g.\oo)\geqslant\sum\limits_{1\leqslant i\leqslant k}\dhy(\oo,\g_i.\oo)+\dhy(\oo,\overline{\g}.\oo)-2kC$$
\noindent
which yields $\dhy(\oo,\g.\oo)\geqslant k\geqslant\frac{1}{2l_0}n$;
 \item[2)]if $n\leqslant2l_0$, the discretness of $\G$ implies 
 $$B:=\underset{1\leqslant n\leqslant2l_0}\inf\underset{\g\in\G(n)}\inf\dhy(\oo,\g.\oo)>0,$$
 \noindent
 hence $\dhy(\oo,\g.\oo)\geqslant B\geqslant\left(\frac{B}{2l_0}\right)n$.
\end{itemize}
\noindent
The result follows with $A:=\min\left(\frac{1}{2l_0},\frac{B}{2l_0}\right)$. 
\end{proof}
 \begin{coro}\label{actioncontractante}
 There exist $r\in]0,1[$ and a constant $C>0$ such that for any $n\geqslant1$, any $\g\in\G(n)$ and $x,y\in\LL_\G\cap\left(\partial\xx\setminus D_{l(\g)}\right)$, 
 $$\dhy_{\oo}(\g.x,\g.y)\leqslant Cr^n\dhy_{\oo}(x,y).$$
\end{coro}
\subsubsection{Coding of the limit set}
Denote by $\Sigma_{\mathcal{A}}^+$ the set of \emph{$\mathcal{A}$-admissible sequences} $(\alpha_n)_{n\geqslant1}$, {\it i.e.} sequences for which each letter $\alpha_n$ belongs to the 
alphabet $\mathcal{A}$ and such that no two consecutive letters belong to the same subgroup $\G_j,1\leqslant j\leqslant p+q$. Fix a point $x_0\in\partial\xx\setminus D$. 
We may find in \cite{BP2} the following result.
\begin{prop}\
\begin{itemize}
 \item[1)]For any $\boldsymbol{\alpha}=(\alpha_n)\in\Sigma_{\mathcal{A}}^+$, the sequence $\left(\alpha_1...\alpha_n.x_0\right)_n$ converges to a point 
$\pi(\boldsymbol{\alpha})\in\LL_\G$, which does not depend on the choice of $x_0$.
 \item[2)]The map $\pi\ :\ \Sigma_{\mathcal{A}}^+\longrightarrow\LL_{\G}$ is one-to-one and $\pi\left(\Sigma_{\mathcal{A}}^+\right)$ is included in the radial limit set of $\G$.
 \item[3)]The complement of $\pi\left(\Sigma_{\mathcal{A}}^+\right)$ in $\LL_{\G}$ is countable and consists of the $\G$-orbit of the union of the limit sets 
 $\LL_{\G_j}$ (each of which being finite here). In particular $\sigma_\oo\left(\LL_\G\setminus\pi\left(\Sigma_{\mathcal{A}}^+\right)\right)=0$, where $\sigma_\oo$ is the Patterson-Sullivan measure on $\partial\xx$.
\end{itemize}
\end{prop}
Let us denote $\LL^0:=\pi\left(\Sigma_{\mathcal{A}}^+\right)$ and $\LL_j^0:=\LL^0\cap D_j=\{\pi(\boldsymbol{\alpha})\ |\ \alpha_1\in\G_j^*\}$ for $1\leqslant j\leqslant p+q$. 
Let us emphasize that $\LL_j^0$ does not equal to the
limit set of the group $\G_j$, but nevertheless $\LL_{\G_j}\subset\overline{\LL_j^0}=\LL_\G\cap D_j$. The sets 
$D_j$, $1\leqslant j\leqslant p+q$, being disjoint, the sets $\left(\LL_j^0\right)_j$ have disjoint closures. The following description of $\LL^0$ will be useful:
\begin{enumerate}
 \item[(1)] $\LL^0$ is the finite union of the sets $\LL_1^0,...,\LL_{p+q}^0$; 
 \item[(2)] each of sets $\LL_j^0$ is partitioned into a countable number of subsets with disjoint closures: indeed, for any $j\in[\![1,p+q]\!]$
$$\LL_j^0=\bigcup\limits_{\alpha\in\G_j^*}\bigcup\limits_{k\neq j} \alpha.\LL_k^0.$$
\end{enumerate}
\noindent
The \emph{shift} $\Theta$ on the symbolic space $\Sigma_{\mathcal{A}}^+$ is defined by:
$$\forall\boldsymbol{\alpha}\in\Sigma_{\mathcal{A}}^+,\ \Theta(\boldsymbol{\alpha})=(\alpha_{k+1})_{k\geqslant1}.$$
\noindent
This operator $\Theta$ induces a transformation $\T :\LL^0\rightarrow\LL^0$ whose action is defined for all  
$x=\pi\left(\boldsymbol{\alpha}\right)$ by
$$\T.x=\alpha_1^{-1}.x.$$
\noindent
As a consequence of Corollary \ref{actioncontractante}, the map $\T$ is expanding on $\LL^ 0$ (see Corollary II.4 in \cite{DP}).
\subsubsection{Coding of the geodesic flow}
In Paragraph 2.1.3, we recalled how to define the action of the geodesic flow $(g_t)_{t\in\R}$ on  $\partial\xx\overset{\Delta}{\times}\partial\xx\times\R/\G\simeq\T^1\mm$. We propose here a coding of the geodesic flow on a $(g_t)$-invariant subset $\Omega^0$ of $\Omega_\G$ defined by 
$\Omega^0:=\LL^0\overset{\Delta}{\times}\LL^0\times\R/\G$. We first conjugate the action of $\G$ on $\LL^0\overset{\Delta}{\times}\LL^0\times\R$ 
with the action of a single transformation. Observe that the subset $\mathcal{D}^0:=\bigcup_{k\neq j}\LL_k^0\times\LL_j^0$ of 
$\LL^0\overset{\Delta}{\times}\LL^0$ is in one-to-one correspondence with the symbolic space $\Sigma_{\mathcal{A}}$ of 
bi-infinite $\mathcal{A}$-admissible sequences $\left(\alpha_n\right)_{n\in\Z}$. Moreover the shift of $\Sigma_{\mathcal{A}}$ induces a transformation still denoted by $\T$ on this set $\mathcal{D}^0$ whose action is
given by
$$\T.(y,x)=(\alpha_1^{-1}.y,\alpha_1^{-1}.x)\ \mathrm{if}\ x=\pi(\boldsymbol{\alpha}).$$
\noindent
In \cite{BP} it is proved that the action of $\G$ on $\LL^0\overset{\Delta}{\times}\LL^0$ is orbit-equivalent with the action of $\T$ on $\mathcal{D}^0$. 
Similarly, the action of $\G$ on the space $\LL^0\overset{\Delta}{\times}\LL^0\times\R$ is orbit-equivalent to the action of the transformation $\T_{\gol}$ on 
$\mathcal{D}^0\times\R$ given by
\begin{equation}\label{actionT_gol}
\T_{\gol}.(y,x,r)=(\T.(y,x),r-\gol(x))
\end{equation}
\noindent
where $\gol(x)=-\mathcal{B}_x(\alpha_1.\oo,\oo)$ when $x=\pi(\boldsymbol{\alpha})$. Let us write $S_k\gol(x)=\gol(x)+\gol(\T.x)+...+\gol(\T^{k-1}.x)$ for all $k\in\N\setminus\{0\}$ and all $x\in\LL^0$. For all $k\geqslant$, one gets
\begin{equation}\label{rajoutsamuel}\T_{\gol}^k.(y,x,r)=\left(\T^k.(y,x),r-S_k\gol(x)\right).\end{equation}
\noindent
When $\gol>0$, a fundamental domain $\mathcal{D}_{\gol}^0$ for the action of $\T_{\gol}$ on $\mathcal{D}^0\times\R$ is given by 
$$\mathcal{D}_{\gol}^0=\left\{(y,x,r)\in\mathcal{D}^0\times\R\ |\ 0\leqslant r<\gol(x)\right\}.$$
\begin{figure}[htbp]
\begin{center}
\includegraphics[height=6cm]{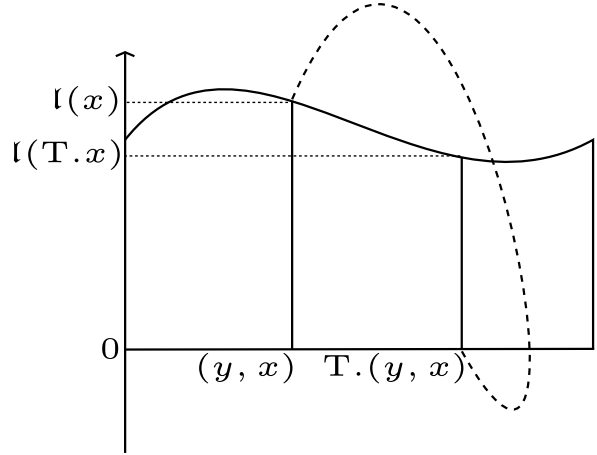}
\caption{\label{domfondgol}Action of $\T_\gol$ when $\gol>0$}
\end{center}
\end{figure}
In general, the function $\gol$ is not positive; nevertheless we have the following property 
\begin{lem}\label{proprietesgol}
The roof function $\gol$ satisfies: 
\begin{itemize}
 \item[$\bullet$]$\gol$ is uniformly bounded from below by $-C$, where the constant $C>0$ depends only on $\xx$ and $\G$; 
 \item[$\bullet$]there exists $k_0\geqslant1$ such that $S_k\gol(x)>0$ for any 
 $k\geqslant k_0$ and $x\in\LL^0$.
\end{itemize}
\end{lem}
\begin{proof}
Let $x\in\LL^0$ and $\boldsymbol{\alpha}\in\Sigma_{\mathcal{A}}^+$ such that $x=\pi\left(\boldsymbol{\alpha}\right)$. Since 
$\gol(x)=\mathcal{B}_{\alpha_1^{-1}.x}\left(\alpha_1^{-1}.\oo,\oo\right)$, Corollary \ref{busedist} implies 
$\gol(x)\geqslant\dhy(\oo,\alpha_1^{-1}.\oo)-C$, which proves the first assertion. Concerning the second one, let us notice that
\begin{align*}
S_k\gol(x)= & \mathcal{B}_{\alpha_1^{-1}.x}(\alpha_1^{-1}.\oo,\oo)+\mathcal{B}_{\alpha_2^{-1}\alpha_1^{-1}.x}(\alpha_2^{-1}.\oo,\oo)+...+
\mathcal{B}_{\alpha_k^{-1}...\alpha_1^{-1}.x}(\alpha_k^{-1}.\oo,\oo)\\
= & \mathcal{B}_{\alpha_k^{-1}...\alpha_1^{-1}.x}(\alpha_k^{-1}...\alpha_1^{-1}.\oo,\oo)\geqslant\dhy(\oo,\alpha_k^{-1}...\alpha_1^{-1}.\oo)-C.
\end{align*}
\noindent
Since the group $\G$ is discrete, the sums $S_k\gol(x)$ are positive for $k$ large enough, uniformly
in $x\in\LL^0$.
\end{proof}
Using this lemma, a classical argument in Ergodic Theory allows us to explicit a fundamental domain for the action of $\T_{\gol}$. 
\begin{prop}
The function $\gol$ is \emph{cohomologous} to a positive function $\mathfrak{L}$, {\it i.e.} there exists a measurable function $f\ :\ \LL^0\longrightarrow\R$ such that $\gol=\mathfrak{L}+f-f\circ\T$.
\end{prop}
\begin{proof}
By Lemma \ref{proprietesgol}, there exists $k_0\geqslant1$ such that for any $k\geqslant k_0$ and any $x\in\LL^0$, one gets
$S_k\gol(x)>0$. Denote by $\varepsilon=\frac{1}{k_0}$ and let us introduce
$a_i=1-i\varepsilon$ for any $i\in[\![0,k_0]\!]$. We notice that $a_0=1$, $a_{k_0}=0$ and $a_i-a_{i-1}=-\varepsilon$ for any $i\in[\![1,k_0]\!]$. Fix $x\in\LL^0$ and set 
$$f(x)=\sum\limits_{i=0}^{k_0-1}a_i\gol\left(\T^i.x\right).$$
\noindent
Therefore
\begin{align*}
f(x)-f\left(\T.x\right) & = \sum\limits_{i=0}^{k_0-1}a_i\gol\left(\T^i.x\right)-\sum\limits_{i=0}^{k_0-1}a_i\gol\left(\T^{i+1}.x\right)\\
& = a_0\gol(x)-a_{k_0}\gol\left(\T^{k_0}.x\right)+\sum\limits_{i=1}^{k_0}a_i\gol\left(\T^i.x\right)-\sum\limits_{i=1}^{k_0}a_{i-1}\gol\left(\T^i.x\right),
\end{align*}
\noindent
which yields
$$f(x)-f\left(\T.x\right)=\gol(x)-\varepsilon\sum\limits_{i=1}^{k_0}\gol\left(\T^i.x\right).$$
\noindent
Let us denote 
$\mathfrak{L}(x)=\varepsilon\sum\limits_{i=1}^{k_0}\gol\left(\T^i.x\right)=\dfrac{1}{k_0}S_{k_0}\gol\left(\T.x\right)$;
Lemma \ref{proprietesgol} implies that the function $\mathfrak{L}$ is positive, which ends the proof of the lemma.
\end{proof}
\noindent
The set
$$\mathcal{D}_{f,\mathfrak{L}}^0=\left\{(y,x,r)\in\mathcal{D}^0\times\R\ |\ f(x)\leqslant r< f(x)+\mathfrak{L}(x)\right\}$$
\noindent
is a fundamental domain for the action of $\T_{\gol}$ on $\mathcal{D}^0\times\R$: indeed
$$\T_{\gol}.\left(y,x,f(x)+\mathfrak{L}(x)\right)=\left(\T.(y,x),f(x)+\mathfrak{L}(x)-\gol(x)\right)=
\left(\T.(y,x),f\left(\T.x\right)\right).$$
\begin{figure}[htbp]
\begin{center}
\includegraphics[height=6cm]{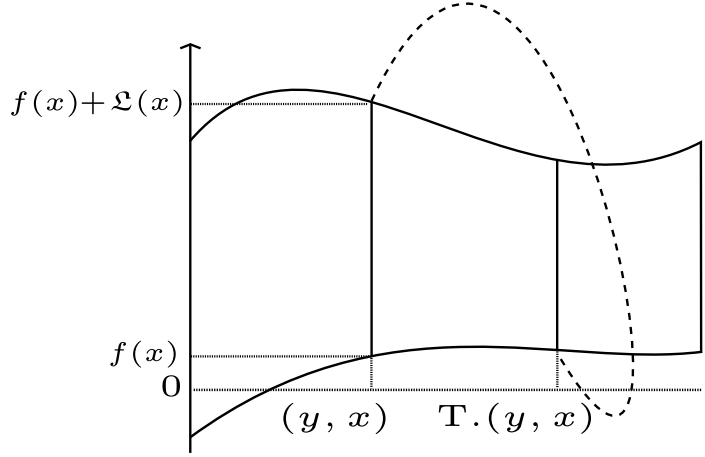}
\caption{\label{domfondgoldeux}Domain $\mathcal{D}_{f,\mathfrak{L}}^0$}
\end{center}
\end{figure}
Let $\tilde{\phi}_t$ denote the transformation, whose action on triplets $(y,x,r)\in\mathcal{D}^0\times\R$ is given by translation 
of $t$ on the third coordinate. The actions of $\left(\tilde{\phi}_t\right)_t$ and $\T_{\gol}$ commute and 
define a \emph{special flow} $\left(\phi_t\right)_t$ on 
$\mathcal{D}^0\times\R/\langle\T_{\gol}\rangle$. Identifying $\mathcal{D}^0\times\R/\langle\T_{\gol}\rangle$ with
$\mathcal{D}_{f,\mathfrak{L}}^0$, for any $(y,x,r)\in\mathcal{D}_{f,\mathfrak{L}}^0$ and $t>0$, one gets from \eqref{rajoutsamuel}
\begin{equation}\label{actiondeT_golsurlesousdécalage}
\phi_t.\left(y,x,r\right)=\left(y,x,r+t\right)=\left(\T^k.(y,x),r+t-S_k\mathfrak{l}(x)\right)=\T_{\gol}^k.\left(y,x,r+t\right),
\end{equation}
\noindent
where $k\in\Z$ is the unique integer such that $f(x)\leqslant r+t-S_k\mathfrak{l}(x)<f(x)+\mathfrak{L}(x)$. We finally deduce the following lemma
\begin{lem}\label{representationsymboliqueflot}\
\begin{enumerate}
 \item[i)]The spaces $\LL^0\overset{\Delta}{\times}\LL^0\times\R/\G$ and $\mathcal{D}^0\times\R/\langle\T_{\gol}\rangle$ are in one-to-one correspondence.
 \item[i)]The geodesic flow on $\Omega^0$ is conjugated to the special flow on $\mathcal{D}^0\times\R/\langle\T_{\gol}\rangle$.
\end{enumerate}
\end{lem}
This lemma implies in particular that there is a one-to-one correspondence between the primitive periodic orbits of the geodesic flow on $\T^1\xx/\G$ and 
the primitive periodic orbits of the special flow $(\phi_t)_t$ on $\mathcal{D}^0\times\R/\langle\T_\gol\rangle$. 
This correspondence allows us to characterize periodic orbits of $(g_t)_t$. Let $L>0$ and $(y,x,f(x))\in\mathcal{D}_{f,\mathfrak{L}}^0$ a 
$\phi_L$-periodic triplet: the equality $\phi_L.\left(y,x,f(x)\right)=(y,x,f(x))$ may be written
\begin{align*}
(y,x,f(x)+L)\sim(y,x,f(x)).
\end{align*}
\noindent
Therefore $L\geqslant\mathfrak{L}(x)$ and there exists an integer $k\geqslant1$ satisfying
$$f(x)\leqslant f(x)+L-S_k\mathfrak{l}(x)<f(x)+\mathfrak{L}(x).$$
\noindent
The unique representative of $(y,x,f(x)+L)$ in $\mathcal{D}_{f,\mathfrak{L}}^0$ is given by
$$\left(\T^k.y,\T^k.x,f(x)+L-S_k\mathfrak{l}(x)\right);$$
\noindent
it follows
$$\T^k.(y,x)=(y,x)\ \text{and}\ L=S_k\mathfrak{l}(x).$$
\noindent
These equalities determine $k$ periodic couples for $\T^k$
$$(y,x),\left(\T.y,\T.x\right),\dots,\left(\T^{k-1}.y,\T^{k-1}.x\right)$$
\noindent
in $\mathcal{D}^0$ and the length $L$ of the orbit is given by $L=S_k\mathfrak{l}(x)$. 

Furthermore, the closed geodesics on $\xx/\G$ are in one-to-one correspondence with the periodic orbits of the geodesic flow
$(g_t)_t$ on $\T^1\xx/\G$. Let $\mathfrak{g}$ be a closed geodesic. If it is not the projection of the axis of a hyperbolic isometry of some 
Schottky factor $\G_j$, $j\in[\![p+1,p+q]\!]$, a lift of $\mathfrak{g}$ in
$\mathcal{D}^0\times\R/\langle\T_\gol\rangle$ corresponds to a periodic orbit for the special flow. There thus exist 
$(y,x)\in\mathcal{D}^0$ and $k\geqslant2$ such that 
$$\T^k.(y,x)=(y,x)\ \mathrm{and}\ l(\mathfrak{g})=S_{k}\gol(x),$$
\noindent
where $l(\mathfrak{g})$ is the length of $\mathfrak{g}$. The couple $(y,x)$ is associated to 
a $k$-periodic two-sided admissible sequence $\left(\alpha_n\right)_{n\in\Z}$ of $\Sigma_\mathcal{A}$: the point $x$ is the attractive fixed point 
of the hyperbolic isometry $\alpha_1\alpha_2...\alpha_k$ and the closed geodesic $\mathfrak{g}$ is the projection of the axis of this isometry.

We may notice that the couples $(y,x),\T.(y,x),\dots,\T^{k-1}.(y,x)$ lead to the same orbit of the special flow, and thus define the same 
closed geodesic. We conclude that the closed geodesics which do not correspond to hyperbolic generators of $\G$ are in one-to-one correspondence 
with the orbits of $\T^k$-periodic couples of $\mathcal{D}^0$, $k\geqslant2$. 
\subsubsection{The dynamical system $(\LL^0,\mathrm{T},\nu)$}
Recall that by the identification $\T^1\xx\simeq\partial\xx\overset{\Delta}{\times}\partial\xx\times\R$ given by Hopf coordinates (see Paragraph 2.1.3), the Bowen-Margulis measure $m_\G$ on $\T^1\xx/\G$ is given by the quotient under the action of $\G$ of the measure $\tilde{m}_\G$ defined by 
$$\dd\tilde{m}_\G(y,x,t)=\dfrac{\dd\sigma_{\oo}(y)\dd\sigma_{\oo}(x)}{\dhy_{\oo}(y,x)^{\frac{2\delta_\G}{a}}}\dd t=\dd\mu(y,x)\dd t$$
\noindent
where $\sigma_\oo$ is the Patterson-Sullivan measure seen from $\oo$, $\dhy_\oo$ the Gromov distance on $\partial\xx$ and $a>0$ is such that the curvatures of $\xx$ is less than 
$-a^2$. Notice that under our hypotheses, the measure $m_\G$ has infinite total mass and 
$m_\G\left(\T^1\mm\setminus\Omega^0\right)=0$. Up to multiplying the Patterson density $\left(\sigma_{{\bf x}}\right)_{{\bf x}\in\xx}$ by a constant, we may assume that 
$\mu\left(\mathcal{D}^0\right)=1$. 

The dynamical system $\left(\Omega^0,\left(g_t\right)_{t\in\R},m_\G\right)$ is conjugated to the special flow 
$\left(\mathcal{D}^0\times\R/\langle\T_\gol\rangle,\left(\phi_t\right)_{t\in\R},\overline{m}_\G\right)$ where 
$\overline{m}_\G$ denotes the projection of $\left(\tilde{m}_\G\right)_{|\mathcal{D}^0\times\R}$ to $\mathcal{D}^0\times\R/\langle\T_\gol\rangle$ under the action of $\T_\gol$.

We have $\mathcal{D}^0\subset\LL^0\times\LL^0$ and may consider the measure $\nu$ on $\LL^0$, obtained as the projection of $\mu_{|\mathcal{D}^0}$ on the second coordinates. This measure $\nu$ is absolutely continuous with respect to the Patterson-Sullivan measure $\sigma_\oo$.
\begin{prop}\label{defifonctionh}
The map $h\ :\ \LL^0\longrightarrow\R_+^*$ defined by: for all $j\in\{1,...,p+q\}$ and $x\in\LL_j^0$ 
\begin{equation}
h(x)=\int\limits_{\bigcup\limits_{l\neq j}\LL_l^0}\dfrac{\dd\sigma_{\oo}(y)}{\dhy_{\oo}(x,y)^{\frac{2\delta_\G}{a}}},
\end{equation}
\noindent
is the density of $\nu$ with respect to $\sigma_\oo$; moreover, the measure $\nu$ is $\T$-invariant.
\end{prop}
\begin{proof}
Let $\varphi\ :\ \LL^0\longrightarrow\R$ be a Borel function. Denote by $p$ the projection 
$$p\ :\ \left\{\begin{array}{ll}
               \mathcal{D}^0 & \longrightarrow\LL^0\\
               (x,y) & \longmapsto y
               \end{array}\right..
$$
\noindent
We write
\begin{align*}
\nu\left(\varphi\right)=\int_{\LL^0}\varphi(x)\nu(\dd x)= &
\sum\limits_{j=1}^{p+q}\int_{\LL_j^0}\int_{\LL^0\setminus\LL_j^0}\varphi(x)\dd\mu(x,y)\\
= &\sum\limits_{j=1}^{p+q}\int_{\LL_j^0}\int_{\LL^0\setminus\LL_j^0}\varphi(x)\dfrac{\dd\sigma_{\oo}(x)\dd\sigma_{\oo}(y)}{\dhy_{\oo}(x,y)^{\frac{2\delta_\G}{a}}}\\
= &\sum\limits_{j=1}^{p+q}\int_{\LL_j^0}\varphi(x)h(x)\dd\sigma_\oo(x)=\int_{\LL^0}h(x)\varphi(x)\dd\sigma_\oo(x).
\end{align*}

Since the measure $\mu$ is $\G$-invariant, the measure $\mu_{|\mathcal{D}^0}$ is $\T$-invariant on $\mathcal{D}^0$: indeed, the family
$\left(\alpha\LL_i\times\beta\LL_j\right)_{i\neq j,\alpha,\beta\in\G}$ is a partition of $\mathcal{D}^0$ and the action of $\T$ on 
$\mathcal{D}^0$ is given by the action of an isometry $\g\in\G$ on each atom of this partition. 
\end{proof}
\begin{rem}
We may extend the function $h$ on $\LL_\G$ setting for any $j\in[\![1,p+q]\!]$ and $x\in\overline{\LL_j^0}$
$$h(x)=\dfrac{1}{\mu\left(\mathcal{D}^0\right)}\int\limits_{\bigcup\limits_{l\neq j}\LL_l^0}\dfrac{\dd\sigma_{\oo}(y)}{\dhy_{\oo}(x,y)^{\frac{2\delta_\G}{a}}}.$$
\end{rem}

\subsection{On transfer operators}
Let us now introduce a family of transfer operators $(\ot_z)$ associated to the transformation $\T$. The interest of such operators to study hyperbolic flows has already been widely illustrated in the literature, see for instance \cite{Bow}, \cite{PP} or \cite{Zin}. These operators use the non-injectivity of the shift on $\Sigma_{\mathcal{A}}^+$ to describe the dynamic of 
$\T$ on $\LL^0$. To define these transfer operators, we will associate a weight to each point $x\in\LL^0$ to take into account the number of its antecedents under the action of $\T$. The operators $\ot_z$ may thus be seen as transition kernels ruled by the action of inverse branches of $\T$ on $\LL^0$. Hence, for $n$ large enough, the study of $\T^n$ is strongly related to the behaviour of a Markov chain on $\LL^0$ (see \cite{BDP} and \cite{HeHe}).

We first present the family of transfer operators which we will consider and the spaces on which they act. We then study their relationship with the dynamical system $\left(\LL^0,\T,\nu\right)$ presented in the previous paragraph. Eventually, we study the spectral properties and the regularity of the family of operators. 
\subsubsection{Definition and first properties}
Let us introduce the family $(\ot_z)$ of \emph{transfer operators} defined formally for a 
parameter $z\in\C$ and a function $\varphi\ :\ \LL_\G\longrightarrow\C$ by 
$$\ot_z\varphi(x)=\sum\limits_{\T.y=x}e^{-z\gol(y)}\varphi(y).$$
\noindent
These operators are associated with the roof-function $\gol$ defined in \eqref{actionT_gol} and allows us to describe the dynamical system $(\LL^0,\T,\nu)$ from an analytic viewpoint. 
We first explicit them more precisely and check that it acts on the space $\mathcal{C}(\LL_\G)$ of continuous functions from $\LL_\G$ to $\C$ equipped with the norm of uniform convergence $|\cdot|_{\infty}$. To get some critical gap property for their spectrum, we will consider their restriction to some subspace of $\left(\mathcal{C}(\LL_\G),|\cdot|_{\infty}\right)$. 

In order to enlight the text, we will denote by $\LL$ the limit set $\LL_\G$, by 
$\LL_j$ the set $\overline{\LL_j^0}=D_j\cap\LL_\G$ for any $j\in[\![1,p+q]\!]$ and $\delta=\delta_\G$; similarly, the quantity $b(\g,x)$ will stand for $\mathcal{B}_x(\g^{-1}.\oo,\oo)$ for any $x\in\partial\xx$ and $\g\in\G$.

Fix $z\in\C$ and $l\in[\![1,p+q]\!]$. If $x$ belongs to $\LL_l^0$, its pre-images by $\T$ are the points $y=\alpha.x$ with 
$\alpha\in\bigcup_{j\neq l}\G_j^*$ and $\gol(y)=b(\alpha,x)$. Consequently for any bounded Borel function $\varphi\ :\ \LL\longrightarrow\C$
\begin{equation}\label{defoptrans}\ot_z\varphi(x)=\sum\limits_{\underset{j\neq l}{\alpha\in\G_j^*}}e^{-zb(\alpha,x)}\varphi(\alpha.x)=
\sum\limits_{j=1}^{p+q}\un_{\LL_j^c}(x)\sum\limits_{\alpha\in\G_j^*}e^{-zb(\alpha,x)}\varphi(\alpha.x).\end{equation}
\noindent
Combining Hypotheses $(P_2)$ and $(N)$ with Lemma \ref{busedist}, we notice that this 
quantity is finite for $\re{z}\geqslant\delta$ and defines a continuous function 
on $\LL^0$. Since the convergence of the series appearing in \eqref{defoptrans} is normal on $\LL$ when $\re{z}\geqslant\delta$, the 
function $\ot_z\varphi$ may be continuously extended on $\LL$.
The operator $\ot_\delta$ is positive on $\LL$. Furthermore
\begin{lem}\label{dualityTtransfer}
The function $h$ defined in Proposition \ref{defifonctionh} satisfies $\ot_\delta h=h$. Moreover,
for any $\varphi,\psi$ in $\mathcal{C}(\LL)$
$$\int\limits_{\LL}\varphi(x)\psi(\T.x)\sigma_{\oo}(\dd x)=\int\limits_{\LL}\ot_{\delta}\varphi(x)\psi(x)\sigma_{\oo}(\dd x).$$
\noindent
In particular, the measure $\sigma_\oo$ is $\ot_\delta$-invariant. 
\end{lem}
\begin{proof}
The property $\ot_\delta h=h$ follows from Section 7 in \cite{DP}. Now, using Property \eqref{deltaconformité} of the family of measures $\left(\sigma_{{\bf x}}\right)_{{\bf x}\in\xx}$, we obtain
\begin{align*}
\int\limits_{\LL}\varphi(x)\psi(\T.x)\sigma_{\oo}(\dd x)= & 
\sum\limits_{j=1}^{p+q}\sum\limits_{\alpha\in\G_j^*}\int_{\alpha.\left(\LL\setminus\LL_j\right)}\varphi(x)\psi(\alpha^{-1}.x)\sigma_{\oo}(\dd x)\\
= &\sum\limits_{j=1}^{p+q}\sum\limits_{\alpha\in\G_j^*}\int_{\LL}\un_{\LL_j^c}(y)
\varphi(\alpha.y)\psi(y)e^{-\delta\mathcal{B}_y(\alpha^{-1}.\oo,\oo)}\sigma_{\oo}(\dd y)\\
= & \int\limits_{\LL}\ot_{\delta}\varphi(y)\psi(y)\sigma_{\oo}(\dd y).
\end{align*}
The equality $\sigma_\oo\ot_\delta=\sigma_\oo$ follows with $\psi=\un_\LL$.
\end{proof}
\noindent
Let us now introduce the normalized operator $P:=\frac{1}{h}\ot_\delta\left(h\cdot\right)$. It is a positive Markov operator,
({\it i.e.} $P\un_\LL=\un_\LL$).
By the previous lemma, we deduce that $P$ satisfies: for any $\varphi,\psi$ in $\mathcal{C}(\LL)$
\begin{equation}
\int\limits_{\LL}\varphi(x)\psi(\T.x)\nu(\dd x)=\int\limits_{\LL}P\varphi(x)\psi(x)\nu(\dd x),
\end{equation}
\noindent
where $\nu=h\sigma_\oo$. A similar property will be useful in the proof of Theorem A
for some suitable extension $\tilde{P}$ of the operator $P$ defined as follows: for all continuous function $\varphi\ :\ \LL\longrightarrow\R$ and $u:\ \R\longrightarrow\R$ with compact support, any $x\in\LL$ and $t\in\R$
\begin{equation}\label{defiP1}\tilde{P}(\varphi\otimes u)(x,t)=\sum\limits_{j=1}^{p+q}\un_{\LL_j^c}(x)\sum\limits_{\alpha\in\G_j^*}
e^{-\delta b(\alpha,x)}\dfrac{h(\alpha.x)\varphi(\alpha.x)}{h(x)}u(t+b(\alpha,x)).\end{equation}
\noindent
By density, the operator $\tilde{P}$ extends continuously to the space of continuous maps with compact support on $\LL\times\R$.
\begin{lem}\label{defiP}
The operator $\tilde{P}$ is the adjoint of the transformation $(x,s)\longmapsto(\T.x,s-\gol(x))$ with respect to the measure $\nu\otimes\dd s$, \emph{i.e.} for all continuous maps $\Phi,\Psi$ on $\LL\times\R$ with compact support
$$\int_{\LL\times\R}\Phi(x,s)\Psi(\T.x,s-\gol(x))\dd\nu(x)\dd s=
\int_{\LL\times\R}\tilde{P}\Phi(x,s)\Psi(x,s)\dd\nu(x)\dd s.$$
\end{lem}
\noindent
In order to study the regularity of the family $\left(\ot_z\right)_z$, let us introduce the following 
\emph{weight functions}: for any $z\in\C$ and $\g\in\G$, let 
$$w_z(\g,x)=\left\{\begin{array}{ll}
                        0 & \mathrm{if}\ x\in\LL_{l(\g)},\\
                        e^{-zb(\g,x)} & \mathrm{if}\ x\in\LL_{j},\ j\neq l(\g).
                       \end{array}\right.$$
\noindent
These weight functions satisfy the following cocycle relation: if $\alpha_1,\alpha_2\in\mathcal{A}$ do not belong to the same group $\G_j$, then 
\begin{equation}\label{cocycle}
 w_z(\alpha_1\alpha_2,x)=w_z(\alpha_1,\alpha_2.x)w_z(\alpha_2,x).
\end{equation}
\noindent
This implies that for any $k\geqslant1$, the $k$-th iterate of $\ot_z$ is given by: for any $\varphi\in\left(\mathcal{C}(\LL),|\cdot|_{\infty}\right)$, for any 
$z\in\C$ with $\re{z}\geqslant\delta$ and for any $x\in\LL$,
$$\ot_z^k\varphi(x)=\sum\limits_{\T^k.y=x}e^{-zS_k\gol(y)}\varphi(y)=\sum\limits_{\g\in\G(k)}w_z(\g,x)\varphi(\g.x).$$
\noindent
We will need to control the regularity of the family of functions $(x\mapsto b(\g,x))_{\g\in\G}$. The following proposition is proved in \cite{BP}:
\begin{prop}\label{equilipcocycle}
Let $E\subset\partial\xx$ and $F\subset\xx$ be two sets with disjoint closures in $\xx\cup\partial\xx$. Then the family $(x\longmapsto\mathcal{B}_x({\bf p},\oo))_{{\bf p}\in F}$ is equi-Lipschitz continuous on $\left(E,\mathrm{d}_{\oo}\right)$.
\end{prop}
Let us now consider the restriction of the operator $\ot_z$ to the subspace of Lipschitz functions from $\LL$ to 
$\C$, defined by
$$\mathrm{Lip}(\LL)=\{\varphi\in\mathcal{C}(\LL)\ |\ ||\varphi||=|\varphi|_{\infty}+[\varphi]<+\infty\}\subset
\mathcal{C}(\LL)$$
\noindent
where $[\varphi]=\underset{1\leqslant j\leqslant p+q}{\sup}\ \underset{\underset{x\neq y}{x,y\in\LL_j}}{\sup}\dfrac{|\varphi(x)-\varphi(y)|}{\dhy_{\oo}(x,y)}$.
The space $(\mathrm{Lip}(\LL),||\cdot||)$ is a $\C$-Banach space; it follows from Ascoli's theorem that the canonical one-to-one map from 
$(\mathrm{Lip}(\LL),||\cdot||)$ into $(\mathcal{C}(\LL),|\cdot|_{\infty})$ is compact. One readily may 
check that the function $h$ belongs to $\mathrm{Lip}\left(\LL\right)$.
The following proposition may be proved as Lemma 2.1 in \cite{BP}.
\begin{prop}\label{cocycle1lip}
For all $z\in\C$, the weight $w_z(\g,\cdot)$ belongs to $(\mathrm{Lip}(\LL),||\cdot||)$ and there exists a constant $C=C(z)>0$ such that for any $\g$ in $\G^*$,
$$||w_z(\g,\cdot)||\leqslant Ce^{-\re{z}\dhy(\oo,\g.\oo)}.$$
\end{prop}
This yields the following.
\begin{coro}\label{hfonctionpropredeotdelta}
The operator $\ot_z$ is bounded on $\mathrm{Lip}(\LL)$ when $\re{z}\geqslant\delta$.
\end{coro}
From Lemma \ref{dualityTtransfer} and from Corollary \ref{hfonctionpropredeotdelta}, we deduce that $1$ is an eigenvalue of 
$\ot_\delta$ on $\mathrm{Lip}\left(\LL\right)$. We need more informations about the spectrum of $\ot_\delta$ on this space. Since the operator 
$\ot_\delta$ is positive, the spectral radius $\rho_{\infty}(\delta)$ of $\ot_\delta$ on $\left(\mc\left(\LL\right),|\cdot|_{\infty}\right)$ is given by
$$\rho_{\infty}(\delta)=\limsup\limits_{n\longrightarrow+\infty}\left|\ot_{\delta}^n\un_\LL\right|_{\infty}^{\frac{1}{n}}.$$
\noindent
The function $h$ being continuous and positive on $\LL$, we have
$$\left|\ot_\delta^n\un_\LL\right|_{\infty}\asymp\left|\ot_\delta^nh\right|_{\infty}=\left|h\right|_{\infty},$$
\noindent
hence $\rho_{\infty}(\delta)=1$. Denote now by $\rho(\delta)$ the spectral radius of $\ot_\delta$ on $\mathrm{Lip}\left(\LL\right)$. The incomming 
proposition gives more details about the spectrum of $\ot_\delta$ on $\mathrm{Lip}\left(\LL\right)$. Its proof relies on the notion of quasi-compacity.  
\begin{defi}
Let $\left(\mathcal{B},||\cdot||\right)$ be a Banach space and $Q$ a bounded operator on $\mathcal{B}$ with spectral radius 
$\rho(Q)$. The operator $Q$ is said to be quasi-compact if $\mathcal{B}$ may be splitted into $Q$-stable 
subspaces $\mathcal{B}=F\oplus H$, where $F$ has finite dimension and $Q_{|F}$ admits only 
eigenvalues of modulus $\rho(Q)$, whereas $\rho\left(Q_{|H}\right)<\rho(Q)$.
\end{defi}
\noindent
This notion is stable under small perturbation (\cite{Hen2}): we will use this fact in the sequel. 
\begin{prop}\label{spectredeotdelta}
The operator $\ot_\delta$ is quasi-compact. The spectral radius $\rho(\delta)$ is a simple and isolated eigenvalue in the spectrum of $\ot_\delta$, 
satisfying $\rho(\delta)=\rho_{\infty}(\delta)=1$; this is the unique eigenvalue with modulus $\rho(\delta)$. Moreover, the rest of the spectrum of $\ot_\delta$ is included in a disc of radius $<1$.
\end{prop}
\noindent
The proof of this proposition is identical to the proof of Proposition III.4 in \cite{BP2}. We will give a full proof of an analogous proposition for an extension of these transfer operators in section 8: see Proposition \ref{propriétésopeétendu}. We can thus write 
$$\ot_\delta=\Pi_\delta+R_\delta$$
\noindent
where $\Pi_\delta\ :\ \mathrm{Lip}\left(\LL\right)\longrightarrow\C h$ is the spectral projection on the eigenspace associated to 1 and  $R_\delta=\ot_\delta-\Pi_\delta$ satisfies $\Pi_\delta R_\delta=R_\delta\Pi_\delta=0$ and has a spectral radius $<1$.
There thus exists a linear form $\sigma_\delta\ :\ \mathrm{Lip}\left(\LL\right)\longrightarrow\C$ such that
$\Pi_\delta(\cdot)=\sigma_\delta(\cdot)h$. It follows that
$\Pi_\delta\left(\ot_\delta\varphi\right)=\sigma_\delta\left(\ot_\delta\varphi\right)h$ for 
any $\varphi\in\mathrm{Lip}\left(\LL\right)$ on the one hand and 
$\Pi_\delta\left(\ot_\delta\varphi\right)=\ot_\delta\Pi_\delta\left(\varphi\right)=\sigma_\delta\left(\varphi\right)h$ 
on the other hand, which implies that the measure $\sigma_\delta$ is $\ot_\delta$-invariant. 
\begin{rem}
The measure $\sigma_\delta$ corresponds to the Patterson measure $\sigma_{\oo}$.
Indeed, for any $\varphi\in\mathrm{Lip}\left(\LL\right)$ and $k\geqslant1$,
$$\sigma_{\oo}\left(\varphi\right)=\sigma_{\oo}\left(\ot_\delta^k\varphi\right)=\sigma_\oo\left(\sigma_\delta\left(\varphi\right)h\right)+
\sigma_\oo\left(R_\delta^k\varphi\right).$$
\noindent
By definition of $h$, one gets
$$\sigma_{\oo}\left(\varphi\right)=\sigma_\delta\left(\varphi\right)+\sigma_\oo\left(R_\delta^k\varphi\right)
\longrightarrow\sigma_\delta\left(\varphi\right).$$
\noindent 
The remark thus follows from the density of the space $\mathrm{Lip}\left(\LL\right)$ in $\mathrm{L}^1\left(\LL\right)$.
\end{rem}
\subsubsection{Study of perturbations of $\ot_\delta$}
In this subsection, we extend the previous spectral gap property to small perturbations of $\ot_\delta$ given by $z\longmapsto\ot_z$ for 
$z\in\C$ with $\re{z}\geqslant\delta$. We first prove the following
\begin{prop}\label{continuityoftransfert}
Under assumptions $\left(H_\beta\right)$, for any compact subset $K$ of $\R$, there exists a constant $C=C_K>0$ such that for any $s,t\in K$ and 
$\kappa$ small enough
\begin{itemize}
 \item[1)]if $\beta\in]0,1[$ 
        \begin{align*}
        &a.\ ||\ot_{\delta+it}-\ot_{\delta+is}||\leqslant C|s-t|^{\beta}L\left(\dfrac{1}{|s-t|}\right);\\
        &b.\ ||\ot_{\delta+\kappa+it}-\ot_{\delta+it}||\leqslant C\kappa^{\beta}L\left(\dfrac{1}{\kappa}\right); 
        \end{align*}
 \item[2)]if $\beta=1$ 
       \begin{align*}
       &a.\ ||\ot_{\delta+it}-\ot_{\delta+is}||\leqslant C|s-t|\tilde{L}\left(\dfrac{1}{|s-t|}\right);\\
       &b.\ ||\ot_{\delta+\kappa+it}-\ot_{\delta+it}||\leqslant C\kappa\tilde{L}\left(\dfrac{1}{\kappa}\right),
       \end{align*}
\noindent
where $\tilde{L}(x)=\displaystyle{\int_1^x}\frac{L(y)}{y}\dd y$.
\end{itemize}
\end{prop}
\begin{proof}
We only detail the proof of assertion 1.$a.$, the arguments being similar for the others. Let $\varphi\in\mathrm{Lip}\left(\LL\right)$ : it is sufficient to check that
$$||\ot_{\dit}\varphi-\ot_{\delta+is}\varphi||\leqslant C|s-t|^{\beta}L\left(\dfrac{1}{|s-t|}\right)||\varphi||.$$
\noindent
Proposition \ref{etapecontroltransfert} is the key-point to obtain such estimates.
For any  $1\leqslant j\leqslant p+q$ and $x\in\LL\setminus\LL_j$, we define the following measure
$$\mu_j^x=\dfrac{1}{M_j(x)}\sum\limits_{\alpha\in\G_j^*}e^{-\delta b(\alpha,x)}\mathrm{D}_{b(\alpha,x)}$$
\noindent
where $\mathrm{D}_{b(\alpha,x)}$ is the Dirac mass at 
$b(\alpha,x)$ and $M_j(x):=\sum_{\alpha\in\G_j^*}e^{-\delta b(\alpha,x)}$. These measures are supported on $[-C,+\infty[$ where $C>0$ is the constant appearing in Corollary \ref{busedist}. We also deduce from this corollary that
$$e^{-\delta C}\sum\limits_{\alpha\in\G_j^*}e^{-\delta\dhy(\oo,\alpha.\oo)}\leqslant M_j(x)\leqslant
e^{\delta C}\sum\limits_{\alpha\in\G_j^*}e^{-\delta\dhy(\oo,\alpha.\oo)},$$
\noindent
then 
$$\dfrac{e^{-2\delta C}}{\sum\limits_{\alpha\in\G_j^*}e^{-\delta\dhy(\oo,\alpha.\oo)}}\sum\limits_{\underset{\dhy(\oo,\alpha.\oo)>T+C}{\alpha\in\G_j^*}}e^{-\delta\dhy(\oo,\alpha.\oo)}
\leqslant 1-\mu_j^x\left([-C,T]\right)$$
\noindent
and
$$1-\mu_j^x\left([-C,T]\right)\leqslant
\dfrac{e^{2\delta C}}{\sum\limits_{\alpha\in\G_j^*}e^{-\delta\dhy(\oo,\alpha.\oo)}}\sum\limits_{\underset{\dhy(\oo,\alpha.\oo)>T-C}{\alpha\in\G_j}}e^{-\delta\dhy(\oo,\alpha.\oo)}.$$
\noindent
Hence for $T$ large enough, one gets
$$\dfrac{e^{-2\delta C}}{\sum\limits_{\alpha\in\G_j^*}e^{-\delta\dhy(\oo,\alpha.\oo)}}\dfrac{L(T)}{T^{\beta}}\leqslant1-\mu_j^x\left([-C,T]\right)\leqslant
\dfrac{e^{2\delta C}}{\sum\limits_{\alpha\in\G_j^*}e^{-\delta\dhy(\oo,\alpha.\oo)}}\dfrac{L(T)}{T^{\beta}}$$
\noindent
for $j\in[\![1,p]\!]$ and
$$1-\mu_j^x\left([-C,T]\right)=o_j(T)\dfrac{L(T)}{T^{\beta}}$$
\noindent
for $j\in[\![p+1,p+q]\!]$.
From Assertion 1.$a)$ of Proposition \ref{etapecontroltransfert}, we deduce that
$$\int_{0}^{+\infty}\left|e^{ity}-1\right|\mu_j^x(\dd y)\preceq|t|^\beta L\left(\dfrac{1}{|t|}\right)$$
\noindent
for $j\in[\![1,p]\!]$, whereas for $j\in[\![p+1,p+q]\!]$, we have 
$$\int_{0}^{+\infty}\left|e^{ity}-1\right|\mu_j^x(\dd y)= o\left(|t|^\beta L\left(\dfrac{1}{|t|}\right)\right),$$
\noindent
uniformly in $x\in\LL\setminus\LL_j$. These estimates may also be written as
\begin{equation}\label{besoin1}
\sum\limits_{\alpha\in\G_j^*}\left|e^{itb(\alpha,x)}-1\right|e^{-\delta b(\alpha,x)}=|t|^{\beta}L\left(\dfrac{1}{|t|}\right)
\end{equation}
\noindent
for $j\in[\![1,p]\!]$ and
\begin{equation}\label{besoin}
\sum\limits_{\alpha\in\G_j^*}\left|e^{itb(\alpha,x)}-1\right|e^{-\delta b(\alpha,x)}=|t|^{\beta}L\left(\dfrac{1}{|t|}\right)o\left(\dfrac{1}{|t|}\right)
\end{equation}
\noindent
for $j\in[\![p+1,p+q]\!]$. Therefore, for any $j\in[\![1,p+q]\!]$ and $x\in\LL\setminus\LL_j$ 
\begin{align*}
\sum\limits_{\alpha\in\G_j^*}\left|w_{\delta+it}(\alpha,x)\varphi(\alpha.x)-w_{\delta+is}(\alpha,x)\varphi(\alpha.x)\right| & \leqslant
\left(\sum\limits_{\alpha\in\G_j}\left|e^{i(t-s)b(\alpha,x)}-1\right|e^{-\delta b(\alpha,x)}\right)|\varphi|_{\infty}\\
& \preceq|t-s|^{\beta}L\left(\dfrac{1}{|t-s|}\right)||\varphi||
\end{align*}
\noindent
which finally implies 
$$
\left|\ot_{\delta+it}\varphi-\ot_{\delta+is}\varphi\right|_{\infty}\preceq
|t-s|^{\beta}L\left(\dfrac{1}{|t-s|}\right)||\varphi||.
$$
\noindent
In order to control the Lipschitz coefficient of the function $x\longmapsto\ot_{\delta+it}\varphi(x)-\ot_{\delta+is}\varphi(x)$, we first notice that
\begin{equation}\label{majorationintermédiairelip}
\left[\ot_{\delta+it}\varphi-\ot_{\delta+is}\varphi\right]\leqslant
\underset{j\in[\![1,p+q]\!]}{\sup}\underset{\underset{y\neq x}{x,y\in\LL_j}}{\sup}
\sum\limits_{\underset{l\neq j}{\alpha\in\G_l}}\dfrac{A_{j}(l,\alpha,x,y)}{\dhy_{\oo}(x,y)}
\end{equation}
\noindent
where 
\begin{align*}
A_{j}(l,\alpha,x,y):= &\left|w_{\delta+it}(\alpha,x)\varphi(\alpha.x)-w_{\delta+is}(\alpha,x)\varphi(\alpha.x)\right.\\
& -\left.\left(w_{\delta+it}(\alpha,y)\varphi(\alpha.y)-w_{\delta+is}(\alpha,y)\varphi(\alpha.y)\right)\right|
\end{align*}
\noindent
for any $j,l\in[\![1,p+q]\!]$, $l\neq j$, $x,y\in\LL_j$, $y\neq x$ and $\alpha\in\G_l$. We observe that 
\begin{align*}
A_{j}(l,\alpha,x,y)\leqslant &\left|w_{\delta+it}(\alpha,x)-w_{\delta+is}(\alpha,x)
-\left(w_{\delta+it}(\alpha,y)-w_{\delta+is}(\alpha,y)\right)\right||\varphi(\alpha.x)|\\
&+\left|w_{\delta+it}(\alpha,y)-w_{\delta+is}(\alpha,y)\right||\varphi(\alpha.y)-\varphi(\alpha.x)|\\
\leqslant & B_{j}(l,\alpha,x,y)+C_{j}(l,\alpha,x,y)
\end{align*}
\noindent
where
$$
B_{j}(l,\alpha,x,y)=\left|w_{\delta+it}(\alpha,x)-w_{\delta+is}(\alpha,x)
-\left(w_{\delta+it}(\alpha,y)-w_{\delta+is}(\alpha,y)\right)\right||\varphi(\alpha.x)|
$$
\noindent
and
$$
C_{j}(l,\alpha,x,y)=\left|w_{\delta+it}(\alpha,y)-w_{\delta+is}(\alpha,y)\right||\varphi(\alpha.y)-\varphi(\alpha.x)|.
$$
\noindent
On the one hand
\begin{align*}
B_{j}(l,\alpha,x,y) \leqslant  & e^{-\delta b(\alpha,x)}\left|e^{i(t-s)b(\alpha,x)}-e^{i(t-s)b(\alpha,y)}\right|||\varphi||\\
& +\left|e^{-(\delta+it)b(\alpha,x)}-e^{-(\delta+it)b(\alpha,y)}\right|\left|e^{i(t-s)b(\alpha,y)}-1\right|||\varphi||\\
\leqslant & e^{-\delta b(\alpha,x)}\left|e^{i(t-s)(b(\alpha,x)-b(\alpha,y))}-1\right|||\varphi||\\
& +e^{-\delta b(\alpha,y)}\left|e^{(\delta+it)(b(\alpha,y)-b(\alpha,x))}-1\right|\left|e^{i(t-s)b(\alpha,y)}-1\right|||\varphi||\\
\preceq & \left(e^{-\delta b(\alpha,x)}\left[b(\alpha,\cdot)\right]|t-s|\dhy_{\oo}(x,y)\right)||\varphi||\\
& + \left(e^{-\delta b(\alpha,y)}|\delta+it|\left[b(\alpha,\cdot)\right]\left|e^{i(t-s)b(\alpha,y)}-1\right|\dhy_{\oo}(x,y)\right)||\varphi||.
\end{align*}
\noindent
Since the sequence $\left(\left[b(\g,\cdot)\right]\right)_{\g\in\G_j}$ is bounded and $t\in K$, we deduce that
\begin{align*}
\dfrac{B_{j}(l,\alpha,x,y)}{\dhy_{\oo}(x,y)} \preceq 
\left(e^{-\delta b(\alpha,x)}|t-s|\right)||\varphi||
+ \left(e^{-\delta b(\alpha,y)}\left|e^{i(t-s)b(\alpha,y)}-1\right|\right)
||\varphi||.
\end{align*}
\noindent
On the other hand
\begin{align*}
\dfrac{C_{j}(l,\alpha,x,y)}{\dhy_{\oo}(x,y)} \preceq 
e^{-\delta b(\alpha,y)}\left|e^{i(t-s)b(\alpha,y)}-1\right|||\varphi||,
\end{align*}
\noindent
so that for any $j\in[\![1,p+q]\!]$ and $x,y\in\LL_j$, one gets
\begin{align*}
\sum\limits_{\underset{l\neq j}{\alpha\in\G_l}}\dfrac{A_j(l,\alpha,x,y)}{\dhy_{\oo}(x,y)} & \leqslant
\sum\limits_{\underset{l\neq j}{\alpha\in\G_l}}\dfrac{B_j(l,\alpha,x,y)+C_j(l,\alpha,x,y)}{\dhy_{\oo}(x,y)}\\
& \preceq \sum\limits_{\underset{l\neq j}{\alpha\in\G_l}}\left(e^{-\delta b(\alpha,x)}|t-s|+e^{-\delta b(\alpha,y)}\left|e^{i(t-s)b(\alpha,y)}-1\right|\right)||\varphi||.
\end{align*}
\noindent
We deduce from \eqref{besoin1} and \eqref{besoin} that
$$\sum\limits_{\underset{l\neq j}{\alpha\in\G_l}}\dfrac{A_j(l,\alpha,x,y)}{\dhy_{\oo}(x,y)}\preceq|t-s|^{\beta}L\left(\dfrac{1}{|t-s|}\right)||\varphi||$$
\noindent
and this last estimate combined with \eqref{majorationintermédiairelip} yields
%
$$\left[\ot_{\delta+it}\un_{\LL}-\ot_{\delta+is}\un_{\LL}\right]\preceq |t-s|^{\beta}L\left(\dfrac{1}{|t-s|}\right)||\varphi||.$$
\end{proof}
Combining the previous proposition with the proof of Proposition 2.2 in \cite{BP}, we deduce the following corollary.
\begin{coro}\label{continuityoftransferteverywhere}
The application $z\longmapsto\ot_z$ is continuous on $\{z\in\C\ |\ \re{z}\geqslant\delta\}$.
\end{coro}
Let us now show the existence of a simple dominant eigenvalue for $\ot_z$, isolated in its spectrum, uniformly in $z$ close enough 
to $\delta$. Denote by $\rho(z)$ the spectral radius of $\ot_z$ on $\mathrm{Lip}\left(\LL\right)$ and by $\left|x+iy\right|_{\infty}=\max\left(|x|,|y|\right)$ for any $x,y\in\R$.
\begin{prop}[\cite{BP} p.$92$]\label{spectreperturbation}
There exist $\varepsilon>0$ and $\rho_\varepsilon\in]0,1[$ such that for all $z\in\C$ satisfying $|z-\delta|_\infty<\varepsilon$ and $\re{z}\geqslant\delta$, one gets:
\begin{enumerate}
 \item[-]$\rho(z)>\rho_\varepsilon$;
 \item[-]$\ot_{z}$ has a unique eigenvalue $\lambda_z$ with modulus $\rho(z)$;
 \item[-]this eigenvalue is simple and $\lim\limits_{z\longrightarrow\delta}\lambda_z=1$;
 \item[-]the rest of the spectrum is included in a disc of radius $\rho_\varepsilon$.
\end{enumerate}
Furthermore for any $A>0$, there exists $\rho(A)<1$ such that $\rho(z)<\rho(A)$ as soon as 
$z\in\C$ satisfies $|z-\delta|_\infty\geqslant\varepsilon$, $\re{z}\geqslant\delta$ and $\left|\im{z}\right|\leqslant A$.
At last, if $z\in\C$ satisfies $\re{z}\geqslant\delta$, then $\rho(z)\leqslant1$ with equality if and only if $z=\delta$.
\end{prop}
\begin{rem}\label{sertA1etB1}
In the proof of Propositions A.1, A.2, B.1 and C.1, we will use Potter lemma and thus choose 
$\varepsilon>0$ small enough in such a way that
$$\dfrac{L(x)}{L(y)}\leqslant\max\left(\dfrac{y}{x},\dfrac{x}{y}\right)^{\frac{\beta}{2}}$$
\noindent
for any $x,y\geqslant\frac{1}{\varepsilon}$.
\end{rem}
We denote by $h_z$ the unique eigenfunction of $\ot_z$ associated to $\lambda_z$ satisfying $\sigma_{\oo}(h_z)=1$; let 
$\Pi_z\ :\ \mathrm{Lip}(\LL)\longrightarrow\mathrm{Lip}(\LL)$ denote the spectral projection associated to $\lambda_z$. There exists a unique 
linear form $\sigma_z\ :\ \mathrm{Lip}(\LL)\longrightarrow\C$ such that $\Pi_z(\cdot)=\sigma_z(\cdot)h_z$ and $\sigma_z(h_z)=1$. We set 
$R_z:=\ot_z-\Pi_z$. By perturbation theory, the maps $z\longmapsto\lambda_z$, $z\longmapsto h_z$ and $z\longmapsto\Pi_z$
have the same regularity as $z\longmapsto\ot_z$.
\subsubsection{Regularity of the dominant eigenvalue}
In proof of Theorems A and B, we will have to deal with quantities like
\begin{equation}\label{quantitesaetudierjustifiantetuderayonspec}
\sum\limits_{\T^k.y=x}e^{-\delta S_k\gol(y)}\varphi(S_k\gol(y)-R),
\end{equation}
\noindent 
for functions $\varphi$ with a $\mathcal{C}^\infty$ Fourier transform and with compact support. The inverse Fourier transform leads us to write
$$\sum\limits_{\T^k.y=x}e^{-\delta S_k\gol(y)}\varphi(S_k\gol(y)-R)=\dfrac{1}{2\pi}\int_\R e^{itR}\ot_{\delta+it}^k\left(\un_{\LL}(x)\right)\widehat{\varphi}(t)\dd t.$$
\noindent
The study of the behaviour when $R\longrightarrow+\infty$ of the quantity \eqref{quantitesaetudierjustifiantetuderayonspec}
thus involves a precise knowledge of the one of the function $t\longmapsto\lambda_{\delta+it}$ in a neighbourhood of $0$. 

We first establish a result concerning some probability measures depending on the Schottky factors $\G_i$, $1\leqslant i\leqslant p+q$.
\begin{prop}\label{theoreticlocalexp}
Let $j\in[\![1,p+q]\!]$ and $x\in\LL\setminus\LL_j$. Denote by $N_j(x)=\sum\limits_{\alpha\in\G_j^*}h(\alpha.x)e^{-\delta b(\alpha,x)}$ and let us introduce 
the following probability measure $\nu_j^x$ on $[-C,+\infty[$ 
$$\nu_j^x:=\dfrac{1}{N_j(x)}\sum\limits_{\alpha\in\G_j^*}h(\alpha.x)e^{-\delta b(\alpha,x)}D_{b(\alpha,x)}$$
\noindent
where $D_{b(\alpha,x)}$ is the Dirac mass at $b(\alpha,x)\in\R$. This measure satisfies one of the following assertions.
\begin{itemize}
 \item[1)]For $j\in[\![1,p]\!]$ and $t\longrightarrow0$, 
 \begin{itemize}
  \item[-] if $\beta\in]0,1[$
         \begin{align*}
         \widehat{\nu_j^x}(t)=1-\left(\dfrac{C_j}{N_j(x)}h(x_j)e^{2\delta\left(x_j|x\right)_\oo}\right)e^{-i\mathrm{sign}(t)\frac{\beta\pi}{2}}\G(1-\beta)|t|^{\beta}L\left(\dfrac{1}{|t|}\right)(1+o(1));
         \end{align*}
  \item[-]if $\beta=1$
         \begin{align*}
&\ast\widehat{\nu_j^x}(t)=1+\left(\dfrac{C_j}{N_j(x)}h(x_j)e^{2\delta\left(x_j|x\right)_\oo}\right) i|t|
\tilde{L}\left(\dfrac{1}{|t|}\right)(1+o(1));\\
&\ast\re{1-\widehat{\nu_j^x}(t)}=\dfrac{\pi}{2}\left(\dfrac{C_j}{N_j(x)}h(x_j)e^{2\delta\left(x_j|x\right)_\oo}\right)|t|
L\left(\dfrac{1}{|t|}\right)(1+o(1)),
         \end{align*}
\end{itemize}
\noindent where $x_j$ is the fixed point of the parabolic group $\G_j$ and the $(C_j)_{1\leqslant j\leqslant p}$ are the constants appearing in Assumption $(P_2)$;
 \item[2)] For $j\in[\![p+1,p+q]\!]$, there exists a function $f_j\ :\ \LL\longrightarrow\R$ satisfying $\sigma_\oo(\un_{\LL_j^c}f_j)<+\infty$ such that for any $t\longrightarrow0$, one gets
  \begin{itemize}
  \item[-]if $\beta\in]0,1[$
         \begin{align*}
         \widehat{\nu_j^x}(t)=1-f_j(x)o\left(|t|^{\beta}L\left(\dfrac{1}{|t|}\right)\right);
         \end{align*}
  \item[-]if $\beta=1$
         \begin{align*}
         &\ast\widehat{\nu_j^x}(t)=1+f_j(x)o\left(|t|\tilde{L}\left(\dfrac{1}{|t|}\right)\right);\\
         &\ast\re{1-\widehat{\nu_j^x}(t)}=f_j(x)o\left(|t|L\left(\dfrac{1}{|t|}\right)\right).
         \end{align*}
\end{itemize}
\end{itemize}
\end{prop}
\begin{proof}
We just detail the proof when $\beta\in]0,1[$. We first consider $j\in[\![1,p]\!]$. Fix $x\in\LL\setminus\LL_j$. By Corollary \ref{busedist}, there exist two constants $m,M>0$ such that
\begin{equation}\label{encadrementconstanterenormalisation}
m\sum\limits_{\alpha\in\G_j}e^{-\delta\dhy(\oo,\alpha.\oo)}\leqslant N_j(x)\leqslant M\sum\limits_{\alpha\in\G_j}e^{-\delta\dhy(\oo,\alpha.\oo)}.
\end{equation}
\noindent
On the other hand, Assumption $(P_2)$ gives
$$\sum\limits_{\underset{\dhy(\oo,\alpha.\oo)>T}{\alpha\in\G_j}}e^{-\delta\dhy(\oo,\alpha.\oo)}\sim C_j\dfrac{L(T)}{T^\beta}.$$
\noindent
We want to show that the distribution function $F_j^x$ of $\nu_j^x$ satisfies
\begin{equation}\label{prooftheoreticlocalexp1}
1-F_j^x\left(T\right)\sim\left(\dfrac{C_j}{N_j(x)}h(x_j)e^{2\delta\left(x_j|x\right)_\oo}\right)\dfrac{L(T)}{T^\beta}\ \text{uniformly in}\ x\notin\LL_j.
\end{equation}
\noindent
Fix $\varepsilon>0$. There exists $T_0\gg1$ such that for any $T\geqslant T_0$, $x\in\LL\setminus\LL_j$ and $\alpha\in\G_j$ satisfying $\dhy(\oo,\alpha.\oo)>T$
\begin{align*}
&i)\ -\varepsilon\leqslant b(\alpha,x)-\dhy(\oo,\alpha.\oo)+2\left(x_j|x\right)_\oo\leqslant\varepsilon\ 
\left(\text{Lemma 6.7 in \cite{DPPS}}\right);\\
&ii)\ (1-\varepsilon)h(x_j)\leqslant h(\alpha.x)\leqslant (1+\varepsilon)h(x_j);\\
&iii)\ (1-\varepsilon)C_j\dfrac{L(T)}{T^\beta}
\leqslant\sum\limits_{\underset{\dhy(\oo,\alpha.\oo)>T}{\alpha\in\G_j}}e^{-\delta\dhy(\oo,\alpha.\oo)}
\leqslant(1+\varepsilon)C_j\dfrac{L(T)}{T^\beta};
\end{align*}
\noindent
hence
$$
(1-\varepsilon)^2e^{-\delta\varepsilon}\dfrac{C_j}{N_j(x)}h(x_j)e^{2\delta\left(x_j|x\right)_\oo}
\dfrac{L\left(T+\varepsilon+2\left(x_j|x\right)_\oo\right)}{\left(T+\varepsilon+2\left(x_j|x\right)_\oo\right)^\beta}
\leqslant 1-F_j^x\left(T\right),$$
\noindent
and
$$
1-F_j^x\left(T\right)\leqslant(1+\varepsilon)^2e^{\delta\varepsilon}\dfrac{C_j}{N_j(x)}h(x_j)e^{2\delta\left(x_j|x\right)_\oo}
\dfrac{L\left(T-\varepsilon+2\left(x_j|x\right)_\oo\right)}{\left(T-\varepsilon+2\left(x_j|x\right)_\oo\right)^\beta}.
$$
\noindent
Since $\left(x_j|x\right)_\oo\asymp\dhy(\oo,(x_jx))$, this quantity is bounded uniformly in $x\in\LL\setminus\LL_j$; moreover, the 
functions 
$$t\longmapsto\dfrac{L(T+t)}{L(T)}\ \mathrm{and}\ t\longmapsto\dfrac{T+t}{T}$$
\noindent
tend to $1$ uniformly on each compact subset of $\R$. Hence 
$$(1-\varepsilon)^3e^{-\delta\varepsilon}\left(\dfrac{C_j}{N_j(x)}h(x_j)e^{2\delta\left(x_j|x\right)_\oo}\right)
\dfrac{L(T)}{T^\beta}
\leqslant 1-F_j^x\left(T\right)$$
\noindent
and
$$1-F_j^x\left(T\right)
\leqslant (1+\varepsilon)^3e^{\delta\varepsilon}\left(\dfrac{C_j}{N_j(x)}h(x_j)e^{2\delta\left(x_j|x\right)_\oo}\right)
\dfrac{L(T)}{T^\beta},$$
\noindent
for $T$ large enough. Therefore \eqref{prooftheoreticlocalexp1} is true for $j\in[\![1,p]\!]$.

Recall now the following result exposed in \cite{Er}.
\begin{prop}[\cite{Er}]
Let $\nu$ be a probability measure on $\R^+$ such that there exist $\beta\in]0,1]$ and a slowly varying function $L$ satisfying 
$\mu\left([T,+\infty[\right)\sim C\frac{L(T)}{T^{\beta}}$ when $T\longrightarrow+\infty$. Then the characteristic function
$\widehat{\nu}(t):=\int_0^{+\infty}e^{itx}\dd\nu(x)$ has the following behaviour in a neighbourhood of $0$: 
\begin{itemize}
 \item[-] if $\beta\in]0,1[$
          $$\widehat{\nu}(t)=1-Ce^{-i\mathrm{sign}(t)\frac{\beta\pi}{2}}\G(1-\beta)|t|^{\beta}L\left(\dfrac{1}{|t|}\right)(1+o(1));$$
 \item[-] if $\beta=1$
          \begin{align*}
          &\ast\widehat{\nu}(t)=1+iC|t|\tilde{L}\left(\dfrac{1}{|t|}\right)(1+o(1));\\
          &\ast\re{1-\widehat{\nu}(t)}=\dfrac{\pi}{2}C|t|L\left(\dfrac{1}{|t|}\right)(1+o(1)),
          \end{align*}
 \noindent
where $\tilde{L}(t)=\displaystyle{\int_1^t}\frac{L(x)}{x}\dd x$.
\end{itemize}
\end{prop}
\noindent
The measures $\nu_j^x$, $1\leqslant j\leqslant p$, satisfy \eqref{prooftheoreticlocalexp1}; Assertion 
1) of Proposition \ref{theoreticlocalexp} thus follows from the previous statement.
%

Now fix $j\in[\![p+1,p+q]\!]$. When $\G_j$ is parabolic, the previous arguments still work, but Assumption $(N)$ imposes
$$\sum\limits_{\underset{\dhy(\oo,\alpha.\oo)>T}{\alpha\in\G_j}}e^{-\delta\dhy(\oo,\alpha.\oo)}=o\left(\dfrac{L(T)}{T^\beta}\right),$$
\noindent
which implies
$$1-F_j^x\left(T\right)=\left(\dfrac{1}{N_j(x)}h(x_j)e^{2\delta\left(x_j|x\right)_\oo}\right)\dfrac{L(T)}{T^\beta}o(T)$$
\noindent
with $\lim\limits_{T\longrightarrow+\infty}o(T)=0$ uniformly in $x\notin\LL_j$. The second part of the result follows from \cite{Er}, with  
$f_j(x)$ given in that case by
$$f_j(x)=\dfrac{1}{N_j(x)}h(x_j)e^{2\delta\left(x_j|x\right)_\oo}.$$
\noindent
Inequality \eqref{encadrementconstanterenormalisation} yields 
\begin{align*}
\sigma_{\oo}\left(\un_{\LL_j^c}f_j\right) = & \displaystyle{\int_{\LL\setminus\LL_j}}
\dfrac{1}{N_j(x)}h(x_j)e^{2\delta\left(x_j|x\right)_\oo}\dd\sigma_{\oo}\left(x\right)\\
\preceq & \dfrac{1}{\sum\limits_{\alpha\in\G_j}e^{-\delta\dhy(\oo,\alpha.\oo)}}\int_{\LL\setminus\LL_j}
e^{2\delta\left(x_j|x\right)_\oo}\dd\sigma_{\oo}\left(x\right)\\
\preceq & \int_{\LL\setminus\LL_j}\dfrac{\dd\sigma_{\oo}\left(x\right)}{\dhy_{\oo}(x_j,x)^{\frac{2\delta}{a}}}
< +\infty.
\end{align*}
When $\G_j$ is a hyperbolic group, with attractive fixed point (respectively repulsive) 
$x_j^+$ (resp. $x_j^-$), we write
$$\sum\limits_{\underset{\dhy(\oo,\alpha.\oo)>T}{\alpha\in\G_j}}e^{-\delta\dhy(\oo,\alpha.\oo)}\asymp
\sum\limits_{\underset{|n|l_j\succeq T}{n\in\Z^*}}e^{-\delta l_jn}\asymp e^{-\delta T}$$
\noindent
where $l_j$ is the length of the axis of the generator of $\G_j$. The arguments are the same as for the non-influent parabolics. In that case, the function 
$f_j(x)$ is given by
$$f_j(x)=\dfrac{1}{2N_j(x)}\left(h(x_j^+)e^{2\delta\left(x_j^+|x\right)_\oo}+
h(x_j^-)e^{2\delta\left(x_j^-|x\right)_\oo}\right).$$
\noindent
The quantity $\sigma_{\oo}\left(\un_{\LL_j^c}f_j\right)$ is finite for the same reasons. This ends the proof of 
Proposition \ref{theoreticlocalexp}.
\end{proof}
The following proposition specifies the local behaviour in $0$ of the function $t\longmapsto\lambda_{\delta+it}$.
\begin{prop}\label{localexp}
There exists a constant $\Cg>0$ such that for any  $t$ small enough
\begin{itemize}
 \item[-]if $\beta\in]0,1[$ 
    $$\lambda_{\delta+it}=1-\Cg\G(1-\beta)e^{+i\mathrm{sign}(t)\frac{\beta\pi}{2}}|t|^{\beta}L\left(\dfrac{1}{|t|}\right)(1+o(1));$$
 \item[-]if $\beta=1$
 \begin{align*}
&\ast\lambda_{\delta+it}=1-\Cg \mathrm{sign}(t)i|t|\tilde{L}\left(\dfrac{1}{|t|}\right)(1+o(1));\\
&\ast\re{1-\lambda_{\delta+it}}=\dfrac{\pi}{2}\Cg|t|L\left(\dfrac{1}{|t|}\right)(1+o(1)).
\end{align*}
\end{itemize}
\end{prop}
\begin{proof}
As previously, we only detail the proof for $\beta\in]0,1[$.  We first write
$$\lambda_{\delta+it}=\sigma_{\oo}(\ot_{\delta+it} h_{\delta+it})=\sigma_{\oo}(\ot_{\delta+it} h)+\sigma_{\oo}\left((\ot_{\delta+it}-\otd)(h_{\delta+it}-h)\right).$$
\noindent
By Proposition \ref{continuityoftransfert}, the second term of the right member is bounded from above by
$\sigma_{\oo}(\LL)\left(|t|^{\beta}L\left(\dfrac{1}{|t|}\right)\right)^2$ and it remains to precise the behaviour of the first one near $0$. We write
\begin{align*}
\sigma_{\oo}(\ot_{\dit} h) & =1+\sigma_{\oo}(\ot_{\dit} h)-1=1+\sigma_{\oo}(\ot_{\dit} h)-\sigma_{\oo}(h)\\
& =1+\sigma_{\oo}\left(\left(\ot_{\dit}-\otd\right)h\right)=1+\sum\limits_{j=1}^{p+q}S_j
\end{align*}
\noindent
where
\begin{align*}
S_j:= & \sum\limits_{\alpha\in\G_j^*}\displaystyle{\int_{\LL\setminus\LL_j}}h(\alpha.x)e^{-\delta b(\alpha,x)}(e^{-itb(\alpha,x)}-1)\dd\sigma_{\oo}(x)\\
= & \displaystyle{\int_{\LL\setminus\LL_j}}N_j(x)\left(\widehat{\nu_j^x}(-t)-1\right)\dd\sigma_{\oo}(x),
\end{align*}
\noindent
for $j\in[\![1,p+q]\!]$. If follows from \ref{theoreticlocalexp} that  $S_j=o\left(|t|^{\beta}
L\left(\dfrac{1}{|t|}\right)\right)$ for $j\in[\![p+1,p+q]\!]$, whereas for $j\in[\![1,p]\!]$
\begin{align*}
S_j= & -C_jh(x_j)\left(\int_{\LL\setminus\LL_j}e^{2\delta\left(x_j|x\right)_\oo}\dd\sigma_{\oo}(x)\right)e^{i\mathrm{sign}(t)\frac{\beta\pi}{2}}
\G(1-\beta)|t|^{\beta}L\left(\dfrac{1}{|t|}\right)(1+o(1))\\
= & -C_j\dfrac{1}{\mu\left(\mathcal{D}^0\right)}\left(\int_{\LL\setminus\LL_j}\dfrac{\dd\sigma_\oo(x)}{\dhy_\oo(x,x_j)^{\frac{2\delta}{a}}}\right)^2e^{i\mathrm{sign}(t)\frac{\beta\pi}{2}}
\G(1-\beta)|t|^{\beta}L\left(\dfrac{1}{|t|}\right)(1+o(1)).
\end{align*}
\noindent
The result follows with the constant $\Cg$ given by
\begin{align*}
\Cg  = \dfrac{1}{\mu(\mathcal{D}^0)}\sum\limits_{1\leqslant j\leqslant p}
C_j\left(\displaystyle{\int\limits_{\LL\setminus\LL_j}}
\dfrac{\dd\sigma_{\oo}(x)}{\dhy_{\oo}(x,x_j)^{\frac{2\delta}{a}}}\right)^2.
\end{align*}
\end{proof}
\subsubsection{About the resolvant operator when $\beta=1$}
Let us conclude this section by some complement when $\beta=1$.
In the proof of Theorem A for $\beta=1$, we will use the operator $Q_z=\left(\mathrm{Id}-\ot_z\right)^{-1}$ for $z\in\C$ such that $\re{z}\geqslant\delta$. The following properties come from \cite{MT}. 
\begin{prop}\label{controlQ}
There exist $\varepsilon>0$ and $C>0$ such that $\left|\left|Q_z-(1-\lambda_z)^{-1}\Pi_{z}\right|\right|\leqslant C$ when
$|z-\delta|_{\infty}<\varepsilon$ and 
$\left|\left|Q_z\right|\right|\leqslant C$ for $z$ such that $|z-\delta|_{\infty}\geqslant\varepsilon$. Moreover, for any $t$ close enough to $0$
$$
Q_{\delta+it}=\dfrac{1}{\Cg\mathrm{sign}(t)i|t|\tilde{L}\left(\dfrac{1}{|t|}\right)}(1+o(1))\Pi_0+O(1)
$$
\end{prop}
\begin{proof}
Let $z\in\C$ such that $\re{z}\geqslant\delta$, $z\neq\delta$ and $|z-\delta|_\infty<\varepsilon$, where $\varepsilon$ 
is chosen as in Proposition \ref{spectreperturbation}.
Writing $\ot_z=\lambda_z\Pi_z+R_z=\lambda_z\Pi_z+\ot_z(\mathrm{Id}-\Pi_z$, one gets
$$Q_z=\left(\mathrm{Id}-\ot_z\right)^{-1}=\left(1-\lambda_z\right)^{-1}\Pi_z+\left(\mathrm{Id}-\ot_z\right)^{-1}\left(\mathrm{Id}-\Pi_z\right).$$
\noindent
Proposition \ref{spectreperturbation} implies $\left|\left|\left(\mathrm{Id}-\ot_z\right)^{-1}\left(\mathrm{Id}-\Pi_z\right)\right|\right|\leqslant C$ 
for such a $z$. Proposition \ref{spectreperturbation} also implies $\left|\left|\left(\mathrm{Id}-\ot_z\right)^{-1}\right|\right|\leqslant C$ when $z$ is
far enough to $\delta$; therefore for $t$ close enough to $0$, we get
\begin{align*}
Q_{\delta+it}= & \left(1-\lambda_{\delta+it}\right)^{-1}\Pi_{\delta+it}+O(1)\\
= & \left(1-\lambda_{\delta+it}\right)^{-1}\Pi_{\delta}+\left(1-\lambda_{\delta+it}\right)^{-1}\left(\Pi_{\delta+it}-\Pi_\delta\right)+O(1).
\end{align*}
\noindent
The regularity of the function $t\longmapsto\ot_{\delta+it}$ given in Proposition \ref{continuityoftransfert} and the local expansion of 
$\lambda_{\delta+it}$ given in \eqref{localexp} implies that the second term is a $O(1)$. Finally
$$Q_{\delta+it}=\left(1-\lambda_{\delta+it}\right)^{-1}\Pi_{\delta}+O(1)$$
\noindent
and the result follows from \eqref{localexp}.
\end{proof}
\begin{coro}\label{integrabilityreQ}
The function $t\longmapsto\re{Q_{\delta+it}}$ is integrable in $0$.
\end{coro}
\begin{proof}
By the previous proposition, we split $\re{Q_{\delta+it}}$ into 
$$\re{Q_{\delta+it}-(1-\lambda_{\delta+it})^{-1}\Pi_\delta}+\re{(1-\lambda_{\delta+it})^{-1}}\Pi_\delta.$$
\noindent
The first part is bounded by a constant $C>0$. Concerning the term $\re{(1-\lambda_{\delta+it})^{-1}}$, we write
$$\re{(1-\lambda_{\delta+it})^{-1}}=\dfrac{\re{1-\lambda_{\delta+it}}}{\left|1-\lambda_{\delta+it}\right|^2}.$$
\noindent
The local expansions \eqref{localexp} yield
$$\re{(1-\lambda_{\delta+it})^{-1}}=\dfrac{\pi}{2\Cg}\dfrac{L\left(\frac{1}{|t|}\right)}{|t|\tilde{L}\left(\frac{1}{|t|}\right)^2}(1+o(1)).$$
\noindent
This function is integrable near $0$: indeed for any $\varepsilon>0$
\begin{align*}
\int_{-\varepsilon}^\varepsilon\dfrac{L\left(\frac{1}{|t|}\right)}{|t|\tilde{L}\left(\frac{1}{|t|}\right)^2}\dd t= & 
2\int_{0}^\varepsilon\dfrac{L\left(\frac{1}{t}\right)}{t\tilde{L}\left(\frac{1}{t}\right)^2}\dd t
= \int_{\frac{1}{\varepsilon}}^{+\infty}\dfrac{L(y)}{y\tilde{L}(y)^2}\dd y\\
= & \left[\dfrac{-1}{\tilde{L}(y)^2}\right]_{\frac{1}{\varepsilon}}^{+\infty}=\dfrac{1}{\tilde{L}\left(\frac{1}{\varepsilon}\right)}<+\infty,
\end{align*}
\noindent
because $\lim\limits_{x\longrightarrow+\infty}\tilde{L}(x)=+\infty$.
\end{proof}                                                                                                      

\section{Theorem A: mixing for $\beta\in]0,1[$}

This section is devoted to the mixing properties of the geodesic flow $\left(g_t\right)_{t\in\R}$ on the unit tangent bundle 
$\T^1\mm$ of the quotient manifold $\mm=\xx/\G$. 
We precise here the speed 
of convergence to $0$ of $m_\G\left(\mathfrak{A}\cap g_{-t}.\mathfrak{B}\right)$ as 
$t\longrightarrow\pm\infty$.

The group $\G$ being of divergent type, the theorem of Hopf, Tsuji and Sullivan (\cite{Rob}) ensures that the geodesic flow is totally conservative: we thus 
do not need to formulate additional assumptions about the sets $\mathfrak{A}$ and $\mathfrak{B}$ to avoid the examples constructed by Hajan and Kakutani 
(\cite{HaKa}). 

We prove in this section the Theorem A.
\begin{thmA}
Let $\G$ be a Schottky group satisfying the hypotheses $(H_\beta)$ for some $\beta\in]0,1[$.
Let $\mathfrak{A},\mathfrak{B}\subset\mathrm{T}^1\xx/\G$ be two $m_\G$-measurable subsets of finite measure. Then, as $t\longrightarrow\pm\infty$
$$m_\G(\mathfrak{A}\cap g_{-t}.\mathfrak{B})\sim\dfrac{\sin(\beta\pi)}{\pi\Cg}\dfrac{m_\G(\mathfrak{A})m_\G(\mathfrak{B})}{|t|^{1-\beta}L(|t|)}.$$ 
\noindent
where
\begin{equation}\label{valeurdelaconstantemelange}
\Cg  = \dfrac{1}{\mu(\mathcal{D}^0)}\sum\limits_{1\leqslant j\leqslant p}
C_j\left(\displaystyle{\int\limits_{\LL\setminus\LL_j}}
\dfrac{\dd\sigma_{\oo}(x)}{\dhy_{\oo}(x,x_j)^{\frac{2\delta}{a}}}\right)^2.
\end{equation}
\end{thmA}
For any $m_\G$-integrable function $f$, we set $m_\G(f)=\int_{\T^1\xx/\G}f\dd m_\G$. The previous theorem may be reformulated as follows
\begin{thmA}
For any functions $A$,$B$ in
$\mathbb{L}^1\left(\T^1\xx/\G,m_\G\right)\cap\mathbb{L}^2\left(\T^1\xx/\G,m_\G\right)$, 
as $R\longrightarrow\pm\infty$
$$m_\G(A\cdot B\circ g_R)\sim\dfrac{\sin(\beta\pi)}{\pi\Cg}\dfrac{m_\G(A)m_\G(B)}{|R|^{1-\beta}L(|R|)}.$$
\end{thmA}
From now on, denote $R\in\R$ the parameter of the flow to emphasize the similarity of the proof of Theorem A with Theorem C. For $A,B\in\mathbb{L}^1(\mathrm{T}^1\xx/\G,m_\G)\cap\mathbb{L}^2(\mathrm{T}^1\xx/\G,m_\G)$, we set
\begin{align*}
M\left(R;A,B\right):= & m_\G(A\cdot B\circ g_R)\\
= & \int_{\Omega}A(\left[x_-,x_+,s\right])B\left(g_R(\left[x_-,x_+,s\right])\right)\dd m_\G(\left[x_-,x_+,s\right])
\end{align*}
\noindent
where $\left[x_-,x_+,s\right]$ stands for the $\G$-orbit of $(x_-,x_+,s)$. In the next paragraph, we will express the quantity 
$M\left(R;A,B\right)$ in terms of the iterates of a transfer operator, via the coding described in Section 4.
\subsection{Study of $\boldsymbol{M\left(R;A,B\right)}$}
The spaces $\Omega^0=\LL^0\overset{\Delta}{\times}\LL^0\times\R/\G$ and
$\mathcal{D}^0\times\R/\langle\T_{\gol}\rangle$ are in one-to-one correspondence and the geodesic flow on $\Omega^0$ is conjugated to the special flow
$(\phi_R)_{R\in\R}$ defined on $\mathcal{D}^0\times\R/\langle\T_{\gol}\rangle$. If we denote by $\mathfrak{c}$ the bijection between $\mathcal{D}^0\times\R/\langle\T_{\gol}\rangle$ and $\Omega^0$, we may write
$$M\left(R;A,B\right)=\int_{\mathcal{D}^0\times\R/\langle\T_\gol\rangle}A([[x_-,x_+,s]])B\left(\phi_R.([[x_-,x_+,s]])\right)
\dd \overline{m}_\G ([[x_-,x_+,s]])$$
\noindent
where $[[x_-,x_+,s]]$ is the $\langle\T_\gol\rangle$-orbit of $(x_-,x_+,s)$ and $A$ and $B$ are identified with $A\circ\mathfrak{c}$ and $B\circ\mathfrak{c}$ respectively. The following strategy is inspired of \cite{GuiH}. Let $S\subset\mathcal{D}^0\times\R$ be a fundamental domain for the action of $\langle\T_{\gol}\rangle$. The vector space generated by the functions $\varphi\otimes u$, where $\varphi$ is Lipschitz on $\mathcal{D}^0$ and $u\ :\ \R\longrightarrow\R$ is continuous with compact
support, is dense in $\mathbb{L}^1(S,\bar{\nu}\otimes\dd s)\cap\mathbb{L}^2(S,\bar{\nu}\otimes\dd s)$; we thus assume that
$A=\varphi\otimes u$ and $B=\psi\otimes v$, with $\varphi,\psi,u,v$ as above. From the definition of $\T_{\gol}$ and the fact that $S$ is a fundamental domain, we deduce that for any $(y,x,s)\in S$ and $R\in\R$, there exists a unique integer $k=k\left(R,x_-,x_+,s\right)\in\Z$ such that $T_{\gol}^k.(x_-,x_+,s+R)\in S$. Hence, for any $(y,x,s)\in S$ and $R\in\R$
$$\psi\otimes v\left(\tilde{\phi}_R(x_-,x_+,s)\right)=\sum\limits_{k\in\Z}\psi\otimes v(T_{\gol}^k.(x_-,x_+,s+R)).$$
\noindent
In the sequel, we will denote $\bar{\nu}:=\mu_{|\mathcal{D}^0}$, so that 
$\left(\tilde{m}_\G\right)_{|\mathcal{D}^0\times\R}=\bar{\nu}\otimes\dd s$ (see Paragraph 4.1.3). 
We decompose $M(R;\varphi\otimes u,\psi\otimes v)$ into $M^+(R;\varphi\otimes u,\psi\otimes v)+M^-(R;\varphi\otimes u,\psi\otimes v)$ where 
\begin{align*}
M^+(R;\varphi\otimes & u,\psi\otimes v)\\
& = \sum\limits_{k\geqslant0}\int_{\mathcal{D}^0\times\R}\varphi(x_-,x_+)u(s)\psi\otimes v(T_{\gol}^k.(x_-,x_+,s+R))\dd\bar{\nu}(x_-,x_+)\dd s
\end{align*}
\noindent
and
\begin{align*}
M^-(R;\varphi\otimes & u,\psi\otimes v)\\
& = \sum\limits_{k\geqslant1}\int_{\mathcal{D}^0\times\R}\varphi(x_-,x_+)u(s)\psi\otimes v(T_{\gol}^{-k}.(x_-,x_+,s+R))\dd\bar{\nu}(x_-,x_+)\dd s.
\end{align*}
\noindent
We first prove the following
\begin{lem}\label{reductiondeMa}\
\begin{itemize}
 \item[1)] $R^{1-\beta}L(R)M^-(R;\varphi\otimes u,\psi\otimes v)=0$ for $R$ large enough;
 \item[2)] $R^{1-\beta}L(R)M^+(R;\varphi\otimes u,\psi\otimes v)=0$ for $-R$ large enough.
\end{itemize}
\end{lem}
\begin{proof}\
Since the measure $\bar{\nu}\otimes\dd s$ is $\T_{\gol}$-invariant, we write
\begin{align*}
M^- & (R;\varphi\otimes u,\psi\otimes v)\\
& = \sum\limits_{k\geqslant1}\int_{\mathcal{D}^0\times\R}\varphi\otimes u\left(\T_{\gol}^k.(x_-,x_+,s)\right)\psi\otimes v(x_-,x_+,s+R)\dd\bar{\nu}(x_-,x_+)\dd s.
\end{align*}
\noindent
Recall that the coding of $\mathcal{D}^0$ identifies the couple $(x_-,x_+)$ with a two-sided sequence $\left(\alpha_n\right)_{n\in\Z}$. By a classical 
density argument in Ergodic Theory, it is sufficient to prove that  
$R^{1-\beta}L(R)M^-(R;\varphi\otimes u,\psi\otimes v)=0$ for functions $\varphi,\psi\ :\ \mathcal{D}^0\longrightarrow\R$ only depending on $(\alpha_n)_{n\geqslant-q}$ for some $q\geqslant0$. Using the  $\T_{\gol}$-invariance of
$\bar{\nu}\otimes\dd s$, one will impose $q=0$ in the sequel. Hence
\begin{align*}
&M^-(R;\varphi\otimes u,\psi\otimes v)\\
& = \sum\limits_{k\geqslant1}\int_{\mathcal{D}^0\times\R}\varphi\otimes u\left(\T_{\gol}^k.(p(x_-,x_+),s)\right)\psi\otimes v(p(x_-,x_+),s+R)\dd\bar{\nu}(x_-, x_+)\dd s
\end{align*}
\noindent
where $p : \mathcal{D}^0\longrightarrow\LL$ is the projection on the second coordinate. Finally
$$M^-(R;\varphi\otimes u,\psi\otimes v)=\sum\limits_{k\geqslant1}\int_{\LL\times\R}\varphi\otimes u\left(\T^k.x,s-S_k\gol(x)\right)\psi(x)v(s+R)\dd\nu(x)\dd s.$$
\noindent
Recall the definition of the operator $\tilde{P}$ given in \eqref{defiP1}: for any $x\in\LL$ and $t\in\R$
$$\tilde{P}(\psi\otimes v)(x,t)=\sum\limits_{j=1}^{p+q}\un_{\LL_j^c}(x)\sum\limits_{\alpha\in\G_j^*}e^{-\delta b(\alpha,x)}\dfrac{h\psi(\alpha.x)}{h(x)}v(t+b(\alpha,x)).$$
\noindent
By Lemma \ref{defiP}, this operator is the  adjoint of the transformation $(x,s)\longmapsto(\T.x,s-\gol(x))$ with respect to the measure $\nu\otimes\dd s$. Since the supports of $u$ and $v$ are compact, setting $R_0=\max\mathrm{supp}\ v+C-\min\mathrm{supp}\ u+1$, one gets for all $R\geqslant R_0$ and for any $k\geqslant1$
$$\tilde{P}^k\left(\psi\otimes v\right)(x,s+R)=0.$$
\noindent 
Therefore, when $R\geqslant R_0$
\begin{align*}
M^-(R;\varphi\otimes u,\psi\otimes v)=  \sum\limits_{k\geqslant1}\int_{\LL\times\R}\varphi\otimes u\left(x,s\right)\tilde{P}^k\left(\psi\otimes v\right)(x,s+R)\dd\nu(x)\dd s= 0
\end{align*}
\noindent
which proves 1). The argument is similar for 2).
\end{proof}
\noindent
In other words, we have
\begin{itemize}
 \item[$\bullet$]when $R\longrightarrow+\infty$ 
\begin{align*}
M(R;\varphi\otimes u,\psi\otimes v)= & M^+(R;\varphi\otimes u,\psi\otimes v)\\
= & \sum\limits_{k\geqslant0}\int_{\LL\times\R}\tilde{P}^k\left(\varphi\otimes u\right)(x,s-R)\psi\otimes v\left(x,s\right)\dd\nu(x)\dd s
\end{align*}
\noindent
with the convention $\tilde{P}^0\left(\varphi\otimes u\right)=\varphi\otimes u$;
 \item[$\bullet$]when $R\longrightarrow-\infty$ 
\begin{align*}
M(R;\varphi\otimes u,\psi\otimes v)= & M^-(R;\varphi\otimes u,\psi\otimes v)\\
= & \sum\limits_{k\geqslant1}\int_{\LL\times\R}\varphi\otimes u\left(x,s\right)\tilde{P}^k\left(\psi\otimes v\right)(x,s+R)\dd\nu(x)\dd s.
\end{align*}
\end{itemize}
%
\noindent
The investigation of an asymptotic for $M(R;\varphi\otimes u,\psi\otimes v)$ thus relies 
on similar arguments in case $R\longrightarrow+\infty$ or $R\longrightarrow-\infty$; in the sequel, we just explain how to obtain the asymptotic 
for $R\longrightarrow+\infty$ and assume $R\geqslant R_0$ to ensure $M(R;\varphi\otimes u,\psi\otimes v)=M^+(R;\varphi\otimes u,\psi\otimes v)$. From now on, we omit the symbol $+$. The following subsection is devoted to the proof of the asymptotic of $M(R;\varphi\otimes u,\psi\otimes v)$ when $\beta\in]0,1[$.
\subsection{Theorem A for $\boldsymbol{\beta\in]0,1[}$}
We have
\begin{align*}
M(R;\varphi\otimes u,\psi\otimes v)=\sum\limits_{k\geqslant0}M_k(R;\varphi\otimes u,\psi\otimes v)
\end{align*}
\noindent
where, for any $k\geqslant0$,
$$M_k(R;\varphi\otimes u,\psi\otimes v)=\int_{\LL\times\R}\tilde{P}^k\left(\varphi\otimes u\right)(x,s-R)\psi\otimes v\left(x,s\right)\dd\nu(x)\dd s.$$ 
\noindent
We follow the steps of the proof of Theorem  1.4 in \cite{Gou}. Let $(a_k)_{k\geqslant1}$ satisfying 
$kL(a_k)=a_k^{\beta}$, where $L$ is the slowly varying function given in the family of assumptions $(H_\beta)$. We postpone the proofs of the following propositions in Paragraphs 5.2.2 and 5.2.3.
\begin{propA1}\label{M.1}
Let $\varphi,\psi\ :\ \LL\longrightarrow\R$ be two Lipschitz functions and $u,v\ :\ \R\longrightarrow\R$ be two continuous functions with compact support.
Uniformly in $K\geqslant2$ and $R\in[0,Ka_k]$, one gets when $k\longrightarrow+\infty$,
$$M_k(R;\varphi\otimes u,\psi\otimes v)=\dfrac{1}{\cg a_k}\left(\Psi_{\beta}\left(\dfrac{R}{\cg a_k}\right)m_\G(\varphi\otimes u)m_\G(\psi\otimes v)+o_k(1)\right),$$
\noindent
where $\Psi_{\beta}$ is the density of the fully asymmetric stable law with parameter $\beta$ and 
$\cg=\Cg^{\frac{1}{\beta}}$.
\end{propA1}
\begin{propA2}\label{M.2}
Let $\varphi,\psi\ :\ \LL\longrightarrow\R$ be two Lipschitz functions and $u,v\ :\ \R\longrightarrow\R$ be two continuous functions with compact support. When $R\geqslant a_k$, there exists a constant $C>0$ depending on $u$ such that
$$\left|M_k(R;\varphi\otimes u,\psi\otimes v)\right|\leqslant Ck\dfrac{L(R)}{R^{1+\beta}}\left|\left|\varphi\otimes u\right|\right|_{\infty}\left|\left|\psi\otimes v\right|\right|_{\infty}.$$
\end{propA2}
We now explain how they imply Theorem A.
\subsubsection{Asymptotic of $M(R;\varphi\otimes u,\psi\otimes v)$}
Using Proposition A.1, we decompose $ M(R;\varphi\otimes u,\psi\otimes v)$ as
$$M^1(R;\varphi\otimes u,\psi\otimes v)+M^2(R;\varphi\otimes u,\psi\otimes v)+M^3(R;\varphi\otimes u,\psi\otimes v)$$
\noindent
where
\begin{align*}
M^1(R;\varphi\otimes u,\psi\otimes v) & := m_\G(\varphi\otimes u)m_\G(\psi\otimes v)\sum\limits_{k\ |\ R<Ka_k}\dfrac{1}{\cg a_k}\Psi_{\beta}\left(\dfrac{R}{\cg a_k}\right),\\
M^2(R;\varphi\otimes u,\psi\otimes v) & := \sum\limits_{k\ |\ R<Ka_k}\dfrac{o_k(1)}{\cg a_k},\\ M^3(R;\varphi\otimes u,\psi\otimes v) & := \sum\limits_{k\ |\ R\geqslant Ka_k}M_k(R;\varphi\otimes u,\psi\otimes v).
\end{align*}
\noindent
\begin{itemize}
\item[a)]{\it Contribution of $M^1(R;\varphi\otimes u,\psi\otimes v)$.} Following \cite{Gou}, we introduce the measure
$\mu_R=\sum\limits_{0<\frac{R}{a_k}\leqslant K} D_{\frac{R}{a_k}}$ on $\R$ so that
\begin{equation}\label{mesureequivalent}
\sum\limits_{k\ |\ R<Ka_k}\dfrac{1}{\cg a_k}\Psi_{\beta}\left(\dfrac{R}{\cg a_k}\right)=
\dfrac{1}{R}\displaystyle{\int_0^K}\dfrac{z}{\cg}\Psi_{\beta}\left(\dfrac{z}{\cg}\right)\dd\mu_R(z).
\end{equation}
\noindent
The definition of $\mu_R$ implies $\mu_R([x,y])=\sum\limits_{k\ |\ x\leqslant\frac{R}{a_k}\leqslant y}1=\sum\limits_{k\ |\ \frac{R}{y}\leqslant a_k\leqslant\frac{R}{x}}1$, for any $[x,y]\subset]0,K]$. Recall some properties of functions $A(t)=\frac{t^{\beta}}{L(t)}$ and of its pseudo-inverse $A^*$ given in Section 3. First of all $A(a_n)=n$. Moreover, the function $A^*$ is a regularly varying function with exponent $\frac{1}{\beta}$ which satisfies $a_n=A^*(n)$.
\begin{lem}
For $R$ large enough
\begin{align*}
&\left\{k\ |\ A\left(\dfrac{R}{y}\right)\leqslant k\leqslant A\left(\dfrac{R}{x}-1\right)\right\}\subset
\left\{k\ |\ \dfrac{R}{y}\leqslant a_k\leqslant\dfrac{R}{x}\right\}\\
\text{and}\ &\left\{k\ |\ \dfrac{R}{y}\leqslant a_k\leqslant\dfrac{R}{x}\right\}\subset
\left\{k\ |\ A\left(\dfrac{R}{y}\right)\leqslant k\leqslant A\left(\dfrac{R}{x}\right)\right\}.
\end{align*}
\end{lem}
\begin{proof}
Since the function $A^*$ is increasing and satisfies $A^*(A(t))\geqslant t$ and $A^*(A(t-1))\leqslant t$ for any $t\gg1$ 
(see \cite{Tom}), we obtain
\begin{itemize}
 \item[-]$A\left(\frac{R}{y}\right)\leqslant k$ implies $\frac{R}{y}\leqslant A^*(k)$;
 \item[-]$k\leqslant A\left(\frac{R}{x}-1\right)$ implies $A^*(k)\leqslant\frac{R}{x}$.
\end{itemize}
\end{proof}
\noindent
We deduce from this lemma that
$$\sum\limits_{A\left(\frac{R}{y}\right)\leqslant k\leqslant A\left(\frac{R}{x}-1\right)}1\leqslant
\mu_R([x,y])\leqslant  \sum\limits_{A\left(\frac{R}{y}\right)\leqslant k\leqslant A\left(\frac{R}{x}\right)}1,$$
which yields
$$A\left(\frac{R}{x}-1\right)-A\left(\frac{R}{y}\right)\leqslant
\mu_R([x,y])\leqslant  A\left(\frac{R}{x}\right)-A\left(\frac{R}{y}\right);$$
\noindent
hence
$$
\mu_R([x,y])\sim  A\left(\dfrac{R}{x}\right)-A\left(\dfrac{R}{y}\right)
\sim\dfrac{R^{\beta}}{L(R)}\left(x^{-\beta}-y^{-\beta}\right)
$$
\noindent
and finally
\begin{equation}\label{mesureequivalent2}
R^{-\beta}L(R)\mu_R([x,y])\sim\displaystyle{\int_x^y}\beta z^{-\beta-1}\dd z.
\end{equation}
\noindent
We now want to control quantities of the form
$$\sum\limits_{k\ |\ R<a_k\varepsilon}\dfrac{1}{\cg a_k}\Psi_\beta\left(\dfrac{R}{\cg a_k}\right)$$
\noindent
for $\varepsilon>0$ small enough. For any $R\geqslant1$ and $\varepsilon>0$, 
we write
\begin{equation}\label{ahminceoubliee}
\left|R^{1-\beta}L(R)\sum\limits_{k\ |\ R<a_k\varepsilon}\dfrac{1}{\cg a_k}\Psi_\beta\left(\dfrac{R}{\cg a_k}\right)\right|\preceq
R^{1-\beta}L(R)\sum\limits_{k\ |\ R<\varepsilon a_k}\dfrac{1}{a_k}.
\end{equation}
\noindent
Let $N=N(\varepsilon,R)$ denote the first integer such that $R<\varepsilon a_k$: 
$N$ is increasing in $R$. Karamata's lemma \ref{regvar} then implies
$$\left|\sum\limits_{k\ |\ R<\varepsilon a_k}\dfrac{1}{a_k}\right|=\sum\limits_{k\geqslant N}\dfrac{1}{a_k}\sim 
\dfrac{N}{a_N}.$$
\noindent
From $a_N^\beta=NL(a_N)$, we deduce $\frac{N}{a_N}=\frac{1}{a_N^{1-\beta}L(a_N)}$, so that
\begin{equation}\label{ohouiouicchiantepsilon}
R^{1-\beta}L(R)\dfrac{N}{a_N}\leqslant\varepsilon^{1-\beta}\dfrac{L(R)}{L(a_N)}\leqslant
\varepsilon^{1-\beta}\max\left(\dfrac{R}{a_N},\dfrac{a_N}{R}\right)^{\frac{1-\beta}{2}},
\end{equation}
\noindent
where the last inequality is a consequence of Potter's lemma \ref{PB} with $B=1$, $\rho=\frac{1-\beta}{2}$, $x=R$ and $y=a_N$. It follows from the definition of $N$ that $\frac{R}{a_N}<\varepsilon$ and $\varepsilon a_{N-1}\leqslant R$; hence
$$\dfrac{a_N}{R}=\dfrac{a_{N-1}}{R}\dfrac{a_N}{a_{N-1}}\leqslant
\dfrac{1}{\varepsilon}\dfrac{a_N}{a_{N-1}}\preceq\dfrac{1}{\varepsilon}$$
\noindent
for $N$ large enough. Finally, this last estimate combined with \eqref{ohouiouicchiantepsilon} yields
$$
R^{1-\beta}L(R)\dfrac{N}{a_N}\preceq\varepsilon^{\frac{1-\beta}{2}}.
$$
\noindent
By \eqref{ahminceoubliee}, for any arbitrarily small $\eta>0$, there exists $\varepsilon_\eta>0$ such that, as $\varepsilon<\varepsilon_\eta$  
\begin{equation}\label{indicationGouezel2}
\left|R^{1-\beta}L(R)\sum\limits_{k\ |\ R<a_k\varepsilon}\dfrac{1}{\cg a_k}\Psi_\beta\left(\dfrac{R}{\cg a_k}\right)\right|\preceq\eta.
\end{equation}
\noindent
Properties \eqref{mesureequivalent}, \eqref{mesureequivalent2} and \eqref{indicationGouezel2} imply, for $R$ large enough
\begin{align*}
R^{1-\beta}L(R)\sum\limits_{k\ |\ R<Ka_k} & \dfrac{1}{\cg a_k}\Psi_\beta\left(\dfrac{R}{\cg a_k}\right)\\
= & O(\eta)+R^{1-\beta}L(R)\int_{\varepsilon_\eta}^K\dfrac{z}{\cg}\Psi_\beta\left(\dfrac{z}{\cg}\right)\dd\mu_R(z)\\
= & O(\eta)+\left(\dfrac{\beta}{\cg}\int_{\varepsilon_\eta}^Kz^{-\beta}\Psi_\beta\left(\dfrac{z}{\cg}\right)\dd z\right)(1+o(1)),
\end{align*}
\noindent
with $\lim\limits_{R\longrightarrow+\infty}o(1)=0$. From the integrability of $z\longmapsto z^{-\beta}\Psi_\beta(z)$ on $[0,+\infty[$ (see \cite{Zol}), we deduce that
$$\left|\int_0^{\varepsilon_\eta}z^{-\beta}\Psi\left(\dfrac{z}{\cg}\right)\dd z\right|\leqslant\eta$$
\noindent
up to a reduction of $\varepsilon_\eta$; hence when $R\longrightarrow+\infty$
\begin{equation}\label{indicationGouzelconclusion}
R^{1-\beta}L(R)\sum\limits_{k\ |\ R<Ka_k}\dfrac{1}{\cg a_k}\Psi_\beta\left(\dfrac{R}{\cg a_k}\right)\sim
\dfrac{\beta}{\cg}\int_{0}^Kz^{-\beta}\Psi_\beta\left(\dfrac{z}{\cg}\right)\dd z\pm\eta.
\end{equation}
\noindent
If thus follows from the definition of $M^{1}(R;\varphi\otimes u,\psi\otimes v)$ and from \eqref{indicationGouzelconclusion} that, when $R\longrightarrow+\infty$
\begin{equation}\label{premierepartdominant}
R^{1-\beta}L(R)M^{1}(R;\varphi\otimes u,\psi\otimes v)\sim
\dfrac{\beta m_\G(\varphi\otimes u)m_\G(\psi\otimes v)}{\Cg}\displaystyle{\int_0^{\frac{K}{\cg}}}z^{-\beta}\Psi_{\beta}(z)\dd z
\end{equation}
\noindent
with $\Cg=\cg^{\beta}$.  
\item[b)]{\it Contribution of $M^{2}(R;\varphi\otimes u,\psi\otimes v)$.} Let $N=N(K,R)$ be the smallest integer such that $Ka_N>R$: the function $R\longmapsto N(K,R)$ is increasing in $R$. Let $\varepsilon>0$; for $R$ large enough and any $k\geqslant N$, we get 
$|o_k(1)|\leqslant\varepsilon$. Karamata's lemma thus implies
$$\left|\sum\limits_{k\ |\ R<Ka_k}\dfrac{o_k(1)}{\cg a_k}\right|\leqslant\sum\limits_{k\geqslant N}\dfrac{\varepsilon}{a_k}\sim 
\varepsilon\dfrac{N}{a_N}.$$
%
%
%
%
%
%
\noindent
Following the same steps as in a) for the negligible parts of  
$M^{1}\left(R;\varphi\otimes u,\psi\otimes v\right)$, we deduce from $\frac{N}{a_N}=\frac{1}{a_n^{1-\beta}L(a_N)}$ that 
$$R^{1-\beta}L(R)\dfrac{N}{a_N}\leqslant K^{1-\beta}\dfrac{L(R)}{L(a_N)}$$
\noindent
and Potter's lemma with $B=\rho=1$ and $x=R$ and $y=a_N$ yields 
$$R^{1-\beta}L(R)\frac{N}{a_N}\preceq K^{2-\beta},$$
\noindent 
so that when $R\longrightarrow+\infty$
\begin{equation}\label{premierepartnegligeable}
R^{1-\beta}L(R)M^{2}\left(R;\varphi\otimes u,\psi\otimes v\right)=o_K(1).
\end{equation}
\item[c)]{\it Contribution of $M^3(R;\varphi\otimes u,\psi\otimes v)$.} We write
$$\left|M^3(R;\varphi\otimes u,\psi\otimes v)\right|\leqslant\sum\limits_{k\ |\ R\geqslant Ka_k}\left|M_k(R;\varphi\otimes u,\psi\otimes v)\right|.$$
\noindent
It follows from Proposition A.2 that
$$M^3(R;\varphi\otimes u,\psi\otimes v)\leqslant \left|\left|\varphi\otimes u\right|\right|_{\infty}\left|\left|\psi\otimes v\right|\right|_{\infty}C\dfrac{L(R)}{R^{1+\beta}}\sum\limits_{k\ |\ R\geqslant Ka_k}k,$$
\noindent
and from the definition of $A$, we deduce that
$$\sum\limits_{k\ |\ R\geqslant Ka_k}k=\sum\limits_{k\ |\ k\leqslant A\left(\frac{R}{K}\right)}k\leqslant A\left(\dfrac{R}{K}\right)^2,$$
\noindent
therefore 
$$\left|M^3(R;\varphi\otimes u,\psi\otimes v)\right|\leqslant\left|\left|\varphi\otimes u\right|\right|_{\infty}
\left|\left|\psi\otimes v\right|\right|_{\infty}C\dfrac{L(R)}{R^{1+\beta}}\dfrac{R^{2\beta}}{K^{2\beta}}\dfrac{1}{L\left(\frac{R}{K}\right)^2}.$$
\noindent
Potter's lemma \eqref{PB} implies $\frac{L(R)}{L(R/K)}\leqslant K^{\frac{\beta}{2}}$ for $R$ large enough, so that 
\begin{equation}\label{deuxiemepart}
R^{1-\beta}L(R)M^3(R;\varphi\otimes u,\psi\otimes v)\leqslant C\left|\left|\varphi\otimes u\right|\right|_{\infty}\left|\left|\psi\otimes v\right|\right|_{\infty}K^{-\beta}.
\end{equation}
\end{itemize}
Combining \eqref{premierepartdominant}, \eqref{premierepartnegligeable} and \eqref{deuxiemepart}, it follows that
\begin{align*}
R^{1-\beta}L(R)M(R;\varphi\otimes u, & \psi\otimes v)\\
= & \dfrac{\beta m_\G(\varphi\otimes u)m_\G(\psi\otimes v)}{\Cg}\displaystyle{\int_0^{\frac{K}{\cg}}}z^{-\beta}\Psi_{\beta}(z)\dd z(1+o(1))\\
& + o_K(1) + O(K^{-\beta}),\\
\end{align*}
\noindent
with $\lim\limits_{R\longrightarrow+\infty}o(1)=0$ and $\lim\limits_{R\longrightarrow+\infty}o_K(1)=0$ for any fixed $K$. Then letting $R\longrightarrow+\infty$, we obtain
$$R^{1-\beta}L(R)M(R;\varphi\otimes u,\psi\otimes v)\sim\dfrac{\beta m_\G(\varphi\otimes u)m_\G(\psi\otimes v)}{\Cg}\displaystyle{\int_0^{\frac{K}{\cg}}}z^{-\beta}\Psi_{\beta}(z)\dd z+O(K^{-\beta}).$$
\noindent
Letting $K\longrightarrow+\infty$ and using $\displaystyle{\int_0^{+\infty}}z^{-\beta}\Psi_{\beta}(z)\dd z=\dfrac{\sin(\beta\pi)}{\beta\pi}$ 
(see \cite{Zol}), this achieves the proof of Theorem A in the case $\beta\in]0,1[$. 
\subsubsection{Proof of Proposition A.1}
We want to prove the following local limit theorem
\begin{propA1}
Let $\varphi,\psi\ :\ \LL\longrightarrow\R$ be two Lipschitz functions and $u,v\ :\ \R\longrightarrow\R$ be two continuous functions with compact support.
Uniformly in $K\geqslant2$ and $R\in[0,Ka_k]$, one gets when $k\longrightarrow+\infty$
$$M_k(R;\varphi\otimes u,\psi\otimes v)=\dfrac{1}{\cg a_k}\left(\Psi_{\beta}\left(\dfrac{R}{\cg a_k}\right)m_\G(\varphi\otimes u)m_\G(\psi\otimes v)+o_k(1)\right),$$
\noindent
where $\Psi_{\beta}$ is the density of the fully asymmetric stable law with parameter $\beta$ and 
$\cg=\Cg^{\frac{1}{\beta}}$.
\end{propA1}
Let us fix until the end of this paragraph $K\geqslant$ and $R\gg1$. For all $k\in\N$ such that $Ka_k>R$,

$$M_k(R;\varphi\otimes u,\psi\otimes v)=\int_{\LL\times\R}\tilde{P}^k\left(\varphi\otimes u\right)(x,s-R)\psi\otimes v\left(x,s\right)\nu(\dd x)\dd s.$$
\noindent
We have to prove that the following sequence of measures 

$$\left(a_kM_k(R;\varphi\otimes\bullet,\psi\otimes v)-\dfrac{1}{\cg}\Psi_{\beta}\left(\dfrac{R}{\cg a_k}\right)\left(\nu(\varphi)\int_\R\bullet(x)\dd x\right)m_\G(\psi\otimes v)\right)_{k|Ka_k>R},$$
\noindent
weakly converges to $0$ when $k\longrightarrow+\infty$, uniformly in $K$ and $R$. Using an argument of Stone (\cite{BP} p.106), it is sufficient to show that 
$$a_kM_k(R;\varphi\otimes u,\psi\otimes v)-\dfrac{1}{\cg}\Psi_{\beta}\left(\dfrac{R}{\cg a_k}\right)m_\G(\psi\otimes v)\nu(\varphi)\widehat{u}(0)\longrightarrow0$$
\noindent
for functions $u\ :\ \R\longrightarrow\R$ such that $u$ and $|u|$ have a $\mathcal{C}^{\infty}$ Fourier transform with compact support. More precisely, let us introduce the 
\begin{defi}\label{classedefonctionstest}
Let $\mathscr{U}$ be the set of test functions $u$ of the form $u(x)=e^{itx}u_0(x)$ where $t\in\R$ and $u_0$ belongs to the set of positive integrable function from $\R$ to $\R$, whose Fourier transform is $\mathcal{C}^{\infty}$ and with compact support.
\end{defi}
We first notice that $a_kM_k(R;\varphi\otimes u,\psi\otimes v)$ is finite for $\varphi\in\mathscr{U}$. By the Fourier inverse formula and the definition of 
$\tilde{P}$ given in Lemma \ref{defiP}, we write
\begin{align*}
\tilde{P}^k(\varphi\otimes u)(x,s-R) = & \sum\limits_{\g\in\G(k)}\un_{\LL_{l(\g)}^c}(x)e^{-\delta b(\g,x)}\dfrac{h\varphi(\g.x)}{h(x)}u(s-R+b(\g,x))\\
= & \dfrac{1}{2\pi h(x)}\int_{\R}e^{it(R-s)}\left(\ot_{\delta+it}^kh\varphi\right)(x)\widehat{u}(t)\dd t.
\end{align*}
\noindent
This quantity is bounded from above by $||h\varphi||_{\infty}||\widehat{u}||_1$, up to a multiplicative constant. Since the function $\psi\otimes v$ is 
integrable on $\LL\times\R$, the quantity $M_k(R;\varphi\otimes u,\psi\otimes v)$ is finite for any $u\in\mathscr{U}$.

Proposition A.1 is a consequence of the following lemma, combined with the Lebesgue dominated convergence theorem
\begin{lem}\label{butmelange}For all $u\in\mathscr{U}$, when $k\longrightarrow+\infty$,
$$\dfrac{a_k}{2\pi h(x)}\int_{\R}e^{it(R-s)}\left(\ot_{\delta+it}^kh\varphi\right)(x)\widehat{u}(t)\dd t-
\dfrac{1}{\cg}\Psi_{\beta}\left(\dfrac{R}{\cg a_k}\right)\nu(\varphi)\widehat{u}(0)\longrightarrow0,$$
\noindent
uniformly in $x\in\LL$, $s\in\mathrm{supp}\ v$, $R$ and $K$.
\end{lem}
\begin{proof}
Fix $\varepsilon>0$ according to Proposition \ref{spectreperturbation}. By the Fourier inverse formula applied to $\Psi_\beta$, we may write
$$\dfrac{a_k}{2\pi h(x)}\int_{\R}e^{it(R-s)}\left(\ot_{\delta+it}^kh\varphi\right)(x)\widehat{u}(t)\dd t-
\dfrac{1}{\cg}\Psi_{\beta}\left(\dfrac{R}{\cg a_k}\right)\nu(\varphi)\widehat{u}(0)=K_1(k)+K_2(k)$$
\noindent
where
$$K_1(k)=\dfrac{a_k}{2\pi h(x)}\int_{[-\varepsilon,\varepsilon]^c}e^{it(R-s)}\left(\ot_{\delta+it}^kh\varphi\right)(x)\widehat{u}(t)\dd t$$ 
\noindent 
and
\begin{align*}
K_2(k)= & \dfrac{a_k}{2\pi h(x)}\int_{[-\varepsilon,\varepsilon]}e^{it(R-s)}\left(\ot_{\delta+it}^kh\varphi\right)(x)\widehat{u}(t)\dd t-
\dfrac{1}{2\pi}\int_{\R}e^{it\frac{R}{a_k}}g_{\beta}(\cg t)\nu(\varphi)\widehat{u}(0)\dd t\\
= &\dfrac{1}{2\pi h(x)}\int_{-\varepsilon a_k}^{\varepsilon a_k}
e^{it\frac{R-s}{a_k}}\left(\ot_{\delta+i\frac{t}{a_k}}^kh\varphi\right)(x)\widehat{u}\left(\dfrac{t}{a_k}\right)\dd t\\
&-\dfrac{1}{2\pi}\int_{\R}e^{it\frac{R}{a_k}}g_{\beta}(\cg t)\nu(\varphi)\widehat{u}(0)\dd t. 
\end{align*}
\noindent
There exists $\rho\in]0,1[$ such that
$\left|\left|\ot_{\delta+it}\right|\right|\leqslant\rho$ for any
$t\in(\mathrm{supp}\ \widehat{u})\cap\left(\R\setminus[-\varepsilon,\varepsilon]\right)$. Hence $\left|K_1(k)\right|\preceq a_k\rho^k$, quantity which goes to $0$ 
uniformly in $x$, $s$, $R$ and $K$ when $k\longrightarrow+\infty$.

Let us deal with $K_2(k)$. Using the spectral decomposition of $\ot_{\delta+i\frac{t}{a_k}}$, we write for any $t\in[-\varepsilon a_k,\varepsilon a_k]$
$$\ot_{\delta+i\frac{t}{a_k}}^k(h\varphi)=\lambda_{\delta+i\frac{t}{a_k}}^k\Pi_{\delta+i\frac{t}{a_k}}(h\varphi)+R_{\delta+i\frac{t}{a_k}}^k(h\varphi)$$
\noindent
where the spectral radius of $R_{\delta+i\frac{t}{a_k}}$ is $\leqslant\rho_\varepsilon<1$. The quantity $K_2(k)$ may be splited into $L_1(k)+L_2(k)+L_3(k)$ where
$$L_1(k)=\dfrac{1}{2\pi h(x)}\int_{-\varepsilon a_k}^{\varepsilon a_k}
e^{it\frac{R-s}{a_k}}R_{\delta+i\frac{t}{a_k}}^k(h\varphi)(x)\widehat{u}\left(\dfrac{t}{a_k}\right)\dd t,$$
$$
L_2(k)=\dfrac{1}{2\pi h(x)}\int_{-\varepsilon a_k}^{\varepsilon a_k}
e^{it\frac{R-s}{a_k}}\lambda_{\delta+i\frac{t}{a_k}}^k\left(\Pi_{\delta+i\frac{t}{a_k}}(h\varphi)(x)-\Pi_{\delta}(h\varphi)(x)\right)
\widehat{u}\left(\dfrac{t}{a_k}\right)\dd t 
$$
\noindent
and
$$
L_3(k)=\dfrac{1}{2\pi h(x)}\int_{-\varepsilon a_k}^{\varepsilon a_k}
e^{it\frac{R-s}{a_k}}\lambda_{\delta+i\frac{t}{a_k}}^k\Pi_{\delta}(h\varphi)(x)\widehat{u}\left(\dfrac{t}{a_k}\right)\dd t
-\dfrac{1}{2\pi}\int_{\R}e^{it\frac{R}{a_k}}g_{\beta}(\cg t)\nu(\varphi)\widehat{u}(0)\dd t.
$$
First $\left|L_1(k)\right|\preceq a_k\rho_\varepsilon^k$, quantity which goes to $0$ uniformly in $x$, $s$, $R$ and $K$ when $k\longrightarrow+\infty$. The characteristic function $g_\beta$ of a stable law with parameter $\beta$ (see Section 3) is given by $g_{\beta}(t)=\exp{\left(-\G(1-\beta)e^{i\mathrm{sign}(t)\frac{\beta\pi}{2}}|t|^{\beta}\right)}$. We may notice that $\left|g_{\beta}(t)\right|\leqslant e^{-(1-\beta)\G(1-\beta)|t|^{\beta}}$ for any 
$t\in\R$, which ensures that $g_{\beta}$ is integrable on $\R$. Moreover, Proposition \ref{localexp} implies that for any $t$ close to $0$, the dominant 
eigenvalue $\lambda_{\delta+it}$ satisfies
$$\lambda_{\delta+it}=e^{-\G(1-\beta)e^{i\mathrm{sign}(t)\frac{\beta\pi}{2}}|\cg t|^{\beta}L\left(\frac{1}{|t|}\right)(1+o(1))}.$$
\noindent 
The regularity of $t\longmapsto\Pi_{\delta+it}$ implies that the integrand of $L_2(k)$ goes to $0$ uniformly in $x$, $s\in\mathrm{supp}\ v$, $R$ and $K$; 
thus, it is sufficient to bound it by an integrable function. By Proposition \ref{continuityoftransfert}, we obtain
$$\left|\Pi_{\delta+i\frac{t}{a_k}}(h\varphi)(x)-\Pi_{\delta}(h\varphi)(x)\right|\leqslant
\left|\left|\Pi_{\delta+i\frac{t}{a_k}}-\Pi_{\delta}\right|\right|\left|\left|h\varphi\right|\right|\preceq\dfrac{|t|^{\beta}}{a_k^{\beta}}L\left(\dfrac{a_k}{|t|}\right),$$
\noindent
with
$$\dfrac{L\left(\frac{a_k}{|t|}\right)}{L(a_k)}\leqslant\max\left(\dfrac{1}{|t|},|t|\right)^{\frac{\beta}{2}}$$
\noindent
for any $k$ large enough and uniformly in $t\in[-\varepsilon a_k,\varepsilon a_k]$, where $\varepsilon>0$ is chosen small enough according to Remark \ref{sertA1etB1}. Therefore
$$\left|\Pi_{\delta+i\frac{t}{a_k}}(h\varphi)(x)-\Pi_{\delta}(h\varphi)(x)\right|\leqslant
\left\{\begin{array}{ll}
              &\dfrac{L(a_k)}{a_k^{\beta}}|t|^{\frac{\beta}{2}}\ \text{if}\ |t|\leqslant1\\
              &\dfrac{L(a_k)}{a_k^{\beta}}|t|^{\frac{3\beta}{2}}\ \text{if}\ |t|>1\\
       \end{array}\right..$$
\noindent
Similarly, the inequalities
$$\left|\lambda_{\delta+i\frac{t}{a_k}}^k\right|\leqslant e^{-(1-\beta)\G(1-\beta)|\cg t|^{\beta}\frac{k}{a_k^{\beta}}L(a_k)\frac{L\left(a_k/|t|\right)}{L(a_k)}(1+o(1))}$$
\noindent
and
$$\min\left(\dfrac{1}{|t|},|t|\right)^{\frac{\beta}{2}}\leqslant\dfrac{L\left(\frac{a_k}{|t|}\right)}{L(a_k)}$$
\noindent
yield
$$\left|\lambda_{\delta+i\frac{t}{a_k}}^k\right|\leqslant
\left\{\begin{array}{ll}
              &e^{-\frac{1}{4}(1-\beta)\G(1-\beta)|\cg t|^{\frac{3\beta}{2}}} \text{if}\ |t|\leqslant1\\
              &e^{-\frac{1}{4}(1-\beta)\G(1-\beta)|\cg t|^{\frac{\beta}{2}}}\ \text{if}\ |t|>1\\
       \end{array}\right..$$
\noindent
Finally, for $k$ large enough, the integrand of $L_2(k)$ may be bounded from above by the function
$$l(t):=\left\{\begin{array}{ll}
        &|t|^{\frac{\beta}{2}}e^{-\frac{1}{4}(1-\beta)\G(1-\beta)|\cg t|^{\frac{3\beta}{2}}} \text{if}\ |t|\leqslant1\\
        &|t|^{\frac{3\beta}{2}}e^{-\frac{1}{4}(1-\beta)\G(1-\beta)|\cg t|^{\frac{\beta}{2}}}\ \text{if}\ |t|>1\\
\end{array}\right.,$$
\noindent
up to a multiplicative constant. 

On the other hand, since $\Pi_\delta(h\varphi)=\nu(\varphi)h$, we decompose $L_3(k)$ into $M_1(k)+M_2(k)+M_3(k)$ where 
$$M_1(k)=\dfrac{\nu(\varphi)\widehat{u}(0)}{2\pi}\int_{[-\varepsilon a_k,\varepsilon a_k]^c}e^{it\frac{R}{a_k}}g_{\beta}(\cg t)\dd t,$$
$$M_2(k)=\dfrac{\nu(\varphi)}{2\pi}\int_{-\varepsilon a_k}^{\varepsilon a_k}
e^{it\frac{R}{a_k}}g_{\beta}(\cg t)\left(\widehat{u}(0)-\widehat{u}\left(\dfrac{t}{a_k}\right)\right)\dd t$$
\noindent
and
$$M_3(k)=\dfrac{\nu(\varphi)}{2\pi}\int_{-\varepsilon a_k}^{\varepsilon a_k}
e^{it\frac{R}{a_k}}\left(e^{-it\frac{s}{a_k}}\lambda_{\delta+i\frac{t}{a_k}}^k-g_{\beta}(\cg t)\right)\widehat{u}\left(\dfrac{t}{a_k}\right)\dd t.$$
\noindent
The term $M_1(k)$ goes to $0$ uniformly in $x$, $s$, $R$ and $K$ and so does $M_2(k)$, thanks to the mean value relation applied to $\widehat{u}$ on $[-\varepsilon,\varepsilon]$ combined with the Lebesgue dominated convergence theorem. Similarly, the integrand of $M_3(k)$ goes to $0$ uniformly in $s\in\mathrm{supp}\ v$, and we bound 
$\left|e^{-it\frac{s}{a_k}}\lambda_{\delta+i\frac{t}{a_k}}^k-g_{\beta}(\cg t)\right|$ from above by 
$e^{-\frac{1}{4}(1-\beta)\G(1-\beta)|\cg t|^{\frac{3\beta}{2}}}+\left|g_{\beta}(\cg t)\right|$ if $|t|\leqslant 1$ and by
$e^{-\frac{1}{4}(1-\beta)\G(1-\beta)|\cg t|^{\frac{\beta}{2}}}+\left|g_{\beta}(\cg t)\right|$ otherwise. This achieves the proof of Lemma 
\ref{butmelange}.
\end{proof}
\subsubsection{Proof of Proposition A.2}
We now give a control of the non-influent terms $M_k$ appearing in the proof of Theorem A. Let us fix 
$k\in\N$ such that $Ka_k\leqslant R$. Once again, we recall that
$$M_k(R;\varphi\otimes u,\psi\otimes v)=\int_{\LL\times\R}\tilde{P}^k\left(\varphi\otimes u\right)(x,s-R)\psi\otimes v\left(x,s\right)\nu(\dd x)\dd s,$$
\noindent
where $\tilde{P}$ is given in \eqref{defiP1}. To show Proposition A.2., it is sufficient to check that
$$\left|\tilde{P}^k\left(\varphi\otimes u\right)(x,s-R)\right|\leqslant Ck\dfrac{L(R)}{R^{1+\beta}}\left|\left|\varphi\otimes u\right|\right|_{\infty}$$
\noindent
uniformly in $x\in\LL$ and $s\in\mathrm{supp}\ v$. We write
\begin{align*}
\left|\tilde{P}^k\left(\varphi\otimes u\right)(x,s-R)\right| \leqslant & 
\dfrac{1}{h(x)}\sum\limits_{\g\in\G(k)}\un_{\LL_{l(\g)}^c}(x)e^{-\delta b(\g,x)}\left|(h\varphi)(\g.x)u(s-R+b(\g,x))\right|\\
\preceq & \left|\left|\varphi\otimes u\right|\right|_{\infty}\sum\limits_{\underset{b(\g,x)\overset{M}{\sim} R-s}{\g\in\G(k)}}\un_{\LL_{l(\g)}^c}(x)e^{-\delta b(\g,x)}
\end{align*}
\noindent
where $a\overset{M}{\sim} b$ means $|a-b|\leqslant M$ and the parameter $M$ satisfies $\mathrm{supp}\ u\subset[-M,M]$.
This notation also emphasizes that the only $\g$ which really appear in the above
sum are the ones such that $b(\g,x)$ has the same order than $R-s$. Therefore we only have to show that
\begin{equation}\label{butA2}
\sum\limits_{\underset{b(\g,x)\overset{M}{\sim} R-s}{\g\in\G(k)}}\un_{\LL_{l(\g)}^c}(x)e^{-\delta b(\g,x)}\leqslant Ck\dfrac{L(R)}{R^{1+\beta}}
\end{equation}
\noindent
where $C$ depends on the support of $u$. The proof is inspired of the one of Theorem 1.6 in \cite{Gou}. We will need the two following lemmas.
\begin{lem}\label{A21}
There exists a constant $C>0$ such that for any $k\geqslant1$ and $x\in\LL$
$$\sum\limits_{\g\in\G(k)}\un_{\LL_{l(\g)}^c}(x)e^{-\delta b(\g,x)}\leqslant C.$$
\end{lem}
\begin{proof}
From properties of the function $h$, we derive
$$\sum\limits_{\g\in\G(k)}\un_{\LL_{l(\g)}^c}(x)e^{-\delta b(\g,x)}\preceq
\sum\limits_{\g\in\G(k)}\un_{\LL_{l(\g)}^c}(x)e^{-\delta b(\g,x)}h(\g.x)\preceq
\ot_{\delta}^kh(x)\leqslant|h|_{\infty}.$$
\end{proof}
\begin{lem}\label{A22}
Let $\Delta>0$. There exists a constant $C_\Delta>0$ such that for any $k\geqslant1$, any $x\in\LL$ and
$\zeta\in[\dfrac{a_k}{2},\infty]$
$$\sum\limits_{\underset{\forall i,\ \dhy(\oo,\alpha_i.\oo)\leqslant\zeta}{\underset{R-\Delta\leqslant b(\g,x)\leqslant R+\Delta}{\g=\alpha_1...\alpha_k}}}
\un_{\LL_{l(\g)}^c}(x)e^{-\delta b(\g,x)}\leqslant C_\Delta\dfrac{e^{-\frac{R}{\zeta}}}{a_k}.$$
\end{lem}
\begin{proof}
Let $f\ :\ \R\longrightarrow\R$ be a positive function whose Fourier transform has compact support. We write
\begin{align*}
\sum\limits_{\underset{\forall i,\ \dhy(\oo,\alpha_i.\oo)\leqslant\zeta}{\underset{R-\Delta\leqslant b(\g,x)\leqslant R+\Delta}{\g=\alpha_1...\alpha_k}}} & 
\un_{\LL_{l(\g)}^c}(x)e^{-\delta b(\g,x)}\\
& \leqslant\dfrac{e^{-\frac{R}{\zeta}}}{\underset{[-\Delta,\Delta]}{\min} f}
\sum\limits_{\underset{\forall i,\ \dhy(\oo,\alpha_i.\oo)\leqslant\zeta}{\underset{R-\Delta\leqslant b(\g,x)\leqslant R+\Delta}{\g=\alpha_1...\alpha_k}}}
\un_{\LL_{l(\g)}^c}(x)e^{\frac{R}{\zeta}}e^{-\delta b(\g,x)}f(R-b(\g,x)).
\end{align*}
\noindent
The inequality $R-\Delta\leqslant b(\g,x)\leqslant R+\Delta$ implies that the sum in the right member may be bounded from above by
$$e^{\frac{\Delta}{\zeta}}\sum\limits_{\underset{\forall i,\ \dhy(\oo,\alpha_i.\oo)\leqslant\zeta}{\underset{R-\Delta\leqslant b(\g,x)\leqslant R+\Delta}{\g=\alpha_1...\alpha_k}}}
\un_{\LL_{l(\g)}^c}(x)e^{-\left(\delta-\frac{1}{\zeta}\right)b(\g,x)}f(R-b(\g,x)).$$
\noindent
Since $\zeta$ is large, the quantity $e^{\frac{\Delta}{\zeta}}$ is close to $1$. From the Fourier inverse formula, it follows
\begin{align*}
\sum\limits_{\underset{\forall i,\ \dhy(\oo,\alpha_i.\oo)\leqslant\zeta}{\underset{R-\Delta\leqslant b(\g,x)\leqslant R+\Delta}{\g=\alpha_1...\alpha_k}}} 
\un_{\LL_{l(\g)}^c}(x)e^{-\delta b(\g,x)}\preceq\dfrac{e^{-\frac{R}{\zeta}}}{2\pi}
\displaystyle{\int\limits_{\R}}e^{itR}\left(\ot_{\delta-\frac{1}{\zeta}+it,\zeta}^k\un_{\LL}\right)(x)\widehat{f}(t)\dd t
\end{align*}
\noindent
where $\ot_{\delta-\frac{1}{\zeta}+it,\zeta}$ is defined, for any $\varphi\in\mathrm{Lip}\left(\LL\right)$ and $x\in\LL$, by
$$\ot_{\delta-\frac{1}{\zeta}+it,\zeta}\left(\varphi\right)(x)=\sum\limits_{\underset{\dhy(\oo,\alpha.\oo)\leqslant\zeta}{\alpha\in\mathcal{A}}}\un_{\LL_{l(\alpha)}^c}e^{-\left(\delta-\frac{1}{\zeta}+it\right)b(\alpha,x)}
\varphi(\alpha.x).$$
\noindent
It remains to show that the integral 
$\displaystyle{\int\limits_{\R}}e^{itR}\left(\ot_{\delta-\frac{1}{\zeta}+it,\zeta}^k\un_{\LL}\right)(x)\widehat{f}(t)\dd t$ is 
$\leqslant\frac{C}{a_k}$. Let us split it into $I_1+I_2$ where
$$
I_1:=\displaystyle{\int\limits_{[-\varepsilon,\varepsilon]^c}}
e^{itR}\left(\ot_{\delta-\frac{1}{\zeta}+it,\zeta}^k\un_{\LL}\right)(x)\widehat{f}\left(t\right)\dd t$$
\noindent 
and
$$I_2:=\displaystyle{\int\limits_{-\varepsilon}^{\varepsilon}}
e^{itR}\left(\ot_{\delta-\frac{1}{\zeta}+it,\zeta}^k\un_{\LL}\right)(x)\widehat{f}\left(t\right)\dd t$$
\noindent
for the $\varepsilon$ given in Proposition \ref{spectreperturbation}. We may first notice that 
$\ot_{\delta-\frac{1}{\zeta}+it,\zeta}$ is a continuous perturbation of $\ot_{\delta+it}$ for any $t\in\R$. Indeed
\begin{align*}
\left|\left|\ot_{\delta-\frac{1}{\zeta}+it,\zeta}-\ot_{\delta+it}\right|\right| & \leqslant
\sum\limits_{\underset{\dhy(\oo,\alpha.\oo)\leqslant\zeta}{\alpha\in\mathcal{A}}}
\left|\left|w_{\delta-\frac{1}{\zeta}+it}(\alpha,\cdot)-w_{\delta+it}(\alpha,\cdot)\right|\right|
+\sum\limits_{\underset{\dhy(\oo,\alpha.\oo)>\zeta}{\alpha\in\mathcal{A}}}
\left|\left|w_{\delta+it}(\alpha,\cdot)\right|\right|\\
& \leqslant 
\sum\limits_{\underset{\dhy(\oo,\alpha.\oo)\leqslant\zeta}{\alpha\in\mathcal{A}}}
\left|\left|\un_{\LL_{l(\alpha)}^c}e^{-(\delta+it)b(\alpha,\cdot)}\left(e^{\frac{1}{\zeta}b(\alpha,\cdot)}-1\right)\right|\right|+
C\dfrac{L(\zeta)}{\zeta^{\beta}}.
\end{align*}
\noindent
The two inequalities
$$\left|\left|\un_{\LL_{l(\alpha)}^c}e^{-(\delta+it)b(\alpha,\cdot)}\right|\right|\preceq e^{-\delta\dhy(\oo,\alpha.\oo)}\ 
\text{and}\ \left|\left|e^{\frac{1}{\zeta}b(\alpha,\cdot)}-1\right|\right|\preceq\dfrac{1}{\zeta}\left(1+\dhy(\oo,\alpha.\oo)\right)e^{\frac{1}{\zeta}\dhy(\oo,\alpha.\oo)}$$
\noindent
yield
$$\left|\left|\ot_{\delta-\frac{1}{\zeta}+it,\zeta}-\ot_{\delta+it}\right|\right|\preceq\dfrac{L(\zeta)}{\zeta^{\beta}}.$$
\noindent
Potter's lemma thus implies that for $k$ large enough and $\zeta\geqslant\dfrac{a_k}{2}$
\begin{equation}\label{perturbationxi}
\left|\left|\ot_{\delta-\frac{1}{\zeta}+it,\zeta}-\ot_{\delta+it}\right|\right|\leqslant Ca_k^{-\beta}L(a_k)\leqslant\dfrac{C}{k}.
\end{equation}
\noindent
Combining \eqref{perturbationxi} and the fact that $\rho(\delta+it)<1$ for $|t|\in\mathrm{supp}\ \widehat{f}\setminus[-\varepsilon,\varepsilon]$, it follows that there exists $\rho\in]0,1[$ such that
\begin{align*}
\left|I_1\right|\leqslant C\displaystyle{\int\limits_{|t|\geqslant\varepsilon}}
\left|\left|\ot_{\delta-\frac{1}{\zeta}+it,\zeta}^k\right|\right|\widehat{f}(t)\dd t
\leqslant \dfrac{C}{a_k}a_k\rho^k\displaystyle{\int\limits_{|t|\geqslant\varepsilon}}
\widehat{f}(t)\dd t\preceq\dfrac{1}{a_k},
\end{align*}
\noindent
since the sequence $(a_k\rho^k)$ converges to $0$. 

Moreover, from Proposition \ref{spectreperturbation} and \eqref{perturbationxi}, we deduce that for $k$ large enough and $t$ close to $0$, 
the operator $\ot_{\delta-\frac{1}{\zeta}+it,\zeta}$ admits a unique dominant eigenvalue
$\lambda_{t,\mathrm{\frac{1}{\zeta}}}$ close to $1$, isolated in the spectrum of $\ot_{\delta-\frac{1}{\zeta}+it,\zeta}$ and 
satisfying $|\lambda_{t,\mathrm{\frac{1}{\zeta}}}-\lambda_{\delta+it}|\leqslant\dfrac{C}{k}$. To estimate $I_2$, it is thus sufficient to check that
\begin{equation}
J:=\displaystyle{\int\limits_{-\varepsilon}^{\varepsilon}}\left|\left|\ot_{\delta+\frac{1}{\zeta}+it,s}^k\right|\right|\dd t\leqslant\dfrac{C}{a_k}.
\end{equation}
\noindent
The integral $J$ may be splitted into $J_1+J_2$ where
$$J_1=\displaystyle{\int\limits_{-\frac{C_1}{a_k}}^{\frac{C_1}{a_k}}}\left|\left|\ot_{\delta+\frac{1}{\zeta}+it,\zeta}^k\right|\right|\dd t$$
\noindent
and
$$J_2=\displaystyle{\int\limits_{[-\varepsilon,\varepsilon]\setminus[-\frac{C_1}{a_k},\frac{C_1}{a_k}]}}
\left|\left|\ot_{\delta+\frac{1}{\zeta}+it,\zeta}^k\right|\right|\dd t,$$
\noindent
for a constant $C_1>0$, which will be chosen later. For $J_1$, the inequality
$|\lambda_{\delta+it}|\leqslant 1$ yields $|\lambda_{t,\frac{1}{\zeta}}|\leqslant 1+\dfrac{C}{k}$; hence
$\left|\left|\ot_{\delta+\frac{1}{\zeta}+it,\zeta}^k\right|\right|\leqslant C$ 
and $J_1$ is thus bounded from above by $\dfrac{CC_1}{a_k}$ for some $C_1>0$.

From Proposition \ref{localexp}, we deduce the existence of a constant $c>0$ such that for any 

\noindent
$|t|\in[\dfrac{C_1}{a_k},\varepsilon]$

$$|\lambda_{t,\frac{1}{\zeta}}|\leqslant |\lambda_{t}|+\dfrac{C}{k}\leqslant 1-c|t|^{\beta}L\left(\dfrac{1}{|t|}\right)+\dfrac{C}{k}.$$
\noindent
Moreover for $k$ large enough
$$\dfrac{1}{k}\sim a_k^{-\beta}L\left(a_k\right)\leqslant C_1^{\frac{\beta}{2}}a_k^{-\beta}L\left(\dfrac{a_k}{C_1}\right)\leqslant
\dfrac{|t|^{\beta}}{C_1^{\frac{\beta}{2}}}L\left(\dfrac{1}{|t|}\right).$$
\noindent
Therefore, for $C_1$ large enough
$$|\lambda_{t,\frac{1}{\zeta}}|\leqslant 1-c'|t|^{\beta}L\left(\dfrac{1}{|t|}\right)$$
\noindent
where $c'>0$; hence
$$
J_2\leqslant C\displaystyle{\int\limits_{\frac{C_1}{a_k}}^{\varepsilon}}\left(1-c't^{\beta}L\left(\dfrac{1}{t}\right)\right)^k\dd t
\leqslant C\displaystyle{\int\limits_{\frac{C_1}{a_k}}^{\varepsilon}}e^{-kc't^{\beta}L(\frac{1}{t})}\dd t.
$$
\noindent
Setting $u=ta_k$, it follows from Remark \ref{sertA1etB1} that if $\varepsilon>0$ is small enough, then 
$\frac{L\left(\frac{a_k}{u}\right)}{L(a_k)}\leqslant\max\left(u^{+\frac{\beta}{2}},u^{-\frac{\beta}{2}}\right)$, which yields
\begin{align*}
\displaystyle{\int\limits_{\frac{C_1}{a_k}}^{\varepsilon}}e^{-kc't^{\beta}L(\frac{1}{t})}\dd t= &
\dfrac{1}{a_k}\displaystyle{\int\limits_{C_1}^{\varepsilon a_k}}e^{-kc'|u|^{\beta}|a_k|^{-\beta}L(\frac{a_k}{u})}\dd u\\
\leqslant &\dfrac{1}{a_k}\displaystyle{\int\limits_{C_1}^{\varepsilon a_k}}e^{-kc'|u|^{\beta\pm\frac{\beta}{2}}|a_k|^{-\beta}L(a_k)}\dd u.
\end{align*}
\noindent
One achieves the proof of Lemma \ref{A22} noticing that $a_k^\beta=kL(a_k)$ so that
$$J_2\leqslant \dfrac{C}{a_k}\displaystyle{\int\limits_{C_1}^{\varepsilon a_k}}e^{-c'u^{\beta\pm\frac{\beta}{2}}}\dd u\leqslant
\dfrac{C}{a_k}\displaystyle{\int\limits_{C_1}^{\infty}}e^{-c'u^{\beta\pm\frac{\beta}{2}}}\dd u.$$
\end{proof}
Let us now deal with the proof of Proposition A.2. {\bf For any $\boldsymbol{j\in[\![1,p+q]\!]}$, we fix $\boldsymbol{y_j\in\LL\setminus\LL_j}$ and 
we denote by $\boldsymbol{\G(k,j)}$ the set of isometries $\boldsymbol{\g\in\G}$ with symbolic length $\boldsymbol{|\g|=k}$ 
such that $\boldsymbol{l(\g)}$ equals $\boldsymbol{j}$}. Let us introduce the following notations.
\begin{enumerate}
 \item For any $t>0$ and $\Delta>0$, let
 $$\mathcal{A}(t,\Delta):=\{\g\in\G\ |\ t-\Delta\leqslant\dhy(\oo,\g.\oo)< t+\Delta\};$$
 \item for any $l\in\N^*$, $c,e\in\R^{+\ast}$, let
 $$Q(l,c,e)=
 \sum\limits_{j=1}^{p+q}
 \sum\limits_{\underset{\g\in\mathcal{A}(c,e)}{\g\in\G(l,j)}}e^{-\delta b(\g,y_j)};$$
 \item let $\Omega_r(j_1,j_2,j_3)\subset\G(k)$ be the set defined for any $r\in[\![1,k]\!]$ and any $j_1,j_2,j_3\in[\![1,p+q]\!]$, $j_1\neq j_2$ and $j_2\neq j_3$ by
 \begin{align*}
   &-\Omega_1(j_1,j_2,j_3):=\{\alpha_1...\alpha_k\ |\ \alpha_{1}\in\G_{j_2}^*,\alpha_{k}\in\G_{j_3}^*\};\\
   &-\Omega_r(j_1,j_2,j_3):=\{\alpha_1...\alpha_k\ |\ \alpha_{r-1}\in\G_{j_1}^*,\alpha_{r}\in\G_{j_2}^*,\alpha_{r+1}\in\G_{j_3}^*\},\ \text{if}\ 2\leqslant r\leqslant k-1;\\ 
   &-\Omega_k(j_1,j_2,j_3):=\{\alpha_1...\alpha_k\ |\ \alpha_{k-1}\in\G_{j_1}^*,\alpha_{k}\in\G_{j_2}^*\}.
 \end{align*}
 \item if $\g=\alpha_1...\alpha_k\in\G(k)$, $k\geqslant1$, we denote by $\g_{(0)}=\g^{(k+1)}=\mathrm{Id}$ and for any $j\in[\![1,k]\!]$, we denote 
 by $\g_{(j)}=\alpha_1...\alpha_j$ and by $\g^{(j)}=\alpha_j...\alpha_k$.
\end{enumerate}
\noindent
Since $Ka_k\leqslant R$, we may write $R=wa_k$ for some $w\geqslant1$; we introduce the following truncation level: 
$\zeta=w^{\theta}\dfrac{a_k}{2}\in\left[\frac{a_k}{2},\frac{R}{2}\right]$, for $\theta\in(0,1)$ close to $1$, which will be precised in the proof. We use it to split the set
\begin{equation}\label{defiensembleJ}
\mathfrak{J}=\{(\alpha_1,...,\alpha_k)\ |\ \alpha_1...\alpha_k\ \mathrm{admissible},\ b(\g,x)\overset{M}{\sim} R-s\}
\end{equation}
\noindent
into $\mathfrak{J}=\mathfrak{J}_1\cup \mathfrak{J}_2\cup \mathfrak{J}_3\cup \mathfrak{J}_4$ where
\begin{align*}
\mathfrak{J}_1 & =\{(\alpha_1,...,\alpha_k)\in J |\ \exists r,\ \dhy(\oo,\alpha_r.\oo)\geqslant\dfrac{R}{2}\};\\
\mathfrak{J}_2 & =\{(\alpha_1,...,\alpha_k)\in J\ |\ \forall j,\ \dhy(\oo,\alpha_j.\oo)<\dfrac{R}{2};\ \exists r<t,\ \dhy(\oo,\alpha_r.\oo),\dhy(\oo,\alpha_t.\oo)\geqslant\zeta\};\\
\mathfrak{J}_3 & =\{(\alpha_1,...,\alpha_k)\in J\ |\ \forall j,\ \dhy(\oo,\alpha_j.\oo)<\dfrac{R}{2};\ \exists! r,\ \dhy(\oo,\alpha_r.\oo)\geqslant\zeta\};\\
\mathfrak{J}_4 & =\{(\alpha_1,...,\alpha_k)\in J\ |\ \forall j,\ \dhy(\oo,\alpha_j.\oo)<\zeta\}.
\end{align*}
\noindent
We are going to prove that there exists a constant $C>0$ such that for any $i\in\{1,2,3,4\}$ 
$$\Sigma_i:=\sum\limits_{\underset{(\alpha_1,...,\alpha_k)\in \mathfrak{J}_i}{\g=\alpha_1...\alpha_k}}\un_{\LL_{l(\g)}^c}(x)e^{-\delta b(\g,x)}
\leqslant Ck\dfrac{L(R)}{R^{1+\beta}}.$$
\medskip
{\it Contribution of $\Sigma_1$.} By definition of $\mathfrak{J}_1$, if $\g=\alpha_1...\alpha_k$ with $(\alpha_1,...,\alpha_k)\in\mathfrak{J}_1$, 
there exists $r\in[\![1,k]\!]$ such that $\dhy(\oo,\alpha_r.\oo)\geqslant\frac{R}{2}$. The cocycle property of $b(\g,x)$ furnishes
\begin{align*}
\Sigma_1= & \sum\limits_{r=1}^{k}
\sum\limits_{\underset{\underset{\dhy(\oo,\alpha_r.\oo)\geqslant\frac{R}{2}}{b(\g,x)\overset{M}{\sim} R-s}}{\g=\alpha_1...\alpha_k}}\un_{\LL_{l(\g)}^c}(x)
e^{-\delta b(\g_{(r-1)},\g^{(r)}.x)}e^{-\delta b(\alpha_{r},\g^{(r+1)}.x)}
e^{-\delta b(\g^{(r+1)},x)}\\
= & \sum\limits_{r=1}^{k}\ \sum\limits_{\underset{j_1\neq j_2,j_2\neq j_3}{j_1,j_2,j_3}}\ \sum\limits_{\underset{b(\g,x)\overset{M}{\sim}R-s}{\g\in\Omega_r(j_1,j_2,j_3)}}
\un_{\LL_{l(\g)}^c}(x)
e^{-\delta b(\g_{(r-1)},\g^{(r)}.x)}e^{-\delta b(\alpha_{r},\g^{(r+1)}.x)}
e^{-\delta b(\g^{(r+1)},x)}.
\end{align*}
\noindent
Proposition \ref{equilipcocycle} implies the existence of $D>0$ such that
$$\left\{\begin{array}{lll}
          \left|b(\g_{(r-1)},\g^{(r)}.x)-b(\g_{(r-1)},y_{j_1})\right|\leqslant D\\
          \left|b(\alpha_r,\g^{(r+1)}.x)-b(\alpha_{r},y_{j_2})\right|\leqslant D\\
          \left|b(\g^{(r+1)},x)-b(\g^{(r+1)},y_{j_3})\right|\leqslant D\\
         \end{array}\right.
         $$
\noindent
for any $r\in[\![1,k]\!]$, any $j_1,j_2,j_3\in[\![1,p+q]\!]$, $j_1\neq j_2$, $j_2\neq j_3$ and $\g\in\Omega_r(j_1,j_2,j_3)$. Hence
$$\Sigma_1\preceq\sum\limits_{r=1}^{k}\ \sum\limits_{\underset{j_1\neq j_2,j_2\neq j_3}{j_1,j_2,j_2}}\ 
\sum\limits_{\underset{b(\g,x)\overset{M}{\sim}R-s}{\g\in\Omega_r(j_1,j_2,j_3)}}
\un_{\LL_{l(\g)}^c}(x)
e^{-\delta b(\g_{(r-1)},y_{j_1})}e^{-\delta b(\alpha_{r},y_{j_2})}
e^{-\delta b(\g^{(r+1)},y_{j_3})}.$$
\noindent
Combining $b(\g,x)\overset{M}{\sim} R-s$ with the cocycle property of $b(\g,x)$, we obtain
$$b(\g_{(r-1)},\g^{(r)}.x)+b(\alpha_{r},\g^{(r+1)}.x)+b(\g^{(r+1)},x)\overset{M}{\sim} R-s,$$
\noindent
then
$$b(\g_{(r-1)},y_{j_1})+b(\alpha_{r},y_{j_2})+b(\g^{(r+1)},y_{j_3})\overset{M+3D}{\sim} R-s.$$
\noindent
By Corollary \ref{busedist}, the condition $\dhy(\oo,\alpha_r.\oo)\geqslant\frac{R}{2}$ yields 
$$b(\alpha_{r},y_{j_2})\geqslant \frac{R}{2}-C.$$
\noindent
Therefore
$b(\g_{(r-1)},y_{j_1})+b(\g^{(r+1)},y_{j_3})\leqslant\frac{R}{2}-s+M+3D+C$. Denote by $\Delta=M+3D+C$. Let
$m,n\in\N^*$ such that
$b(\g_{(r-1)},y_{j_1})\in[(m-1)\Delta,(m+1)\Delta]$ and
$b(\g^{(r+1)},y_{j_3})\in[(n-1)\Delta,(n+1)\Delta]$ for $m\leqslant N$ and $n\leqslant N-m$, where 
$$N:=\left[\dfrac{\frac{R}{2}-s+\Delta}{2\Delta}\right]+1.$$
\noindent
Using the previous notations, we may bound $\Sigma_1$ from above by
\begin{align*}
\sum\limits_{r=1}^k\sum\limits_{m+n\leqslant N}\Biggl(Q(r-1, & m\Delta,2\Delta+2C)\\
& Q(1,R-s-(m+n)\Delta,3\Delta+2C)Q(k-r,n\Delta,2\Delta+2C)\Biggr).
\end{align*}
\noindent
For $0\leqslant n,m\leqslant N$ such that $n+m\leqslant N$, Corollary \ref{busedist} implies
\begin{align*}
Q(1,R-s-(m+n)\Delta,3\Delta+2C)= & \sum\limits_{j_2=1}^{p+q}
\sum\limits_{\underset{\alpha\in\mathcal{A}(R-s-(m+n)\Delta,3\Delta+2C)}{\alpha\in\G_{j_2}^*}}
e^{-\delta b(\alpha,y_{j_2})}\\
\preceq &
\sum\limits_{\underset{\alpha\in\mathcal{A}(R-s-(m+n)\Delta,3\Delta+2C)}{|\alpha|=1}}e^{-\delta\dhy(\oo,\alpha.\oo)}.
\end{align*}
\noindent
Combining Assumption $(S)$ of the family $(H_\beta)$ and Potter's lemma, we obtain
\begin{align*}\label{dominantSigma3}
Q(1,R-s-(m+n)\Delta,3\Delta+2C)\leqslant
\underset{t\geqslant\frac{R}{2}-s}{\sup}\sum\limits_{\underset{\alpha\in\mathcal{A}(t,3\Delta+2C)}{|\alpha|=1}}e^{-\delta\dhy(\oo,\alpha.\oo)}
\preceq \underset{t\geqslant\frac{R}{2}-s}{\sup}\dfrac{L(t)}{t^{1+\beta}}
\preceq \dfrac{L(R)}{R^{1+\beta}}.
\end{align*}
\noindent
Thus $\frac{R^{1+\beta}}{L(R)}\Sigma_1$ may be bounded up to a multiplicative constant by
$$\sum\limits_{r=1}^k
\left(\sum\limits_{j_1=1}^{p+q}\sum\limits_{\g_1\in\G(r-1)}\un_{\LL_{l(\g_1)}^c}(y_{j_1})e^{-\delta b(\g_1,y_{j_1})}\right)
\left(\sum\limits_{j_3=1}^{p+q}\sum\limits_{\g_2\in\G(k-r)}\un_{\LL_{l(\g_2)}^c}(y_{j_3})e^{-\delta b(\g_2,y_{j_3})}\right)$$
\noindent
and Lemma \ref{A21} finally gives $\Sigma_1\preceq kR^{-1-\beta}L(R)$.

\medskip
{\it Contribution of $\Sigma_2$}. If $\g=\alpha_1...\alpha_k$ with $(\alpha_1,...,\alpha_k)\in\mathfrak{J}_2$, there exist
$r<t$ in $[\![1,k]\!]$ such that $\dhy(\oo,\alpha_r.\oo),\dhy(\oo,\alpha_t.\oo)>\zeta$. We decompose $\Sigma_2$ as $\Sigma_1$ according to the values of 
$r$ and $t$, which leads us to the following upper bound for $\Sigma_2$
\begin{align*}
\sum\limits_{r<t} & \sum\limits_{m+n+l\leqslant N}\Biggl(Q(r-1,m\Delta,2\Delta+2C)Q(1,R-\zeta-s-(m+n+l)\Delta,3\Delta+2C)\\
& Q(r-t-1,n\Delta,2\Delta+2C)Q(k-t,l\Delta,2\Delta+2C)
\sum\limits_{j=1}^{p+q}\sum\limits_{\underset{\dhy(\oo,\alpha.\oo)\geqslant\zeta}{\alpha\in\G_j^*}}e^{-\delta b(\alpha,y_j)}\Biggl),
\end{align*}
%
\noindent
where $\Delta=M+5D+2C$ and $N=\left[\dfrac{R-2\zeta-s+\Delta}{2\Delta}\right]+1$. We write
\begin{align*}
Q(1,R-\zeta-s-(m+n+l)\Delta,3\Delta+2C)= & \sum\limits_{j_2=1}^{p+q}
\sum\limits_{\underset{\alpha\in\mathcal{A}(R-\zeta-s-(m+n+l)\Delta,3\Delta+2C)}{\alpha\in\G_{j_2}^*}}e^{-\delta b(\alpha,y_{j_2})}\\
\preceq & \dfrac{L(\zeta)}{\zeta^{1+\beta}}.
\end{align*}
\noindent
Assumptions $(P_2)$ and $(N)$ combined with Corollary \ref{busedist} imply
$$\sum\limits_{j=1}^{p+q}\sum\limits_{\underset{\dhy(\oo,\alpha.\oo)\geqslant\zeta}{\alpha\in\G_{j}^*}}e^{-\delta b(\alpha,y_{j})}\preceq 
\dfrac{L(\zeta)}{\zeta^{\beta}}.$$
\noindent
We bound
\begin{align*}
\sum\limits_{m+n+l\leqslant N}Q(r-1,m\Delta,2\Delta+2C)Q(r-t-1,n\Delta,2\Delta+2C)Q(k-t,l\Delta,2\Delta+2C)
\end{align*}
\noindent
from above by $\leqslant C^3$ using Lemma \ref{A21}. Summing over $r<t$, we obtain
$$\Sigma_2\preceq k^2\zeta^{-2\beta-1}L(\zeta)^2.$$
\noindent
Since $k\sim \dfrac{a_k^{\beta}}{L(a_k)}$, $\dfrac{2\zeta}{a_k}=w^{\theta}$ and $\dfrac{R}{\zeta}=2w^{1-\theta}$, the last inequality may be reformulated as follows
$$\Sigma_2\preceq k\dfrac{a_k^{\beta}}{L(a_k)}\dfrac{L(\zeta)}{\zeta^{\beta}}\dfrac{L(\zeta)}{\zeta^{\beta+1}}\dfrac{R^{\beta+1}}{L(R)}R^{-\beta-1}L(R).$$
\noindent
By Potter's lemma, for $\varepsilon>0$, one gets
$$\dfrac{L(\zeta)}{L(a_k)}\leqslant w^{\varepsilon}\ \mathrm{and}\ \dfrac{L(\zeta)}{L(R)}\leqslant w^{\varepsilon};\ \text{therefore}\ \Sigma_2\preceq k\cdot w^{-\beta\theta+\varepsilon}\cdot w^{(1-\theta)(\beta+1)+\varepsilon}R^{-\beta-1}L(R).$$
\noindent
If $\beta\theta>(1-\theta)(\beta+1)$ $\left(i.e.\ \theta>\dfrac{1+\beta}{1+2\beta}\right)$, the power of $w$ may be chosen negative for $\varepsilon$ small enough. 
Finally $\Sigma_2\preceq kR^{-1-\beta}L(R)$.

\medskip
{\it Contribution of $\Sigma_3$.} If $\g=\alpha_1...\alpha_k$ with $(\alpha_1,...,\alpha_k)\in\mathfrak{J}_3$, there exists a unique integer
$r\in[\![1,k]\!]$ such that $\dhy(\oo,\alpha_r.\oo)>\zeta$. We deal separately with the cases $w\leqslant k$ and $w>k$. 

When $w\leqslant k$, either $r\leqslant\frac{k}{w}$ or $r\geqslant k+1-\frac{k}{w}$, or $r\in\left[\frac{k}{w},k+1-\frac{k}{w}\right]$. 
 \begin{itemize} 
  \item[a)]If $r\leqslant\frac{k}{w}$ or $r\geqslant k+1-\frac{k}{w}$, we bound 
\begin{equation}\label{quantiteanonyme}
\sum\limits_{\underset{\underset{\dhy(\oo,\alpha_r.\oo)\geqslant\frac{R}{2}}{b(\g,x)\overset{M}{\sim} R-s}}{\g=\alpha_1...\alpha_k}}\un_{\LL_{l(\g)}^c}(x)
e^{-\delta b(\g_{(r-1)},\g^{(r)}.x)}e^{-\delta b(\alpha_{r},\g^{(r+1)}.x)}
e^{-\delta b(\g^{(r+1)},x)}
\end{equation}
\noindent
from above by
\begin{align*}
\sum\limits_{m+n\leqslant N}\Biggl(Q(r-1,m\Delta,2\Delta+2C)Q(1, & R-s-(m+n)\Delta,3\Delta+2C)\\
& Q(k-r,n\Delta,2\Delta+2C)\Biggr),
\end{align*}
\noindent
where $\Delta=M+3D+C$ and $N=\left[\frac{R-\zeta-s+\Delta}{2\Delta}\right]+1$.
As for $\Sigma_1$, for any $m,n$ such that $m+n\leqslant N$, we bound $Q(1,R-s-(m+n)\Delta,3\Delta+2C)$ from above by
$C\zeta^{-1-\beta}L(\zeta)$. Moreover Lemma \ref{A21} allows us to bound from above the quantity
$$\sum\limits_{m+n\leqslant N}\left(Q(r-1,m\Delta,2\Delta+2C)Q(k-r,n\Delta,2\Delta+2C)\right).$$ 
\noindent 
There are at most $\dfrac{2k}{w}$ such terms; their contribution is thus less than
\begin{align*}
C\dfrac{k}{w}\zeta^{-\beta-1}L(\zeta)|& = Ckw^{-1}\dfrac{L(\zeta)}{\zeta^{\beta+1}}\dfrac{R^{\beta+1}}{L(R)}R^{-\beta-1}L(R)\\
                                 & \preceq kw^{-1}w^{(1-\theta)(\beta+1)+\varepsilon}R^{-\beta-1}L(R),
\end{align*}
\noindent
since
$$\dfrac{R^{\beta+1}}{\zeta^{\beta+1}}=2^{\beta+1}w^{(1-\theta)(\beta+1)}\ \mathrm{and}\ \dfrac{L(\zeta)}{L(R)}\leqslant
\max\left(2w^{1-\theta},\dfrac{1}{2}w^{\theta-1}\right)^{\frac{\varepsilon}{1-\theta}}=2^{\frac{\varepsilon}{1-\theta}}w^{\varepsilon}.$$
\noindent
For $\theta$ close enough to $1$ $\left(i.e.\ \theta>\dfrac{\beta}{1+\beta}\ \text{here}\right)$, the power of $w$ is negative and the 
contribution is finally $\preceq kR^{-\beta-1}L(R)$.
\item[b)]Assume now that $r\in\left[\dfrac{k}{w},k+1-\dfrac{k}{w}\right]$. The condition 
$$R-s-M\leqslant b(\g,x)\leqslant R-s+M$$
\noindent
and the cocycle property of $b(\g,x)$ both imply
$$b(\g_{(r-1)},y_{j_1})+b(\alpha_{r},y_{j_2})+b(\g^{(r+1)},y_{j_3})\overset{M+3D}{\sim} R-s$$
\noindent
for any $r\in[\![1,k]\!]$, any $(j_1,j_2,j_3)\in[\![1,p+q]\!]$, $j_1\neq j_2$, $j_2\neq j_3$ and any $\g\in\Omega_r(j_1,j_2,j_3)$. Fix $r$, 
$j_1$, $j_2$, $j_3$ and $\g\in\Omega_r(j_1,j_2,j_3)$ as above. From $\dhy(\oo,\alpha_r.\oo)<\frac{R}{2}$, we deduce $b(\alpha_r,y_{j_2})<\frac{R}{2}+C$, hence
$$b(\g_{(r-1)},y_{j_1})+b(\g^{(r+1)},y_{j_3})\geqslant\dfrac{R}{2}-s-M-3D-C.$$
\noindent
This last upper bound yields
$$1)\ b(\g_{(r-1)},y_{j_1})\geqslant \dfrac{R}{4}-\dfrac{s}{2}-\dfrac{M+3D+C}{2}$$
\noindent
or 
$$2)\ b(\g^{(r+1)},y_{j_3})\geqslant \dfrac{R}{4}-\dfrac{s}{2}-\dfrac{M+3D+C}{2}.$$
\noindent
We only detail the arguments concerning the control of the sum in the case 1); the other one 
may be treated similarly. Set $\Delta=M+3D$ and let $m,n,l\in\N^*$ such that
$$\left\{\begin{array}{lll}
         b(\g_{(r-1)},y_{j_1})\in[(m-1)\Delta,(m+1)\Delta]\\
         b(\alpha_{r},y_{j_2})\in[(n-1)\Delta,(n+1)\Delta]\\
         b(\g^{(r+1)},y_{j_3})\in[(l-1)\Delta,(l+1)\Delta]\\
         \end{array}\right.$$
\noindent
for $m\leqslant N$, $n\leqslant N-m$ and $l=N-m-n$, where $N=\left[\frac{R-s}{\Delta}\right]+1$. The sum \eqref{quantiteanonyme} for 
$r\in\left[\dfrac{k}{w},k+1-\dfrac{k}{w}\right]$ may thus be bounded from above by
\begin{align*}
\sum\limits_{m+n+l=N}
\biggl(Q(r-1,m\Delta,2\Delta+2C) & \sum\limits_{j_2=1}^{p+q}
\sum\limits_{\underset{\dhy(\oo,\alpha.\oo)>\zeta}{\alpha\in\mathcal{A}(l\Delta,3\Delta+2C)\cap\G_{j_2}^*}}e^{-\delta b(\alpha,y_{j_2})} \\
& Q(k-r,n\Delta,2\Delta+2C)\biggr),
\end{align*}
\noindent
%
\noindent
which is smaller than
\begin{align*}
(\star)\quad\underset{m}{\sup}\left(Q(r-1,m\Delta,2\Delta+2C)\right) & \times
\biggl(\sum\limits_{j_2=1}^{p+q}\sum\limits_{{\underset{\dhy(\oo,\alpha.\oo)>\zeta}{\alpha\in\G_{j_2}^*}}}e^{-\delta b(\alpha,y_{j_2})}\biggr)\\
 \times &\left(\sum\limits_{j_3=1}^{p+q}\sum\limits_{|\g_2|=k-r}\un_{\LL_{l(\g_2)}^c}(y_{j_3})e^{-\delta b(\g_2,y_{j_3})}\right)
\end{align*}
\noindent
where the supremum is taken over $m\in\N^*$ such that $m\Delta\geqslant\frac{R}{4}-\frac{s}{2}-\frac{3}{2}(\Delta+C)$. We combined 
Lemma \ref{A22} and the fact that $s$ lies in a compact subset of $\R$ to control the first 
factor, Assumptions $(P_2)$ and $(N)$ for the second factor and 
Lemma \ref{A21} for the third one, which allows us to bound the quantity $(\star)$ from above by
$$C\dfrac{e^{-\frac{R}{4\zeta}}}{a_{r-1}}\zeta^{-\beta}L(\zeta),$$ 
\noindent 
where $C$ only depends on the support of $\varphi$. Using the regular variation of $(a_r)_r$, we obtain $\frac{a_r}{a_{r-1}}\leqslant C$ for some $C>0$ and the quantity $(\star)$ is bounded from above, up to a multiplicative constant, by
$$e^{-\frac{R}{4\zeta}}\zeta^{-\beta}\dfrac{L(\zeta)}{a_r}.$$
Since $\dfrac{k}{w}\leqslant r\leqslant k+1-\frac{k}{w}$ and $(a_i)_i$ is regularly varying with exponent $\dfrac{1}{\beta}$, Potter's lemma implies
$a_r\geqslant\dfrac{a_k}{w^{\frac{1}{\beta}+\varepsilon}}$; combining with the equalities  $a_k=\dfrac{R}{w}$ and $\dfrac{R}{\zeta}=2w^{1-\theta}$, this yields to an upper bound of the form
$$Ce^{-\frac{w^{1-\theta}}{2}}w^{C'}R^{-\beta-1}L(R)\ \left(\text{with}\ C'=\dfrac{1}{\beta}+(1-\theta)(1+\beta)+\varepsilon\right),$$
\noindent
which is smaller than $CR^{-\beta-1}L(R)$ since $\theta<1$.
The integer $r$ taking at most $k$ values, the result follows for $\Sigma_3$.
\end{itemize}
When $k\leqslant w$, then $r\in\left[\dfrac{k}{w},k+1-\dfrac{k}{w}\right]$ and the proof is the same as for the case b).

\medskip
{\it Contribution of $\Sigma_4$.} By Lemma \ref{A22}  
$$\Sigma_4\leqslant C\dfrac{e^{-\frac{R}{\zeta}}}{a_k},$$
\noindent
with $C\frac{e^{-\frac{R}{\zeta}}}{a_k}=C\frac{e^{-2w^{1-\theta}}}{a_k}$ and
$kR^{-\beta-1}L(R)\sim \dfrac{a_k^{\beta}}{L(a_k)}(wa_k)^{-\beta-1}L(wa_k)=\dfrac{w^{-\beta-1\pm\varepsilon}}{a_k}$.
This achieves the proof for $\Sigma_4$ and the proof of Proposition A.2.

\section{Theorem A: mixing for $\beta=1$}
This section is devoted to the proof of Theorem A when $\beta=1$. Let $\G$ be a Schottky group satisfying the family of hypotheses $(H_\beta)$ for $\beta=1$. 
The arguments here are slightly different from the case $\beta\in]0,1[$. Indeed, recall that 
Karamata's lemma \ref{regvar} asserts that $\tilde{L}$ is a slowly varying function, which additionally satisfies
$$\lim\limits_{x\longrightarrow+\infty}\dfrac{L(x)}{\tilde{L}(x)}=0,$$
\noindent
and the fact that the Bowen-Margulis measure $m_\G$ is infinite implies 
$\lim\limits_{x\longrightarrow+\infty}\tilde{L}(x)=+\infty$. The proofs of this section are inspired from
the one of Theorem 2.1 in the case $\beta=1$ of \cite{MT}. 

The argument of Subsection 5.1 applies verbatim in our setting. Therefore, still writing $M(R;A,B)=m_\G(A.B\circ g_R)$, we have to prove that, as $R\longrightarrow\pm\infty$,
$$M(R;\varphi\otimes u,\psi\otimes v)\sim\dfrac{1}{\Cg}\dfrac{m_\G(\varphi\otimes u)m_\G(\psi\otimes v)}{\tilde{L}(|R|)}$$
\noindent
where $\varphi,\psi\ :\ \mathcal{D}^0\longrightarrow\R$ are Lipschitz functions on $\LL$ and $u,v\ :\ \R\longrightarrow\R$ are continuous functions on $\R$ with compact support. 
The arguments exposed in Subsection 5.1 allow us to treat only the case $R\longrightarrow+\infty$; Lemma \ref{reductiondeMa} implies in particular, for $R$ large enough
\begin{align*}
M(R;\varphi\otimes u,\psi\otimes v)= & M^+(R;\varphi\otimes u,\psi\otimes v)\\
= & \sum\limits_{k\geqslant0}\int_{\LL\times\R}\tilde{P}^k\left(\varphi\otimes u\right)(x,s-R)\psi\otimes v\left(x,s\right)\nu(\dd x)\dd s.
\end{align*}
\noindent
In the sequel, we would need to consider the integral near $0$ of the function $t\longmapsto Q_{\delta+it}$, where 
$Q_z=\left(\mathrm{Id}-\ot_z\right)^{-1}$ for any $z\in\C$ with $\re{z}\geqslant\delta$. Nevertheless, Proposition \ref{controlQ} ensures that $t\longmapsto Q_{\delta+it}$ is not integrable in $0$. To overcome this matter, we proceed as in the proof of Theorem 6.1 in \cite{DPPS} and we introduce the symmetrized quantity $\left(\tilde{P}^{\mathrm{sym}}\right)^k\left(\varphi\otimes u\right)(x,s-R)$ of 
$\tilde{P}^k\left(\varphi\otimes u\right)(x,s-R)$, defined as follows: for any $k\geqslant1$
\begin{align*}
\left(\tilde{P}^{\mathrm{sym}}\right)^k & \left(\varphi\otimes u\right)(x,s-R)\\
& =\sum\limits_{\g\in\G(k)}\un_{\LL_{l(\g)}^c}(x)e^{-\delta b(\g,x)}\dfrac{h\varphi(\g.x)}{h(x)}
\left[u(s-R+b(\g,x))+u(s-R-b(\g,x))\right].
\end{align*}
\noindent
Thus we study the term $M^{\mathrm{sym}}(R;\varphi\otimes u,\psi\otimes v)$ defined by
$$M^{\mathrm{sym}}(R;\varphi\otimes u,\psi\otimes v)=\sum\limits_{k\geqslant0}\int_{\LL\times\R}\left(\tilde{P}^{\mathrm{sym}}\right)^k\left(\varphi\otimes u\right)(x,s-R)\psi\otimes v\left(x,s\right)\nu(\dd x)\dd s$$
\noindent
with the convention $\left(\tilde{P}^{\mathrm{sym}}\right)^0\left(\varphi\otimes u\right)=\varphi\otimes u$. Since $u$ has a compact support, for any $s\in\mathrm{supp}\ v$ and for any $k\in\N$, we have as $R\longrightarrow+\infty$
$$\left(\tilde{P}^{\mathrm{sym}}\right)^k\left(\varphi\otimes u\right)(x,s-R)=
\tilde{P}^k\left(\varphi\otimes u\right)(x,s-R).$$
\noindent
It is thus sufficient to prove that as $R\longrightarrow+\infty$
\begin{equation}\label{thmD}
M^{\mathrm{sym}}(R;\varphi\otimes u,\psi\otimes v)\sim\dfrac{1}{\Cg}\dfrac{m_\G(\varphi\otimes u)m_\G(\psi\otimes v)}{\tilde{L}(R)}.
\end{equation}
\begin{rem}
If we fix $\varphi,\psi$ (resp. the function $v$) in the space of Lipschitz functions (resp. in the space of continuous functions on $\R$ with compact support), statement \eqref{thmD} is equivalent to the weak convergence to $\frac{1}{\Cg}m_\G(\varphi\otimes\bullet)m_\G(\psi\otimes v)$ of the sequence of measures $\left(\tilde{L}(R)M^{\mathrm{sym}}(R;\varphi\otimes\bullet,\psi\otimes v)\right)_R$. By the argument of Stone already mentionned in the proof of 
Proposition A.1, it is thus sufficient to prove that $\tilde{L}(R)M^{\mathrm{sym}}(R;\varphi\otimes u,\psi\otimes v)$ is finite and converges to $\frac{1}{\Cg}m_\G(\varphi\otimes u)m_\G(\psi\otimes v)$ 
for any function $u\ :\ \R\longrightarrow\R$ in the set of test functions $\mathscr{U}$.
\end{rem}
\subsubsection{Proof of $\eqref{thmD}$}
Let $u\in\mathscr{U}$. We first introduce the following quantity: for $\xi>\delta$, 
 $k\geqslant1$, $x\in\LL$ and $s\in\R$
\begin{align*}
\left(\tilde{P}_{\xi}^{\mathrm{sym}}\right)^k & \left(\varphi\otimes u\right)(x,s-R)\\
& :=\sum\limits_{\g\in\G(k)}\un_{\LL_{l(\g)}^c}(x)e^{-\xi b(\g,x)}\dfrac{h\varphi(\g.x)}{h(x)}
\left(u(s-R+b(\g,x))+u(s-R-b(\g,x))\right).
\end{align*}
\noindent
Similarly, we set 
\begin{align}\label{defdeMtildexi}
M_{\xi}^{\mathrm{sym}}(R;\varphi\otimes u,\psi\otimes v)
:=\sum\limits_{k\geqslant0}\int_{\LL\times\R}\left(\tilde{P}_{\xi}^{\mathrm{sym}}\right)^k\left(\varphi\otimes u\right)(x,s-R)\psi\otimes v\left(x,s\right)\nu(\dd x)\dd s
\end{align}
\noindent
with the convention 
$\left(\tilde{P}_{\xi}^{\mathrm{sym}}\right)^0\left(\varphi\otimes u\right)=\varphi\otimes u$. The convergence of the Poincaré series of $\G$ at $\xi>\delta$ yields
$$M_{\xi}^{\mathrm{sym}}(R;\varphi\otimes u,\psi\otimes v)=\int_{\LL\times\R}\left(\sum\limits_{k\geqslant0}\left(\tilde{P}_{\xi}^{\mathrm{sym}}\right)^k\left(\varphi\otimes u\right)(x,s-R)\right)\psi(x)v(s)\nu(\dd x)\dd s;$$
\noindent
moreover
\begin{align*}
\sum\limits_{k\geqslant0} \left(\tilde{P}_{\xi}^{\mathrm{sym}}\right)^k & \left(\varphi\otimes u\right)(x,s-R)\\ 
= & \dfrac{1}{2\pi h(x)}\sum\limits_{k\geqslant0}\int_{\R}e^{it(R-s)}
\left(\ot_{\xi+it}^k+\ot_{\xi-it}^k\right)(h\varphi)(x)\widehat{u}(t)\dd t\\
= & \dfrac{1}{2\pi h(x)}\int_{\R}e^{it(R-s)}\left[\sum\limits_{k\geqslant0}\left(\ot_{\xi+it}^k+\ot_{\xi-it}^k\right)(h\varphi)(x)\right]\widehat{u}(t)\dd t
\end{align*}
\noindent
for any  $x\in\LL$ and $s\in\R$, so that the term $M_{\xi}^{\mathrm{sym}}(R;\varphi\otimes u,\psi\otimes v)$ equals
\begin{align*}
\int_{\LL\times\R}&\left(\dfrac{1}{\pi h(x)}\int_{\R}e^{it(R-s)}\re{Q_{\xi+it}}(h\varphi)(x)\widehat{u}(t)\dd t\right)\psi(x)v(s)\nu(\dd x)\dd s.
\end{align*}
\noindent
To show  \eqref{thmD}, we have to understand how to relate the quantities $M_{\xi}^{\mathrm{sym}}(R;\varphi\otimes u,\psi\otimes v)$ and
$M^{\mathrm{sym}}(R;\varphi\otimes u,\psi\otimes v)$. This is the purpose of the following proposition.
\begin{prop}\label{convergencexiversdelta}
\begin{align*}
\lim\limits_{\xi\searrow\delta}M_{\xi}^{\mathrm{sym}}(R;\varphi\otimes u, & \psi\otimes v)\\
= & \int_{\LL\times\R}\left(\dfrac{1}{\pi h(x)}\int_{\R}e^{it(R-s)}\re{Q_{\delta+it}}(h\varphi)(x)\widehat{u}(t)\dd t\right)\psi(x)v(s)\nu(\dd x)\dd s\\
= & M^{\mathrm{sym}}(R;\varphi\otimes u,\psi\otimes v)
\end{align*}
\end{prop}
\begin{proof}
This result relies on the two following remarks:
\begin{itemize}
 \item[1)]$\underset{\xi\searrow\delta}{\lim}M_{\xi}^{\mathrm{sym}}(R;\varphi\otimes u,\psi\otimes v)$ exists and is equal to 
\begin{align*}
&\int_{\LL\times\R}\left(\dfrac{1}{\pi h(x)}\int_{\R}e^{it(R-s)}\re{Q_{\delta+it}}(h\varphi)(x)\widehat{u}(t)\dd t\right)v(x)\psi(s)\nu(\dd x)\dd s
\end{align*}
\item[2)]this limit also equals $M^{\mathrm{sym}}(R;\varphi\otimes u,\psi\otimes v)$.
\end{itemize}
\noindent
To prove the first point, it is sufficient to check that
\begin{equation}\label{propmegarelou}
\int_{\R}e^{it(R-s)}\re{Q_{\xi+it}}(h\varphi)(x)\widehat{u}(t)\dd t\underset{\xi\searrow\delta}{\longrightarrow}
\int_{\R}e^{it(R-s)}\re{Q_{\delta+it}}(h\varphi)(x)\widehat{u}(t)\dd t
\end{equation}
\noindent
uniformly in $x\in\LL$ and $s\in\mathrm{supp}\ v$. We postpone the proof in the next paragraph 6.0.5. Let us show how the second assertion follows from the first one. First, assume that the functions $\varphi,\psi,u$ and $v$ are positive. From \eqref{defdeMtildexi}, the monotone convergence theorem implies that
$M_{\xi}^{\mathrm{sym}}(R;\varphi\otimes u,\psi\otimes v)$ converges to $M^{\mathrm{sym}}(R;\varphi\otimes u,\psi\otimes v)$. Combining the uniqueness of the limit for 
$M_{\xi}^{\mathrm{sym}}(R;\varphi\otimes u,\psi\otimes v)$ and Corollary \ref{integrabilityreQ}, we deduce that 
$M^{\mathrm{sym}}(R;\varphi\otimes u,\psi\otimes v)$ is finite for any positive $\varphi,\psi,u$ and $v$. We use the Lebesgue dominated convergence theorem to prove that it is also the case for arbitrary functions $\varphi,\psi,u$ and $v$: indeed, we bound
$$\int_{\LL\times\R}\left(\tilde{P}_{\xi}^{\mathrm{sym}}\right)^k(\varphi\otimes u)(x,s-R)\psi\otimes v(x,s)\nu(\dd x)\dd s$$
\noindent
from above by
\begin{equation}
\label{dominerelou}\int_{\LL\times\R}\left(\tilde{P}^{\mathrm{sym}}\right)^k(|\varphi|\otimes |u|)(x,s-R)|\psi|\otimes |v|(x,s)\nu(\dd x)\dd s,
\end{equation}
\noindent
and notice that $|u|\in\mathscr{U}$, it follows from the finiteness of $M^{\mathrm{sym}}(R;|\varphi|\otimes|u|,|\psi|\otimes|v|)$ that the sum of the family of terms \eqref{dominerelou} does exist.
%
\end{proof}
To achieve the proof of \eqref{thmD}, it is thus sufficient to show 
\begin{prop}\label{butmixingbeta1utensphi}
Uniformly in $x\in\LL$ and $s\in\mathrm{supp}\ v$, as $R\longrightarrow+\infty$
$$I:=\dfrac{1}{\pi h(x)}\int_{\R}e^{it(R-s)}\re{Q_{\delta+it}}(h\varphi)(x)\widehat{u}(t)\dd t\sim
\dfrac{m_\G(\varphi\otimes u)}{\Cg\tilde{L}(R)}.$$
\end{prop}
\noindent
\begin{proof}
Let $A>0$. We split $I$ into $I_1+I_2$ where
$$I_1=\dfrac{1}{\pi h(x)}\int_{|t|>\frac{A}{R-s}}e^{it(R-s)}\re{Q_{\delta+it}}(h\varphi)(x)\widehat{u}(t)\dd t$$
\noindent
and
$$I_2=\dfrac{1}{\pi h(x)}\int_{|t|\leqslant\frac{A}{R-s}}e^{it(R-s)}\re{Q_{\delta+it}}(h\varphi)(x)\widehat{u}(t)\dd t.$$
\noindent
We first deal with $I_1$. We decompose this integral according the sign of $t$; we only detail the arguments for 
$$J=\dfrac{1}{\pi h(x)}\int_{t>\frac{A}{R-s}}e^{it(R-s)}\re{Q_{\delta+it}}(h\varphi)(x)\widehat{u}(t)\dd t.$$
\noindent
Setting $t=y-\frac{\pi}{R-s}$ in $J$, we may write
$$J=-\dfrac{1}{\pi h(x)}\displaystyle{\int_{y>\frac{A+\pi}{R-s}}}e^{iy(R-s)}\re{Q_{\delta+i\left(y-\frac{\pi}{R-s}\right)}}(h\varphi)(x)\widehat{u}
\left(y-\dfrac{\pi}{R-s}\right)\dd y,$$
\noindent
hence
\begin{align*}
2J = & \dfrac{1}{\pi h(x)}\displaystyle{\int_{\frac{A}{R-s}}^{\frac{A+\pi}{R-s}}}e^{it(R-s)}\re{Q_{\delta+it}}(h\varphi)(x)\widehat{u}(t)\dd t\\
& +\dfrac{1}{\pi h(x)}\displaystyle{\int_{t>\frac{A+\pi}{R-s}}}e^{it(R-s)}\re{Q_{\delta+i\left(t-\frac{\pi}{R-s}\right)}}(h\varphi)(x)\left(\widehat{u}(t)-
\widehat{u}\left(t-\dfrac{\pi}{R-s}\right)\right)\dd t\\
& +\dfrac{1}{\pi h(x)}\displaystyle{\int_{t>\frac{A+\pi}{R-s}}}e^{it(R-s)}\left(\re{Q_{\dit}}-\re{Q_{\delta+i\left(t-\frac{\pi}{R-s}\right)}}\right)(h\varphi)(x)
\widehat{u}(t)\dd t\\
=: & K_1+K_2+K_3.
\end{align*}
Let us first deal with $K_1$. By Corollary \ref{integrabilityreQ}, for any $t$ close to $0$
$$\left|\re{Q_{\delta+it}}(h\varphi)(x)\right|\preceq\dfrac{L\left(\frac{1}{|t|}\right)}{|t|\tilde{L}\left(\frac{1}{|t|}\right)^2};$$
\noindent
therefore 
\begin{align*}
|K_1| & \preceq\displaystyle{\int_{\frac{A}{R-s}}^{\frac{A+\pi}{R-s}}}\dfrac{L\left(\frac{1}{t}\right)}{t\tilde{L}\left(\frac{1}{t}\right)^2}\dd t\\
& \preceq \dfrac{1}{\tilde{L}(R)}\dfrac{L(R)}{\tilde{L}(R)}\displaystyle{\int_A^{A+\pi}}
\dfrac{1}{t}\dfrac{L\left(\frac{R-s}{t}\right)}{L(R)}\dfrac{\tilde{L}(R)^2}{\tilde{L}\left(\frac{R-s}{t}\right)^2}\dd t.
\end{align*}
\noindent
Potter's lemma \ref{PB} combined with the fact that $s$ belongs to a compact subset of $\R$ thus implies 
\begin{align}\label{mapremièreétoile}
|K_1|\preceq \dfrac{1}{\tilde{L}(R)}\dfrac{L(R)}{\tilde{L}(R)}\displaystyle{\int_A^{A+\pi}}\dfrac{\dd t}{t^{\frac{1}{4}}}
\preceq \dfrac{1}{\tilde{L}(R)}\dfrac{L(R)}{\tilde{L}(R)}\left(A+\pi\right)^{\frac{3}{4}}.
\end{align}
\noindent
Similarly 
\begin{equation}\label{madeuxièmeétoile}|K_2|\preceq \dfrac{1}{\tilde{L}(R)}\dfrac{L(R)}{\tilde{L}(R)}.\end{equation}
\noindent
Concerning $|K_3|$, we get
\begin{align*}
|K_3|\leqslant & \dfrac{1}{\pi h(x)}\displaystyle{\int_{t>\frac{A+\pi}{R-s}}}
\left|\left|Q_{\delta+it}\right|\right|\left|\left|Q_{\delta+i\left(t-\frac{\pi}{R-s}\right)}\right|\right|\left|\left|\ot_{\dit}-\ot_{\delta+i\left(t-\frac{\pi}{R-s}\right)}\right|\right|
|\widehat{u}(t)|\dd t.
\end{align*}
\noindent
There exists $M>0$ such that the support of $u$ is included in $[-M,M]$; we thus deduce from Propositions \ref{continuityoftransfert} and
\ref{controlQ} that
\begin{align*}|K_3|\preceq & \dfrac{\tilde{L}\left(\dfrac{R-s}{\pi}\right)}{R-s}\displaystyle{\int_{\frac{A+\pi}{R-s}}^M}\dfrac{1}{t\tilde{L}\left(\frac{1}{t}\right)}
\dfrac{1}{\left(t-\frac{\pi}{R-s}\right)\tilde{L}\left(\frac{1}{t-\frac{\pi}{R-s}}\right)}\dd t\\
\preceq &\dfrac{\tilde{L}(R)}{R}\displaystyle{\int_{\frac{A}{R-s}}^{M-\frac{\pi}{R-s}}}\dfrac{1}{\left(t+\frac{\pi}{R-s}\right)\tilde{L}\left(\frac{1}{t+\frac{\pi}{R-s}}\right)}
\dfrac{1}{t\tilde{L}\left(\frac{1}{t}\right)}\dd t.\\
\end{align*}
\noindent
Noticing that $\frac{1}{t}=\frac{R-s}{t(R-s)}$, Potter's lemma implies that when $R\longrightarrow+\infty$
$$\tilde{L}\left(\dfrac{1}{t}\right)\preceq\tilde{L}\left(\dfrac{1}{t+\frac{\pi}{R-s}}\right),$$
\noindent
hence
\begin{align*}|K_3|\preceq & \dfrac{\tilde{L}(R)}{R}\displaystyle{\int_{\frac{A}{R}}^{M}}
\dfrac{1}{t^2\tilde{L}\left(\frac{1}{t}\right)^2} \dd t
\preceq \tilde{L}(R)\displaystyle{\int_{A}^{RM}}\dfrac{1}{t^2\tilde{L}\left(\frac{R}{t}\right)^2} \dd t\\
\preceq & \frac{1}{\tilde{L}(R)}\displaystyle{\int_{A}^{RM}}\dfrac{\tilde{L}(R)^2}{t^2\tilde{L}\left(\frac{R}{t}\right)^2} \dd t
\preceq \frac{1}{\tilde{L}(R)}\displaystyle{\int_{A}^{RM}}\dfrac{1}{t^{\frac{3}{2}}} \dd t,
\end{align*}
\noindent
which yields 
\begin{equation}\label{matroisièmeétoile}|K_3|\preceq\frac{1}{\tilde{L}(R)}\dfrac{1}{\sqrt{A}}.\end{equation}
Combining \eqref{mapremièreétoile}, \eqref{madeuxièmeétoile} and \eqref{matroisièmeétoile}, we deduce
$\lim\limits_{A\longrightarrow+\infty}\lim\limits_{R\longrightarrow+\infty}\tilde{L}(R)|J|=0$ uniformly in $x\in\LL$ and $s\in\mathrm{supp}\ v$. Therefore
\begin{equation}\label{firstpartmixingbeta1}
\lim\limits_{A\longrightarrow+\infty}\lim\limits_{R\longrightarrow+\infty}\tilde{L}(R)|I_1|=0
\end{equation}
\noindent
uniformly in $x\in\LL$ and $s\in\mathrm{supp}\ v$. We now explain why the contribution of $I_2$ is dominant. We write
\begin{align*}
I_2 = & \dfrac{1}{\pi h(x)}\displaystyle{\int_{|t|\leqslant\frac{A}{R-s}}}e^{it(R-s)}\left[\re{Q_{\delta+it}}-\re{\left(1-\lambda_{\delta+it}\right)^{-1}}\Pi_{\delta}\right](h\varphi)(x)\widehat{u}(t)\dd t\\
& + \dfrac{1}{\pi h(x)}\displaystyle{\int_{|t|\leqslant\frac{A}{R-s}}}e^{it(R-s)}\left[\re{\left(1-\lambda_{\delta+it}\right)^{-1}}\Pi_{\delta}\right](h\varphi)(x)\widehat{u}(t)\dd t\\
=: & K_1+K_2.
\end{align*}
\noindent
It follows from Proposition \ref{controlQ} that $\left|K_1\right|\leqslant 2\dfrac{A}{R}$ for $R$ large enough. We split $K_2$ into $L_1+L_2$ where
$$L_1=\dfrac{1}{\pi h(x)}\int_{|t|\leqslant\frac{A}{R-s}}\left(e^{it(R-s)}-1\right)\left[\re{\left(1-\lambda_{\delta+it}\right)^{-1}}\Pi_{\delta}\right](h\varphi)(x)\widehat{u}(t)\dd t$$
\noindent
and
$$L_2=\dfrac{1}{\pi h(x)}\int_{|t|\leqslant\frac{A}{R-s}}\left[\re{\left(1-\lambda_{\delta+it}\right)^{-1}}\Pi_{\delta}\right](h\varphi)(x)\widehat{u}(t)\dd t.$$
\noindent
The equality $\Pi_{\delta}(h\varphi)(x)=\sigma_{\oo}(h\varphi)h(x)=\nu(\varphi)h(x)$ yields
$$L_1=\dfrac{\nu(\varphi)}{\pi}\int_{|t|\leqslant\frac{A}{R-s}}\left(e^{it(R-s)}-1\right)\re{\left(1-\lambda_{\delta+it}\right)^{-1}}\widehat{u}(t)\dd t.$$
\noindent
and the local expansion of $\re{\left(1-\lambda_{\delta+it}\right)^{-1}}$ given in \eqref{localexp} thus implies
\begin{align*}
\tilde{L}(R)L_1 = &\ \tilde{L}(R)\dfrac{\nu(\varphi)}{\Cg}\displaystyle{\int_{|y|\leqslant A}}\dfrac{(e^{iy}-1)}{y}\dfrac{L\left(\frac{R-s}{y}\right)}{\tilde{L}\left(\frac{R-s}{y}\right)^2}
\widehat{u}\left(\dfrac{y}{R-s}\right)\dd y\\
\sim &\ \dfrac{\nu(\varphi)}{\Cg}\dfrac{L(R-s)}{\tilde{L}(R-s)}\displaystyle{\int_{|y|\leqslant A}}\dfrac{(e^{iy}-1)}{y}\dfrac{L\left(\frac{R-s}{y}\right)}{L(R-s)}
\dfrac{\tilde{L}(R-s)^2}{\tilde{L}\left(\frac{R-s}{y}\right)^2}\widehat{u}\left(\dfrac{y}{R-s}\right)\dd y,\\
\preceq &\ A\dfrac{L(R)}{\tilde{L}(R)}.
\end{align*}
\noindent
It remains now to deal with $L_2$; we decompose it according the sign of $t$. We only detail the arguments for $t>0$, \emph{i.e.} we control
$$M=\dfrac{1}{\pi h(x)}\int_0^{\frac{A}{R-s}}\left[\re{\left(1-\lambda_{\delta+it}\right)^{-1}}\Pi_{\delta}\right](h\varphi)(x)\widehat{u}(t)\dd t.$$
\noindent
We follow \cite{MT} and denote by $H$ the function defined as follows:
$$\re{(1-\lambda_{\delta+it})^{-1}}=\dfrac{\pi}{2}\dfrac{1}{\Cg}\dfrac{L\left(\frac{1}{t}\right)}{t\tilde{L}^2\left(\frac{1}{t}\right)}(1+H(t))$$
\noindent
where $H(t)=o(1)$ when $t$ is close to $0$. We obtain
\begin{align*}
M = & \dfrac{\nu(\varphi)}{\pi}\displaystyle{\int_{0}^{\frac{A}{R-s}}}\re{(1-\lambda_{\delta+it})^{-1}}\widehat{u}(t)\dd t\\
= &\dfrac{\nu(\varphi)}{2\Cg}\displaystyle{\int_{0}^{\frac{A}{R-s}}}\dfrac{L\left(\frac{1}{t}\right)}{t\tilde{L}^2\left(\frac{1}{t}\right)}\widehat{u}(t)\dd t +
\dfrac{\nu(\varphi)}{2\Cg}\displaystyle{\int_{0}^{\frac{A}{R-s}}}\dfrac{L\left(\frac{1}{t}\right)}{t\tilde{L}^2\left(\frac{1}{t}\right)}H(t)\widehat{u}(t)\dd t\\
= & \dfrac{\nu(\varphi)}{2\Cg}\displaystyle{\int_{0}^{\frac{A}{R-s}}}\dfrac{L\left(\frac{1}{t}\right)}{t\tilde{L}^2\left(\frac{1}{t}\right)}\widehat{u}(t)\dd t +
O\left(\underset{0\leqslant t\leqslant\frac{A}{R-s}}{\sup}|H(t)|
\displaystyle{\int_{0}^{\frac{A}{R-s}}}\dfrac{L\left(\frac{1}{t}\right)}{t\tilde{L}^2\left(\frac{1}{t}\right)}\dd t\right).
\end{align*}
\noindent
Setting $y=\frac{1}{t}$, it follows from an integration by parts that
$$\displaystyle{\int_{0}^{\frac{A}{R-s}}}\dfrac{L\left(\frac{1}{t}\right)}{t\tilde{L}^2\left(\frac{1}{t}\right)}\widehat{u}(t)\dd t=
\dfrac{\widehat{u}\left(\frac{A}{R-s}\right)}{\tilde{L}\left(\frac{R-s}{A}\right)}+\displaystyle{\int_{\frac{R-s}{A}}^{+\infty}}
\widehat{u}'\left(\frac{1}{y}\right)\tilde{L}(y)^{-1}\dfrac{\dd y}{y^2}.$$
\noindent
Using the properties of slowly varying functions, we deduce that the second term of the right member is negligeable with respect to the first one when 
$R\longrightarrow+\infty$. The regularity of $\widehat{u}$ yields 
$$\lim\limits_{R\longrightarrow+\infty}\widehat{u}\left(\dfrac{A}{R-s}\right)=\widehat{u}(0)$$
\noindent
uniformly in $s\in\mathrm{supp}\ v$. Finally
$$\tilde{L}(R)M\underset{R\longrightarrow+\infty}{\sim}\dfrac{\nu(\varphi)\widehat{u}(0)}{2\Cg};$$
\noindent
hence
$$\tilde{L}(R)L_2\underset{R\longrightarrow+\infty}{\sim}\dfrac{1}{\Cg}m_\G(\varphi\otimes u).$$
\noindent
Combining this result with \eqref{firstpartmixingbeta1}, we finally obtain
$$\lim\limits_{R\longrightarrow+\infty}\dfrac{\tilde{L}(R)}{\pi h(x)}\int_{\R}e^{it(R-s)}\re{Q_{\delta+it}}(h\varphi)(x)\widehat{u}(t)\dd t=
\dfrac{m_\G(\varphi\otimes u)}{\Cg}$$
\noindent
uniformly in $x\in\LL$ and $s\in\mathrm{supp}\ v$. This achieves the proof of Proposition \ref{butmixingbeta1utensphi}.
\end{proof}
\subsubsection{Proof of Proposition \ref{convergencexiversdelta}}
We follow the steps of paragraph 6 in \cite{MT}. To show the convergence, we use the local expansion of $(1-\lambda_{\xi+it})^{-1}$ near $0$ to bound the function
$t\longmapsto\re{Q_{\xi+it}}(h\varphi)(x)$ from above by an integrable function. Let us write $\xi=\delta+\kappa$ where $\kappa>0$.
\begin{prop}\label{localexpb1}
There exist finite measures $\nu_j$ on $\R^+$ with masses $P_j$, $j\in[\![1,p]\!]$ such that
$$1-\lambda_{\delta+\kappa+it}=\sum\limits_{j=1}^p P_j\left[\left(\kappa I_{C_j}+t I_{S_j}\right)-i\left(\kappa I_{S_j}-t I_{C_j}\right)\right]
+o\left(|t|\tilde{L}\left(\dfrac{1}{|t|}\right)+\kappa\tilde{L}\left(\dfrac{1}{\kappa}\right)\right),$$
\noindent
where, for any $j\in[\![1,p]\!]$, $I_{S_j}$ and $I_{C_j}$ are defined as in Proposition \ref{asymptoticISetIC} by
$$I_{S_j}:=\int_0^{+\infty}e^{-\kappa y}\sin(ty)\left(1-\nu_j([0,y])\right)\dd y$$
\noindent
and
$$I_{C_j}:=\int_0^{+\infty}e^{-\kappa u}\cos(ty)\left(1-\nu_j([0,y])\right)\dd y.$$
\end{prop}
\begin{proof}
The steps are the same as for the local expansion of $\lambda_{\delta+it}$ (in \ref{localexp}). 
For $z=\delta+\kappa+it$, we write
$$\lambda_z=\sigma_{\oo}(\ot_z h_z)=\sigma_{\oo}(\ot_z h)+\sigma_{\oo}\left((\ot_z-\otd)(h_z-h)\right).$$
\noindent
By Proposition \ref{continuityoftransfert}, the contribution of the second term is
$\preceq\left(\kappa \tilde{L}\left(\dfrac{1}{\kappa}\right)+|t|\tilde{L}\left(\dfrac{1}{|t|}\right)\right)^2$. The first term may be decomposed as
\begin{align*}
\sigma_{\oo}(\ot_{\delta+\kappa+it} h) & =1+\sigma_{\oo}(\ot_{\delta+\kappa+it} h)-1=1+\sigma_{\oo}(\ot_{\delta+\kappa+it} h)-\sigma_{\oo}(\otd h)\\
& =1+\sum\limits_{j=1}^{p+q}S_j
\end{align*}
\noindent
where
$$S_j:=\sum\limits_{\alpha\in\G_j^*}\displaystyle{\int_{\LL\setminus\LL_j}}h(\alpha.x)e^{-\delta b(\alpha,x)}(e^{-(it+\kappa)b(\alpha,x)}-1)\dd\sigma_{\oo}(x).$$
\noindent
Using the notations of Proposition \ref{theoreticlocalexp}, we obtain
$$S_j=\int_{\LL\setminus\LL_j}M_j(x)\left(\widehat{\mu_j^x}(-t+i\kappa)-1\right)\dd\sigma_\oo(x)$$
\noindent 
and we deduce from Proposition \ref{etapecontroltransfert} that the contribution of this quantity is
$o\left(\kappa \tilde{L}\left(\dfrac{1}{\kappa}\right)\right)+o\left(|t|\tilde{L}\left(\dfrac{1}{|t|}\right)\right)$
for $j\in[\![p+1,p+q]\!]$.

Let $j\in[\![1,p]\!]$ and $x_j$ be the fixed point of the elementary parabolic group $\G_j$ and set
$\Delta_{\alpha}(x):=b(\alpha,x)-\dhy(\oo,\alpha.\oo)$ for any $\alpha\in\G_j^*$ and $x\in\LL\setminus\LL_j$. The sequence 
$\left(\Delta_\alpha(x)\right)_{\alpha\in\G_j}$ converges to $-2\left(x|x_j\right)_{\oo}$ when 
$\dhy(\oo,\alpha.\oo)\longrightarrow+\infty$, uniformly in $x\notin\LL_j$.
We thus split $S_j$ into $S_{j1}+S_{j2}$ where
%
$$S_{j1}:=\sum\limits_{\alpha\in\G_j^*}e^{-\del\dhy(\oo,\alpha.\oo)}\displaystyle{\int_{\LL\setminus\LL_j}}h(\alpha.x)e^{-\delta\Delta_{\alpha}(x)}\left(e^{-(it+\kappa)b(\alpha,x)}-e^{-(it+\kappa)\dhy(\oo,\alpha.\oo)}\right)\dd\sigma_{\oo}(x)$$
\noindent
and
$$S_{j2}:=\sum\limits_{\alpha\in\G_j^*}e^{-\del\dhy(\oo,\alpha.\oo)}\left(e^{-(it+\kappa)\dhy(\oo,\alpha.\oo)}-1\right)\displaystyle{\int_{\LL\setminus\LL_j}}h(\alpha.x)e^{-\delta\Delta_{\alpha}(x)}\dd\sigma_{\oo}(x).$$
\noindent
Hence
$$
\left|S_{j1}\right|\leqslant Ce^{\delta C}(|t|+\kappa)\sum\limits_{\alpha\in\G_j^*}e^{-\del\dhy(\oo,\alpha.\oo)}=O(|t|+\kappa)=o\left(|t|\tilde{L}\left(\dfrac{1}{|t|}\right)+\kappa \tilde{L}\left(\dfrac{1}{\kappa}\right)\right),
$$
\noindent
since $\lim\limits_{x\longrightarrow+\infty}\tilde{L}(x)=+\infty$. We now decomposed $S_{j2}$ as 
$S_{j2}=P_j\widehat{\nu_j}(-t+i\kappa)-P_j$ where
$$\nu_j=\dfrac{1}{P_j}\sum\limits_{\alpha\in\G_j}e^{-\delta\dhy(\oo,\alpha.\oo)}\left(\displaystyle{\int_{\LL\setminus\LL_j}}h(\alpha.x)
e^{-\delta\Delta_{\alpha}(x)}\dd\sigma_{\oo}(x)\right)D_{\dhy(\oo,\alpha.\oo)}.$$
\noindent
The normalizing coefficient $P_j$ is given by
$$P_j=\sum\limits_{\alpha\in\G_j}e^{-\delta\dhy(\oo,\alpha.\oo)}\left(\displaystyle{\int_{\LL\setminus\LL_j}}h(\alpha.x)
e^{-\delta\Delta_{\alpha}(x)}\dd\sigma_{\oo}(x)\right).$$
\noindent
It is finite by Assumption $(P_1)$ and Corollary \ref{busedist} implies
$\left|\Delta_{\alpha}(x)\right|\leqslant C$ uniformly in $x\notin\LL_j$. In addition, these measures satisfy
\begin{align}\label{equivalencepourmesurenu}
1-\nu_j([0,T])=
\dfrac{1}{P_j}\sum\limits_{\alpha\ |\ \dhy(\oo,\alpha.\oo)>T}e^{-\delta\dhy(\oo,\alpha.\oo)}\left(\displaystyle{\int_{\LL\setminus\LL_j}}
h(\alpha.x)e^{-\delta\Delta_\alpha(x)}\dd\sigma_{\oo}(x)\right)
\sim \dfrac{\mathcal{C}_j}{P_j}\dfrac{L(T)}{T^{\beta}},
\end{align}
\noindent
when $T\longrightarrow+\infty$, with
$$\mathcal{C}_j=\lim\limits_{\alpha\longrightarrow+\infty}\displaystyle{\int_{\LL\setminus\LL_j}}h(\alpha.x)
e^{-\delta\Delta_{\alpha}(x)}\sigma_{\oo}(\dd x)=
h(x_{\mathfrak{a}_j})\displaystyle{\int\limits_{\LL\setminus\LL_j}}\dfrac{\dd\sigma_{\oo}(x)}{\dhy_{\oo}(x,x_{\mathfrak{a}_j})^{\frac{2\delta}{a}}}.$$
\noindent
Therefore
\begin{align*}
\dfrac{S_{j2}}{P_j} = & \displaystyle{\int_0^{+\infty}}\left(e^{i(-t+i\kappa)y}-1\right)\dd\nu_j(y)\\
= & -(\kappa+it)\displaystyle{\int_0^{+\infty}}e^{i(-t+i\kappa)y}(1-\nu_j[0,y])\dd y\\
= & \left(-\kappa I_{Cj}-tI_{Sj}\right)+i\left(\kappa I_{Sj}-tI_{Cj}\right)
\end{align*}
\noindent
where
$$I_{Cj}:=\displaystyle{\int_0^{+\infty}}e^{-\kappa y}\cos(ty)(1-\nu_j[0,y])\dd y$$
\noindent
and
$$I_{Sj}:=\displaystyle{\int_0^{+\infty}}e^{-\kappa y}\sin(ty)(1-\nu_j[0,y])\dd y.$$
\noindent
Finally
$$1-\lambda_{\delta+\kappa+it}=\sum_{j=1}^pP_j\left[\left(\kappa I_{Cj}+tI_{Sj}\right)-i\left(\kappa I_{Sj}-tI_{Cj}\right)\right]+
o\left(|t|\tilde{L}\left(\dfrac{1}{|t|}\right)+\kappa \tilde{L}\left(\dfrac{1}{\kappa}\right)\right).$$
\end{proof}
As a consequence of Propositions \ref{asymptoticISetIC} and \ref{localexpb1}, we may state
\begin{prop}\label{propositioncomportement1moinslambda}
Let $\kappa>0$ and $\varepsilon>0$ be as in Proposition \ref{spectreperturbation}. For any $t\in\R$ such that $\max\left(\kappa,|t|\right)<\varepsilon$, the dominant eigenvalue $\lambda_{\delta+\kappa+it}$ satisfies
\begin{align*}
1) & \left|1-\lambda_{\delta+\kappa+it}\right|^{-1} \preceq \left|\re{1-\lambda_{\delta+\kappa+it}}\right|^{-1}\preceq
\dfrac{1}{\sum\limits_{j=1}^p\mathcal{C}_j}\dfrac{1}{\kappa\tilde{L}\left(\dfrac{1}{\kappa}\right)}\ when\ |t|\leqslant\kappa;\\
2) & \left|1-\lambda_{\delta+\kappa+it}\right|^{-1}\preceq
\dfrac{1}{\sum\limits_{j=1}^p\mathcal{C}_j}\dfrac{1}{(\kappa+|t|)\tilde{L}\left(\dfrac{1}{|t|}\right)}\ when\ |t|>\kappa;\\
3) & \left|\re{\left(1-\lambda_{\delta+\kappa+it}\right)^{-1}}\right|\preceq
\dfrac{1}{\sum\limits_{j=1}^p\mathcal{C}_j}
\left(\dfrac{\kappa}{(\kappa^2+|t|^2)\tilde{L}\left(\dfrac{1}{|t|}\right)}+
\dfrac{|t|L\left(\dfrac{1}{|t|}\right)}{(\kappa^2+|t|^2)\tilde{L}\left(\dfrac{1}{|t|}\right)^2}\right)\\
&\ when\ |t|>\kappa.
\end{align*}
\end{prop}
\begin{proof}
From Proposition \ref{localexpb1}, we deduce
$$\re{1-\lambda_{\delta+\kappa+it}}=\sum\limits_{j=1}^pP_j\left[\kappa I_{C_j}+t I_{S_j}\right]+o\left(|t|\tilde{L}\left(\dfrac{1}{|t|}\right)\right)+o\left(\kappa\tilde{L}\left(\dfrac{1}{\kappa}\right)\right)$$
\noindent
and
$$\im{\lambda_{\delta+\kappa+it}}=-\im{1-\lambda_{\delta+\kappa+it}}=\sum\limits_{j=1}^pP_j\left[\kappa I_{S_j}-t I_{C_j}\right]+o\left(|t|\tilde{L}\left(\dfrac{1}{|t|}\right)\right)+o\left(\kappa\tilde{L}\left(\dfrac{1}{\kappa}\right)\right).$$
\begin{itemize}
 \item[1)]Combining \eqref{equivalencepourmesurenu} and assertions i) and iii) of Proposition \ref{asymptoticISetIC}, we obtain
$$\re{1-\lambda_{\delta+\kappa+it}}\sim\sum\limits_{j=1}^{p}\mathcal{C}_j\left[\kappa\tilde{L}\left(\dfrac{1}{\kappa}\right)(1+o(1))+
\kappa O\left(\dfrac{|t|}{\kappa}L\left(\dfrac{1}{\kappa}\right)\right)+t I_{S_j}\right].$$
\noindent
Since $|t|\leqslant\kappa$, it follows that
$$\kappa\dfrac{|t|}{\kappa}L\left(\dfrac{1}{\kappa}\right)=
o\left(\kappa\tilde{L}\left(\dfrac{1}{\kappa}\right)\right),$$
\noindent
and Karamata's lemma yields
$$\left|I_{S_j}\right|\preceq\dfrac{|t|^2}{\kappa}L\left(\dfrac{1}{\kappa}\right)\preceq\kappa L\left(\dfrac{1}{\kappa}\right)=
o\left(\kappa\tilde{L}\left(\dfrac{1}{\kappa}\right)\right),$$
\noindent
hence
$$\re{1-\lambda_{\delta+\kappa+it}}\sim\left(\sum\limits_{j=1}^p\mathcal{C}_j\right)\kappa\tilde{L}\left(\dfrac{1}{\kappa}\right)\ \text{as}\ t\longrightarrow0.$$
\noindent
The first assertion follows.
\item[2)]We use properties ii) and iv) of Proposition \ref{asymptoticISetIC}. In this case
$$\re{1-\lambda_{\delta+\kappa+it}}\sim\sum\limits_{j=1}^{p}\mathcal{C}_j\left[\kappa\tilde{L}\left(\dfrac{1}{|t|}\right)(1+o(1))+
\kappa O\left(\dfrac{\kappa}{|t|}L\left(\dfrac{1}{|t|}\right)\right)+t I_{S_j}\right].$$
\noindent
Inequalities
$$\left|tI_{S_j}\right|\preceq|t|L\left(\dfrac{1}{|t|}\right)=O\left(|t|L\left(\dfrac{1}{|t|}\right)\right)$$
\noindent
and
$$\dfrac{\kappa^2}{|t|}L\left(\dfrac{1}{|t|}\right)\leqslant|t|L\left(\dfrac{1}{|t|}\right)=O\left(|t|L\left(\dfrac{1}{|t|}\right)\right),$$
\noindent
imply that, as $t\longrightarrow0$
\begin{equation}\label{partiereelle}
\re{1-\lambda_{\delta+\kappa+it}}\sim\left(\sum\limits_{j=1}^p\mathcal{C}_j\right)
\left(\kappa\tilde{L}\left(\dfrac{1}{|t|}\right)+O\left(|t|L\left(\dfrac{1}{|t|}\right)\right)\right).
\end{equation}
\noindent
Similarly 
$$\im{\lambda_{\delta+\kappa+it}}=\sum\limits_{j=1}^p\mathcal{C}_j\left[-t\tilde{L}\left(\dfrac{1}{|t|}\right)(1+o(1))-
tO\left(\dfrac{\kappa}{|t|}L\left(\dfrac{1}{|t|}\right)\right)+\kappa I_{S_j}\right]$$
\noindent
and Karamata's lemma implies
\begin{equation}\label{partieimaginaire}
\im{\lambda_{\delta+\kappa+it}}\sim-\left(\sum\limits_{j=1}^p\mathcal{C}_j\right)t\tilde{L}\left(\dfrac{1}{|t|}\right).
\end{equation}
\noindent
Since $\left|1-\lambda_{\delta+\kappa+it}\right|\geqslant
\frac{1}{2}\left(\left|\re{1-\lambda_{\delta+\kappa+it}}\right|+\left|\im{1-\lambda_{\delta+\kappa+it}}\right|\right)$, the estimates 
\eqref{partiereelle} and \eqref{partieimaginaire} combined with  Karamata's lemma furnish
$$\left|1-\lambda_{\delta+\kappa+it}\right|\succeq
\left(\sum\limits_{j=1}^p\mathcal{C}_j\right)\left(\kappa+|t|\right)\tilde{L}\left(\dfrac{1}{|t|}\right).$$
\noindent
The second assertion follows.
\item[3)] It is sufficient to notice that
$$\re{\left(1-\lambda_{\delta+\kappa+it}\right)^{-1}}=\dfrac{\re{1-\lambda_{\delta+\kappa+it}}}{\left|1-\lambda_{\delta+\kappa+it}\right|^2}$$
\noindent
and to use \eqref{partiereelle} combined with $2)$.
\end{itemize}
\end{proof}
The previous proposition leads us to the following property.
\begin{prop}\label{fonctiondedominationpourrelle}
There exists $\kappa_0>0$ such that for any $0<\kappa<\kappa_0$ and any compact subset $K$ of $\R$, the function 
$t\longmapsto\re{Q_{\delta+\kappa+it}}(h\varphi)(x)$ is bounded by $b_\kappa(t)$ defined by
\begin{align*}
b_{\kappa}\ :\ t\longmapsto & \dfrac{\kappa \tilde{L}\left(\frac{1}{\kappa}\right)}{\kappa+|t|}+
\dfrac{1}{\kappa\tilde{L}\left(\frac{1}{\kappa}\right)}\un_{|t|\leqslant\kappa}\\
& + \dfrac{\kappa}{(\kappa^2+t^2)\tilde{L}\left(\frac{1}{|t|}\right)}+
\dfrac{ L\left(\frac{1}{|t|}\right)}{|t|\tilde{L}\left(\frac{1}{|t|}\right)^2},
\end{align*}
uniformly in $x\in\LL$ and $t\in K$.
\end{prop}
\begin{proof}
Let $\kappa\in[0,1]$ and $t\in\R$ such that $\max\left(\kappa,|t|\right)<\varepsilon$, where
$\varepsilon$ is given in Proposition \ref{spectreperturbation}. Using the arguments exposed in the proof of Proposition 
\ref{controlQ}, we obtain

\begin{equation}\label{decompoQlourde}
Q_{\delta+\kappa+it}=\left(1-\lambda_{\delta+\kappa+it}\right)^{-1}\Pi_{\delta}+
\left(1-\lambda_{\delta+\kappa+it}\right)^{-1}\left(\Pi_{\delta+\kappa+it}-\Pi_\delta\right)+O(1).
\end{equation}
\noindent
By assertions 1) and 3) of the previous proposition, we deduce
\begin{align*}
\re{\left(1-\lambda_{\delta+\kappa+it}\right)^{-1}}\leqslant & \left|\re{\left(1-\lambda_{\delta+\kappa+it}\right)^{-1}}\right|\un_{|t|>\kappa}+
\left|\re{\left(1-\lambda_{\delta+\kappa+it}\right)^{-1}}\right|\un_{|t|\leqslant\kappa}\\
\preceq & \left[\dfrac{\kappa}{(\kappa^2+|t|^2)\tilde{L}\left(\frac{1}{|t|}\right)}+
\dfrac{|t|\tilde{L}\left(\frac{1}{|t|}\right)}{(\kappa^2+|t|^2)\tilde{L}\left(\frac{1}{|t|}\right)^2}\right]\un_{]\kappa,\varepsilon]}(|t|)\\
& +\dfrac{1}{\kappa\tilde{L}\left(\frac{1}{\kappa}\right)}\un_{|t|\leqslant\kappa}\\
\preceq & \dfrac{\kappa}{(\kappa^2+|t|^2)\tilde{L}\left(\frac{1}{|t|}\right)}+
\dfrac{\tilde{L}\left(\frac{1}{|t|}\right)}{|t|\tilde{L}\left(\frac{1}{|t|}\right)^2}+
\dfrac{1}{\kappa\tilde{L}\left(\frac{1}{\kappa}\right)}\un_{|t|\leqslant\kappa}.
\end{align*}
\noindent
The functions $t\longmapsto\Pi_{\delta+it}$ and $\kappa\longmapsto\Pi_{\delta+\kappa+it}$ inherit the regularity of the functions 
$t\longmapsto\ot_{\delta+it}$ and $\kappa\longmapsto\ot_{\delta+\kappa+it}$, therefore Proposition 
\ref{propositioncomportement1moinslambda} combined with Potter's lemma yields to the following 
estimate of the second term in \eqref{decompoQlourde}
\begin{align*}
\bigg| & \left(1-\lambda_{\delta+\kappa+it}\right)^{-1}\left(\Pi_{\delta+\kappa+it}-\Pi_\delta\right)\bigg|\\
& \preceq
\left(\dfrac{1}{\kappa\tilde{L}\left(\frac{1}{\kappa}\right)}\un_{|t|\leqslant\kappa}+
\dfrac{1}{(\kappa+|t|)\tilde{L}\left(\frac{1}{|t|}\right)}\un_{]\kappa,\varepsilon]}(|t|)\right)
\left(\kappa\tilde{L}\left(\dfrac{1}{\kappa}\right)+|t|\tilde{L}\left(\dfrac{1}{|t|}\right)\right)\\
& \preceq 1+\dfrac{1}{\kappa\tilde{L}\left(\frac{1}{\kappa}\right)}\un_{|t|\leqslant\kappa}+
\dfrac{\kappa\tilde{L}\left(\frac{1}{\kappa}\right)}{\kappa+|t|}\\
& \preceq \dfrac{1}{\kappa\tilde{L}\left(\frac{1}{\kappa}\right)}\un_{|t|\leqslant\kappa}+
\dfrac{\kappa\tilde{L}\left(\frac{1}{\kappa}\right)}{\kappa+|t|}
\end{align*}
\noindent
since $1=o\left(\dfrac{1}{\kappa\tilde{L}\left(\frac{1}{\kappa}\right)}\right)$ as $\kappa\searrow0$.
\end{proof}
\noindent
We now achieve the proof of \eqref{propmegarelou}. For any $x\in\LL$ and $t\neq0$, the sequence $\left(\re{Q_{\delta+\kappa+it}}(h\varphi)(x)\right)_{\kappa>0}$ tends to $\re{Q_{\delta+it}}(h\varphi)(x)$ when $\kappa\searrow0$. Moreover, the previous proposition implies that for any $\kappa\searrow0$ and any $t$ in the compact support of $\widehat{\varphi}$
$$\left|\re{Q_{\delta+\kappa+it}}(h\varphi)(x)\right|\leqslant h_\kappa(t)+
\dfrac{ L\left(\frac{1}{|t|}\right)}{|t|\tilde{L}\left(\frac{1}{|t|}\right)^2}$$
\noindent
where
\begin{equation}\label{defihkappa}h_\kappa(t)= \dfrac{\kappa \tilde{L}\left(\frac{1}{\kappa}\right)}{\kappa+|t|}+
\dfrac{1}{\kappa\tilde{L}\left(\frac{1}{\kappa}\right)}\mathrm{1}_{|t|\leqslant\kappa}
+ \dfrac{\kappa}{(\kappa^2+t^2)\tilde{L}\left(\frac{1}{|t|}\right)}.\end{equation}
\noindent
Corollary \ref{integrabilityreQ} implies that the function 
$t\longmapsto\dfrac{ L\left(\frac{1}{|t|}\right)}{|t|\tilde{L}\left(\frac{1}{|t|}\right)^2}$ 
is integrable on any compact. Since the support of $\widehat{\varphi}$ is compact, it is sufficient 
to prove that $\lim\limits_{\kappa\searrow0}\int_{-M}^M h_\kappa(t)\widehat{\varphi}(t)\dd t=0$ for 
any $M>0$. We write on the one hand
\begin{align*}
\int_{-M}^M\dfrac{\kappa\tilde{L}\left(\frac{1}{\kappa}\right)}{\kappa+|t|}\dd t \preceq & 
\kappa\tilde{L}\left(\frac{1}{\kappa}\right)\int_{0}^M\dfrac{\dd t}{\kappa+t}\\
\preceq & \kappa\tilde{L}\left(\frac{1}{\kappa}\right)\left(\ln(\kappa+M)-\ln(\kappa)\right),
\end{align*}
\noindent
quantity which goes to $0$ when $\kappa\searrow0$. On the other hand
\begin{align*}
\int_{-M}^M\dfrac{1}{\kappa\tilde{L}\left(\frac{1}{\kappa}\right)}\un_{[0,\kappa]}(|t|)\dd t \preceq & 
\dfrac{1}{\kappa\tilde{L}\left(\frac{1}{\kappa}\right)}\int_{0}^{\kappa}\dd t\\
\preceq & \dfrac{1}{\tilde{L}\left(\frac{1}{\kappa}\right)}
\end{align*}
\noindent
which goes to $0$ when $\kappa\searrow0$. The control of the integral of the last term of \eqref{defihkappa} relies on the following trick used in  \cite{MT}. Let $A\in[0,M]$; we may write
\begin{align*}
\int_{-M}^M\dfrac{\kappa}{(\kappa^2+|t|^2)\tilde{L}\left(\frac{1}{|t|}\right)}\dd t \preceq & 
\int_{0}^M\dfrac{\kappa}{(\kappa^2+t^2)\tilde{L}\left(\frac{1}{t}\right)}\dd t\\
\preceq & \int_{0}^A\dfrac{\kappa}{(\kappa^2+t^2)\tilde{L}\left(\frac{1}{t}\right)}\dd t+
\int_{A}^M\dfrac{\kappa}{(\kappa^2+t^2)\tilde{L}\left(\frac{1}{t}\right)}\dd t\\
\preceq & \dfrac{A}{\kappa\tilde{L}\left(\frac{1}{A}\right)}+\dfrac{\kappa}{A}.
\end{align*}
\noindent
We may choose $A=A(\kappa)$ 
such that $\kappa=\dfrac{A}{\sqrt{\tilde{L}\left(\frac{1}{A}\right)}}$. From properties of slowly varying functions, we deduce that
$A(\kappa)\longrightarrow0$ when $\kappa\longrightarrow0$. For such $A$, we get
$$\int_{-M}^M\dfrac{\kappa}{(\kappa^2+|t|^2)\tilde{L}\left(\frac{1}{|t|}\right)}\dd t\preceq 
\dfrac{1}{\sqrt{\tilde{L}\left(\frac{1}{A}\right)}}\underset{\kappa\searrow0}{\longrightarrow}0.$$
\section{Theorem B: closed geodesics for $\beta\in]0,1[$}
In this section, we will assume $\beta\in]0,1[$ and we will establish an asymptotic lower bound for the number of closed geodesics of $\xx/\G$ with length $\leqslant R$. 
Let $\mathscr{G}$ denote the set of all closed geodesics of $\xx/\G$ and set $\mathrm{N}_{\mathscr{G}}(R):=\sharp\{\wp\in\mathscr{G}\ |\ l(\wp)\leqslant R\}$. 
%
%
%
A closed geodesic of $\xx/\G$ is the projection on $\xx/\G$ of the axis of an hyperbolic isometry $\g\in\G$. The contribution of geodesics with length $\leqslant R$ associated to powers of generators of some subgroup $\G_j,p+1\leqslant j\leqslant p+q$ is polynomial in $R$ and thus negligeable with respect to $\frac{e^{\delta R}}{R}$.
Therefore, in the sequel, we will only consider the hyperbolic isometries $\g$ of $\G$ with symbolic length $\geqslant2$. Up to a conjugacy, we may assume 
that $i(\g)\neq l(\g)$. Let $\tilde{\mathrm{N}}_{\mathscr{G}}(R)$ the number of closed geodesics $\wp$ with length $\leqslant R$ associated to a 
hyperbolic isometry $\g_{\wp}$ with symbolic length $\geqslant2$. Theorem $B$ may be reformulated as follows:
\begin{thmB}
With the previous notations, for any $\beta\in]0,1[$, as $R\longrightarrow+\infty$,
$$\liminf\limits_{R\longrightarrow+\infty}\dfrac{\delta R}{\beta e^{\delta R}}\tilde{\mathrm{N}}_{\mathscr{G}}(R)\geqslant1.$$
\end{thmB}
\subsection{Proof of Theorem B}
The coding of the geodesic flow given in Section 4 allows us to write $\tilde{\mathrm{N}}_{\mathscr{G}}(R)$ as
$$\tilde{\mathrm{N}}_{\mathscr{G}}(R)=\sum_{\wp\in\mathscr{G},\ |\g_{\wp}|\geqslant2}\un_{l(\wp)\leqslant R}=\sum\limits_{k\geqslant2}\dfrac{1}{k}\sum\limits_{\T^kx=x}\un_{[0,R]}(S_k\gol(x)).$$
\noindent
Following \cite{La} and \cite{BDP}, the idea is to approach $\tilde{\mathrm{N}}_{\mathscr{G}}(R)$ by quantities of the form 
$$\sum\limits_{k\geqslant2}\dfrac{1}{k}\sum\limits_{\T^ky=x}\un_{[0,R]}(S_k\gol(y))$$
\noindent
for suitable points $x$. Any point $x\in\LL^0$ may be obtained as a limit
$\lim\limits_{p\longrightarrow+\infty}\alpha_1...\alpha_k.x_0$ for a unique sequence $\mathcal{A}$-admissible $\left(\alpha_k\right)_{k\geqslant1}$. We will call $\alpha_1,...,\alpha_k$ the ``first $k$ letters of $x$''. For all $\g\in\G$, denote by $\LL_\g^0=\g\left(\LL^0\setminus\LL_{l(\g)}^0\right)$ and fix a point $x^\g\in\LL_\g^0$.
For technical reasons, we assume that each point $x^\g$ is non-periodic. We also may notice  that for any $n\geqslant1$, the family $\left(\LL_\g^0\right)_{\g\in\G(n)}$ is a partition of $\LL^0$: the larger $n$ is, the finer is this partition. Fix $n\in\N$, $\g\in\G(n)$ and a point $y\in\LL_\g^0$ such that $\T^ky=x^\g$ for some $k\geqslant2$. There exists a unique $k$-periodic point $x\in\LL_\g^0$ such that the points $x$ and $y$ have the same $k$ first letters. From the assumptions on $y$, we deduce that these two points have in fact the same $k+n$ letters. 
Thus there exist a finite $\mathcal{A}$-admissible sequence $(\alpha_1,...,\alpha_{k+n})$ and points $x',y'\in\LL\setminus\LL_{l(\alpha_{k+n})}^0$ such that
$x=\alpha_1...\alpha_{k+n}.x'$ and $y=\alpha_1...\alpha_{k+n}.y'$ and we get
\begin{align*}
\left|S_k\gol(y)-S_k\gol(x)\right| \leqslant & \left|\gol(y)-\gol(x)\right|+...+\left|\gol(\T^{k-1}.y)-\gol(\T^{k-1}.x)\right|\\
\leqslant & \left|b(\alpha_1...\alpha_{k+n},y')-b(\alpha_1...\alpha_{k+n},x')\right|+...\\
& +\left|b(\alpha_k...\alpha_{k+n},y')-b(\alpha_k...\alpha_{k+n},x')\right|\\
\preceq & r^{k+n}+...+r^{n+1} \preceq \dfrac{r^{n+1}}{1-r}:=\varepsilon_n
\end{align*}
\noindent
for $r\in]0,1[$ given in Corollary \ref{actioncontractante}. It follows that
\begin{equation}\label{çafinitsuruneligne}
\sum\limits_{\g\in\G(n)}\sum\limits_{k\geqslant2}\dfrac{1}{k}\sum\limits_{\T^ky=x^\g}\un_{\LL_\g}(y)\un_{[0,R-\varepsilon_n]}(S_k\gol(y))
\leqslant\tilde{\mathrm{N}}_{\mathscr{G}}(R)
\end{equation}
\noindent
Theorem B will thus be a consequence of the following proposition.
\begin{prop}\label{propintermediairegeodfermees} For any function $\varphi\in\mathrm{Lip}\left(\LL\right)$ and uniformly in $x\in\LL^0$, we have as $R\longrightarrow+\infty$
$$\sum\limits_{k\geqslant2}\dfrac{1}{k}\sum\limits_{\T^ky=x}\varphi(y)\un_{[0,R]}(S_k\gol(y))\sim
\beta\sigma_\oo\left(\varphi\right)h(x)\dfrac{e^{\delta R}}{\delta R}.$$
\end{prop}
\begin{proof}
The quantity $N\left(\varphi,x,R\right):=\sum\limits_{k\geqslant2}\frac{1}{k}\sum\limits_{\T^ky=x}\varphi(y)\un_{[0,R]}\left(S_k\gol(y)\right)$
may be written as 
$$N\left(\varphi,x,R\right)=\sum\limits_{k\geqslant2}\dfrac{1}{k}\sum\limits_{\T^ky=x}\varphi(y)e^{-\delta S_k\gol(y)}u_R\left(S_k\gol(y)\right)$$
\noindent
where $u_R(t)=e^{\delta t}\un_{[0,R]}(t)$. Let us fix $\eta>0$ and set $P=\left[\frac{R}{\eta}\right]$. From
$$\sum\limits_{p=0}^{P-1}e^{\delta\eta p}\un_{[\eta p,\eta(p+1)[}(t)\leqslant
u_R(t),$$
\noindent
we deduce
\begin{align*}
\sum\limits_{p=0}^{P-1}e^{\delta\eta p}\sum\limits_{k\geqslant2}\dfrac{1}{k}\sum\limits_{\T^ky=x}\varphi(y)e^{-\delta S_k\gol(y)}\un_{[0,\eta[}\left(S_k\gol(y)-\eta p\right)\leqslant
N\left(\varphi,x,R\right).
\end{align*}
\noindent
Let us first admit the following proposition, whose proof is postponed to subsection 7.2.
\begin{prop}\label{butdessectionsgeodfermees}
For any functions $\varphi\in\mathrm{Lip}\left(\LL\right)$ and $u\ :\ \R\longrightarrow\R$ with compact support, one gets
$$ 
\lim\limits_{p\longrightarrow+\infty}\underset{x\in\LL^0}{\sup}
\left|p\sum\limits_{k\geqslant2}\dfrac{1}{k}\sum\limits_{\T^ky=x}\varphi(y)e^{-\delta S_k\gol(y)}u\left(S_k\gol(y)-\eta p\right)-
\beta\sigma_\oo\left(\varphi\right)h(x)\dfrac{\widehat{u}(0)}{\eta}\right|=0.$$
\end{prop}
\noindent
We also need the following Lemma (see \cite{BDP}). 
\begin{lem}\label{dedalbopeigne}
Let $(V_p(x))_{p\geqslant0}$ a sequence of positive functions defined on a compact set $K$ and $\left(v_p\right)_{p\geqslant0}$ a divergent positive series. If
$\displaystyle{\lim_{p\longrightarrow+\infty}\sup_{x\in K}\left|\frac{V_p(x)}{v_p}-1\right|=0}$, then
$$\lim\limits_{N\longrightarrow+\infty}\sup\limits_{x\in K}\left|\dfrac{\sum\limits_{p=0}^NV_p(x)}{\sum\limits_{p=0}^Nv_p}-1\right|=0.$$
\end{lem}
\noindent
Proposition \ref{butdessectionsgeodfermees} and Lemma \ref{dedalbopeigne} imply that, as $P\longrightarrow+\infty$
$$\sum\limits_{p=1}^{P}e^{\delta\eta p}\sum\limits_{k\geqslant2}\dfrac{1}{k}\sum\limits_{\T^ky=x}\varphi(y)e^{-\delta S_k\gol(y)}\un_{[0,\eta[}\left(S_k\gol(y)-\eta p\right)
\sim\left(\sum\limits_{p=1}^{P}\dfrac{e^{\delta\eta p}}{p}\right)\beta\sigma_\oo\left(\varphi\right)h(x)$$
\noindent
uniformly in $x\in\LL^0$. Since $\displaystyle{\sum_{p=1}^{P}\frac{e^{\delta\eta p}}{p}\sim \frac{e^{\delta\eta I}}{\delta\eta I}}$ and $\eta$ is arbitrary chosen, 
we obtain when $R\longrightarrow+\infty$
$$N\left(\varphi,x,R\right)\sim \beta\sigma_\oo\left(\varphi\right)h(x)\dfrac{e^{\delta R}}{\delta R}$$
\noindent
uniformly in $x\in\LL^0$.
\end{proof}

Theorem B is a direct consequence of Proposition \ref{propintermediairegeodfermees}; combining 
\eqref{çafinitsuruneligne} and Fatou's lemma yields
$$e^{-\delta\varepsilon_n}\sum\limits_{\g\in\G(n)}\sigma_\oo\left(\un_{\LL_\g}\right)h(x^\g)\leqslant
\liminf\limits_{R\longrightarrow+\infty}\dfrac{\delta R}{\beta e^{\delta R}}\tilde{\mathrm{N}}_{\mathscr{G}}(R).$$
\noindent
Theorem B follows noticing that
$$\lim\limits_{n\longrightarrow+\infty}e^{-\delta\varepsilon_n}\sum\limits_{\g\in\G(n)}\sigma_\oo\left(\un_{\LL_\g}\right)h(x^\g)=\int_{\LL^0}h(x)\dd\sigma_\oo(x)=1.$$
\noindent
\begin{rem}
We also may have an upper bound of $\tilde{\mathrm{N}}_{\mathscr{G}}(R)$ of the form 
$$\sum\limits_{\g\in\G(n)}\sum\limits_{k\geqslant2}\dfrac{1}{k}\sum\limits_{\T^ky=x^\g}\un_{\LL_\g}(y)\un_{[0,R+\varepsilon_n]}(S_k\gol(y))$$
\noindent
as the minor bound in \eqref{çafinitsuruneligne}. Unfortunately, we do not find a domination by an 
integrable function in order to use Proposition \ref{propintermediairegeodfermees}.
\end{rem}
\subsection{Proposition \ref{butdessectionsgeodfermees}}
The proof is inspired from the one of Theorem 1.4 in \cite{Gou}. Let us consider a sequence $(a_k)_{k\geqslant1}$ of non negative real numbers such that $a_k^\beta=kL(a_k)$, where $L$ is the slowly varying function given in assumptions $(H_\beta)$. Let $\varphi\in\mathrm{Lip}\left(\LL\right)$ and $u\ :\ \R\longrightarrow\R$ a function with compact support. For any $k\geqslant1$, denote by 
$$Z_k\left(\varphi,u,x,p\right)=\sum\limits_{\T^ky=x}\varphi(y)e^{-\delta S_k\gol(y)}u\left(S_k\gol(y)-\eta p\right).$$
\noindent
We first admit the two following propositions.
\begin{propB1}
Let $\varphi\in\mathrm{Lip}\left(\LL\right)$ and $u\ :\ \R\longrightarrow\R$ with compact support. For all $\eta>0$, uniformly in $K\geqslant2$, $p\in[0,Ka_k]$
and $x\in\LL^0$, as $k\longrightarrow+\infty$
$$Z_k\left(\varphi,u,x,p\right)=\dfrac{1}{\cg a_k}\left(\Psi_{\beta}\left(\dfrac{\eta p}{\cg a_k}\right)\sigma_\oo\left(\varphi\right)h(x)\widehat{u}(0)+o_k(1)\right),$$
\noindent
where $\Psi_{\beta}$ stands for the density of the fully asymmetric stable law with parameter $\beta$ and $\cg$ equals to $\Cg^{\frac{1}{\beta}}$.
\end{propB1}
\begin{propB2}
Let $\varphi\in\mathrm{Lip}\left(\LL\right)$ and $u\ :\ \R\longrightarrow\R$ with compact support. There exists a constant $C>0$, which depends only on $\eta$ and on the support of $\varphi$ such that, 
when $p\geqslant Ka_k$
$$\left|Z_k\left(\varphi,u,x,p\right)\right|\leqslant Ck\dfrac{L(p)}{p^{1+\beta}}|\varphi|_{\infty}|u|_{\infty}.$$
\end{propB2}
We postpone the proof of these propositions to Subsections 7.2.1 and 7.2.2 and first explain how Proposition \ref{butdessectionsgeodfermees} follows. Let us set for any $x\in\LL$ and $p\geqslant1$
$$D(x;p)=\left|p\sum\limits_{k\geqslant2}\dfrac{1}{k}Z_k\left(\varphi,u,x,p\right)-
\beta\sigma_\oo\left(\varphi\right)h(x)\dfrac{\widehat{u}(0)}{\eta}\right|,$$
\noindent
where $\eta>0$ is fixed as in the proof of Proposition \ref{propintermediairegeodfermees}.
By Proposition B.1, the quantity $D(x;p)$ is bounded from above by $D(x;p)^1+D(x;p)^2+D(x;p)^3+D(x;p)^4$ where
$$D^1(x;p)=\left|p\sum\limits_{k\ |\ \frac{a_k}{K}\leqslant p<Ka_k}\dfrac{1}{\cg ka_k}\Psi_{\beta}\left(\dfrac{\eta p}{\cg a_k}\right)\sigma_\oo\left(\varphi\right)h(x)\widehat{u}(0)-\beta\sigma_\oo\left(\varphi\right)h(x)\dfrac{\widehat{u}(0)}{\eta}\right|,$$
\begin{align*}
& D^2(x;p)=\left|p\sum\limits_{k\ |\ p<\frac{a_k}{K}}\dfrac{1}{\cg ka_k}\Psi_{\beta}\left(\dfrac{\eta p}{\cg a_k}\right)\sigma_\oo\left(\varphi\right)h(x)\widehat{u}(0)\right|,\\
& D^3(x;p)=\left|p\sum\limits_{k\ |\ p<Ka_k}\dfrac{o_k(1)}{\cg ka_k}\right|\\
\text{and}\ & D^4(x;p)=\left|p\sum\limits_{k\ |\ p\geqslant Ka_k}\dfrac{1}{k}Z_k\left(\varphi,u,x,p\right)\right|.
\end{align*}
\begin{itemize}
 \item[a)]{\it Study of $D^1(x;p)$.} There exists $C>0$ which depends only on $\varphi$ and $u$ such that
 $$D^1(x;p)\leqslant C\left|\dfrac{p}{\cg}\sum_{k\ |\ \frac{a_k}{K}\leqslant p<Ka_k}\dfrac{1}{ka_k}\Psi_{\beta}\left(\dfrac{\eta p}{\cg a_k}\right)-\frac{\beta}{\eta}\right|.$$
 \noindent
 Since $k=\frac{a_k^\beta}{L(a_k)}$, it follows 
 \begin{align}\label{etap1geodfermees}
 \sum_{k\ |\ \frac{a_k}{K}\leqslant p<Ka_k}\dfrac{1}{ka_k}\Psi_{\beta}\left(\dfrac{\eta p}{\cg a_k}\right) & =\dfrac{L(p)}{p^{1+\beta}}
 \sum_{k\ |\ \frac{a_k}{K}\leqslant p<Ka_k}\dfrac{L(a_k)}{L(p)}\dfrac{p^{1+\beta}}{a_k^{1+\beta}}\Psi_{\beta}\left(\dfrac{\eta p}{\cg a_k}\right)\nonumber\\
 & \sim \dfrac{L(p)}{p^{1+\beta}}
 \sum_{k\ |\ \frac{a_k}{K}\leqslant p<Ka_k}\dfrac{p^{1+\beta}}{a_k^{1+\beta}}\Psi_{\beta}\left(\dfrac{\eta p}{\cg a_k}\right)
 \end{align}
 \noindent
 uniformly in $p$ such that $\frac{a_k}{K}\leqslant p<Ka_k$. Using the measure
 $\mu_p=\sum_{k}D_{\frac{p}{a_k}}$ defined on the interval $\left[\frac{1}{K},K\right]$, we may rewrite the right member of \eqref{etap1geodfermees} as
 \begin{equation}\label{etap2geodfermees}
 \dfrac{L(p)}{p^{1+\beta}}\sum_{k\ |\ \frac{a_k}{K}\leqslant p<Ka_k}\dfrac{p^{1+\beta}}{a_k^{1+\beta}}\Psi_{\beta}\left(\dfrac{\eta p}{\cg a_k}\right)=
 \dfrac{L(p)}{p^{1+\beta}}\int_{\frac{1}{K}}^Kz^{\beta+1}\Psi_\beta\left(\dfrac{\eta}{\cg}z\right)\dd\mu_p(z)
 \end{equation}
 \noindent
 and from the arguments given in the proof of Theorem A in a) Section 5.2.1, we deduce that 
 the sequence of measures $\left(\frac{L(p)}{p^{1+\beta}}\mu_p\right)_p$ weakly converges on $\left[\frac{1}{K},K\right]$ to the measure $\beta x^{-1-\beta}\dd x$. Combining 
 \eqref{etap1geodfermees} and \eqref{etap2geodfermees}, we get 
 $$D^1(x;p)\leqslant C\beta
 \left|\dfrac{1}{\cg}\left(\int_{\frac{1}{K}}^K\Psi_\beta\left(\dfrac{\eta}{\cg}z\right)\dd z\right)(1+o(1))-\frac{1}{\eta}\right|$$
 \noindent
 where $\lim\limits_{K\longrightarrow+\infty}\lim\limits_{p\longrightarrow+\infty}o(1)=0$. Finally  
 $$D^1(x;p)\leqslant C\dfrac{\beta}{\eta}
 \left|\left(\int_{\frac{\eta}{K\cg}}^{\frac{\eta K}{\cg}}\Psi_\beta\left(z\right)\dd z\right)(1+o(1))-1\right|$$
 \noindent
 which implies when $K\longrightarrow+\infty$, since $\Psi_\beta$ is a probability density, 
 $$\lim\limits_{K\longrightarrow+\infty}\lim\limits_{p\longrightarrow+\infty}\underset{x\in\LL^0}{\sup}\ D^1(x;p)=0.$$
\item[b)]{\it Study of $D^2(x;p)$.} Recall that $k=A(a_k)$ and $A$ is increasing. The inequality $pK<a_k$ thus gives $A(pK)\leqslant k$ and yields
\begin{align*}
D^2(x;p) \leqslant \dfrac{p}{\cg A(pK)}\left|\sum\limits_{k\ |\ p<\frac{a_k}{K}}\dfrac{1}{a_k}\Psi_\beta\left(\dfrac{\eta p}{\cg a_k}\right)\right|
\leqslant C\dfrac{p^{1-\beta}}{K^\beta}L(pK)\left|\sum\limits_{k\ |\ p<\frac{a_k}{K}}\dfrac{1}{a_k}\Psi_\beta\left(\dfrac{\eta p}{\cg a_k}\right)\right|.
\end{align*}
\noindent
From the weak convergence on $\left]0,K\right]$ of the sequence of measures $\left(\frac{L(p)}{p^{1+\beta}}\mu_p\right)_p$ to the measure $\beta x^{-1-\beta}\dd x$, we deduce
\begin{align*}
\dfrac{p^{1-\beta}}{K^\beta}L(pK) \sum\limits_{k\ |\ p<\frac{a_k}{K}}\dfrac{1}{a_k}\Psi_\beta\left(\dfrac{\eta p}{\cg a_k}\right)
& =\dfrac{p^{-\beta}}{K^\beta}L(pK)\int_{0}^{\frac{1}{K}}z\Psi_\beta\left(\frac{\eta}{\cg}z\right)\dd\mu_p(z)\\
& =\dfrac{\beta}{K^\beta}\dfrac{L(pK)}{L(p)}\left(\int_{0}^{\frac{1}{K}}z\Psi_\beta\left(\frac{\eta}{\cg}z\right)\dd z\right)(1+o(1))
\end{align*}
\noindent
where $\lim\limits_{p\longrightarrow+\infty}o(1)=0$;  hence
$$\lim\limits_{K\longrightarrow+\infty}\lim\limits_{p\longrightarrow+\infty}\underset{x\in\LL^0}{\sup}\ D^2(x;p)=0.$$
\item[c)] {\it Study of $D^3(x;p)$.} Fix $\varepsilon>0$ and let $N=N(p)$ be the smallest integer such that $Ka_N>p$. The map $p\longmapsto N(p)$ is increasing. For $p$ large enough and any $k\geqslant N(p)$, one gets $|o_k(1)|\leqslant\varepsilon$. Karamata's lemma \ref{regvar} implies
\begin{align*}
D^3(x;p) \leqslant\dfrac{\varepsilon}{\cg} p\sum\limits_{k\geqslant N}\dfrac{1}{ka_k}\leqslant \dfrac{\varepsilon}{\cg}\dfrac{p}{a_N}.
\end{align*}
\noindent
Noticing that $\frac{1}{a_N}=\frac{a_N^{\beta-1}}{a_N^\beta}$ with 
$a_{N-1}\leqslant\frac{p}{K}$ and $p<Ka_N$ and using asymptotic properties of regularly varying functions, we deduce 
$$\frac{1}{a_N}=\dfrac{a_N^{\beta-1}}{a_{N-1}^{\beta-1}}\dfrac{a_{N-1}^{\beta-1}}{a_N^{\beta}}\preceq \dfrac{p^{\beta-1}}{K^{\beta-1}}\dfrac{K^\beta}{p^\beta}=\dfrac{K}{p}.$$
\noindent
Finally $\underset{x\in\LL^0}{\sup}\ D^3(x;p)\preceq\varepsilon K$ and thus $\underset{x\in\LL^0}{\sup}\ D^3(x;p)=o_K(1)$ where 
$\lim\limits_{p\longrightarrow+\infty}o_K(1)=0$ for any fixed $K$; hence
$\lim\limits_{K\longrightarrow+\infty}\lim\limits_{p\longrightarrow+\infty}\underset{x\in\LL^0}{\sup}\ D^3(x;p)=0.$
\item[d)] {\it Study of $D^4(x;p)$.} By Proposition B.2, we bound $D^4(x;p)$ from above by $Cp^{-\beta}L(p)\sum_{k\ |\ Ka_k\leqslant p}1$ where  
the constant $C$ only depends on $\varphi$ and $u$. Using the same arguments as in b), the inequality $a_k\leqslant\frac{p}{K}$ implies $k\leqslant\left(\frac{p}{K}\right)^\beta L\left(\frac{p}{K}\right)^{-1}$; therefore
$$D^4(x;p)\leqslant CK^{-\beta}\dfrac{L(p)}{L\left(\frac{p}{K}\right)}.$$
\noindent
Potter's lemma with $B=1$, $\rho=\frac{\beta}{2}$, $x=p$ and $y=\frac{p}{K}$ finally implies $D^4(x;p)\leqslant CK^{-\frac{\beta}{2}}$, so that 
$\lim\limits_{K\longrightarrow+\infty}\lim\limits_{p\longrightarrow+\infty}\underset{x\in\LL^0}{\sup}\ D^4(x;p)=0.$
\end{itemize}
\subsubsection{Proof of Proposition B.1}
We follow the proof of Proposition A.1 in paragraph 5.2.2.

Let us fix $p\gg1$ and consider the integers $k$ such that $Ka_k>p$, where $K>2$ is fixed. For any such $k\in\N$, any $\varphi\in\mathrm{Lip}\left(\LL\right)$ and $u\ :\ \R\longrightarrow\R$ continuous with compact support, we write 
$$Z_k\left(\varphi,u,x,p\right)=\sum\limits_{\T^ky=x}\varphi(y)e^{-\delta S_k\gol(y)}u\left(S_k\gol(y)-\eta p\right).$$
\noindent
We want to show that, as $k\longrightarrow+\infty$
\begin{equation}
a_kZ_k\left(\varphi,u,x,p\right)-\dfrac{1}{\cg}\Psi_{\beta}\left(\dfrac{\eta p}{\cg a_k}\right)\sigma_\oo\left(\varphi\right)h(x)\widehat{u}(0)
\longrightarrow0
\end{equation}
\noindent
uniformly in $K$, $p$ and $x\in\LL^0$.
\begin{pro}\label{nouvellepropriété}
The sequence of measures 
$$\left(\varphi\longmapsto a_kZ_k\left(\varphi,\bullet,x,p\right)-\dfrac{1}{\cg}\Psi_{\beta}\left(\dfrac{\eta p}{\cg a_k}\right)\sigma_\oo\left(\varphi\right)h(x)\int_\R\bullet(y)\dd y\right)_{k|Ka_k>p}$$
\noindent
converges weakly to $0$ as $k\longrightarrow+\infty$.
\end{pro}
\noindent
By Stone's argument (see the proof of Proposition A.1), it is sufficient to check that $a_kZ_k\left(\varphi,\bullet,x,p\right)$ is finite and that the convergence of Property \ref{nouvellepropriété} holds.
The Fourier inverse formula furnishes 
$$
Z_k\left(\varphi,u,x,p\right)=\dfrac{1}{2\pi}\int_\R e^{it\eta p}\ot_{\delta+it}^k\varphi(x)\widehat{u}(t)\dd t;
$$
\noindent
we thus deduce 
$\left|Z_k\left(\varphi,u,x,p\right)\right|\preceq||\varphi||_\infty||\widehat{u}||_1<+\infty$. To prove Property \ref{nouvellepropriété}, let us fix $\varepsilon>0$ satisfying the conclusion of 
Proposition \ref{spectreperturbation} and let us decompose
$$
a_kZ_k\left(\varphi,u,x,p\right)-\dfrac{1}{\cg}\Psi_{\beta}\left(\dfrac{\eta p}{\cg a_k}\right)\sigma_\oo\left(\varphi\right)h(x)\widehat{u}(0)
$$
\noindent
as $K_1(k)+K_2(k)$ where 
$$K_1(k):=\dfrac{a_k}{2\pi}\displaystyle{\int\limits_{[-\varepsilon,\varepsilon]^c}}e^{it\eta p}\ot_{\delta+it}^k\varphi(x)\widehat{u}(t)\dd t$$
\noindent
and
\begin{align*}
K_2(k):= & \dfrac{a_k}{2\pi}\displaystyle{\int\limits_{-\varepsilon}^{\varepsilon}}e^{it\eta p}\ot_{\dit}^k\varphi(x)\widehat{u}(t)\dd t-
\dfrac{1}{2\pi}\displaystyle{\int\limits_{\R}}e^{it\frac{\eta p}{a_k}}g_{\beta}(\cg t)\sigma_\oo\left(\varphi\right)h(x)\widehat{u}(0)\dd t\\
= & \dfrac{1}{2\pi}\displaystyle{\int\limits_{-\varepsilon a_k}^{\varepsilon a_k}}e^{is\frac{\eta p}{a_k}}\ot_{\delta+i\frac{s}{a_k}}^k\varphi(x)\widehat{u}\left(\dfrac{s}{a_k}\right)\dd s
-\dfrac{1}{2\pi}\displaystyle{\int\limits_{\R}}e^{it\frac{\eta p}{a_k}}g_{\beta}(\cg t)\sigma_\oo\left(\varphi\right)h(x)\widehat{u}(0)\dd t.
\end{align*}
Proposition \ref{spectreperturbation} combined with the fact that $\widehat{u}$ has a compact support implies the existence of $\rho\in]0,1[$ such that $||\ot_{\delta+it}^k||\preceq\rho^k$ for any $t\in[-\varepsilon,\varepsilon]^c\cap\mathrm{supp}\ \widehat{u}$; hence $\left|K_1(k)\right|\preceq\left|\left|\widehat{u}\right|\right|_{\infty}
\left|\left|\varphi\right|\right|_{\infty}\rho^ka_k\longrightarrow0$ as $k\longrightarrow+\infty$, uniformly in $K$, $p$ and $x\in\LL^0$. 

We now deal with  $K_2(k)$. The spectral decomposition of $\ot_{\delta+i\frac{u}{a_k}}$ furnishes
$$\ot_{\delta+i\frac{s}{a_k}}^k\varphi=\lambda_{\delta+i\frac{s}{a_k}}^k\Pi_{\delta+i\frac{s}{a_k}}\varphi+R_{\delta+i\frac{s}{a_k}}^k\varphi$$ 
\noindent
where the spectral radius of $R_{\delta+i\frac{s}{a_k}}$ is $\leqslant\rho_\varepsilon<1$.
We split $K_2(k)$ into $L_1(k)+L_2(k)+L_3(k)$ where 
$$L_1(k)=\dfrac{1}{2\pi}\int\limits_{-\varepsilon a_k}^{\varepsilon a_k}e^{is\frac{\eta p}{a_k}}R_{\delta+i\frac{s}{a_k}}^k\varphi(x)\widehat{u}\left(\dfrac{s}{a_k}\right)\dd s,$$
\noindent
$$L_2(k)=\dfrac{1}{2\pi}\int\limits_{-\varepsilon a_k}^{\varepsilon a_k}e^{is\frac{\eta p}{a_k}}\lambda_{\delta+i\frac{s}{a_k}}^k\left(\Pi_{\delta+i\frac{s}{a_k}}\varphi(x)-\Pi_{\delta}\varphi(x)\right)\widehat{u}\left(\dfrac{s}{a_k}\right)\dd s$$
\noindent
and
$$L_3(k)=\dfrac{\sigma_\oo\left(\varphi\right)h(x)}{2\pi}\int\limits_{-\varepsilon a_k}^{\varepsilon a_k}e^{is\frac{\eta p}{a_k}}\lambda_{\delta+i\frac{s}{a_k}}^k\widehat{u}\left(\dfrac{s}{a_k}\right)\dd s-\dfrac{\sigma_\oo\left(\varphi\right)h(x)}{2\pi}\int\limits_{\R}e^{it\frac{\eta p}{a_k}}g_{\beta}(\cg t)\widehat{u}(0)\dd t.$$
\noindent
First $\left|L_1(k)\right|\preceq a_k\rho_\varepsilon^k\left|\left|\widehat{u}\right|\right|_{\infty}\left|\left|\varphi\right|\right|$, hence $L_1(k)$ tends to
$0$ uniformly in $K$, $p$  and $x\in\LL^0$ when $k\longrightarrow+\infty$. We need the Lebesgue dominated convergence theorem for $L_2(k)$; as for the quantity $L_2(k)$ appearing in the proof of Proposition A.1, the local expansion of $\lambda_{\delta+it}$ given in Proposition \ref{localexp} implies that the integrand of $L_2(k)$ is
bounded from above up to a multiplicative constant by 
$$l(t)=\left\{\begin{array}{ll}
        &|t|^{\frac{\beta}{2}}e^{-\frac{1}{4}(1-\beta)\G(1-\beta)|\cg t|^{\frac{3\beta}{2}}} \text{if}\ |t|\leqslant1\\
        &|t|^{\frac{3\beta}{2}}e^{-\frac{1}{4}(1-\beta)\G(1-\beta)|\cg t|^{\frac{\beta}{2}}}\ \text{if}\ |t|>1\\
\end{array}\right..$$
\noindent
The term $L_3(k)$ may be treated similarly as the quantity $L_3(k)$ in the proof of Proposition A.1.
\subsubsection{Proof of Proposition B.2}
The quantity $Z_k\left(\varphi,u,x,p\right)$ may be written as 
$$Z_k\left(\varphi,u,x,p\right)=\sum\limits_{\g\in\G(k)}\un_{\LL_{l(\g)}}(x)\varphi(\g.x)e^{-\delta b(\g,x)}u\left(b(\g,x)-\eta p\right).$$
\noindent
Thus there exists a constant $M>0$ such that
$$\left|Z_k\left(\varphi,u,x,p\right)\right|\leqslant\left|\varphi\right|_{\infty}\left|u\right|_{\infty}
\sum\limits_{\underset{b(\g,x)\overset{M}{\sim} \eta p}{\g\in\G(k)}}\un_{\LL_{l(\g)}}(x)e^{-\delta b(\g,x)}.$$
\noindent
It is thus sufficient to prove
$$
\sum\limits_{\underset{b(\g,x)\overset{M}{\sim} \eta p}{\g\in\G(k)}}\un_{\LL_{l(\g)}}(x)e^{-\delta b(\g,x)}\leqslant Ck\dfrac{L(p)}{p^{1+\beta}}
$$
\noindent
where $C$ only depends on $\eta$ and on the support of $\varphi$; this inequality is a consequence of \eqref{butA2}.

\section{Extended coding}
In the sequel, we aim to find an asymptotic for the orbital function $\mathrm{N}_\G(\oo,R)$ of the group $\G$ defined by 
$$\mathrm{N}_\G(\oo,R)=\sharp\{\g\in\G\ |\ \dhy(\oo,\g.\oo)\leqslant R\}.$$
\noindent
The idea is to formalize the fact that $\mathrm{N}(\oo,\G)$ ``behaves'' like a sum of iterates of some extension the transfer operators. Unfortunately, the coding exposed in Section 4 does not take into account the finite words $\g=\alpha_1...\alpha_k$. Adaptating Lalley's approach (in \cite{La}), we first extend this coding to finite sequences $(\alpha_1,...,\alpha_k)$ and then study the corresponding transfer operators: in subsection 8.3, we extend Propositions \ref{spectredeotdelta}, \ref{continuityoftransfert}, \ref{spectreperturbation} and \ref{localexp} to this new coding. As for the mixing, 
these results are essential. Then, the proof of Theorem C follows quite line by line the one of Theorem A using this new coding.
\subsection{Extension of the coding to finite sequences}
In this section, we fix $x_0\in\partial\xx\setminus D$ as in section 4. We denote
$\tilde{\LL}^0$ the set $\tilde{\LL}^0:=\LL^0\cup\G.x_0$ and introduce the following 
symbolic space (called ``extended symbolic space'')
$$\tilde{\Sigma}^+:=\Sigma^+\cup\{\emptyset\}\cup\G^*$$
\noindent
where we denote respectively by $\emptyset$ and $\G^*$ the empty sequence and the set $\G\setminus\{\mathrm{Id}\}$. The set
$\tilde{\Sigma}^+$ is in one-to-one correspondance with the subset $\tilde{\LL}^0$ of $\partial\xx$:
\begin{enumerate}
 \item[-]the point $x_0$ corresponds to the empty sequence;
 \item[-]the point $\alpha_1...\alpha_k.x_0$ corresponds to the admissible finite sequence $(\alpha_1,...,\alpha_k)$;
 \item[-]the infinite sequence $(\alpha_k)_k$ corresponds to the limit point $x:=\lim\limits_{k\longrightarrow+\infty}\alpha_1...\alpha_k.x_0$.
\end{enumerate}
\noindent
Analogously to paragraph 4.1.1, we will use the following description of $\tilde{\LL}^0\setminus\{x_0\}$:
\begin{itemize}
\item[(1)]the set $\tilde{\LL}^0\setminus\{x_0\}$ is the disjoint union of sets $(\tilde{\LL}_j^0)_j$, where
$$
\tilde{\LL}_j^0:=\LL_j^0\cup\left\{g.x_0\ |\ g\ \text{has first letter in}\ \G_j\right\}=\tilde{\LL}^0\cap D_j;
$$
\item[(2)]each subset $\tilde{\LL}_j^0$ is partitioned into a countable number of subsets with disjoint closures: indeed, for any $j\in[\![1,p+q]\!]$
$$\tilde{\LL}_j^0=\bigcup_{\alpha\in\G_j^*}\alpha.\left(\{x_0\}\cup\bigcup_{l\neq j}\tilde{\Lambda}_l^0\right).$$
\end{itemize}
\noindent
Now we extend the cocycle $b(\g,x)=\mathcal{B}_x(\g^{-1}.\oo,\oo)$ on $\tilde{\LL}^0$ in such a way to decompose the distance $\dhy(\oo,\g.\oo)$ for any $\g\in\G$ as a sum of terms expressed with the cocycle. 
For any $\g,g\in\G$, let us set 
$$b^*(\g,g.x_0):=\dhy(\g^{-1}.\oo,g.\oo)-\dhy(\oo,g.\oo).$$
\noindent 
The function $x=g.x_0\in\G.x_0\longmapsto b^*(\g,x)$ satisfies the three following properties: 
\begin{itemize}
\item[i)] when the sequence of points $\left(g.\oo\right)_{g}$ tends to a point $x\in\LL$, so does the sequence $\left(g.x_0\right)_{g}$ and the quantity $\dhy(\g^{-1}.\oo,g.\oo)-\dhy(\oo,g.\oo)$ 
converges to $\mathcal{B}_{x}(\g^{-1}.\oo,\oo)=b(\g,x)$;
\item[ii)]  there exists a constant $C>0$ depending only on $\xx$ and $\G$ such that 
for any $g\in\G$ with $i(g)\neq l(\g)$, one gets 
$$\left|\left(\dhy(\g^{-1}.\oo,g.\oo)-\dhy(\oo,g.\oo)\right)-\dhy(\oo,\g.\oo)\right|\leqslant C,$$ 
\noindent
hence $\left|b^*(\g,g.x_0)-\dhy(\oo,\g.\oo)\right|\leqslant C$;
\item[iii)]$b^*(\g g,x_0)=b^*(\g,g.x_0)+b^*(g,x_0)$ for any $\g,g\in\G$.
\end{itemize}
The function $b^*(\g,\cdot)$ is a cocycle, which extends continuously $b$ to $\tilde{\LL}^0$; in particular, if $\g$ decomposes into $\g=\alpha_1...\alpha_k$, then
\begin{equation}\label{decompo1}\dhy(\oo,\g.\oo)=b^*(\alpha_1,\alpha_2...\alpha_k.x_0)+b^*(\alpha_2,\alpha_3...\alpha_k.x_0)+...+b^*(\alpha_k,x_0).\end{equation}
\noindent
In the sequel, we thus consider the following ``extended cocycle'', also denoted by $b^*$ and defined by: for any $\g\in\G$ and $x\in\tilde{\LL}^0$
\begin{equation}\label{definouveaucocycle}
b^*(\g,x) = \left\{\begin{array}{ll}
                   \mathcal{B}_x(\g^{-1}.\oo,\oo)\ \text{if}\ x\in\LL^0\\
                   \dhy(\g^{-1}.\oo,g.\oo)-\dhy(\oo,g.\oo)\ \text{if}\ x=g.x_0
                   \end{array}\right..
\end{equation}
\noindent
The map $\T$ previously defined on $\LL^0$, may be naturally extended to $\G.x_0$ as follows: 
\begin{itemize}
 \item[-]we use the convention $\T.x_0=x_0$;
 \item[-]we set $\T.(\g.x_0):=\alpha_1^{-1}.(\g.x_0)$ for any $\g\in\G$ with first letter $\alpha_1\in\mathcal{A}$.
\end{itemize}
\noindent
We can similarly extend the roof function $\gol$ to the discrete orbit $\G.x_0$ setting 
\begin{itemize}
 \item[-]$\gol(x_0)=0$;
 \item[-]$\gol(\g.x_0)=b^*(\alpha_1,\alpha_2...\alpha_k.x_0)$ for any $\g=\alpha_1...\alpha_k$ in 
$\G^*$.
\end{itemize}
\noindent
Then for $y=\g.x_0$ the relation \eqref{decompo1} may be rewritten as
\begin{equation}\label{decompo2}\dhy(\oo,\g.\oo)=\gol(y)+\gol(\T.y)+...+\gol(\T^{k-1}.y)=S_{k}\gol(y).\end{equation}
%
%
\subsection{Regularity of the extended cocycle}
To simplify the notations, from now on we denote by
$\tilde{\LL}=\overline{\tilde{\LL}^0}=\LL\cup\G.x_0$ and $\tilde{\LL}_j:=\overline{\tilde{\LL}_j^0}=
\tilde{\LL}\cap D_j$ for any $j\in[\![1,p+q]\!]$. In this subsection, we aim to show the following
\begin{prop}\label{contfinal}
There exists a constant $C=C(x_0)>0$ such that for any $j\in[\![1,p+q]\!]$, $\g\in\G$ with $l(\g)=j$ and $x,y\in\tilde{\LL}_l$, $l\neq j$,
$$\left|b^*(\g,x)-b^*(\g,y)\right|\leqslant C\dhy_\oo(x,y).$$
\end{prop}
The proof is long and technical. It is sufficient to prove this result for $\g\in\G$ with symbolic length $1$. Indeed, assume that $\g=\alpha_1...\alpha_k$ for 
$k\geqslant2$ and denote by $\g^{(j)}=\alpha_j...\alpha_k$ for any $j\in[\![2,k-1]\!]$. Using the cocycle property of $b^*$, we obtain for any $x,y\in\tilde{\LL}\setminus\tilde{\LL}_{j}$
\begin{align*}
\left|b^*(\g,x)-b^*(\g,y)\right| \leqslant &\left|b^*(\alpha_1,\g^{(2)}.x)-b^*(\alpha_1,\g^{(2)}.y)\right|+\left|b^*(\g^{(2)},x)-b^*(\g^{(2)},y)\right|\\
\leqslant & \left|b^*(\alpha_1,\g^{(2)}.x)-b^*(\alpha_1,\g^{(2)}.y)\right|+\left|b^*(\alpha_2,\g^{(3)}.x)-b^*(\alpha_2,\g^{(3)}.y)\right|\\
&+... +\left|b^*(\alpha_k,x)-b^*(\alpha_k,y)\right|.
\end{align*}
\noindent
If Proposition \ref{contfinal} holds for element of $\G$ with symbolic length $1$, we get 
$$\left|b^*(\g,x)-b^*(\g,y)\right|\leqslant C(x_0)\left(\dhy_\oo(\g^{(2)}.x,\g^{(2)}.y)+...+\dhy_\oo(x,y)\right),$$
\noindent
where $C(x_0)$ depends only on $x_0$. By Corollary \ref{actioncontractante}, there exist $r\in]0,1[$ and $C>0$ such that 
$\dhy_\oo(\g^{(2)}.x,\g^{(2)}.y)\leqslant C.r^{k-1}\dhy_{\oo}(x,y),...,\dhy_\oo(\alpha_k.x,\alpha_k.y)\leqslant C.r\dhy_{\oo}(x,y)
$, so that
$$\left|b^*(\g,x)-b^*(\g,y)\right|\leqslant CC(x_0)\dfrac{1}{1-r}\dhy_\oo(x,y),$$
\noindent
which proves that the inequality is still valid when $|\g|\geqslant2$.

Let us fix for this subsection $\alpha\in\G_j^*$, $j\in[\![1,p+q]\!]$, and $x,y\in\tilde{\LL}_l$ for some 
$l\neq j$. There are three cases to consider:
\begin{itemize}
 \item[(a)]the points $x$ and $y$ both belong to $\LL$;
 \item[(b)]$x\in\LL$ and $y\in\G.x_0$;
 \item[(c)]the points $x$ and $y$ belong to $\G.x_0$.
\end{itemize}
The different cases are treated in the next three subsubsections.
\subsubsection{{\bf Case (a): the points ${\bf x}$ and ${\bf y}$ both belong to $\boldsymbol{\LL}$.}} The statement follows from Proposition \ref{equilipcocycle}.
\subsubsection{{\bf Case (b): ${\bf x}\boldsymbol{\in\LL}$ and ${\bf y}\boldsymbol{\in\G.x_0}$.}} Set 
$y=g.x_0$ for some $g\in\G$ with $i(g)=l$. The statement may be thus reformulated as follows\\

\noindent {\bf Proposition \ref{contfinal} in case (b).}$\ ${\it There exists a constant $C=C(x_0)>0$ such that for any $j\in[\![1,p+q]\!]$, $\alpha\in\G_j^*$, $g\in\G$ with $i(g)\neq j$ and $x\in\LL_{i(g)}$, the following inequality holds 
$$\left|\mathcal{B}_x(\alpha^{-1}.\oo,\oo)-b^*(\alpha,g.x_0)\right|\leqslant C\dhy_\oo(x,g.x_0).$$}
\begin{proof}[Proof of the case (b).]
The proof is splitted into two steps. In the first step {\bf (1)}, we assume that 
$x\in g.\left(\LL\setminus\LL_{l(g)}\right)$ and then we prove the case {\bf (b)} without this additionnal assumption in step {\bf (2)}. 
\begin{itemize}
 \item[{\bf (1)}]Denote by $V(x,\oo,t)$ the subset of points $y\in\partial\xx$ whose projection ${\bf \tilde{y}}$ on the geodesic ray  $[\oo x)$ satisfies $\dhy({\bf \tilde{y}},\oo)\geqslant t$; the $\dhy_\oo$-diameter 
of such a set is $\preceq e^{- at}$. The set $\tilde{V}(x,\oo,t)$ stands for a connected and geodesically convex subset of $\overline{\xx}$ whose
intersection with $\partial\xx$ equals to $V(x,\oo,t)$. Recall that the space $\xx$ is a 
Gromov-$\kappa$-hyperbolic space for some $\kappa=\kappa(a)>0$ (see \cite{G-H}). The following properties are proved in \cite{Sch}.
\begin{prop}
Let $\kappa>0$ such that  $\xx$ is a Gromov-$\kappa$-hyperbolic space. 
\begin{enumerate}\label{inclusion}
 \item\label{inclusion1} Let ${\bf p}\in\xx$, $x\in\partial{\xx}$ and $t\geqslant1$. For any $y\in V(x,{\bf p},t+7\kappa)$
$$V(x,{\bf p},t+6\kappa)\subset V(y,{\bf p},t)\subset V(x,{\bf p},t-6\kappa).$$
\item\label{inclusion2} For any $D>0$, denote by $K_2=2D+4\kappa$. Let ${\bf p}$ and ${\bf q}$ be in $\xx$ such that $\dhy({\bf p},{\bf q})\leqslant D$, $x\in\partial\xx$ and $t\geqslant K_2$. Hence
$$V(x,{\bf p},t+K_2)\subset V(x,{\bf q},t)\subset V(x,{\bf p},t-K_2).$$
\end{enumerate}
\end{prop}
Let $y\in\LL\setminus\LL_{l(g)}$ such that $x=g.y$. The conformality 
equation \eqref{mvr} implies
$$\mathrm{d}_{\oo}(x,g.x_0)=\mathrm{d}_{\oo}(g.y,g.x_0)=\sqrt{|g'(y)|_{\oo}|g'(x_0)|_{\oo}}\mathrm{d}_{\oo}(y,x_0).$$
\noindent
Since $y\in\LL\setminus D_{l(g)}$, Corollary \ref{busedist} implies
$$\mathrm{d}_{\oo}(x,g.x_0)\asymp e^{-a\dhy(\oo,g.\oo)}\mathrm{d}_{\oo}(y,x_0)\asymp e^{-a\dhy(\oo,g.\oo)}.$$
\noindent
We now estimate $|\mathcal{B}_x(\alpha^{-1}.\oo,\oo)-b^*(\alpha,g.x_0)|$. Denote by $\xi'$ (respectively $\xi''$) 
the endpoint of the geodesic ray starting from $\oo$ (resp. from $\alpha^{-1}.\oo$) and passing through $g.\oo$. From the definition of 
Busemann functions (see figure \ref{image1cocycle}), we derive
\begin{equation}\label{etoile}
\mathcal{B}_{\xi'}(\alpha^{-1}.\oo,\oo)\leqslant\dhy(\alpha^{-1}.\oo,g.\oo)-\dhy(\oo,g.\oo)\leqslant\mathcal{B}_{\xi''}(\alpha^{-1}.\oo,\oo).
\end{equation}
\begin{figure}[htpb]
\begin{center}
\includegraphics[height=3.2cm]{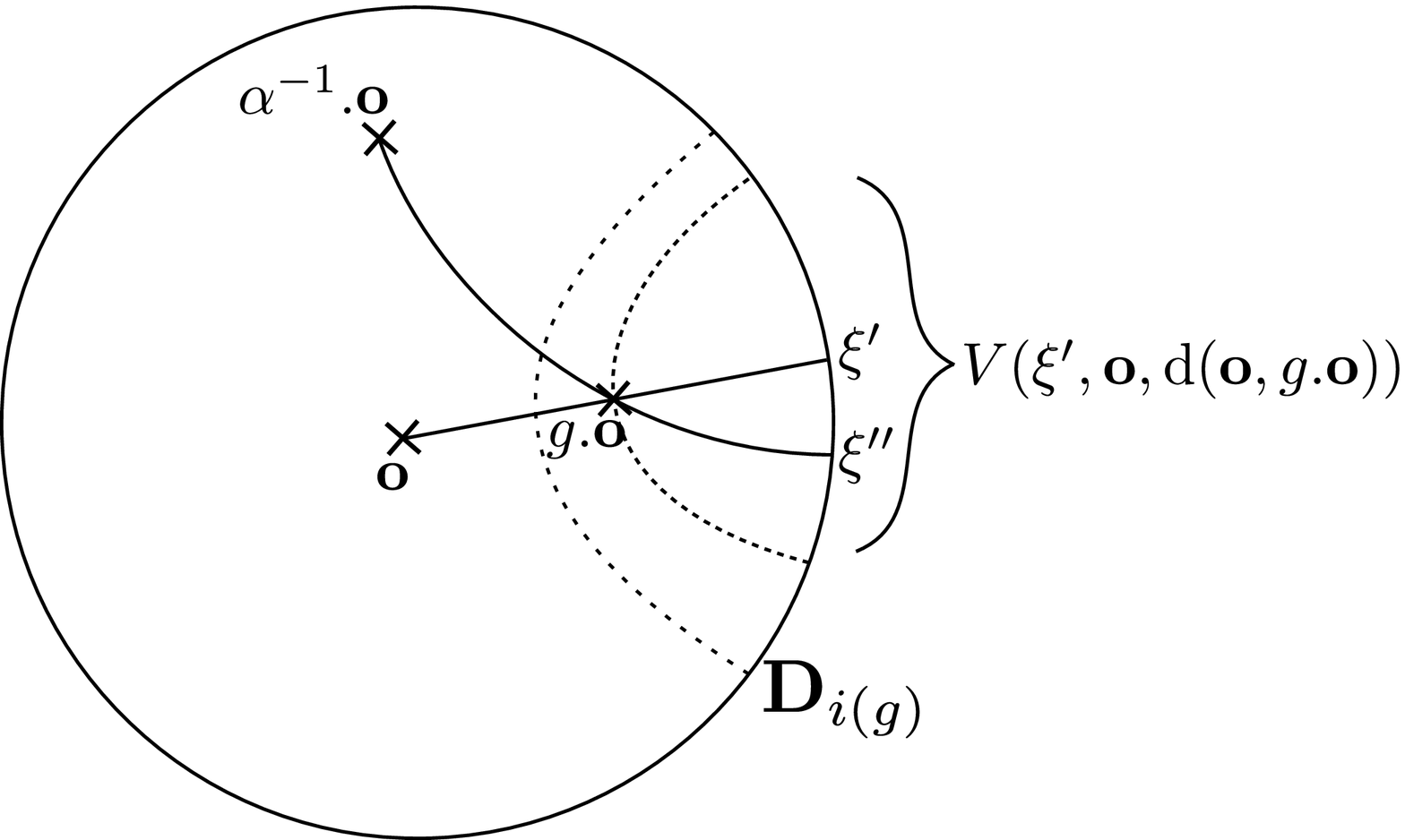}
\end{center}
\caption{\label{image1cocycle}Estimate of $\dhy(\alpha^{-1}.\oo,g.\oo)-\dhy(\oo,g.\oo)$}
\end{figure}
%
The points $\xi'$ and $\xi''$ belong to the path-connected set $V(\xi',\oo,\dhy(\oo,g.\oo))$ which is path-connected. From Proposition \ref{equilipcocycle} and \eqref{etoile}, we deduce 
the existence of $\xi\in D_{i(g)}$ such that $b^*(\alpha,g.x_0)=\mathcal{B}_{\xi}(\alpha^{-1}.\oo,\oo)$. Proposition \ref{equilipcocycle} also implies 
$$|\mathcal{B}_x(\alpha^{-1}.\oo,\oo)-b^*(\alpha,g.x_0)|=
|\mathcal{B}_x(\alpha^{-1}.\oo,\oo)-\mathcal{B}_{\xi}(\alpha^{-1}.\oo,\oo)|\leqslant \mathrm{C}\mathrm{d}_{\oo}(x,\xi)$$
\noindent
for some constant $C>0$. Since $\xi\in V(\xi',\oo,\dhy(\oo,g.\oo))$, we derive $\dhy_\oo(\xi',\xi)\preceq e^{-a\dhy(\oo,g.\oo)}$. In order to obtain a similar conclusion for $\dhy_\oo(x,\xi')$, we have two different cases to consider.
\begin{itemize}
\item[(i)]{\it Assume first that $i(g)\neq l(g)$.} In this case, the isometry $g$ is hyperbolic with attractive (resp. repulsive) 
fixed point $x_g^+$ (resp. $x_g^-$); notice that $x_g^+,y\in\LL\setminus\LL_{l(g)}$, while 
$x_g^-\in\LL_{l(g)}$. There thus exist $E>0$ and a point
$\tilde{\oo}\in(x_{g}^-x_{g}^+)$ such that $\dhy(\oo,\tilde{\oo})\leqslant E$ and the projection of $y$ on the axis $(x_{g}^-x_{g}^+)$ belongs to 
the geodesic ray $[\tilde{\oo}x_{g}^+)$ (see figure \ref{cocyclelip2figure}). 
\begin{figure}[htpb]
\begin{center}
\includegraphics[height=5cm]{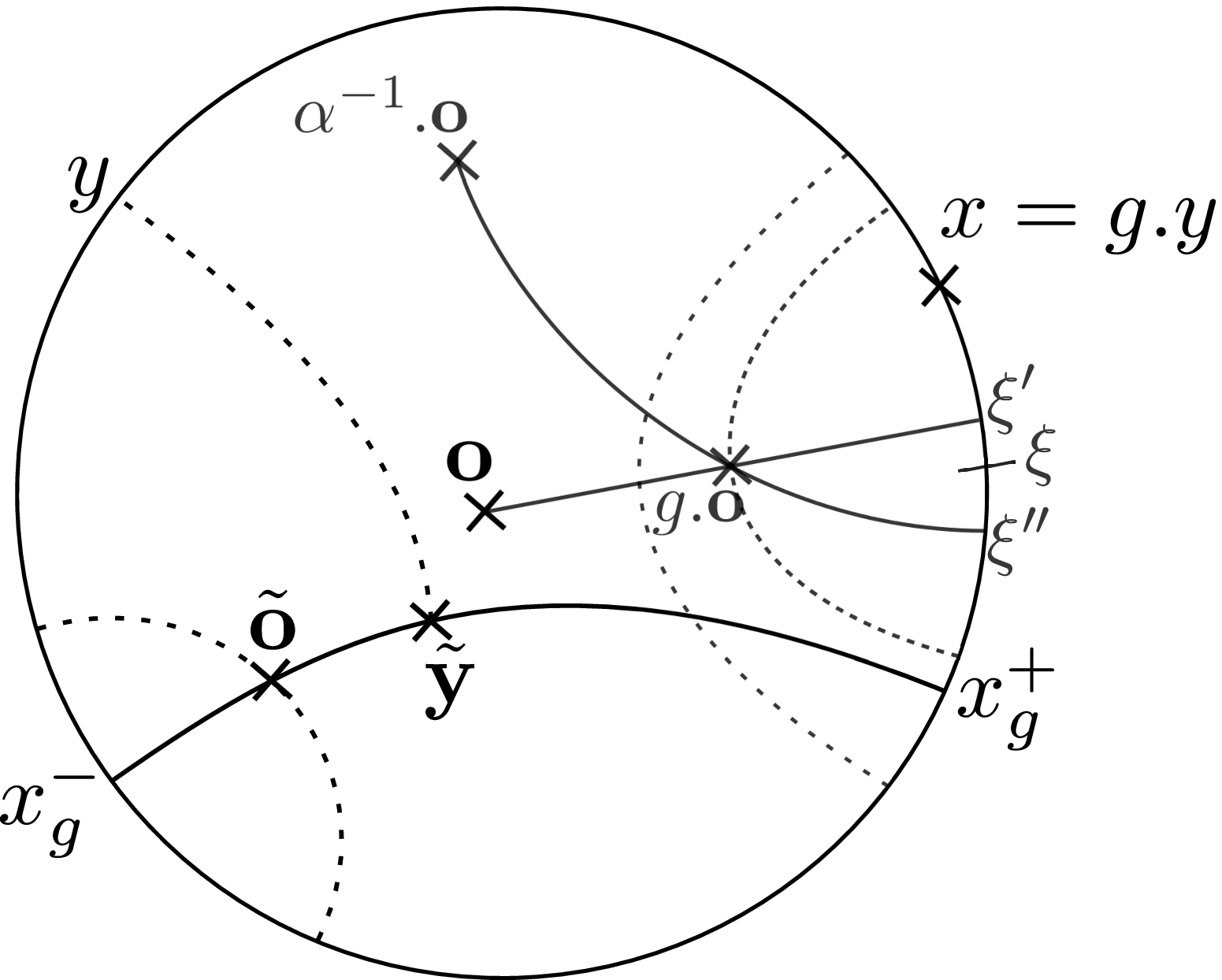}
\end{center}
\caption{\label{cocyclelip2figure}Points $\tilde{\oo}$ and $\tilde{y}$}
\end{figure}
Denote by $L=\dhy(\tilde{\oo},g.\tilde{\oo})$ the length of 
the axis of $g$. Proposition \ref{inclusion}.\ref{inclusion2} with ${\bf q}=\tilde{\oo}$, ${\bf p}=\oo$, $D=E$ and $t=L-E-3\kappa$ yields
\begin{equation}\label{premiereinclusion}
\quad\quad\quad V(x_{g}^+,{\bf \tilde{y}},L-E-3\kappa)\subset V(x_{g}^+,\tilde{\oo},L-E-3\kappa)\subset V(x_{g}^+,\oo,L-3E-7\kappa).
\end{equation}
\noindent
It follows from $g.\tilde{\oo}\in B(g.\oo,E)$ and $B(g.\oo,E)\subset\tilde{V}(\xi',\oo,\dhy(\oo,g.\oo)-E)$, that 
$[g.\tilde{\oo},x_{g}^+)\subset\tilde{V}(\xi',\oo,\dhy(\oo,g.\oo)-E)$, hence $x_g^+\in V(\xi',\oo,\dhy(\oo,g.\oo)-E)$. The triangular inequality implies $L-2E\leqslant \dhy(\oo,g.\oo)$, so that $x_g^+\in V(\xi',\oo,L-3E)$. By Proposition \ref{inclusion}.\ref{inclusion1} with $t=L-3E-7\kappa$, we get
\begin{equation}\label{deuxièmeinclusion}
V(x_{g}^+,\oo,L-3E-7\kappa)\subset V(\xi',\oo,L-3E-13\kappa),
\end{equation}
\noindent
so that combining \eqref{premiereinclusion} and \eqref{deuxièmeinclusion}
$$V(x_{g}^+,{\bf \tilde{y}},L-E)\subset V(\xi',\oo,L-3E-13\kappa).$$
\noindent
Since $x\in V(x_{g}^+,{\bf \tilde{y}},L-E)$, we have $x\in V(\xi',\oo,L-3E-13\kappa)$; moreover, 
the triangular inequality implies $\dhy(\oo,g.\oo)-2E\leqslant L$, hence $x\in V(\xi',\oo,\dhy(\oo,g.\oo)-5E-13\kappa)$. Finally $\dhy_{\oo}(x,\xi')\leqslant K_{\G}'e^{-a\dhy(\oo,g.\oo)}$, 
so that $\dhy_{\oo}(x,\xi)\preceq e^{-a\dhy(\oo,g.\oo)}$: this achieves the proof of part {\bf (1)} when $i(g)\neq l(g)$.
\item[(ii)]{\it When $i(g)=l(g)$:} one applies the previous arguments with $\alpha g$ instead of $g$. We thus obtain $\dhy_\oo(\alpha.x,\alpha.\xi')\preceq e^{-a\dhy(\oo,\alpha.g.\oo)}$. Using \eqref{mvr} and Lemma 
\ref{quasiegalitetriangulaire}, we finally obtained 
$\dhy_\oo(x,\xi')\preceq e^{-a\dhy(\oo,g.\oo)}$. 
\end{itemize}
%
%
\item[{\bf (2)}]We now prove the case {\bf (b)} without additionnal assumption on the position of $x$ in $\LL_{i(g)}$. 
\begin{fait}\label{faitintermediaire}
There exists a constant $C=C(x_0)>0$ such that, for any $g\in\G$ and $x\in\LL_{i(g)}$, there exists $z\in g.\left(\LL\setminus\LL_{l(g)}\right)$ such that $\dhy_\oo(z,x)\leqslant C\dhy_\oo(x,g.x_0)$ and 
$\dhy_\oo(z,g.x_0)\leqslant C\dhy_\oo(x,g.x_0)$.
\end{fait}
\begin{proof}[Proof of Fact \ref{faitintermediaire}.]
Since $x\in\LL_{i(g)}$, there exists $g'\in\G$ with $i(g')=i(g)$ and $x'\in\LL$ such that $x=g'.x'$. Let $k$ be the first index 
$\leqslant\min\left(|g'|,|g|\right)$ for which the $k$-th letters of $g'$ and $g$ are different. There are two cases to consider:
\begin{itemize}
\item[(i)]{\it both $k$-th letters of $g$ and $g'$ do not belong to the same Schottky factor $\G_j,1\leqslant j\leqslant p+q$}; in this case, we may write $g=\alpha_1...\alpha_{k-1}\alpha_k.g_1$ and $g'=\alpha_1...\alpha_{k-1}\alpha_k'.g_1'$ where $\alpha_k$ and $\alpha_k'$ do not belong to the same factor. Fix $w\in\LL\setminus\LL_{l(g)}$ and a point $u$ in the boundary $\partial_{g_1}^{x_0}$ of the connected component of $D_{i(g_1)}$  containing $g_1.x_0$ (see figure \ref{faitchiantpicture}). It follows 
$$\quad\quad\quad\quad\dhy_{\oo}(g.x_0,\alpha_1...\alpha_k.u)=e^{-\frac{a}{2}\mathcal{B}_{g_1.x_0}\left(\alpha_k^{-1}...\alpha_1^{-1}.\oo,\oo\right)}
e^{-\frac{a}{2}\mathcal{B}_{u}\left(\alpha_k^{-1}...\alpha_1^{-1}.\oo,\oo\right)}\dhy_\oo(g_1.x_0,u).$$
\noindent
Setting
$m=\underset{j\in[\![1,p+q]\!]}{\min}\underset{\underset{i(\g)=j}{\g\in\G}}{\min}\underset{u\in\partial_{\g}^{x_0}}{\min}\dhy_\oo(\g.x_0,u)>0$, 
we obtain
$$me^{-\frac{a}{2}\mathcal{B}_{g_1.x_0}\left(\alpha_k^{-1}...\alpha_1^{-1}.\oo,\oo\right)}
e^{-\frac{a}{2}\mathcal{B}_{u}\left(\alpha_k^{-1}...\alpha_1^{-1}.\oo,\oo\right)}\leqslant\dhy_{\oo}(g.x_0,\alpha_1...\alpha_k.u).$$
\noindent
There also exists a constant $M>0$ (which actually is the $\dhy_\oo$-diameter of $\partial\xx$) such that
$$\dhy_{\oo}(g.x_0,g.w)\leqslant M
e^{-\frac{a}{2}\mathcal{B}_{g_1.x_0}\left(\alpha_k^{-1}...\alpha_1^{-1}.\oo,\oo\right)}
e^{-\frac{a}{2}\mathcal{B}_{g_1.w}\left(\alpha_k^{-1}...\alpha_1^{-1}.\oo,\oo\right)}.$$
\noindent
This yields 
$$\quad\dfrac{\dhy_{\oo}(g.x_0,g.w)}{\dhy_{\oo}(g.x_0,\alpha_1...\alpha_k.u)}\leqslant\dfrac{M}{m}e^{-\frac{a}{2}\left(\mathcal{B}_{g_1.w}\left(\alpha_k^{-1}...\alpha_1^{-1}.\oo,\oo\right)-\mathcal{B}_{u}\left(\alpha_k^{-1}...\alpha_1^{-1}.\oo,\oo\right)\right)}\preceq1,$$
\noindent
where the last estimate is uniform in $u\in\partial_{g_1}^{x_0}$ and follows from 
Corollary \ref{busedist}. This upper bound being true for any $u\in\partial_{g_1}^{x_0}$, we deduce from the compactness of $\partial_{g_1}^{x_0}$
that $\dhy_{\oo}(g.x_0,g.w)\leqslant C'(x_0)\dhy_{\oo}(g.x_0,\alpha_1...\alpha_k.\partial_{g_1}^{x_0})$. Since  
$x\notin\alpha_1...\alpha_{k}.D_{i(g_1)}$, we obtain $\dhy_{\oo}(g.x_0,g.w)\leqslant C'(x_0)\dhy_{\oo}(x,g.x_0)$. The triangular 
inequality yields $\dhy_\oo(g.w,x)\leqslant\dhy_\oo(x,g.x_0)+\dhy(g.x_0,g.w)$ and the result follows with $z=g.w$ and $C=1+C'(x_0)$.
\begin{figure}[htpb]
\begin{center}
\includegraphics[height=4.5cm]{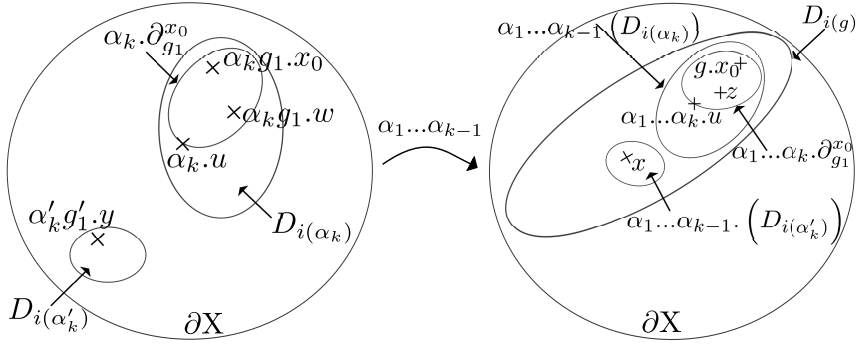}
\end{center}
\caption{\label{faitchiantpicture}Action of $\alpha_1...\alpha_{k-1}$}
\end{figure}

\item[(ii)]{\it both $k$-th letters belong to the same Schottky factor $\G_j,1\leqslant j\leqslant p+q$}: there thus exist $\beta\in\mathcal{A}$ and $n>n'\in\N^*$ such that 
$g=\alpha_1...\alpha_{k-1}.\beta^n.g_1$ and $g'=\alpha_1...\alpha_{k-1}.\beta^{n'}.g_1'$. Assume that $\beta$ generates $\G_l$, $l\in[\![1,p+q]\!]$. 
Fix $u\in\partial_{g_1}^{x_0}$ and $w\in\LL\setminus\LL_{l(g)}$. Similarly as in a), we get
$$\quad\quad\quad\quad\quad me^{-\frac{a}{2}\mathcal{B}_{g_1.x_0}\left(\beta^{-n}\alpha_{k-1}^{-1}...\alpha_1^{-1}.\oo,\oo\right)}
e^{-\frac{a}{2}\mathcal{B}_{u}\left(\beta^{-n}\alpha_{k-1}^{-1}...\alpha_1^{-1}.\oo,\oo\right)}\leqslant\dhy_{\oo}(g.x_0,\alpha_1...\alpha_{k-1}\beta^n.u)$$
\noindent
and
$$\quad\quad\dhy_{\oo}(g.x_0,g.w)\leqslant M
e^{-\frac{a}{2}\mathcal{B}_{g_1.x_0}\left(\beta^{-n}\alpha_{k-1}^{-1}...\alpha_1^{-1}.\oo,\oo\right)}
e^{-\frac{a}{2}\mathcal{B}_{g_1.w}\left(\beta^{-n}\alpha_{k-1}^{-1}...\alpha_1^{-1}.\oo,\oo\right)}.$$
\noindent
Using Corollary \ref{busedist} and noticing that the previous estimates are satisfied for any $u\in\partial_{g_1}^{x_0}$, we deduce the existence of 
a constant $C'(x_0)>0$ such that 
$\dhy_{\oo}(g.x_0,g.w)\leqslant C'(x_0)\dhy_{\oo}(g.x_0,\alpha_1...\alpha_{k-1}\beta^n.\partial_{g_1}^{x_0})$. Since
$x\in\alpha_1...\alpha_{k-1}\beta^{n'}.D_{i(g_1')}$ and 
$\alpha_1...\alpha_{k-1}\beta^{n'}.D_{i(g_1')}\cap\alpha_1...\alpha_{k-1}\beta^n.\partial_{g_1}^{x_0}=\emptyset$, we deduce 
$\dhy_\oo(g.w,g.x_0)\leqslant C'(x_0)\dhy_{\oo}(x,g.x_0)$. We set $z=g.w$. This achieves the proof of 
Fact \ref{faitintermediaire}.
\end{itemize}
\end{proof}
\noindent
Let us now come back to the proof of the part {\bf (2)} of case {\bf (b)}. Fix $x\in\LL_{i(g)}$ and a point $z\in g.\left(\LL\setminus\LL_{l(g)}\right)$ satisfying the conclusions of Fact \ref{faitintermediaire}. We bound $\left|\mathcal{B}_x(\alpha^{-1}.\oo,\oo)-b^*(\alpha,g.x_0)\right|$ from above by 
\begin{equation}\label{bonpointprésentation}\left|\mathcal{B}_x(\alpha^{-1}.\oo,\oo)-\mathcal{B}_z(\alpha^{-1}.\oo,\oo)\right|
+\left|\mathcal{B}_z(\alpha^{-1}.\oo,\oo)-b^*(\alpha,g.x_0)\right|.\end{equation}
\noindent
Proposition \ref{equilipcocycle} and part {\bf (1)} of case {\bf (b)} applied at $z$ thus imply that there exists $C=C(x_0)>0$ such that the bound \eqref{bonpointprésentation} is smaller than 
$C\dhy_\oo(x,g.x_0)$, which concludes the proof of case {\bf (b)}.
\end{itemize}
\end{proof}
\subsubsection{{\bf Case (c): the points ${\bf x}$ and ${\bf y}$ belong to $\boldsymbol{\G.x_0}$.}} Fix $g_1,g_2\in\G$ such that $i(g_1)=i(g_2)=l$, $x=g_1.x_0$ and $y=g_2.x_0$. Set $g_1=g.\beta_1$ and $g_2=g.\beta_2$. 
\begin{itemize}
\item[{\bf (1)}] Assume first that $i(\beta_1)\neq i(\beta_2)$. Let
$z_1\in\LL\setminus\LL_{l(g_1)}$ and $z_2\in\LL\setminus\LL_{l(g_2)}$. From Proposition \ref{equilipcocycle} and from case {\bf (b)} above, we deduce 
that there exists a constant $C'=C'(x_0)>0$ such that
\begin{align*}
\left|b^*(\alpha,g_1.x_0)-b^*(\alpha,g_2.x_0)\right| \leqslant & \left|b^*(\alpha,g_1.x_0)-\bcal_{g_1.z_1}(\alpha^{-1}.\oo,\oo)\right|\\
 & +\left|\bcal_{g_1.z_1}(\alpha^{-1}.\oo,\oo)-\bcal_{g_2.z_2}(\alpha^{-1}.\oo,\oo)\right|\\
& + \left|\bcal_{g_2.z_2}(\alpha^{-1}.\oo,\oo)-b^*(\alpha,g_2.x_0)\right|\\
\leqslant & C'\left(\dhy_{\oo}(g_1.x_0,g_1.z)+\dhy_{\oo}(g_1.z_1,g_2.z_2)\right.\\
& \left.+\dhy_{\oo}(g_2.z_2,g_2.x_0)\right).
\end{align*}
\noindent
Combining Property \eqref{mvr} and Corollary \ref{busedist}, it follows that
\begin{align*}
\left|b^*(\alpha,g_1.x_0)-b^*(\alpha,g_2.x_0)\right| \preceq & e^{-a\dhy(\oo,g_1.\oo)}\dhy_{\oo}(x_0,z_1)+e^{-a\dhy(\oo,\gamma.\oo)}\dhy_{\oo}(\beta_1.z_1,\beta_2.z_2)\\
& + e^{-a\dhy(\oo,g_2.\oo)}\dhy_{\oo}(z_2,x_0) \\
\preceq & \left(e^{-a\dhy(\oo,g_1.\oo)}+e^{-a\dhy(\oo,\g.\oo)}+e^{-a\dhy(\oo,g_2.\oo)}\right).
\end{align*}
\noindent
Lemma \ref{quasiegalitetriangulaire} implies that, for $i\in\{1,2\}$,
$$\dhy(\oo,g.\oo)+\dhy(\oo,\beta_i.\oo)-C\leqslant\dhy(\oo,g_i.\oo)\leqslant\dhy(\oo,g.\oo)+\dhy(\oo,\beta_i.\oo)+C,$$ 
\noindent 
hence
$$\left|b^*(\alpha,g_1.x_0)-b^*(\alpha,g_2.x_0)\right|\preceq e^{-a\dhy(\oo,g.\oo)}.$$ 
\noindent
Since $\min_{(x,y)\in\bigcup\limits_{l\neq j}\tilde{\LL}_l\times\tilde{\LL}_j}\dhy_{\oo}(x,y)>0$, this 
yields
$$\left|b^*(\alpha,g_1.x_0)-b^*(\alpha,g_2.x_0)\right|\preceq e^{-a\dhy(\oo,g.\oo)}\dhy_{\oo}(\beta_1.x_0,\beta_2.x_0).$$
\noindent
We conclude the part {\bf (1)} of case {\bf (c)} using \eqref{mvr} again. 
\item[{\bf (2)}]If $i(\beta_1)=i(\beta_2))=l\in[\![1,p+q]\!]$: there exist $\beta\in\G_l^*$, $n_1,n_2\in\Z^*,\ n_1\neq n_2$ and $\beta_{1,1},\beta_{2,1}\in\G$ such that $\beta_1=\beta^{n_1}\beta_{1,1}$ and $\beta_2=\beta^{n_2}\beta_{2,1}$. 
We need the following fact.
\begin{fait}\label{fact7}
There exists $z\in\LL_l$ such that $\dhy_\oo(\beta_1.x_0,z)\leqslant C\dhy(\beta_1.x_0,\beta_2.x_0)$ and $\dhy_\oo(\beta_2.x_0,z)\leqslant C\dhy(\beta_1.x_0,\beta_2.x_0)$ for a constant $C=C(x_0)>0$.
\end{fait}
\begin{proof}[Proof of Fact \ref{fact7}.]
Fix $w\in\LL_{i(\beta_{2,1})}$ and $u\in\partial_{\beta_{2,1}}^{x_0}$. As in the proof of Fact \ref{faitintermediaire}, we may write 
$$me^{-\frac{a}{2}\mathcal{B}_{\beta_{2,1}.x_0}\left(\beta^{-n_2}.\oo,\oo\right)}
e^{-\frac{a}{2}\mathcal{B}_{u}\left(\beta^{-n_2}.\oo,\oo\right)}\leqslant\dhy_{\oo}(\beta^{n_2}\beta_{2,1}.x_0,\beta^{n_2}.u)$$
\noindent
and
$$\dhy_{\oo}(\beta^{n_2}\beta_{2,1}.x_0,\beta^{n_2}.w)\leqslant M
e^{-\frac{a}{2}\mathcal{B}_{\beta_{2,1}.x_0}\left(\beta^{-n_2}.\oo,\oo\right)}
e^{-\frac{a}{2}\mathcal{B}_{w}\left(\beta^{-n_2}.\oo,\oo\right)}.$$
\noindent
Setting $z=\beta^{n_2}.w$, Corollary \ref{busedist} and the previous estimates yield
$$\dhy_\oo(\beta^{n_2}.\beta_{2,1}.x_0,z)\leqslant C'\dhy_{\oo}(\beta^{n_2}\beta_{2,1}.x_0,\beta^{n_2}.\partial_{\beta_{2,1}}^{x_0})$$
\noindent
for $C'=C'(x_0)>0$. But $\beta^{n_1}\beta_{1,1}.x_0\in\beta^{n_1}.\left(\partial\xx\setminus D_l\right)$ and 
$\beta^{n_1}.\left(\partial\xx\setminus D_l\right)\cap\beta^{n_2}.\partial_{\beta_{2,1}}^{x_0}=\emptyset$, therefore
$$\dhy_\oo(\beta^{n_2}\beta_{2,1}.x_0,z)\leqslant C'\dhy_{\oo}(\beta^{n_2}\beta_{2,1}.x_0,\beta^{n_1}\beta_{1,1}.x_0).$$
\noindent
Fact \ref{fact7} follows with $C=1+C'$.
\end{proof}
\noindent
We now achieve the proof of part {\bf (2)} of case {\bf (c)}. Let $z$ be a point of $\LL_l$ satisfying the conclusion of Fact \ref{fact7} above. Let us split $\left|b^*(\alpha,g_1.x_0)-b^*(\alpha,g_2.x_0)\right|$ into
$$\left|b^*(\alpha,g_1.x_0)-b^*(\alpha,g.z)\right|+\left|b^*(\alpha,g.z)-b^*(\alpha,g_2.x_0)\right|;$$
\noindent
The case {\bf (c)} with $x=g_1.x_0$ and $y=g_2.x_0$ follows using case {\bf (b)}, the conformality equality \eqref{mvr} and the definition of $z$.
\end{itemize}
\subsection{On the extended transfer operator and its spectral properties}
We now introduce the transfer operator associated to the roof function $b^*$. Recall that
$\tilde{\LL}=\LL\cup\G.x_0$ and $\tilde{\LL}_j:=\overline{\tilde{\LL}_j^0}$ for any $j\in[\![1,p+q]\!]$.

For any $z\in\C$ and any function $\varphi\in\mathcal{C}(\tilde{\LL})$, we formally define the operator $\ot_z^*$ as follows: for any $x\in\tLL^0$
$$\ot_z^*\varphi(x)=\sum\limits_{j=1}^{p+q}\sum\limits_{\alpha\in\G_j^*}\un_{\tilde{\LL}_j^c}(x)e^{-z b^*(\alpha,x)}\varphi(\alpha.x)=\sum_{j=1}^{p+q}\sum_{\alpha\in\G_j^*}e^{-\del\dhy(\oo,\alpha.\oo)}\varphi(\alpha.x_0).$$ 
\noindent
Assumptions $(P_1)$ and $(N)$ combined with Corollary \ref{busedist} imply that the previous sums are finite for $\re{z}\geqslant\delta$. The normal convergence of these series for $\re{z}\geqslant\delta$ and the fact that $\varphi\in\mathcal{C}(\tilde{\LL})$ also imply that $\ot_z^*\varphi$ may be continuously extended on $\tilde{\LL}$; in particular, the operator $\ot_{\delta}^*$ acts on $\mathcal{C}(\tLL)$. Denote by 
$\rho_ {\infty}^*$ its spectral radius on this set. The operator $\ot_{\delta}^*$ is positive, hence 
$$\rho_{\infty}^*=\limsup\limits_{k\longrightarrow+\infty}\left|(\otds)^k\un_{\tLL}\right|_{\infty}^{\frac{1}{k}}.$$
\noindent
To obtain a spectral gap, we will study the action of $\ot_{\delta}^*$ on the space $\mathrm{Lip}(\tilde{\LL})$, defined by
$$\mathrm{Lip}(\tilde{\LL}):=\left\{\varphi\in\mathcal{C}(\tLL)\ |\ ||\varphi||:=|\varphi|_{\infty}+[\varphi]<+\infty\right\}$$
\noindent
where
$$[\varphi]=\underset{1\leqslant j\leqslant p+q}{\sup}\ \underset{\underset{x\neq y}{x,y\in\tLL_j}}{\sup}\dfrac{|\varphi(x)-\varphi(y)|}{\dhy_{\oo}(x,y)}.$$
\noindent
The following proposition implies that $\ot_z^*$ is bounded on $\mathrm{Lip}(\tilde{\LL})$ for $z\in\C$ such that $\re{z}\geqslant\delta$. 
\begin{prop}
Each weight $w_z^*(\g,\cdot):=e^{-zb^*(\g,\cdot)}\un_{\tLL_{l(\g)}^c}$ belongs to $\mathrm{Lip}(\tilde{\LL})$ 
and there exists a constant $C=C(z)>0$ such that for any $\g$ in $\G^*$, we have
$$||w_z^*(\g,\cdot)||\leqslant Ce^{-\re{z}\dhy(\oo,\g.\oo)}.$$
\end{prop}
\noindent
The proof is similar to the one of Proposition \ref{cocycle1lip} using Proposition \ref{contfinal}. Let $\rho^*$ denote the spectral radius of $\ot_z^*$ on $\mathrm{Lip}\left(\tilde{\LL}\right)$. 
\begin{prop}\label{propriétésopeétendu}
The operator $\otds$ is quasi-compact on $\mathrm{Lip}(\tilde{\LL})$ and $\rho^*$ is a simple and isolated eigenvalue in its spectrum, associated to a positive eigenfunction; this is the unique eigenvalue with modulus $\rho^*$. Furthermore $\rho^*=\rho_{\infty}^*=1$. 
\end{prop}
\begin{proof}
As in the proof of Proposition \ref{spectredeotdelta}, the first step is to prove that $\otds$ is quasi-compact on
$\mathrm{Lip}(\tilde{\LL})$. By \cite{Hen2}, it is sufficient to prove that 
$\otds$ satisfies the following property.
\begin{defi}[Property DF(s)]
The operator $\otds$ satisfies the DF(s)-property on $(\mathrm{Lip}(\tilde{\LL}),||\cdot||)$ if
\begin{enumerate}
 \item[i)]$\otds$ is compact from $(\mathrm{Lip}(\tilde{\LL}),||\cdot||)$ to $(\mathrm{Lip}(\tilde{\LL}),|\cdot|_{\infty})$;
 \item[ii)]for any $k\in\N\setminus\{0\}$, there exist positive reals $S_k,s_k$ such that 
 $$\left|\left|\left(\otds\right)^k\varphi\right|\right|\leqslant S_k\left|\varphi\right|_\infty+s_k\left|\left|\varphi\right|\right|\ \text{for any}\ \varphi\in\mathrm{Lip}(\tilde{\LL})\ 
 \text{and}\ \liminf_k(s_k)^{\frac{1}{k}}=s<\rho^*.$$
\end{enumerate}
\end{defi}
We follow once again the steps of proofs given in \cite{BP} and \cite{Pe1}. The set
$\tLL=\LL\cup\G.x_0$ being compact, Ascoli's result implies that the inclusion $\left(\mathrm{Lip}(\tilde{\LL}),||\cdot||\right)\hookrightarrow
\left(\mathrm{Lip}(\tilde{\LL}),|\cdot|_{\infty}\right)$ is compact; hence $\otds$ is compact from  
$\left(\mathrm{Lip}(\tilde{\LL}),||\cdot||\right)$ to $\left(\mathrm{Lip}(\tilde{\LL}),|\cdot|_{\infty}\right)$. Let us show the existence of two 
sequences $(s_k)_{k\geqslant1}$ and $(S_k)_{k\geqslant1}$ satisfying properties ii) above. Let $\varphi\in\mathrm{Lip}(\tilde{\LL})$, $k\geqslant1$ and $j\in[\![1,p+q]\!]$. For any $x\in\tLL_j$, 
\begin{equation}\label{boundsup}\left|(\otds)^k\varphi(x)\right|\leqslant\left(\sum\limits_{|\g|=k}||w_\delta^*(\g,\cdot)||\right)\cdot|\varphi|_{\infty}.\end{equation}
For any $x,y\in\tLL_j$, we bound $\left|(\otds)^k\varphi(x)-(\otds)^k\varphi(y)\right|$ from above by $K_1+K_2$ where
$$K_1:=\sum\limits_{\underset{l(\g)\neq j}{\g\in\G(k)}}e^{-\delta b^*(\g,x)}\left|\varphi(\g.x)-\varphi(\g.y)\right|$$
\noindent
and
$$K_2:=\sum\limits_{\underset{l(\g)\neq j}{\g\in\G(k)}}\left|e^{-\delta b^*(\g,y)}-e^{-\delta b^*(\g,x)}\right|\left|\varphi(\g.y)\right|.$$
\noindent
The points $x$ and $y$ belong to $\tilde{\LL}_j$ and $\g$ 
satisfies $l(\g)\neq j$; hence, by Corollary \ref{actioncontractante}, there exists $r<1$ such that
\begin{align*}
K_1\leqslant & Cr^k\left(\sum\limits_{\underset{l(\g)\neq j}{\g\in\G(k)}}e^{-\delta b^*(\g,x)}\right)\left[\varphi\right]\dhy_{\oo}(x,y),
\end{align*}
\noindent
hence
\begin{equation}\label{firstboundlip}
K_1\leqslant Cr^k\left|(\otds)^k\un_{\tilde{\LL}}\right|_{\infty}\left|\left|\varphi\right|\right|\cdot\dhy_{\oo}(x,y).
\end{equation}
\noindent
The second term $K_2$ is bounded from above by
\begin{equation}\label{secondboundlip}
\left(\sum\limits_{\g\in\G(k)}\left|\left|w_\delta^*(\g,\cdot)\right|\right|\right)\left|\varphi\right|_\infty\cdot\dhy_{\oo}(x,y).
\end{equation}
\noindent
Consequently
$$\left|\left|(\otds)^k\varphi\right|\right|\leqslant\left(Cr^k\left|(\otds)^k\un_{\tilde{\LL}}\right|_{\infty}\right)\left|\left|\varphi\right|\right|+
2\left(\sum\limits_{\g\in\G(k)}\left|\left|w_\delta^*(\g,\cdot)\right|\right|\right)\left|\varphi\right|_\infty.$$
\noindent
Set $s_k:=Cr^k\left|(\otds)^k\un_{\tilde{\LL}_{\G}}\right|_{\infty}$ and 
$S_k=2\sum_{\g\in\G(k)}\left|\left|w_\delta^*(\g,\cdot)\right|\right|$. Finally $\otds$ 
satisfies the DF(s)-property and is thus quasi-compact on
$\mathrm{Lip}(\tilde{\LL})$.

In the second step, we prove that $\ros=\rois=1$. The equality $\ros=\rois$ follows from the arguments 
exposed in the proof of Proposition III.4 in \cite{BP2}.
Let us now show $\rho^*=1$. We know by Proposition \ref{spectredeotdelta} that the spectral radius of $\ot_{\delta}$ equals $1$ and is a simple and isolated eigenvalue associated to 
the positive eigenfunction $h$ defined in Section 4. Let $\varphi\in\mathrm{Lip}(\tilde{\LL})$ be an eigenfunction associated to an eigenvalue $\lambda$ with 
modulus $\rho^*$. The function $\varphi_{|_{\LL}}$ belongs to $\mathrm{Lip}(\LL)$ and satisfies for any $x\in\LL$ 
\begin{align*}
\otd\varphi_{|_{\LL}}(x)= & \sum\limits_{\alpha\in\mathcal{A}}\un_{\LL_{l(\alpha)}^c}(x)e^{-\delta b(\alpha,x)}\varphi_{|_{\LL}}(\alpha.x)\\
=& \left(\otds\varphi\right)_{|_{\LL}}(x)\\
=&\lambda\varphi_{|_{\LL}}(x). 
\end{align*}
\noindent
Therefore $\rho^*=\left|\lambda\right|\leqslant1$. On the other hand $\rho^*=\limsup_n\left|\left(\otds\right)^n\un_{\tLL}\right|_{\infty}^{\frac{1}{n}}$. Fix $\varepsilon>0$; there exists a subsequence $(n_k)_k$ such that $\left(\otds\right)^{n_k}\un_{\tLL}(x)\preceq(\rho^*)^{n_k}(1+\varepsilon)^{n_k}$ for any $k\geqslant1$ and $x\in\LL$. By definition of $\otds$, 
we obtain for any $k\geqslant1$ and $x\in\LL$
\begin{equation}\label{rhovautun1}
h(x)=\otd^{n_k}h(x)\asymp\otd^{n_k}\un_{\LL}(x)\preceq(\rho^*)^{n_k}(1+\varepsilon)^{n_k}.
\end{equation}
\noindent
Letting $k\longrightarrow+\infty$, we obtain  
$\rho^*(1+\varepsilon)\geqslant1$ for any $\varepsilon>0$. Finally $\rho^*=1$.

Let $e^{i\theta}$ be an eigenvalue with modulus $1$ for $\otds$: there exists $\varphi\in\mathrm{Lip}\left(\tilde{\LL}\right)$ such that 
$\otds\varphi=e^{i\theta}\varphi$. As before, for any $x\in\LL$, we get
$$\ot_\delta\varphi_{|\LL}(x)=\left(\otds\varphi\right)_{|\LL}(x)=e^{i\theta}\varphi_{|\LL}(x).$$
\noindent
Since $1$ is the unique eigenvalue with modulus $1$ of $\ot_\delta$, it follows $\theta\in2\pi\Z$; hence $1$ is also the unique eigenvalue with modulus $1$ of $\otds$.

The next step is devoted to the proof of the existence of a positive eigenfunction for $1$. Since the operator
$\otds$ is positive, there exists a non-negative function $\phi$ such that $\otds\phi=\phi$. 
Let us show that $\phi$ is positive. Assume that $\phi$ vanishes at $x\in\tLL$. There are two cases.
\begin{itemize}
 \item[-]If $x\in\LL$, there exists $j\in[\![1,p+q]\!]$ such that $x\in\LL_j$. For any $k\geqslant1$, we have 
\begin{equation}\label{renewal}\left(\otds\right)^k\phi(x)=\sum\limits_{\underset{l(\g)\neq j}{\g\in\G(k)}}e^{-\delta b(\g,x)}\phi(\g.x)=\phi(x)=0.\end{equation}
\noindent
Therefore $\phi(\g.x)=0$ for any $k\geqslant1$ and any $\g\in\G(k)$ with $l(\g)\neq j$. By minimality of the action of $\G$ on $\LL$ 
and continuity of $\phi$ on $\LL$, we derive that $\phi$ vanishes at any point of $\LL$. It remains to show that we may draw the same conclusion on $\G.x_0$. Equation
\eqref{renewal} ensures that it is sufficient to prove that $\phi(x_0)=0$, or equivalently, that the sequence $\left(\sum_{\g\in\G(k)}e^{-\delta\dhy(\oo,\g.\oo)}\phi(\g.x_0)\right)_k$ tends to $0$. Let $k\geqslant1$. Corollary \ref{actioncontractante} implies that for any  $|\g|=k$ and any limit point $y\in\LL_{l(\g)}^c$
$$\phi(\g.x_0)=|\phi(\g.x_0)|=|\phi(\g.x_0)-\phi(\g.y)|\leqslant\left[\phi\right]\dhy_{\oo}(\g.x_0,\g.y)\leqslant\left[\phi\right]r^k\dhy_{\oo}(x_0,y).$$
\noindent
Lemma \ref{A21} combined with Corollary \ref{busedist} implies that there exists a constant 
$C>0$ such that $\sum_{\g\in\G(k)}e^{-\delta\dhy(\oo,\g.\oo)}\preceq C$, so that
%
%
$\sum_{\g\in\G(k)}e^{-\delta\dhy(\oo,\g.\oo)}\phi(\g.x_0)\preceq r^k$, which tends to $0$ as 
$k\longrightarrow+\infty$.
\item[-] Otherwize $x\in\G.x_0$. It follows $\phi(\g.x)=0$ for any $x\in\LL$ and any $\g\in\G$ such that $x\notin\tLL_{l(\g)}$. By continuity of $\phi$ and by definition of the limit set, the function $\phi$ vanishes on $\LL$ and the above arguments ensure that 
this is true everywhere.
\end{itemize}

The next step deals with the dimension of the eigenspace associated to $1$. To prove that it is equal to $1$, we introduce the normalized operator 
$P^*$ defined for any $\varphi\in\mathrm{Lip}(\tLL)$ by
$$P^*\varphi=\dfrac{1}{\phi}\otds\left(\varphi\phi\right)$$
\noindent
where $\phi$ is a positive eigenfunction associated to $1$. This operator is 
positive, bounded on the space $\mathrm{Lip}(\tLL)$ and quasi-compact; it also satisfies $P^*(\un_{\tilde{\LL}})=\un_{\tilde{\LL}}$ and its spectral radius is $1$. Let us prove that the eigenspace associated to $1$ for $P^*$ has dimension $1$. Let
$f$ be an eigenfunction associated to $1$ for $P^*$. 
\begin{itemize}
\item[(i)] Assume first that the maximum $|f|_{\infty}$ of $|f|$ is obtained  at a point 
$x\in\LL_j$, $j\in[\![1,p+q]\!]$. Hence
\begin{equation}\label{equalitymarkovian}|f(x)|=|P^*f(x)|\leqslant P^*|f|(x)\leqslant|f(x)|.\end{equation}
\noindent
Using the iterates of $P^*$ and a convexity argument, we obtain $|f(\g.x)|=|f(x)|$ for any $\g\in\G$ with $l(\g)\neq j$. 
By minimality of the action of $\G$ on $\LL$, the function $|f|$ is constant on $\LL$. To check that $\left|f\right|$ is also constant on $\G.x_0$, we rewrite \eqref{renewal} as 
$$1-\dfrac{|f(x_0)|}{M}=\dfrac{1}{\phi(x_0)}\sum\limits_{\g\in\G(k)}e^{-\delta\dhy(\oo,\g.\oo)}\left(1-\dfrac{|f(\g.x_0)|}{M}\right)\phi(\g.x_0),$$
\noindent
hence
\begin{equation}\label{jelerajouteetcmieux}\left|1-\dfrac{|f(x_0)|}{M}\right|\leqslant\dfrac{1}{\phi(x_0)}\sum\limits_{\g\in\G(k)}e^{-\delta\dhy(\oo,\g.\oo)}\dfrac{|M-|f(\g.x_0)||}{|M|}\phi(\g.x_0).\end{equation}
\noindent
For any $\g\in\G(k)$ and any $z\in\LL_{l(\g)}^c$ one gets $M=|f(\g.z)|$. As previously, we bound from above the right term of \eqref{jelerajouteetcmieux} by $Cr^k$, quantity which goes to $0$ as $k\longrightarrow+\infty$. Finally $|f(x_0)|=M$. Applying \eqref{equalitymarkovian} to the iterates of $P^*$ and using a convexity argument once again, we deduce that
 $|f|$ is constant on $\G.x_0$, and finally $|f|$ is constant on $\tLL$. 
\item[(ii)] If $|f|$ reaches its maximum in $\G.x_0$, it follows from the minimality of the action of $\G$ on the boundary that 
$|f|$ is constant on $\LL$, hence on $\tLL$ as above.
\end{itemize}
Repeating the same argument, we show that $f$ is constant on $\tLL$, which ends the proof of the proposition.
\end{proof}
Let us now choose an eigenfunction $h^*$ of $\otds$ associated to the eigenvalue $1$ satisfying $\sigma_\oo\left(h^*\right)=1$.
There exists $\Pi_\delta^*\ :\ \mathrm{Lip}\left(\tLL\right)\longrightarrow\C h^*$ such that for any
$\varphi\in\mathrm{Lip}\left(\tLL\right)$,  
$$\otds\varphi=\Pi_\delta^*\left(\varphi\right)+R_\delta^*\left(\varphi\right)$$
\noindent
where $R^*$ satisfies $R_{\delta}^*\Pi_{\delta}^*=\Pi_{\delta}^*R_{\delta}^*=0$ and has a spectral radius $<1$. There exists a linear form
$\sigma^*\ :\ \mathrm{Lip}\left(\tLL\right)\longrightarrow\C$ satisfying $\sigma^*\left(h^*\right)=1$ such that for any $\varphi\in\mathrm{Lip}(\tLL)$, 
\begin{equation}\label{decspec}\otds\varphi=\sigma^*(\varphi)h^*+R^*\varphi.\end{equation}
\noindent
It follows for any $k\geqslant1$
$$
\left(\otds\right)^k\varphi=\sigma^*\left(\varphi\right)h^*+\left(R^*\right)^k\varphi=\sigma^*\left(\otds\varphi\right)h^*+\left(R^*\right)^{k-1}\otds\varphi.
$$
\noindent
Letting $k\longrightarrow+\infty$, we deduce that $\sigma^*$ is $\otds$-invariant. To identify $\sigma^*$, let us write, for any $\varphi\in\mathrm{Lip}\left(\tLL\right)$ and $k\geqslant1$
$$\ot_\delta^k\varphi_{|\LL}=\left(\left(\otds\right)^k\varphi\right)_{|\LL}=\sigma^*\left(\varphi\right)h_{|\LL}^*+\left(R^*\right)^k\varphi_{|\LL}.$$
\noindent
The equality $\sigma_\oo\left(\varphi\right)=\sigma_\oo\left(\varphi_{|\LL}\right)=\sigma_\oo\left(\ot_\delta^k\varphi_{|\LL}\right),$ readly implies
$$\sigma_\oo\left(\varphi\right)=
\sigma^*\left(\varphi\right)+\sigma_{\oo}\left(\left(R^*\right)^k\varphi_{|\LL}\right)$$
\noindent
and letting $k\longrightarrow+\infty$, we obtain $\sigma_\oo\left(\varphi\right)=\sigma^*\left(\varphi\right)$. Finally
$\sigma^*=\sigma_\oo$ and $h^*$ extends $h$ on $\tLL$.
\begin{rem}\label{valeurhstarenxzero}
One gets
$h^*(x_0)=\lim\limits_{k\longrightarrow+\infty}\sum\limits_{\g\in\G(k)}e^{-\delta\dhy(\oo,\g.\oo)}$;
this quantity does not depend on $x_0\in\partial\xx$.
\end{rem}
Let us now check that the function $z\longmapsto\ot_z^*$ is a continuous perturbation of $\ot_\delta^*$ for $\re{z}\geqslant\delta$ using the following proposition, whose proof is the same 
as the one of Proposition \ref{continuityoftransfert}.
\begin{prop}\label{continuityoftransfertstar}
Under the assumptions $\left(H_\beta\right)$, for any compact $K$ of $\R$, there exists a constant $C=C_K>0$ such that for any $s,t\in K$ 
and $\kappa>0$
\begin{itemize}
 \item[1)]if $\beta\in]0,1[$ 
        \begin{align*}
        &a.\ ||\ot_{\delta+it}^*-\ot_{\delta+is}^*||\leqslant C|s-t|^{\beta}L\left(\dfrac{1}{|s-t|}\right),\\
        &b.\ ||\ot_{\delta+\kappa+it}^*-\ot_{\delta+it}^*||\leqslant C\kappa^{\beta}L\left(\dfrac{1}{\kappa}\right); 
        \end{align*}
\item[2)]if $\beta=1$ 
       \begin{align*}
       &a.\ ||\ot_{\delta+it}^*-\ot_{\delta+is}^*||\leqslant C|s-t|\tilde{L}\left(\dfrac{1}{|s-t|}\right),\\
       &b.\ ||\ot_{\delta+\kappa+it}^*-\ot_{\delta+it}^*||\leqslant C\kappa\tilde{L}\left(\dfrac{1}{\kappa}\right),
       \end{align*}
\noindent
where $\tilde{L}(x)=\displaystyle{\int_1^x}\frac{L(y)}{y}\dd y$.
\end{itemize}
\end{prop}
Now let $\rho^*(z)$ denote the spectral radius of 
$\ot_z^*$ on $\left(\mathrm{Lip}\left(\tLL\right),||\cdot||\right)$. As in \cite{BP}, we may state the
\begin{prop}\label{spectreperturbationstar}
There exist $\varepsilon>0$ and $\rho_\varepsilon\in]0,1[$ such that for any $z\in\C$ satisfying $|z-\delta|_\infty<\varepsilon$ and $\re{z}\geqslant\delta$
\begin{enumerate}
 \item[-]$\rho^*(z)>\rho_\varepsilon$;
 \item[-]$\ot_{z}$ has a unique eigenvalue $\lambda_z^*$ with modulus $\rho^*(z)$;
 \item[-]this eigenvalue is simple and close to $1$;
 \item[-]the remainder of the spectrum is included in a disc of radius $\rho_\varepsilon$.
\end{enumerate}
Moreover, for any $A>0$, there exists $\rho(A)<1$ such that $\rho^*(z)<\rho(A)$ as soon as
$z\in\C$ satisfies $|z-\delta|_\infty\geqslant\varepsilon$, 
$\re{z}\geqslant\delta$ and $\left|\im{z}\right|\leqslant A$. Finally $\rho^*(z)=1$ if and only if $z=\delta$.
\end{prop}
The following proposition specifies the local behaviour of the dominant eigenvalue $\lambda_{\delta+it}^*$; its proof is verbatim the one of Proposition \ref{localexp}.
\begin{prop}\label{localexp2}
For the constant $\Cg>0$ given in Proposition \ref{localexp} and $t\longrightarrow0$
\begin{itemize}
 \item[-]if $\beta\in]0,1[$ 
    $$\lambda_{\delta+it}^*=1-\Cg\G(1-\beta)e^{+i\mathrm{sign}(t)\frac{\beta\pi}{2}}|t|^{\beta}L\left(\dfrac{1}{|t|}\right)(1+o(1));$$
 \item[-]if $\beta=1$
 \begin{align*}
&\bullet\ \lambda_{\delta+it}^*=1-\Cg \mathrm{sign}(t)i|t|\tilde{L}\left(\dfrac{1}{|t|}\right)(1+o(1));\\
&\bullet\ \re{1-\lambda_{\delta+it}^*}=\dfrac{\pi}{2}\Cg|t|L\left(\dfrac{1}{|t|}\right)(1+o(1)).
\end{align*}
\end{itemize}
\end{prop}
\begin{rem}
The fact that the constant $\Cg$ is the one of Proposition \ref{localexp} relies on the equivalent
$\lambda_{\delta+it}^*\sim\sigma_{\oo}\left(\left(\ot_{\delta+it}^*-\ot_{\delta}^*\right)h^*\right)$ and the equality $\sigma_{\oo}\left(\G.x_0\right)=0$. Therefore
$$\sigma_{\oo}\left(\left(\ot_{\delta+it}^*-\ot_{\delta}^*\right)h^*\right)=
\sigma_{\oo}\left(\left(\ot_{\delta+it}-\ot_{\delta}\right)h\right).$$
\end{rem}
In the proof of Theorem C for $\beta=1$, we will use $Q_z^*=\left(\mathrm{Id}-\ot_z^*\right)^{-1}$ for $z\in\C$ such that $\re{z}\geqslant\delta$. 
The following properties are the same as Properties \ref{controlQ} and \ref{integrabilityreQ} proved in Section 4 concerning
$Q_z=\left(\mathrm{Id}-\ot_z\right)^{-1}$. 
\begin{prop}\label{controlQstar}
There exist $\varepsilon>0$ and $C>0$ such that $\left|\left|Q_z^*-(1-\lambda_z^*)^{-1}\Pi_{z}^*\right|\right|\leqslant C$ for $z$ such that
$|z-\delta|_\infty<\varepsilon$ and $\left|\left|Q_z^*\right|\right|\leqslant C$ for $z$ such that $|z-\delta|_\infty\geqslant\varepsilon$. Moreover, as $t\longrightarrow0$,
$$
Q_{\delta+it}^*=\dfrac{1}{\Cg\mathrm{sign}(t)i|t|\tilde{L}\left(\dfrac{1}{|t|}\right)}(1+o(1))\Pi_0+O(1).
$$
\end{prop}
As a direct consequence:
\begin{coro}\label{integrabilityreQstar}
The function $t\longmapsto\re{Q_{\delta+it}^*}$ is integrable at $0$.
\end{coro}

\section{Theorem C: asymptotic of the orbital counting function}

Let us restate Theorem C. 
\begin{thmC}
Let $\G$ be a Schottky group satisfying the assumptions $(H_\beta)$ for some $\beta\in]0,1]$. Then, as 
$R\longrightarrow+\infty$,
\begin{align*}
&\bullet\ \mathrm{N}_{\G}(\oo,R)\sim C\dfrac{e^{\delta R}}{R^{1-\beta}L(R)}\ \text{with}\ 
C=\dfrac{\sin(\beta\pi)}{\pi}\dfrac{h^*(x_0)}{\delta\Cg}\ \text{if}\ \beta\in]0,1[;\\
&\bullet\ \mathrm{N}_{\G}(\oo,R)\sim C'\dfrac{e^{\delta R}}{\tilde{L}(R)}\ \text{with}\ 
C'=\dfrac{h^*(x_0)}{\delta\Cg}\ \text{if}\ \beta=1.
\end{align*}
\end{thmC}
\noindent
We write
\begin{align*}\mathrm{N}_{\G}(\oo,R)= &
e^{\delta R}\sum\limits_{\g\in\G^*}e^{-\delta\dhy(\oo,\g.\oo)}u(\dhy(\oo,\g.\oo)-R)+1\\
= & e^{\delta R}\sum\limits_{k\geqslant1}\sum\limits_{\g\in\G(k)}e^{-\delta\dhy(\oo,\g.\oo)}u(\dhy(\oo,\g.\oo)-R)+1\\
\sim & e^{\delta R}W(R,u),
\end{align*}
\noindent
where $u(x)=e^{\delta x}\un_{\R^-}(x)$ and 
$W(R,u)=\sum\limits_{k\geqslant1}W_k(R,u)=\sum_{\g\in\G(k)}e^{-\delta\dhy(\oo,\g.\oo)}u(\dhy(\oo,\g.\oo)-R)$. To prove Theorem C, it 
is thus sufficient to prove the following 
\begin{prop}\label{butorbital}
For any function $u\ :\ \R\longrightarrow\R$ piecewise continuous and satisfying $\int_\R u(x)\dd x>0$, as $R\longrightarrow+\infty$,
\begin{align*}
&\bullet\ W(R,u)\sim\dfrac{C}{R^{1-\beta}L(R)}\int_{\R}u(x)\dd x\ \text{if}\ \beta\in]0,1[;\\
&\bullet\ W(R,u)\sim\dfrac{C'}{\tilde{L}(R)}\int_{\R}u(x)\dd x\ \text{if}\ \beta=1,
\end{align*}
\noindent
where the constants $C,C'$ are precised in the above Theorem C.
\end{prop}
\begin{proof}[Proof of Theorem C]
Let us explain how Theorem C follows from Proposition \ref{butorbital} for $\beta\in]0,1[$. Up to a finite number of terms, we have 
$$\mathrm{N}_{\G}(\oo,R)=\sharp\{\g\in\G\ |\ 1\leqslant\dhy(\oo,\g.\oo)\leqslant R\}.$$
\noindent
Fix $\varepsilon>0$. Then, for all $R>0$, one gets $m(R)\leqslant \mathrm{N}_{\G}(\oo,R)\leqslant M(R)$ where 
$$m(R)=\sum\limits_{q=0}^{\left[\frac{R}{\varepsilon}\right]-2}
\sharp\{\g\in\G\ |\ 1+q\varepsilon\leqslant\dhy(\oo,\g.\oo)< 1+(q+1)\varepsilon\}$$
\noindent
and
$$M(R)=\sum\limits_{q=0}^{\left[\frac{R}{\varepsilon}\right]-1}
\sharp\{\g\in\G\ |\ 1+q\varepsilon\leqslant\dhy(\oo,\g.\oo)< 1+(q+1)\varepsilon\}.$$
\noindent
Applying Proposition \ref{butorbital} with $u\ :\ t\longmapsto e^{\delta t}\un_{[0,\varepsilon]}(t)$ yields
$$
m(R)\sim C\dfrac{e^{\delta\varepsilon}-1}{\delta\varepsilon}\sum\limits_{q=0}^{\left[\frac{R}{\varepsilon}\right]-2}
\varepsilon\dfrac{e^{\delta(1+q\varepsilon)}}{(1+q\varepsilon)^{1-\beta}L(1+q\varepsilon)}
\underset{\varepsilon\longrightarrow0}{\longrightarrow}C\int_1^R\dfrac{e^{\delta x}}{x^{1-\beta}L(x)}\dd x
$$
\noindent
and
$$
M(R)\sim C\dfrac{e^{\delta\varepsilon}-1}{\delta\varepsilon}\sum\limits_{q=0}^{\left[\frac{R}{\varepsilon}\right]-1}
\varepsilon\dfrac{e^{\delta(1+q\varepsilon)}}{(1+q\varepsilon)^{1-\beta}L(1+q\varepsilon)}
\underset{\varepsilon\longrightarrow0}{\longrightarrow}C\int_1^R\dfrac{e^{\delta x}}{x^{1-\beta}L(x)}\dd x
$$
\noindent
Finally $\mathrm{N}_{\G}(\oo,R)\sim C\frac{e^{\delta R}}{\delta R^{1-\beta}L(R)}$. The same proof holds for $\beta=1$ using the second estimate of Proposition \ref{butorbital}.
\end{proof}
\subsection{Proposition \ref{butorbital} for $\boldsymbol{\beta\in]0,1[}$}
%
%
%
%
We follow step by step the proof of Theorem A. Let $K\geqslant2$ and $(a_k)_{k\geqslant1}$ be a sequence satisfying $kL(a_k)=a_k^{\beta}$. We first admit the following propositions.
\begin{propC1}\label{C1}
Let $u\ :\ \R\longrightarrow\R$ a function with compact support. Uniformly in $K\geqslant2$ and $R\in[0,Ka_k]$, one gets when $k\longrightarrow+\infty$
$$W_k(R,u)=\dfrac{1}{\cg a_k}\left(C_0\Psi_\beta\left(\dfrac{R}{\cg a_k}\right)\widehat{u}(0)+o_k(1)\right),$$
\noindent
where $\Psi_\beta$ is the density of the fully asymmetric stable law with parameter $\beta$, 
$\cg=\Cg^{\frac{1}{\beta}}$ and 
$C_0=h^*(x_0)=\lim\limits_{k\rightarrow+\infty}\sum\limits_{\g\in\G(k)}e^{-\delta\dhy(\oo,\g.\oo)}$.
\end{propC1}
\begin{propC2}\label{C2}
Let $u\ :\ \R\longrightarrow\R$ be a function with compact support. If $R\geqslant a_k$, there exists $C>0$ which only depends on $u$ such that
$$\left|W_k(R,u)\right|\leqslant Ck\dfrac{L(R)}{R^{1+\beta}}|u|_{\infty}.$$
\end{propC2}
Let us now explain how Proposition \ref{butorbital} follows from Propositions C.1 and C.2. Let $u\ :\ \R\longrightarrow\R$ be a function 
with compact support. We want to estimate $W(R,u)=\sum_{k\geqslant1}W_k(R,u)$. By Proposition C.1, we may decompose this quantity into $W^1(R,u)+W^2(R,u)+W^3(R,u)$ where
\begin{align*}
& W^1(R,u):= \dfrac{C_0\widehat{u}(0)}{\cg}\sum\limits_{k\ |\ R<Ka_k}\dfrac{1}{a_k}\Psi_\beta\left(\dfrac{R}{\cg a_k}\right),\\
& W^2(R,u):= \dfrac{1}{\cg}\sum\limits_{k\ |\ R<Ka_k}\dfrac{o_k(1)}{a_k}\\
\text{and}\ & W^3(R,u):=\sum\limits_{k\ |\ R\geqslant Ka_k}W_k(R,u).
\end{align*}
\begin{itemize}
 \item[a)] {\it Contribution of $W^1(R,u)$.}  Following \cite{Gou}, we introduce the measure
$\mu_R=\sum\limits_{0<\frac{R}{a_k}\leqslant K} D_{\frac{R}{a_k}}$ defined on the interval $]0,K]$. One gets
$$\sum\limits_{R<Ka_k}\dfrac{1}{\cg a_k}\Psi_\beta\left(\dfrac{R}{\cg a_k}\right)=
\dfrac{1}{R}\displaystyle{\int_0^K}\dfrac{z}{\cg}\Psi_\beta\left(\dfrac{z}{\cg}\right)\dd\mu_R(z).$$
\noindent
When $R\longrightarrow+\infty$, this measure satisfies
$$R^{-\beta}L(R)\mu_R([x,y])\sim\displaystyle{\int_x^y}\beta z^{-\beta-1}\dd z,$$
\noindent
so that
\begin{align}\label{Dominantorbital1}
R^{1-\beta}L(R)W^1(R,u)\sim & 
\ \dfrac{C_0\beta\widehat{u}(0)}{\cg}\displaystyle{\int_0^K}z^{-\beta}\Psi_\beta\left(\dfrac{z}{\cg}\right)\dd z.\nonumber\\
\sim & \dfrac{\beta\widehat{u}(0)}{\Cg}\displaystyle{\int_0^{\frac{K}{\cg}}}z^{-\beta}\Psi_\beta(z)\dd z\ \text{with}\ \Cg=\cg^{\beta}.
\end{align}
\item[b)] {\it Contribution of $W^2(R,u)$.} This quantity is similar to $M^2(R;\varphi\otimes u,\psi\otimes v)$ 
appearing in the proof of Theorem A; it satisfies  
\begin{equation}\label{Dominantorbital2}
R^{1-\beta}L(R)W^2(R,u)=o_K(1)\ \text{where}\ \lim\limits_{R\longrightarrow+\infty}o_K(1)=0\ \text{for any fixed}\ K.
\end{equation}
\item[c)] {\it Contribution of $W^3(R,u)$.} As for $M^3(R;\varphi\otimes u,\psi\otimes v)$ appearing in the proof of Theorem A, we get
\begin{equation}\label{negligeableorbital}
R^{1-\beta}L(R)W^3(R,u)\leqslant CK^{-\beta}.
\end{equation}
\end{itemize}
Combining \eqref{Dominantorbital1}, \eqref{Dominantorbital2} and \eqref{negligeableorbital}, we obtain
\begin{align*}
R^{1-\beta}L(R)W(R,u)=\dfrac{\beta\widehat{u}(0)C_0}{\Cg}\displaystyle{\int_0^{\frac{K}{\cg}}}z^{-\beta}\psi(z)\dd z(1+o(1)) + 
o_K(1) + O(K^{-\beta}),
\end{align*}
\noindent
with $\lim\limits_{R\longrightarrow+\infty}o(1)=0$. Letting $R\longrightarrow+\infty$ and 
choosing $K$ large enough yield
$$R^{1-\beta}L(R)W(R,u)\sim\dfrac{\beta\widehat{u}(0)C_0}{\Cg}\displaystyle{\int_0^{+\infty}}z^{-\beta}\Psi_\beta(z)\dd z$$
\noindent
and Proposition \ref{butorbital} follows with $C=\frac{C_0}{\Cg}\frac{\sin(\beta\pi)}{\pi}.$
\subsubsection{Proof of Proposition C.1}
Let $K\geqslant2$ and $R>0$ be fixed and choose $k\in\N$ such that $Ka_k\geqslant R$. It is sufficient to prove that, as $k\longrightarrow+\infty$,
\begin{equation}\label{convergenceorbitaleC1}
a_kW_k(R,u)-\dfrac{h^*(x_0)}{\cg}\Psi_{\beta}\left(\dfrac{R}{\cg a_k}\right)\widehat{u}(0)\longrightarrow0,
\end{equation}
\noindent 
for all $u\in\mathscr{U}$, where the set of test functions $\mathscr{U}$ was defined in \ref{classedefonctionstest}. Let us 
first show that the quantity $W_k(R,u)$ is finite for any $u\in\mathscr{U}$. The Fourier inverse formula implies
\begin{align*}
\left|W_k(R,u)\right|\leqslant & \dfrac{1}{2\pi}\int_{\R}\left|e^{itR}\left(\ot_{\delta+it}^*\right)^k\un_{\tLL}(x_0)\widehat{u}(t)\dd t\right|\\
\leqslant & \dfrac{1}{2\pi}\left(\ot_{\delta}^*\right)^k\un_{\tLL}(x_0)\int_{\R}\widehat{u}(t)\dd t\\
\preceq & \left|\left|\widehat{u}\right|\right|_1h^*(x_0)<+\infty.
\end{align*}
\noindent
Let us now prove \eqref{convergenceorbitaleC1}. To simplify notations, we will omit the symbol $^*$. The Fourier inverse formula allows us to split 
$$
a_kW_k(R,u)-\dfrac{h(x_0)}{\cg}\Psi_{\beta}\left(\dfrac{R}{\cg a_k}\right)\widehat{u}(0)
$$
\noindent
into $K_1(k)+K_2(k)$ where 
$$K_1(k):=\dfrac{a_k}{2\pi}\displaystyle{\int\limits_{[-\varepsilon,\varepsilon]^c}}e^{itR}\ot_{\delta+it}^k\un_{\tLL}(x_0)\widehat{u}(t)\dd t$$
\noindent
and
\begin{align*}
K_2(k):= & \dfrac{a_k}{2\pi}\displaystyle{\int\limits_{-\varepsilon}^{\varepsilon}}e^{itR}\ot_{\dit}^k\un_{\tLL}(x_0)\widehat{u}(t)\dd t-
\dfrac{1}{2\pi}\displaystyle{\int\limits_{\R}}e^{it\frac{R}{a_k}}g_{\beta}(\cg t)\widehat{u}(0)h(x_0)\dd t\\
= & \dfrac{1}{2\pi}\displaystyle{\int\limits_{-\varepsilon a_k}^{\varepsilon a_k}}e^{it\frac{R}{a_k}}\ot_{\delta+i\frac{t}{a_k}}^k\un_{\tLL}(x_0)\widehat{u}\left(\dfrac{t}{a_k}\right)\dd t-
\dfrac{1}{2\pi}\displaystyle{\int\limits_{\R}}e^{it\frac{R}{a_k}}g_{\beta}(\cg t)\widehat{u}(0)h(x_0)\dd t,
\end{align*}
where $\varepsilon>0$ satisfies the conclusions of Proposition \ref{spectreperturbationstar}. The spectral properties of $\ot_z$ given in the latter
proposition and the fact that $\widehat{u}$ has compact support imply that $||\ot_{\delta+it}^k||\preceq\rho^k$, for $0<\rho<1$ which
only depends on the support of $u$. Therefore $\left|K_1(k)\right|\preceq\left|\left|\widehat{u}\right|\right|_{\infty}\rho^ka_k\longrightarrow0$ as $k\longrightarrow+\infty$, 
uniformly in $K$ and $R$. 

We now deal with $K_2(k)$. The spectral decomposition of 
$\ot_{\delta+i\frac{t}{a_k}}$ yields
$$\ot_{\delta+i\frac{t}{a_k}}^k\un_{\tLL}=\lambda_{\delta+i\frac{t}{a_k}}^k\Pi_{\delta+i\frac{t}{a_k}}\un_{\tLL}+R_{\delta+i\frac{t}{a_k}}^k\un_{\tLL}x_0$$ 
\noindent
with $\mathrm{spec}(R_{\delta+i\frac{t}{a_k}})\subset B(0,\rho_\varepsilon)$ and $0<\rho_\varepsilon<1$. We decompose $K_2(k)$ as $L_1(k)+L_2(k)+L_3(k)$ where
$$L_1(k)=\dfrac{1}{2\pi}\int\limits_{-\varepsilon a_k}^{\varepsilon a_k}e^{it\frac{R}{a_k}}R_{\delta+i\frac{t}{a_k}}^k\un_{\tLL}(x_0)\widehat{u}\left(\dfrac{t}{a_k}\right)\dd t,$$
\noindent
$$L_2(k)=\dfrac{1}{2\pi}\int\limits_{-\varepsilon a_k}^{\varepsilon a_k}e^{it\frac{R}{a_k}}\lambda_{\delta+i\frac{t}{a_k}}^k\left(\Pi_{\delta+i\frac{t}{a_k}}(\un_{\tLL})(x_0)-\Pi_{\delta}(\un_{\tLL})(x_0)\right)\widehat{u}\left(\dfrac{t}{a_k}\right)\dd t$$
\noindent
and
$$L_3(k)=\dfrac{h(x_0)}{2\pi}\int\limits_{-\varepsilon a_k}^{\varepsilon a_k}e^{it\frac{R}{a_k}}\lambda_{\delta+i\frac{t}{a_k}}^k\widehat{u}\left(\dfrac{t}{a_k}\right)\dd t-\dfrac{h(x_0)}{2\pi}\int\limits_{\R}e^{it\frac{R}{a_k}}g_{\beta}(\cg t)\widehat{u}(0)\dd t.$$
\noindent
First $\left|L_1(k)\right|\preceq a_k\rho_\varepsilon^k\left|\left|\widehat{u}\right|\right|_{\infty}\longrightarrow0$ as $k\longrightarrow+\infty$, uniformly in $K$ and $R$. We use the Lebesgue dominated convergence theorem for $L_2(k)$: as for the integral $L_2(k)$ in the proof of Proposition A.1, choosing $\varepsilon>0$ small enough so that it satisfies Remark \ref{sertA1etB1} and using the local expansion of $\lambda_{\delta+it}$ given in Proposition \ref{localexp2}, we may bound from above the integrand of $L_2(k)$ up to a multiplicative constant by
$$l(t)=\left\{\begin{array}{ll}
        &|t|^{\frac{\beta}{2}}e^{-\frac{1}{4}(1-\beta)\G(1-\beta)|\cg t|^{\frac{3\beta}{2}}} \text{if}\ |t|\leqslant1\\
        &|t|^{\frac{3\beta}{2}}e^{-\frac{1}{4}(1-\beta)\G(1-\beta)|\cg t|^{\frac{\beta}{2}}}\ \text{if}\ |t|>1\\
\end{array}\right.,$$
\noindent
The term $L_3(k)$ may be treated similarly as $L_3(k)$ appearing in the proof of Proposition A.1.
\subsubsection{Proof of Proposition C.2}
It suffices to check that there exist constants $C,M>0$ such that
\begin{equation}
\sum\limits_{\underset{\dhy(\oo,\g.\oo)\overset{M}{\sim}R}{\g\in\G(k)}}e^{-\delta\dhy(\oo,\g.\oo)}\leqslant Ck\dfrac{L(R)}{R^{1+\beta}}.
\end{equation}
\noindent
Corollary \ref{busedist} implies that this estimate is a consequence of Proposition A.2.
\subsection{Proposition \ref{butorbital} for $\boldsymbol{\beta=1}$}
As in the proof of Theorem A for $\beta=1$, we will need to symmetrize the quantity $W(R,u)$. We thus define
$$W^{\mathrm{sym}}(R,u):=\sum\limits_{k\geqslant1}\sum\limits_{\g\in\G(k)}e^{-\delta\dhy(\oo,\g.\oo)}\left(u(\dhy(\oo,\g.\oo)-R)+u(-\dhy(\oo,\g.\oo)-R)\right).$$
\noindent
We may notice that $W^{\mathrm{sym}}(R,u)=W(R,u)$ for $R$ large enough, because the function $u$ has compact support. We first study
$$W_{\xi}^{\mathrm{sym}}(R,u):=\sum\limits_{k\geqslant1}\sum\limits_{\g\in\G(k)}e^{-\xi\dhy(\oo,\g.\oo)}\left(u(\dhy(\oo,\g.\oo)-R)+u(-\dhy(\oo,\g.\oo)-R)\right)$$
\noindent
for $\xi>\delta$. From now on, we assume again that $u$ belongs to the set introduced in Definition \ref{classedefonctionstest}. The fact that $W_{\xi}^{\mathrm{sym}}(R,u)$ is finite for any $u\in\mathscr{U}$ will be a consequence of the forthcomming study. For now, let us notice that 
the Fourier inverse formula combined with the convergence of the Poincaré series of $\G$ at $\xi>\delta$ implies
$$W_{\xi}^{\mathrm{sym}}(R,u)=\dfrac{1}{2\pi}\int_{\R}e^{itR}\sum\limits_{k\geqslant1}\left(\left(\ot_{\xi+it}^*\right)^k+\left(\ot_{\xi-it}^*\right)^k\right)(\un_{\tLL})(x_0)\widehat{u}(t)\dd t$$
\noindent
which is finite for any $u\in\mathscr{U}$. Therefore
$$W_{\xi}^{\mathrm{sym}}(R,u)=\dfrac{1}{\pi}\int_{\R}e^{itR}\re{{Q_{\xi+it}^*}}(\un_{\tLL})(x_0)\widehat{u}(t)\dd t-\dfrac{1}{\pi}\int_{\R}e^{itR}\widehat{u}(t)\dd t.$$
\noindent
We want
\begin{equation}\label{expressionpotentielorbitalbeta1}
W^{\mathrm{sym}}(R,u)=\dfrac{1}{\pi}\int_{\R}e^{itR}\re{Q_{\delta+it}^*}(\un_{\tLL})(x_0)\widehat{u}(t)\dd t-\dfrac{1}{\pi}\int_{\R}e^{itR}\widehat{u}(t)\dd t.
\end{equation}
\noindent
Notice that, as in Proposition \ref{convergencexiversdelta},
$$\dfrac{1}{\pi}\int_{\R}e^{itR}\re{Q_{\xi+it}^*}(\un_{\tLL})(x_0)\widehat{u}(t)\dd t\underset{\xi\searrow\delta}{\longrightarrow}
\dfrac{1}{\pi}\int_{\R}e^{itR}\re{Q_{\delta+it}^*}(\un_{\tLL})(x_0)\widehat{u}(t)\dd t.$$
\noindent
Moreover, for a positive function $u$, the monotone convergence theorem implies that $W_{\xi}^{\mathrm{sym}}(R,u)$ tends to 
$W^{\mathrm{sym}}(R,u)$. Since $t\longmapsto\re{Q_{\delta+it}^*}\left(\un_{\tilde{\LL}}\right)(x_0)$ is integrable in $0$, we deduce that $W^{\mathrm{sym}}(R,u)$ is finite for any positive $u$ (which also implies that it is finite for any $u\in\mathscr{U}$). 
Finally
$$W^{\mathrm{sym}}(R,u)=\dfrac{1}{\pi}\int_{\R}e^{itR}\re{Q_{\delta+it}^*}(\un_{\tLL})(x_0)\widehat{u}(t)\dd t-\dfrac{1}{\pi}\int_{\R}e^{itR}\widehat{u}(t)\dd t.$$
\noindent
By Riemann-Lebesgue's Lemma, the term $\int_{\R}e^{itR}\widehat{u}(t)\dd t$ is negligeable with respect to $\frac{1}{\tilde{L}(R)}$. As previously, let us fix $A>0$. We split
$$\dfrac{1}{\pi}\int_{\R}e^{itR}\re{Q_{\delta+it}^*}(\un_{\tLL})(x_0)\widehat{u}(t)\dd t$$
\noindent
into $I_1+I_2$ where  
$$I_1=\dfrac{1}{\pi}\int_{|t|>\frac{A}{R}}e^{itR}\re{Q_{\delta+it}^*}(\un_{\tLL})(x_0)\widehat{u}(t)\dd t$$
\noindent
and
$$I_2=\dfrac{1}{\pi}\int_{|t|\leqslant\frac{A}{R}}e^{itR}\re{Q_{\delta+it}^*}(\un_{\tLL})(x_0)\widehat{u}(t)\dd t.$$
\noindent
We may decompose $I_1$ according to the sign of $t$. Let $J$ be
$$\dfrac{1}{\pi}\int_{t>\frac{A}{R}}e^{itR}\re{Q_{\delta+it}^*}(\un_{\tLL})(x_0)\widehat{u}(t)\dd t.$$
\noindent
Setting $t=y-\frac{\pi}{R}$ in $J$, it follows
$$J=-\dfrac{1}{\pi}\displaystyle{\int_{y>\frac{A+\pi}{R}}}e^{iyR}\re{Q_{\delta+i\left(y-\frac{\pi}{R}\right)}^*}(\un_{\tLL})(x_0)\widehat{u}
\left(y-\dfrac{\pi}{R}\right)\dd y.$$
\noindent
Hence
\begin{align*}
2J = & \dfrac{1}{\pi}\displaystyle{\int_{\frac{A}{R}}^{\frac{A+\pi}{R}}}e^{itR}\re{Q_{\delta+it}^*}(\un_{\tLL})(x_0)\widehat{u}(t)\dd t\\
& +\dfrac{1}{\pi}\displaystyle{\int_{t>\frac{A+\pi}{R}}}e^{itR}\re{Q_{\delta+i\left(t-\frac{\pi}{R}\right)}^*}(\un_{\tLL})(x_0)\left(\widehat{u}(t)-
\widehat{u}\left(t-\dfrac{\pi}{R}\right)\right)\dd t\\
& +\dfrac{1}{\pi}\displaystyle{\int_{t>\frac{A+\pi}{R}}}e^{itR}\left(\re{Q_{\dit}^*}-\re{Q_{\delta+i\left(t-\frac{\pi}{R}\right)}^*}\right)(\un_{\tLL})(x_0)
\widehat{u}(t)\dd t\\
=: & K_1+K_2+K_3.
\end{align*}
\noindent
As in Proposition \ref{butmixingbeta1utensphi}, one has 
$|K_1|\preceq\frac{1}{\tilde{L}(R)}\frac{L(R)}{\tilde{L}(R)}(A+\pi)^{\frac{3}{4}}$ and
$|K_2|\preceq\frac{1}{\tilde{L}(R)}\frac{L(R)}{\tilde{L}(R)}$. The arguments exposed to control $K_3$ in the proof of Proposition \ref{butmixingbeta1utensphi} give
$$|K_3|\preceq\dfrac{1}{\tilde{L}(R)}\dfrac{1}{\sqrt{A}}.$$
\noindent
Therefore
$\lim\limits_{A\longrightarrow+\infty}\lim\limits_{R\longrightarrow+\infty}\tilde{L}(R)|J|=0$. 
\noindent
Hence
\begin{equation}\label{firstpartorbitalbeta1}
\lim\limits_{A\longrightarrow+\infty}\lim\limits_{R\longrightarrow+\infty}\tilde{L}(R)|I_1|=0.
\end{equation}
\noindent
We rewrite the integral $I_2$ as
\begin{align*}
I_2 = & \dfrac{1}{\pi}\displaystyle{\int_{|t|\leqslant\frac{A}{R}}}e^{itR}\left[\re{Q_{\delta+it}^*}-\re{\left(1-\lambda_{\delta+it}^*\right)^{-1}}\Pi_{\delta}^*\right](\un_{\tLL})(x_0)\widehat{u}(t)\dd t\\
& + \displaystyle{\int_{|t|\leqslant\frac{A}{R}}}e^{itR}\left[\re{\left(1-\lambda_{\delta+it}^*\right)^{-1}}\Pi_{\delta}^*\right](\un_{\tLL})(x_0)\widehat{u}(t)\dd t\\
=: & L_1+L_2.
\end{align*}
\noindent
By Proposition \ref{controlQstar}, for  $R$ large enough, we have $\left|L_1\right|\leqslant 2\dfrac{A}{R}$. The integral $L_2$ may be splitted into $M_1+M_2$ as follows
\begin{align*}
L_2 = & \displaystyle{\int_{|t|\leqslant\frac{A}{R}}}(e^{itR}-1)\left[\re{\left(1-\lambda_{\delta+it}^*\right)^{-1}}\Pi_{\delta}^*\right](\un_{\tLL})(x_0)\widehat{u}(t)\dd t \\
& + \displaystyle{\int_{|t|\leqslant\frac{A}{R}}}\left[\re{\left(1-\lambda_{\delta+it}^*\right)^{-1}}\Pi_{\delta}^*\right](\un_{\tLL})(x_0)\widehat{u}(t)\dd t\\
= & M_1+M_2,
\end{align*}
\noindent
with, as in Proposition \ref{butmixingbeta1utensphi},
$$\tilde{L}(R)M_1\preceq A\dfrac{L(R)}{\tilde{L}(R)}.$$
\noindent
To study the integral $M_2$, it suffices to deal with the case $0\leqslant t\leqslant \frac{A}{R}$.
Denote by
$$N=\displaystyle{\int_{0}^{\frac{A}{R}}}\left[\re{\left(1-\lambda_{\delta+it}^*\right)^{-1}}\Pi_{\delta}^*\right](\un_{\tLL})(x_0)\widehat{u}(t)\dd t.$$
\noindent
We obtain
$$\tilde{L}(R)N\sim\dfrac{h^*(x_0)}{2\Cg}\widehat{u}(0),$$
\noindent
hence
$$\lim\limits_{R\longrightarrow+\infty}\tilde{L}(R)M_2=C_{\G}\widehat{u}(0)$$
\noindent
with $C_{\G}=\dfrac{\lim\limits_{k\longrightarrow+\infty}\sum\limits_{\g\in\G(k)}e^{-\delta\dhy(\oo,\g.\oo)}}{\Cg}$. This concludes the proof of Theorem C.

%

\bibliographystyle{acm}
\bibliography{mabiblio.bib}

\end{document}